\documentclass[a4paper]{article}

\usepackage[utf8]{inputenc} 
\usepackage[T1]{fontenc}
\usepackage{a4}
\usepackage{amsfonts, amssymb, amsmath, amsthm}

\usepackage{enumitem}

\usepackage{mathtools} 
\usepackage{tensor} 
\usepackage{scalerel} 

\usepackage{perpage}
\MakePerPage{footnote}

\usepackage[ngerman, english]{babel}
\useshorthands{"}
\makeatletter
 \defineshorthand[ngerman]{"/}{\mbox{-}\bbl@allowhyphens}
\makeatother
\addto\extrasenglish{\languageshorthands{ngerman}}

\usepackage[x11names]{xcolor}
\usepackage[hyperfootnotes=false, unicode=true, pagebackref, linktocpage=true]{hyperref}
\hypersetup{
	pdfauthor={Seerp Roald Koudenburg},
	pdftitle={A double-dimensional approach to formal category theory},
  linkbordercolor={NavajoWhite1},
  citebordercolor={LightCyan2},
  urlbordercolor={DarkSeaGreen2}
}

\usepackage{tikz} 
\usetikzlibrary{decorations.markings, matrix, arrows, fit, calc}

\tikzstyle{map} = [->, font=\scriptsize]
\tikzstyle{linj} = [left hook->, font=\scriptsize]
\tikzstyle{rinj} = [right hook->, font=\scriptsize]
\tikzstyle{sur} = [->>, font=\scriptsize]
\tikzstyle{cell} = [double,double equal sign distance,-implies, shorten >= 3.75pt, shorten <= 3.75pt, font=\scriptsize]
\tikzstyle{eq} = [double,double equal sign distance]
\tikzstyle{ps} = [shorten >= 2pt]

\tikzstyle{iso} = [above, sloped, inner sep=1.5pt]
\tikzstyle{nat} = [above, sloped, inner sep=2pt]
\tikzstyle{desc} = [fill=white, inner sep=2pt]
\tikzstyle{dots} = [black, font=]

\tikzstyle{small} = [font=\scriptsize]

\tikzstyle{textbaseline} = [baseline=-2.8pt]

\tikzstyle{barred} = [decoration={markings, mark=at position 0.5 with {\draw[-] (0,-1.5pt) -- (0,1.5pt);}}, postaction ={decorate}]

\tikzstyle{math35} = [matrix of math nodes, row sep={3.25em,between origins}, column sep={3.5em,between origins}, text height=1.5ex, text depth=0.25ex, nodes in empty cells]
\tikzstyle{minimath} = [matrix of math nodes, row sep={3em,between origins}, column sep={3.25em,between origins}, font=\scriptsize, text height=1ex, text depth=0.25ex, nodes in empty cells]
\tikzstyle{scheme} = [textbaseline, x=1.6em, y=1.6em, yshift=-2.4em, font=\scriptsize, text depth=0ex, every node/.style={overlay}, execute at end picture = { \useasboundingbox ($(current bounding box.north west) + (0,0.4em)$) rectangle ($(current bounding box.south east) - (0,0.4em)$); }]

\makeatletter
\def\slashedarrowfill@#1#2#3#4#5{%
  $\m@th\thickmuskip0mu\medmuskip\thickmuskip\thinmuskip\thickmuskip
   \relax#5#1\mkern-7mu%
   \cleaders\hbox{$#5\mkern-2mu#2\mkern-2mu$}\hfill
   \mathclap{#3}\mathclap{#2}%
   \cleaders\hbox{$#5\mkern-2mu#2\mkern-2mu$}\hfill
   \mkern-7mu#4$%
}
\def\rightslashedarrowfill@{%
  \slashedarrowfill@\relbar\relbar\mapstochar\rightarrow}
\newcommand\xslashedrightarrow[2][]{%
  \ext@arrow 0055{\rightslashedarrowfill@}{#1}{#2}}
\makeatother

\def\slashedrightarrow{\xslashedrightarrow{}}

\newcommand{\conc}{%
  \mathbin{
    \mathchoice
    {\raisebox{1ex}{\scalebox{.7}{$\frown$}}}
    {\raisebox{1ex}{\scalebox{.7}{$\frown$}}}
    {\raisebox{.7ex}{\scalebox{.5}{$\frown$}}}
    {\raisebox{.7ex}{\scalebox{.5}{$\frown$}}}
  }
}

\DeclareMathOperator*{\Oslash}{\scalerel*{\oslash}{\sum}}

\DeclareMathOperator{\cocart}{cocart}
\DeclareMathOperator{\cart}{cart}

\newtheorem{theorem}{Theorem}[section]
\newtheorem{lemma}[theorem]{Lemma}
\newtheorem{corollary}[theorem]{Corollary}
\newtheorem{proposition}[theorem]{Proposition}

\theoremstyle{definition}
\newtheorem{definition}[theorem]{Definition}

\theoremstyle{remark}
\newtheorem{remark}[theorem]{Remark}
\newtheorem{example}[theorem]{Example}

\providecommand{\corref}[1]{Corollary~\ref{#1}}
\providecommand{\defref}[1]{Definition~\ref{#1}}
\providecommand{\exref}[1]{Example~\ref{#1}}

\providecommand{\lemref}[1]{Lemma~\ref{#1}}
\providecommand{\propref}[1]{Proposition~\ref{#1}}

\providecommand{\thmref}[1]{Theorem~\ref{#1}}

\providecommand{\secref}[1]{Section~\ref{#1}}

\newcommand\defeq{\mathrel{\vcentcolon\Leftrightarrow}}
\renewcommand{\implies}{\Rightarrow}
\renewcommand{\iff}{\Leftrightarrow}

\providecommand{\dfn}{\coloneqq}

\providecommand{\of}{\circ}
\providecommand{\iso}{\cong}
\providecommand{\eq}{\simeq}

\providecommand{\brar}{\slashedrightarrow}
\providecommand{\xrar}[1]{\xrightarrow{#1}}
\providecommand{\xlar}[1]{\xleftarrow{#1}}
\providecommand{\xbrar}[1]{\xslashedrightarrow{#1}}
\providecommand{\Rar}{\Rightarrow}
\providecommand{\xRar}[1]{\xRightarrow{#1}}

\providecommand{\into}{\hookrightarrow}

\providecommand{\eps}{\varepsilon}

\DeclareMathOperator{\dash}{--}

\providecommand{\tens}{\otimes}

\providecommand{\ul}[1]{\underline{#1}{}}
\providecommand{\ull}[1]{\ul{\ul{#1}}{}}

\providecommand{\mf}[1]{\mathfrak{#1}}

\providecommand{\brcs}[1]{\lbrace #1 \rbrace}

\providecommand{\brks}[1]{\lbrack #1 \rbrack}

\providecommand{\pars}[1]{\left(#1\right)}
\providecommand{\bigpars}[1]{\bigl(#1\bigr)}

\providecommand{\lns}[1]{\lvert#1\rvert}

\providecommand{\angles}[1]{\langle#1\rangle}

\providecommand{\gen}[1]{\angles{#1}}

\providecommand{\NN}{\mathbb N}

\providecommand{\set}[1]{\brcs{#1}}

\providecommand{\djunion}{\sqcup}

\providecommand{\Ks}[1]{#1^\sharp} 
\providecommand{\conv}{\to} 
\providecommand{\convord}{\mathord\conv} 
\providecommand{\leqord}{\mathord\leq} 

\DeclareMathOperator{\downset}{\downarrow}

\providecommand{\natarrow}{\Rightarrow}

\providecommand{\map}[3]{#1\colon#2\to#3}

\providecommand{\nat}[3]{#1\colon#2\natarrow#3}
\providecommand{\cell}[3]{#1\colon#2\Rightarrow#3}

\providecommand{\hmap}[3]{#1\colon#2\slashedrightarrow#3}

\providecommand{\inv}[1]{{#1}^{-1}}

\DeclareMathOperator{\id}{id}

\DeclareMathOperator{\ob}{ob}
\DeclareMathOperator{\yon}{y}
\providecommand{\A}{\mathcal A}

\providecommand{\ladj}{\dashv}

\providecommand{\op}[1]{#1^\textup{op}}
\providecommand{\co}[1]{#1^\textup{co}}

\providecommand{\ps}[1]{\widehat{#1}}

\providecommand{\catvar}[1]{\mathcal{#1}}

\providecommand{\2}{\mathsf 2}
\providecommand{\A}{\catvar A}
\providecommand{\B}{\catvar B}

\providecommand{\D}{\catvar D}
\providecommand{\E}{\catvar E}
\providecommand{\J}{\catvar J}
\providecommand{\K}{\catvar K}
\renewcommand{\L}{\catvar L}

\providecommand{\V}{\catvar V}

\providecommand{\Gg}{\mathbb G}

\providecommand{\Ss}{\mathbb S}

\providecommand{\Set}{\mathsf{Set}}
\providecommand{\Cat}{\mathsf{Cat}}
\providecommand{\CAT}{\underline{\Cat}} 
\providecommand{\enCat}[1]{#1\text-\Cat}
\providecommand{\inCat}[1]{\Cat(#1)}
\providecommand{\twoCat}{2\text{-}\Cat}

\providecommand{\MonCat}{\mathsf{MonCat}}
\providecommand{\wkRepModTop}{\mathsf{wkRepModTop}}

\providecommand{\AugVirtDblCat}{\mathsf{AugVirtDblCat}}
\providecommand{\VirtDblCat}{\mathsf{VirtDblCat}}

\providecommand{\Tens}{\bigotimes}

\providecommand{\Rel}{\mathsf{Rel}}
\providecommand{\Span}[1]{\mathsf{Span}(#1)}
\providecommand{\Mat}[1]{#1\text-\mathsf{Mat}}
\providecommand{\Mod}{\mathsf{Mod}}
\providecommand{\ModRel}{\mathsf{ModRel}}
\providecommand{\Prof}{\mathsf{Prof}}
\providecommand{\enProf}[1]{#1\text-\Prof}
\providecommand{\inProf}[1]{\Prof(#1)}

\providecommand{\Rep}{\mathsf{Rep}}
\providecommand{\opRep}{\mathsf{opRep}}

\providecommand{\Alg}[1]{#1\text-\mathsf{Alg}}
\providecommand{\wAlg}[3]{#3\text-\mathsf{Alg}_{(#1,\, #2)}}
\providecommand{\lbcwAlg}[2]{#2\text-\mathsf{Alg}_{(#1,\, \textup{ps},\, \textup{lbc})}}
\providecommand{\clx}{\textup c}
\providecommand{\lax}{\textup l}
\providecommand{\psd}{\textup{ps}}

\providecommand{\Mnd}{\mathsf{Mnd}}

\providecommand{\hc}{\odot}

\providecommand{\tab}[1]{\gen{#1}}

\providecommand{\cur}[1]{#1^{\scriptscriptstyle\lambda}}

\DeclareMathOperator{\fc}{\mathsf{fc}}
\DeclareMathOperator{\bl}{\mathsf{bl}}

\providecommand{\edpi}[3]{#1_{(#2,\, #3)}}
\providecommand{\edpilbc}[2]{#1_{(#2,\, \psd,\, \textup{lbc})}}

\begin{document}
	\title{A double-dimensional approach to formal category theory}
	\author{Seerp Roald Koudenburg\thanks{Part of this paper was written during a visit to Macquarie University, where its main results formed the subject of a series of talks at the Australian Category Seminar; I am grateful to the Macquarie University Research Centre for its funding of my visit. I would like to thank Richard Garner, Mark Weber and Ram\'on Abud Alcal\'a for helpful discussions, and Robert Par\'e for suggesting the term ``augmented virtual double category''.}}
	\date{Draft, October 2019\\[0.5cm]\framebox(320,160){\begin{minipage}{25em}
		\centering{Notes}
		\normalsize\begin{itemize}
			\item The recent papers ``Augmented virtual double categories'' (\href{http://arxiv.org/abs/1910.11189}{\texttt{arxiv:1910.11189}}) and ``Formal category theory in augmented virtual double categories'' (\href{https://arxiv.org/abs/2205.04890}{\texttt{arxiv:2205.04890}}) contain streamlined, corrected and expanded versions of respectively Sections~\ref{augmented virtual double category section}, \ref{restriction section} and \ref{composition section} and Sections~\ref{Kan extension section} and \ref{yoneda embeddings section} below; I encourage readers of these sections to read the new papers instead.
			\item The first version of this draft (November 2015) used the term “hypervirtual double category” for its main notion, where presently “augmented virtual double category” is used instead.
		\end{itemize}
	\end{minipage}}}
	\maketitle
	\begin{abstract}
		Whereas formal category theory is classically considered within a $2$"/category, in this paper a double-dimensional approach is taken. More precisely we develop such theory within the setting of augmented virtual double categories, a notion extending that of virtual double category (also known as $\fc$-multicategory) by adding cells with nullary target.

This paper starts by introducing the notion of augmented virtual double category, followed by describing its basic theory and its relation to other types of double category. After this the notion of `weak' Kan extension within an augmented virtual double category is considered, together with three strengthenings. The first of these generalises Borceux-Kelly's notion of Kan extension along enriched functors, the second one generalises Street's notion of pointwise Kan extension in $2$-categories, and the third is a combination of the other two; these stronger notions are compared. The notion of yoneda embedding is then considered in an augmented virtual double category, and compared to that of a good yoneda structure on a $2$-category; the latter in the sense of Street-Walters and Weber. Conditions are given ensuring that a yoneda embedding $\map\yon A{\ps A}$ defines $\ps A$ as the free small cocompletion of $A$, in a suitable sense.

In the second half we consider formal category theory in the presence of algebraic structures. In detail: to a monad $T$ on an augmented virtual double category $\K$ several augmented virtual double categories $T\text-\mathsf{Alg}_{(v, w)}$ of $T$"/algebras are associated, one for each pair of types of weak coherence satisfied by the $T$"/algebras and their morphisms respectively. This is followed by the study of the creation of, amongst others, left Kan extensions by the forgetful functors $T\text-\mathsf{Alg}_{(v, w)} \to \K$. The main motivation of this paper is the description of conditions ensuring that yoneda embeddings in $\K$ lift along these forgetful functors, as well as ensuring that such lifted algebraic yoneda embeddings again define free small cocompletions, now in $T\text-\mathsf{Alg}_{(v, w)}$. As a first example we apply the previous to monoidal structures on categories, hence recovering Day convolution of presheaves and Im-Kelly's result on free monoidal cocompletion, as well as obtaining a ``monoidal Yoneda lemma''.
	\end{abstract}
	
	\subsection*{Motivation}\addcontentsline{toc}{subsection}{Motivation}
	Central to classical category theory is the Yoneda lemma which, for a locally small category $A$, describes the position of the representable presheaves $A(\dash, x)$ within the category $\ps A$ of all presheaves $\op A \to \Set$. More precisely, in its ``parametrised'' form, the Yoneda lemma states that the yoneda embedding \mbox{$\map\yon A{\ps A}\colon x \mapsto A(\dash, x)$} satisfies the following property: any profunctor $\hmap JAB$---that is a functor $\map J{\op A \times B}\Set$---induces a functor $\map{\cur J}B{\ps A}$ equipped with bijections
	\begin{displaymath}
		\ps A(\yon x, \cur Jy) \iso J(x, y),
	\end{displaymath}
	that combine to form an isomorphism $\ps A(\yon\dash, \cur J\dash) \iso J$ of profunctors. Indeed, $\cur J$ can be defined as $\cur Jy \dfn J(\dash, y)$.
	
	The main observation motivating this paper is that, given a monoidal structure $\tens$ on $A$, the above ``lifts'' to a ``monoidal Yoneda lemma'' as follows. Recall that a monoidal structure on $A$ induces such a structure $\ps\tens$ on $\ps A$, that is given by Day convolution \cite{Day70}
	\begin{displaymath}
		(p \mathbin{\ps\tens} q)(x) \dfn \int^{u, v \in A} A(x, u \tens v) \times pu \times qv, \qquad\qquad\qquad \text{where $p, q \in \ps A$},
	\end{displaymath}
	and with respect to which $\map\yon A{\ps A}$ forms a pseudomonoidal functor---that is, it comes equipped with coherent isomorphisms $\bar\yon \colon \yon x \mathbin{\ps\tens} \yon y \iso \yon(x \tens y)$. Thus a pseudomonoidal functor, $\map{(\yon, \bar\yon)}{(A, \tens)}{(\ps A, \ps\tens)}$ satisfies the following monoidal variant of the property above: any lax monoidal profunctor $\hmap JAB$---equipped with coherent morphisms $\map{\bar J}{J(x_1, y_1) \times J(x_2, y_2)}{J(x_1 \tens x_2, y_1 \tens y_2)}$---induces a lax monoidal functor $\map{\cur J}B{\ps A}$ equipped with an isomorphism of lax monoidal profunctors $\ps A(\yon\dash, \cur J\dash) \iso J$. In detail: we can take $\cur J$ to be as defined before, equipped with coherence morphisms $\nat{\bar{\cur J}}{\cur J y_1 \mathbin{\ps\tens} \cur J y_2}{\cur J(y_1 \tens y_2)}$ that are induced by the composites
	\begin{displaymath}
		A(x, u \tens v) \times J(u, y_1) \times J(v, y_2) \xrar{\id \times \bar J} A(x, u \tens v) \times J(u \tens v, y_1 \tens y_2) \to J(x, y_1 \tens y_2),
	\end{displaymath}
	where the unlabelled morphism is induced by the functoriality of $J$ in $A$.
	
	The principal aim of this paper is to formalise the way in which the classical Yoneda lemma in the presence of a monoidal structure leads to the monoidal Yoneda lemma, as described above; thus allowing us to
	\begin{enumerate}[label = (\alph*)]
		\item generalise the above to other types of algebra, such as double categories and representable topological spaces;
		\item characterise lax monoidal profunctors $J$ whose induced lax monoidal functors $\cur J$ are pseudomonoidal, and likewise for other types of algebraic morphism;
		\item recover classical results in the setting of monoidal categories, such as $(\yon, \bar\yon)$ exhibiting $(\ps A, \ps\tens)$ as the ``free monoidal cocompletion'' of $(A, \tens)$ (see \cite{Im-Kelly86}), and generalise these to other types of algebra.
	\end{enumerate}
	
	\subsection*{Augmented virtual double categories}\addcontentsline{toc}{subsection}{Augmented virtual double categories}
	While the classical Yoneda lemma has been formalised in the setting of $2$"/categories, in the sense of the yoneda structures of Street and Walters \cite{Street-Walters78} (see also \cite{Weber07}), it is unclear how the monoidal Yoneda lemma, as described above, fits into this framework. Indeed the problem is that it concerns two types of morphism between monoidal categories, both lax monoidal functors and lax monoidal profunctors, and it is not clear how to arrange both types into a $2$"/category equipped with a yoneda structure. In particular lax monoidal structures on representable profunctors $C(f\dash, \dash)$, where $\map fAC$ is a functor, correspond to colax monoidal structures on $f$ (see \thmref{representable horizontal T-morphisms} below for the generalisation of this to other types of algebra).
	
	The natural setting in which to consider two types of morphism is that of a (pseudo-)double category, where one type is regarded as being horizontal and the other as being vertical, while one considers square-shaped cells
	\begin{displaymath}
		\begin{tikzpicture}
			\matrix(m)[math35]{A & B \\ C & D \\};
			\path[map]	(m-1-1) edge[barred] node[above] {$J$} (m-1-2)
													edge node[left] {$f$} (m-2-1)
									(m-1-2) edge node[right] {$g$} (m-2-2)
									(m-2-1) edge[barred] node[below] {$K$} (m-2-2);
			\path[transform canvas={xshift=1.75em}]	(m-1-1) edge[cell] node[right] {$\phi$} (m-2-1);
		\end{tikzpicture}
	\end{displaymath}
	between them. Double categories however form a structure that is too strong for some of the types of algebra that we would like to consider, as it requires compositions for both types of morphism while, for example, the composite of ``double profunctors'' $\hmap JAB$ and $\hmap HBE$, between double categories $A$, $B$ and $E$, does in general not exist. While the latter is a consequence of, roughly speaking, the algebraic structures of $J$ and $H$ being incompatible in general, composites of profunctor-like morphisms may also fail to exist because of size issues. Indeed, recall from \cite{Freyd-Street95} that, for a locally small category $A$, the category of presheaves $\ps A$ need not be locally small. Consequently in describing the classical Yoneda lemma for locally small categories it is, on one hand, necessary to consider categories $A, B, \dotsc$ that might have large hom-sets while, on the other hand, the property satisfied by the yoneda embedding is stated in terms of `small' profunctors---that is small-set-valued functors $\map J{\op A \times B}\Set$---between such categories; in general, such profunctors do not compose either. This classical situation is typical: for a description of most of the variations of the Yoneda lemma given in this paper, one considers a double-dimensional setting whose objects $A, B, \dotsc$ are ``large in size'' while the size of the horizontal morphisms $\hmap JAB$ between them is ``small''.
	
	In view of the previous we, instead of double categories, consider a weaker notion as the right notion for our double-dimensional approach to formal category theory. This weaker notion is extends slightly that of `virtual double category' \cite{Cruttwell-Shulman10} which, analogous to the generalisation of monoidal category to `multicategory', does not require a horizontal composition but, instead of the square-shaped cells above, has cells of the form below, with a sequence of horizontal morphisms as horizontal source. Introduced by Burroni \cite{Burroni71} under the name `$\fc$-cat\'egorie', where $\fc$ denotes the `free category'-monad, virtual double categories have also been called `$\fc$-multicategories' \cite{Leinster04}.
	\begin{displaymath}
		\begin{tikzpicture}
			\matrix(m)[math35]{A_0 & A_1 & A_{n'} & A_n \\ C & & & D \\};
			\path[map]	(m-1-1) edge[barred] node[above] {$J_1$} (m-1-2)
													edge node[left] {$f$} (m-2-1)
									(m-1-3) edge[barred] node[above] {$J_n$} (m-1-4)
									(m-1-4) edge node[right] {$g$} (m-2-4)
									(m-2-1) edge[barred] node[below] {$K$} (m-2-4);
			\path[transform canvas={xshift=1.75em}]	(m-1-2) edge[cell] node[right] {$\phi$} (m-2-2);
			\draw				($(m-1-2)!0.5!(m-1-3)$) node {$\dotsb$};
		\end{tikzpicture}
	\end{displaymath}
	Both settings mentioned above can be considered as a virtual double category: there is a virtual double category with (possibly large) categories as objects, functors as vertical morphisms and small profunctors as horizontal morphisms, and likewise one with (possibly large) double categories, `double functors' and `small' double profunctors. These virtual double categories however miss an ingredient crucial to the theory of (double) categories: they do not contain transformations $f \Rar g$ between (double) functors $f$ and $\map gAC$ into (double) categories $C$ with large hom-sets. Indeed, for such transformations to be represented by cells of the form above we need the small (double) profunctor $K$ to consist of the hom-sets of $C$, which is not possible if some of them are large.
	
	In order to remove the inadequacy described above, we extend to the notion of virtual double category to also contain cells with empty horizontal target, as shown below, so that transformations $f \Rar g$, as described above, are represented by cells with both empty horizontal source and target. We will call this extended notion `augmented virtual double category'; a detailed definition is given in the first section below.
	\begin{displaymath}
		\begin{tikzpicture}
			\matrix(m)[math35, column sep={1.75em,between origins}]
				{A_0 & & A_1 & \dotsb & A_{n'} & & A_n \\ & & & C & & & \\};
			\path[map]	(m-1-1) edge[barred] node[above] {$J_1$} (m-1-3)
													edge node[below left] {$f$} (m-2-4)
									(m-1-5) edge[barred] node[above] {$J_n$} (m-1-7)
									(m-1-7) edge node[below right] {$g$} (m-2-4);
			\path				(m-1-4) edge[cell] node[right] {$\phi$} (m-2-4);
		\end{tikzpicture}
	\end{displaymath}
	
	\subsection*{Overview}\addcontentsline{toc}{subsection}{Overview}
	We start, in \secref{augmented virtual double category section}, by introducing the notion of augmented virtual double category. Every augmented virtual double category $\K$ contains both a virtual double category $U(\K)$, consisting of the cells in $\K$ that have nonempty horizontal target, as well as a `vertical 2-category' $V(\K)$, consisting of cells with both horizontal target and source empty. We obtain examples of augmented virtual double categories by considering monoids and bimodules in virtual double categories, just like one does in monoidal categories, followed by restricting the size the bimodules. The archetypal example of an augmented virtual double category, that of small profunctors between large categories, is thus obtained by considering monoids and bimodules in the pseudo double category of spans of large sets, followed by restricting to small-set-valued profunctors. Given a `universe enlargement' $\V \subset \V'$ of monoidal categories, we consider the enriched variant of the previous, resulting in the augmented virtual double category of $\V$-enriched profunctors between $\V'$-enriched categories. The description of the $2$"/category of augmented virtual double categories---in particular its equivalences---closes the first section.
	
	In \secref{restriction section} and \secref{composition section} the basic theory of augmented virtual double categories is introduced which, for a large part, consists of a straightforward generalisation of that for virtual double categories. Invaluable to both theories are the notions of restriction and composition of horizontal morphisms: these generalise respectively the restriction $K \of (\op f \times g)$ of a profunctor $K$, along functors $f$ and $g$, as well as the composition of profunctors. As a special cases of restriction, the notions of companion and conjoint generalise that of the (op-)representable profunctors $C(f\dash, \dash)$ and $C(\dash, f\dash)$ induced by a functor $\map fAC$, and at the end of \secref{restriction section} we characterise the locally full sub-augmented virtual double category of $\K$, obtained by restricting to representable horizontal morphisms, in terms of its vertical $2$"/category $V(\K)$. As part of the theory of composition of horizontal morphisms in \secref{composition section} horizontal units are considered: the small unit profunctor of a category $A$ exists, for example, precisely if $A$ has locally small hom-sets. Finishing the third section is a theorem proving, in the presence of all horizontal units, the equivalence of the notions of augmented virtual double category and virtual double category.
	
	Having introduced the basic theory of augmented virtual double categories we begin studying `formal category theory' within such double categories. We start \secref{Kan extension section} by rewriting the classical notion of left Kan extension in the vertical $2$-category $V(\K)$, in terms of the companions in $\K$, leading to a notion of `weak' left Kan extension in $\K$. We then consider three strengthenings of this notion: the first of these can be thought of as generalising the notion of `Kan extension along enriched functors', in the sense of e.g.\ \cite{Kelly82}; the second as generalising the notion of `pointwise Kan extension in a $2$-category', in the sense of Street \cite{Street74b}; while the third is a combination of the previous two. Besides studying the basic theory of these notions we will compare them among each other, as well as make precise their relation to the aforementioned classical notions.
	
	Being one of the main goals of this paper, \secref{yoneda embeddings section} introduces the notion of yoneda embedding in an augmented virtual double category as a vertical morphism \mbox{$\map\yon A{\ps A}$} satisfying two axioms. As is the case for the classical yoneda embedding, the first of these asks $\yon$ to be dense, while the second is the `yoneda axiom': this formally captures the fact that, in the classical setting, any small profunctor \mbox{$\map JAB$} induces a functor $\map{\cur J}B{\ps A}$, as was described above as part of the motivation. These two conditions are closely related to the axioms satisfied by morphisms that make up a `good yoneda structure' on a $2$-category, in the sense of Weber \cite{Weber07}. Slightly strengthening the original notion of `yoneda structure', introduced by Street and Walters in \cite{Street-Walters78}, a good yoneda structure consists of a collection of yoneda embeddings that satisfy a `yoneda axiom' with respect to a specified collection of `admissible' morphisms (informally these are to be thought of as ``small in size''). In contrast, our position of regarding all horizontal morphisms as being small enables us to consider just a single yoneda embedding. We make precise the relation between the yoneda embeddings in an augmented virtual double category $\K$ and the existence of a good yoneda structure on the vertical $2$-category that it $V(\K)$ contains. Given a yoneda embedding $\map\yon A{\ps A}$, the main result of \secref{yoneda embeddings section} gives conditions ensuring that it defines $\ps A$ as the `free small cocompletion' of $A$, in a suitable sense.
	
	In \secref{algebras section} we consider formal category theory within augmented virtual double categories in the presence of `algebraic structures'; the archetypal example being that of monoidal structures on categories. Like in $2$"/dimensional category theory, algebraic structures in an augmented virtual double category $\K$ are defined by monads on $\K$; in fact, any monad $T$ on $\K$ induces a strict $2$"/monad $V(T)$ on the vertical $2$-category $V(\K)$. To each monad $T$ we will associate several augmented virtual double categories $\wAlg vwT$ of weak algebras of $T$, one for each pair of weaknesses \mbox{$v, w \in \set{\textup{pseudo}, \textup{lax}, \textup{colax}}$} specifying the type of coherence satisfied by the algebra structures and algebraic morphisms. The vertical parts $V(\wAlg vwT)$ of these augmented virtual double categories are defined to coincide with the $2$-categories $\wAlg vw{V(T)}$ of weak $V(T)$"/algebras in the classical sense, while their notion of horizontal morphism generalises that of  `horizontal $T$-morphism' introduced by Grandis and Par\'e in the setting of pseudo double categories \cite{Grandis-Pare04}. The theorem closing this section characterises (op-)representable horizontal $T$-morphisms.
	
	The remaining sections, \secref{creativity section} and \secref{algebraic yoneda embeddings section}, are devoted to describing the creativity of the forgetful functors $\map U{\wAlg vwT}\K$. Generalising the main theorem of \cite{Koudenburg15a} to the setting of augmented virtual double categories, the main result of \secref{creativity section} describes the creation of algebraic left Kan extensions by $U$; it can be regarded as extending Kelly's classical result on `doctrinal adjunctions' \cite{Kelly74} to left Kan extensions. As the main result of this paper, in \secref{algebraic yoneda embeddings section} we describe the lifting of algebraic yoneda embeddings along $U$, followed by making precise its consequences, as listed at the end of the motivation above. We treat in detail the lifting of a monoidal Yoneda embedding $\map{(\yon, \bar\yon)}{(A, \tens)}{(\ps A, \ps\tens)}$, as described in the motivation, and recover Im and Kelly's result \cite{Im-Kelly86} showing that $(\yon, \bar\yon)$ defines $(\ps A, \ps\tens)$ as the `free monoidal cocompletion' of $(A, \tens)$.
	
	\tableofcontents	
	
	\section{Augmented virtual double categories}\label{augmented virtual double category section}
	
	\subsection{Definition of augmented virtual double category}
	We start by introducing the notion of augmented virtual double category. In doing so we use the well-known notion of a \emph{directed graph}, by which we mean a parallel pair of functions $A = \bigl(\mspace{-11mu}\begin{tikzpicture}[textbaseline]
		\matrix(m)[math35, column sep=1em]{A_1 & A_0 \\};
		\path[map]	(m-1-1) edge[above, transform canvas={yshift=2pt}] node {$s$} (m-1-2)
												edge[below, transform canvas={yshift=-2pt}] node {$t$} (m-1-2);
	\end{tikzpicture}\mspace{-11mu}\bigr)$ from a class $A_1$ of \emph{edges} to a class $A_0$ of \emph{vertices}. An edge $e$ with $(s,t)(e) = (x,y)$ is denoted $x \xrar e y$; the vertices $x$ and $y$ are called its \emph{source} and \emph{target}. Remember that any graph $A$ generates a \emph{free category} $\fc A$, whose underlying graph has the same vertices as $A$ while its edges $x \to y$ are (possibly empty) paths \mbox{$\ul e = (x = x_0 \xrar{e_1} x_1 \xrar{e_2} \dotsb \xrar{e_n} x_n = y)$} of edges in $A$; we write $\lns{\ul e} \dfn n$ for their lengths. Composition in $\fc A$ is given by concatenation $(\ul e, \ul f) \mapsto \ul e \conc \ul f$ of paths, while the empty path $(x)$ forms the identity at $x \in A_0$. We denote by $\bl A \subseteq \fc A$ the subgraph consisting of all paths of length $\leq 1$, which we think of as obtained from $A$ by ``freely adjoining base loops''.
	
	Given an integer $n \geq 1$ we write $n' \dfn n - 1$.
	\begin{definition} \label{augmented virtual double category}
		An \emph{augmented virtual double category} $\K$ consists of
		\begin{itemize}[label=-]
			\item a class $\K_0$ of \emph{objects} $A$, $B, \dotsc$
			\item a category $\K_\textup v$ with $\K_{\textup v0} = \K_0$, whose morphisms $\map fAC$, $\map gBD, \dotsc$ are called \emph{vertical morphisms};
			\item a directed graph $\K_\textup h$ with $\K_{\textup h0} = \K_0$, whose edges are called \emph{horizontal morphisms} and denoted by slashed arrows $\hmap JAB$, $\hmap KCD, \dotsc$;
			\item a class of \emph{cells} $\phi$, $\psi, \dotsc$ that are of the form
				\begin{equation} \label{cell}

				\end{equation}
				where $\ul J$ and $\ul K$ are paths in $\K_\textup h$ with $\lns{\ul K} \leq 1$;
			\item	for any path of cells
				\begin{equation} \label{composable sequence}
					%
				\end{equation}
				of length $n \geq 1$ and a cell $\psi$ as on the left below, a \emph{vertical composite} as on the right;
				\begin{equation} \label{vertical composite}
					%
				\end{equation}
			\item	\emph{horizontal identity cells} as on the left below, one for each $\hmap JAB$;
				\begin{displaymath}
					%
				\end{displaymath}
			\item \emph{vertical identity cells} as on the right above, one for each $\map fAC$, that are preserved by vertical composition: $\id_h \of (\id_f) = \id_{h \of f}$; we write $\id_A \dfn \id_{\id_A}$.
		\end{itemize}
		The composition above is required to satisfy the \emph{associativity axiom}
		\begin{multline*}
			\chi \of \bigpars{\psi_1 \of (\phi_{11}, \dotsc, \phi_{1m_1}), \dotsc, \psi_n \of (\phi_{n1}, \dotsc, \phi_{nm_n})} \\
			= \bigpars{\chi \of (\psi_1, \dotsc, \psi_n)} \of (\phi_{11}, \dotsc, \phi_{nm_n}),
		\end{multline*}
		whenever the left-hand side makes sense, as well as the \emph{unit axioms}
		\begin{flalign*}
			&& \id_C \of (\phi) = \phi, \quad \id_K \of (\phi) = \phi, \quad \phi \of (\id_A) &= \phi, \quad \phi \of (\id_{J_1}, \dotsc, \id_{J_n}) = \phi & \\
			\text{and} && \psi \of (\phi_1, \dotsc, \phi_i, \id_{f_i}, \phi_{i+1}, \dotsc, \phi_n) &= \psi \of (\phi_1, \dotsc, \phi_n) &
		\end{flalign*}
		whenever these make sense and where, in the last axiom, $0 \leq i \leq n$.
	\end{definition}
	
	For a cell $\phi$ as in \eqref{cell} above we call the vertical morphisms $f$ and $g$ its \emph{vertical source} and \emph{target} respectively, and call the path of horizontal morphisms $\ul J = (J_1, \dotsc, J_n)$ its \emph{horizontal source}, while we call $\ul K$ its \emph{horizontal target}. We write $\lns \phi \dfn (\lns{\ul J}, \lns{\ul K})$ for the \emph{arity} of $\phi$. A $(n,1)$"/ary cell will be called \emph{unary}, $(n,0)$-ary cells \emph{nullary} and $(0,0)$-ary cells \emph{vertical}.
	
	When writing down paths $(J_1, \dotsc, J_n)$ of length $n \leq 1$ we will often leave out parentheses and simply write $A_0 \dfn (A_0)$ or $J_1 \dfn (A_0 \xbrar{J_1} A_1)$. Likewise in the composition of cells: $\psi \of \phi_1 \dfn \psi \of (\phi_1)$. We will often denote unary cells simply by $\cell\phi{(J_0, \dotsc, J_n)}K$, and nullary cells by $\cell\psi{(J_0, \dotsc, J_n)}C$, leaving out their vertical source and target. When drawing compositions of cells it is often helpful to depict them in full detail and, in the case of nullary cells, draw their horizontal target as a single object, as shown below.
	\begin{displaymath}
		\begin{tikzpicture}[baseline]
			\matrix(m)[math35]{A_0 \\ C \\};
			\path[map]	(m-1-1) edge[bend right=45] node[left] {$f$} (m-2-1)
													edge[bend left=45] node[right] {$g$} (m-2-1);
			\path				(m-1-1) edge[cell] node[right] {$\psi$} (m-2-1);
		\end{tikzpicture} \mspace{6mu} \begin{tikzpicture}[baseline]
			\matrix(m)[math35, column sep={1.625em,between origins}]
				{A_0 & & A_1 & \dotsb & A_{n'} & & A_n \\ & & & C & & & \\};
			\path[map]	(m-1-1) edge[barred] node[above] {$J_1$} (m-1-3)
													edge node[below left] {$f$} (m-2-4)
									(m-1-5) edge[barred] node[above] {$J_n$} (m-1-7)
									(m-1-7) edge node[below right] {$g$} (m-2-4);
			\path				(m-1-4) edge[cell] node[right] {$\psi$} (m-2-4);
		\end{tikzpicture}	\mspace{3mu} \begin{tikzpicture}[baseline]
			\matrix(m)[math35, column sep={1.75em,between origins}]{& A_0 & \\ C & & D \\};
			\path[map]	(m-1-2) edge node[left] {$f$} (m-2-1)
													edge node[right] {$g$} (m-2-3)
									(m-2-1) edge[barred] node[below] {$K$} (m-2-3);
			\path				(m-1-2) edge[cell, transform canvas={yshift=-0.25em}] node[right, inner sep=2.5pt] {$\phi$} (m-2-2);
		\end{tikzpicture} \mspace{15mu} \begin{tikzpicture}[baseline]
			\matrix(m)[math35, column sep={3.25em,between origins}]{A_0 & A_1 & A_{n'} & A_n \\ C & & & D \\};
			\path[map]	(m-1-1) edge[barred] node[above] {$J_1$} (m-1-2)
													edge node[left] {$f$} (m-2-1)
									(m-1-3) edge[barred] node[above] {$J_n$} (m-1-4)
									(m-1-4) edge node[right] {$g$} (m-2-4)
									(m-2-1) edge[barred] node[below] {$K$} (m-2-4);
			\path[transform canvas={xshift=1.625em}]	(m-1-2) edge[cell] node[right] {$\phi$} (m-2-2);
			\draw				($(m-1-2)!0.5!(m-1-3)$) node {$\dotsb$};
		\end{tikzpicture}
	\end{displaymath}
	
	A cell with identities as vertical source and target is called \emph{horizontal}. A horizontal cell $\cell\phi JK$ with unary horizontal source is called \emph{invertible} if there exists a horizontal cell $\cell\psi KJ$ such that $\phi \of \psi = \id_K$ and $\psi \of \phi = \id_J$; in that case we write $\inv\phi \dfn \psi$. When drawing diagrams we shall often depict identity morphisms by the equal sign ($=$), while we leave identity cells empty. Also, because composition of cells is associative, we will usually leave out bracketings when writing down composites.
	
	For convenience we use the `whisker' notation from $2$-category theory and define
	\begin{displaymath}
		h \of (\phi_1, \dotsc, \phi_n) \dfn \id_h \of (\phi_1, \dotsc, \phi_n) \qquad \text{and} \qquad \psi \of f \dfn \psi \of \id_f,
	\end{displaymath}
	whenever the right-hand side makes sense. Moreover, for any path
	\begin{displaymath}
		\begin{tikzpicture}
			\matrix(m)[math35]{A_0 & A_1 & A_{n'} & A_n & B_1 & B_{m'} & B_m \\ C & & & D & & & G \\};
			\path[map]	(m-1-1) edge[barred] node[above] {$J_1$} (m-1-2)
													edge node[left] {$f$} (m-2-1)
									(m-1-3) edge[barred] node[above] {$J_n$} (m-1-4)
									(m-1-4) edge[barred] node[above] {$H_1$} (m-1-5)
													edge node[right] {$g$} (m-2-4)
									(m-1-6) edge[barred] node[above] {$H_m$} (m-1-7)
									(m-1-7) edge node[right] {$h$} (m-2-7)
									(m-2-1) edge[barred] node[below] {$\ul K$} (m-2-4)
									(m-2-4) edge[barred] node[below] {$\ul L$} (m-2-7);
			\path[transform canvas={xshift=1.75em}]	(m-1-2) edge[cell] node[right] {$\phi$} (m-2-2)
									(m-1-5) edge[cell] node[right] {$\psi$} (m-2-5);
			\draw				($(m-1-2)!0.5!(m-1-3)$) node {$\dotsb$}
									($(m-1-5)!0.5!(m-1-6)$) node {$\dotsb$};
		\end{tikzpicture}
	\end{displaymath}
	with $\lns{\ul K} + \lns{\ul L} \leq 1$ we define the \emph{horizontal composite} $\cell{\phi \hc \psi}{\ul J \conc \ul H}{\ul K \conc \ul L}$ by
	\begin{displaymath}
		\phi \hc \psi \dfn \id_{\ul K \conc \ul L} \of (\phi, \psi),
	\end{displaymath}
	where $\id_{\ul K \conc \ul L}$ is to be interpreted as the identity $\map{\id_C}CC$ in case $\ul K \conc \ul L = (C)$. The following lemma follows easily from the associativity of the vertical composition in $\K$.
	\begin{lemma} \label{horizontal composition}
		Horizontal composition $(\phi, \psi) \mapsto \phi \hc \psi$, as defined above, satisfies the \emph{associativity} and \emph{unit axioms}
		\begin{displaymath}
			(\phi \hc \psi) \hc \chi = \phi \hc (\psi \hc \chi), \qquad (\id_f \hc \phi) = \phi \qquad \text{and} \qquad (\phi \hc \id_g) = \phi
		\end{displaymath}
		whenever these make sense. Moreover, horizontal and vertical composition satisfy the \emph{interchange axioms}
		\begin{flalign*}
			&& \bigpars{\psi \of (\phi_1, \dotsc, \phi_n)} \hc \bigl(\chi \of (\xi_1, \dotsc, \xi_m)\bigr ) &= (\psi \hc \chi) \of (\phi_1, \dotsc, \phi_n, \xi_1, \dotsc, \xi_m) & \\
			\text{and} && \psi \of \bigpars{\phi_1, \dotsc, (\phi_{i'} \hc \phi_{i}), \dotsc, \phi_n} &= \psi \of (\phi_1, \dotsc, \phi_{i'}, \phi_i, \dotsc, \phi_n) &
		\end{flalign*}
		whenever they make sense.
	\end{lemma}
	Notice that by removing the nullary cells from \defref{augmented virtual double category} (including the vertical identity cells) we recover the classical notion of \emph{virtual double category}, in the sense of \cite{Cruttwell-Shulman10} or Section 5.1 of \cite{Leinster04}, where it is called $\fc$-multicategory. Virtual double categories where originally introduced by Burroni \cite{Burroni71}, who called them `$\fc$"/cat\'egories'. Likewise if instead we remove all horizontal morphisms, so that the only remaining cells are the vertical ones then, using both compositions $\of$ and $\hc$, we recover the notion of \emph{$2$-category}. We conclude that every augmented virtual double category $\K$ contains a virtual double category $U(\K)$ consisting of its objects, vertical and horizontal morphisms, as well as unary cells. Likewise $\K$ contains a \emph{vertical $2$-category} $V(\K)$, consisting of its objects, vertical morphisms and vertical cells.
	
	Every augmented virtual double category has a horizontal dual as follows.
	\begin{definition} \label{horizontal dual}
		Let $\K$ be an augmented virtual double category. The \emph{horizontal dual} of $\K$ is the augmented virtual double category $\co\K$ that has the same objects and vertical morphisms, that has a horizontal morphism $\hmap{\co J}AB$ for each $\hmap JBA$ in $\K$, and a cell $\co\phi$ as on the left below for each cell $\phi$ in $\K$ as on the right.
		\begin{displaymath}
			\begin{tikzpicture}[baseline]
				\matrix(m)[math35, column sep={3.25em,between origins}]{A_0 & A_1 & A_{n'} & A_n \\ C & & & D \\};
				\path[map]	(m-1-1) edge[barred] node[above] {$\co J_1$} (m-1-2)
														edge node[left] {$f$} (m-2-1)
										(m-1-3) edge[barred] node[above] {$\co J_n$} (m-1-4)
										(m-1-4) edge node[right] {$g$} (m-2-4)
										(m-2-1) edge[barred] node[below] {$\co{\ul K}$} (m-2-4);
				\path[transform canvas={xshift=1.625em}]	(m-1-2) edge[cell] node[right] {$\co\phi$} (m-2-2);
				\draw				($(m-1-2)!0.5!(m-1-3)$) node {$\dotsb$};
			\end{tikzpicture} \qquad\qquad\qquad \begin{tikzpicture}[baseline]
				\matrix(m)[math35, column sep={3.25em,between origins}]{A_n & A_{n'} & A_1 & A_0 \\ D & & & C \\};
				\path[map]	(m-1-1) edge[barred] node[above] {$J_n$} (m-1-2)
														edge node[left] {$g$} (m-2-1)
										(m-1-3) edge[barred] node[above] {$J_1$} (m-1-4)
										(m-1-4) edge node[right] {$f$} (m-2-4)
										(m-2-1) edge[barred] node[below] {$\ul K$} (m-2-4);
				\path[transform canvas={xshift=1.625em}]	(m-1-2) edge[cell] node[right] {$\phi$} (m-2-2);
				\draw				($(m-1-2)!0.5!(m-1-3)$) node {$\dotsb$};
			\end{tikzpicture}
		\end{displaymath}
		Identities and compositions in $\co\K$ are induced by those of $\K$:
		\begin{displaymath}
			\id_{\co J} \dfn \co{(\id_J)}, \quad \id_f \dfn \co{(\id_f)} \quad \text{and} \quad \co\psi \of (\co\phi_1, \dotsc, \co\phi_n) \dfn \co{\bigpars{\psi \of (\phi_n, \dotsc, \phi_1)}}.
		\end{displaymath}
	\end{definition}
	
	We end this subsection with a remark on the associativity of composition of cells in augmented virtual double categories.
	\begin{remark}
	  Consider a configuration of six composable cells described by the scheme below, with the cell $\phi_2$ nullary and the others unary. Notice that there are two ways of composing these cells if we start with composing the top two rows first: the cell $\phi_2$ can then be composed either with $\psi_1$ or with $\psi_2$. On the other hand, if we start by composing the bottom two rows then there is only one way in which we can form the composite.
	  
	  This shows why in \defref{augmented virtual double category} the associativity axiom for a three-row composite has to be ``read from left to right'': reading it in the other direction in general there might be multiple ways in which the cells in the top row can be ``distributed'' over those in the middle row.
	  \begin{displaymath}
	    \begin{tikzpicture}[scheme, x=1.8em, y=1.8em]
	      \draw (0,3) -- (3,3) -- (2.5,2) -- (2.5, 1) -- (2,0) -- (1,0) -- (0.5,1) -- (0.5,2) -- (0,3)
	            (0.5,2) -- (2.5,2)
	            (1,3) -- (1.5,2) -- (1.5,1)
	            (2,3) -- (1.5,2)
	            (0.5,1) -- (2.5,1);
	      \draw (0.75,2.5) node {$\phi_1$}
	            (1.5,2.75) node {$\phi_2$}
	            (2.25,2.5) node {$\phi_3$}
	            (1,1.5) node {$\psi_1$}
	            (2,1.5) node {$\psi_2$}
	            (1.5,0.5) node {$\chi$};
	    \end{tikzpicture}
	  \end{displaymath}
	  
	  More formally the above observation is a manifestation of the fact that, when regarded as monoids, augmented virtual double categories $\K$ (with a fixed vertical category $\K_\textup v = \mathcal C$) are monoids in a \emph{skew-monoidal category}, in the sense of Szlach\'anyi \cite{Szlachanyi12}, instead of monoids in an ordinary monoidal category. Making this statement precise will be left as future work.
	\end{remark}
	
	\subsection{Monoids and bimodules}
	Our main source of augmented virtual double categories is virtual double categories, as we will explain in this subsection. Briefly, given a virtual double category $\K$ we will consider the augmented virtual double category $\Mod(\K)$ of `monoids' and `bimodules' in $\K$, the latter in the sense of Section 5.3 of \cite{Leinster04}; see also Section 2 of \cite{Cruttwell-Shulman10}. Often we will then consider a sub-augmented virtual double category of $\Mod(\K)$ by ``restricting the size of bimodules''. For instance, instead of considering `large profunctors' between large categories, it is preferable to consider `small profunctors' between large categories.
	\begin{definition} \label{monoids and bimodules}
		Let $\K$ be a virtual double category.
		\begin{enumerate}[label =-]
			\item A \emph{monoid} $A$ in $\K$ is a quadruple $A = (A, \alpha, \bar\alpha, \tilde\alpha)$ consisting of a horizontal morphism $\hmap\alpha AA$ in $\K$ equipped with \emph{multiplication} and \emph{unit} cells
			\begin{displaymath}

			\end{displaymath}
			that satisfy the associativity and unit axioms $\bar\alpha \of (\bar\alpha, \id_\alpha) = \bar\alpha \of (\id_\alpha, \bar\alpha)$ and $\bar\alpha \of (\tilde\alpha, \id_\alpha) = \id_\alpha = \bar\alpha \of (\id_\alpha, \tilde\alpha)$.
			\item	A \emph{morphism} $A \to C$ of monoids is a vertical morphism $\map fAC$ in $\K$ that is equipped with a cell
			\begin{displaymath}
				%
			\end{displaymath}
			that satisfies the associativity and unit axioms $\bar\gamma \of (\bar f, \bar f) = \bar f \of \bar \alpha$ and $\tilde \gamma \of f = \bar f \of \tilde \alpha$.
			\item A \emph{bimodule} $A \brar B$ between monoids is a horizontal morphism $\hmap JAB$ in $\K$ that is equipped with left and right \emph{action} cells
			\begin{displaymath}
				%
			\end{displaymath}
			that satisfy the usual associativity, unit and compatibility axioms for bimodules:
			\begin{align*}
				\lambda \of (\bar\alpha, \id_J) & = \lambda \of (\id_\alpha, \lambda); & \rho \of (\id_J, \bar\beta) &= \rho \of (\rho, \id_\beta); \\
				\lambda \of (\tilde\alpha, \id_J) & = \id_J = \rho \of (\id_J, \tilde\beta); &\rho \of (\lambda, \id_\beta) &= \lambda \of (\id_\alpha, \rho).
			\end{align*}
			\item	A cell
			\begin{displaymath}
				%
			\end{displaymath}
			of bimodules, where $n \geq 1$, is given by a cell $\phi$ in $\K$ between the underlying morphisms, satisfying the \emph{external equivariance} axioms
			\begin{align*}
				\phi \of (\lambda, \id_{J_2}, \dotsc, \id_{J_n}) &= \lambda \of (\bar f, \phi) \\
				\phi \of (\id_{J_1}, \dotsc, \id_{J_{n'}}, \rho) &= \rho \of (\phi, \bar g)
			\end{align*}
			and the \emph{internal equivariance} axioms
			\begin{multline*}
				\phi \of (\id_{J_1}, \dotsc, \id_{J_{i''}}, \rho, \id_{J_i}, \id_{J_{i+1}}, \dotsc, \id_{J_n}) \\
					= \phi \of (\id_{J_1}, \dotsc, \id_{J_{i''}}, \id_{J_{i'}}, \lambda, \id_{J_{i+1}}, \dotsc, \id_{J_n})
			\end{multline*}
			for $2 \leq i \leq n$.
			\item A cell
			\begin{displaymath}
				\begin{tikzpicture}
					\matrix(m)[math35, column sep={1.75em,between origins}]{& A & \\ C & & D \\};
					\path[map]	(m-1-2) edge[transform canvas={xshift=-1pt}] node[left] {$f$} (m-2-1)
															edge[transform canvas={xshift=1pt}] node[right] {$g$} (m-2-3)
											(m-2-1) edge[barred] node[below] {$K$} (m-2-3);
					\path				(m-1-2) edge[cell, transform canvas={yshift=-0.25em}] node[right, inner sep=2.5pt] {$\phi$} (m-2-2);
				\end{tikzpicture}	
			\end{displaymath}
			of bimodules is given by a cell $\phi$ in $\K$ between the underlying morphisms, satisfying the \emph{external equivariance} axiom $\lambda \of (\bar f, \phi) = \rho \of (\phi, \bar g)$.
		\end{enumerate}
	\end{definition}
	For the next proposition notice that any module $C = (C, \gamma, \bar\gamma, \tilde\gamma)$ in a virtual double category $\K$ induces a bimodule $\hmap\gamma CC$, both whose actions are given by multiplication $\cell{\bar\gamma}{(\gamma, \gamma)}\gamma$.
	\begin{proposition} \label{(fc, bl)-multicategory of monoids}
		Given a virtual double category $\K$ consider to each cell $\phi$ of bimodules in $\K$, as on the left below, a new unary cell $\bar\phi$ as on the right, that is of the same form, and to each cell $\psi$ of bimodules as on the left, with horizontal target $\gamma$ as described above, a new nullary cell $\bar\psi$ as on the right.
		\begin{displaymath}

		\end{displaymath}
		Monoids in $\K$, the morphisms and bimodules between them, together with the cells $\bar\phi$ and $\bar\psi$ above, form an augmented virtual double category $\Mod(\K)$. Composition of cells in $\Mod(\K)$ is given by
		\begin{displaymath}
			\bar\psi \of (\bar\phi_1, \dotsc, \bar\phi_n) \dfn \overline{\bigpars{\psi' \of (\phi_1, \dotsc, \phi_n)}}
		\end{displaymath}
		where $\psi'$ is any cell of the right form in $\K$ that is obtained by composing $\psi$ with actions, on its horizontal sources and/or target, of the horizontal targets $\gamma_i$ of the cells $\bar\phi_i$ that are nullary. The identity cells in $\Mod(\K)$, for bimodules $\hmap JAB$ and morphisms $\map fAC$ of monoids, are given by
		\begin{displaymath}
			\id_J \dfn \overline{(\id_J)} \qquad \text{and} \qquad \id_f \dfn \overline{(\tilde\gamma \of f)}.
		\end{displaymath}
	\end{proposition}
	That the composition $\bar\psi \of (\bar\phi_1, \dotsc, \bar\phi_n)$ above does not depend on the choice of $\psi'$ follows from the equivariance axioms for $\psi$ (\defref{monoids and bimodules}). It is straightforward to show that these axioms, together with associativity and unitality of composition in $\K$, imply the associativity and unit axioms for the composition in $\Mod(\K)$. Later on we will see that $\K \mapsto \Mod(\K)$ forms the object-function of the composition of a pair of $2$-functors, the first described in \propref{Mod as right pseudo-adjoint} and the second in \thmref{unital fc-multicategories}. 
	
	To give examples, we assume given a category $\Set'$ of \emph{large sets} and functions between them, as well as a full subcategory $\Set \subset \Set'$ of \emph{small sets} and their functions, such that the morphisms in $\Set$ form a large set in $\Set'$.
	\begin{example} \label{enriched profunctors}
		Let $\V = (\V, \tens, I)$ be a monoidal category. The virtual double category $\Mat\V$ of \emph{$\V$-matrices} has large sets and functions as objects and vertical morphisms, while a horizontal morphism $\hmap JAB$ is a $\V$-matrix, given by a family $J(x, y)$ of $\V$"/objects indexed by pairs $(x, y) \in A \times B$. A cell $\phi$ in $\Mat\V$, of the form as above, consists of a family of $\V$-maps
		\begin{displaymath}
			\map{\phi_{(x_0, \dotsc, x_n)}}{J_1(x_0, x_1) \tens \dotsb \tens J_n(x_{n'}, x_n)}{K(fx_0, gx_n)}
		\end{displaymath}
		indexed by sequences $(x_0, \dotsc, x_n) \in A_0 \times \dotsb \times A_n$, where the tensor product is to be interpreted as the unit $I$ in case $n = 0$.
		
		The augmented virtual double category $\enProf\V \dfn \Mod(\Mat\V)$ of monoids and bimodules in $\Mat\V$ is that of \emph{$\V$-enriched} categories, $\V$-functors, $\V$-profunctors and $\V$-natural transformations. In some more detail, a horizontal morphism $\hmap JAB$ in $\enProf\V$, between $\V$"/categories $A$ and $B$, is a \emph{$\V$-profunctor} in the sense of Section~3 of \cite{Lawvere73} (called bimodule there): it consists of a family of $\V$-objects $J(x, y)$, indexed by pairs of objects $x \in A$ and $y \in B$, that is equipped with associative and unital actions
		\begin{displaymath}
			\map\lambda{A(x_1, x_2) \tens J(x_2, y)}{J(x_1, y)} \qquad \text{and} \qquad \map\rho{J(x,y_1) \tens B(y_1, y_2)}{J(x, y_2)}
		\end{displaymath}
		satisfying the usual compatibility axiom for bimodules. If $\V$ is closed symmetric monoidal, so that it can be considered as enriched over itself, then $\V$-profunctors $\hmap JAB$ can be identified with $\V$-functors of the form $\map J{\op A \tens B}\V$.
		
		A vertical cell $\cell\phi fg$ in $\enProf\V$, between $\V$-functors $f$ and $\map gAC$, is a $\V$-natural transformation $f \Rar g$ in the usual sense; see for instance Section 1.2 of \cite{Kelly82}. We conclude that the vertical $2$-category $V(\enProf\V)$, that is contained in $\enProf\V$, equals the $2$-category $\enCat\V$ of $\V$-categories, $\V$-functors and the $\V$-natural transformations between them.
		
		Taking $\V = \Set$ we recover the archetypal augmented virtual double category $\enProf\Set$ of locally small categories, functors, \emph{(small) profunctors} $\map J{\op A \times B}\Set$ and transformations. Analogously $\enProf{\Set'}$ is the augmented virtual double category of categories (with possibly large hom-sets), functors, \emph{large profunctors} $\map J{\op A \times B}\Set'$ and transformations.
	\end{example}
	
	\begin{example} \label{internal profunctors}
		Let $\E$ be a category with pullbacks. The virtual double category $\Span\E$ of \emph{spans} in $\E$ has as objects and vertical morphisms those of $\E$, while its horizontal morphisms $\hmap JAB$ are spans $A \leftarrow J \rightarrow B$ in $\E$. A cell $\phi$ in $\Span\E$, of the form as above, is a morphism
		\begin{displaymath}
			\map \phi{J_1 \times_{A_1} \dotsb \times_{A_{n'}} J_n}K
		\end{displaymath}
		in $\E$ lying over $f$ and $g$. The augmented virtual double category $\inProf\E \dfn \Mod(\Span\E)$ of monoids and bimodules in $\Span\E$ is that of \emph{internal} categories, functors, profunctors and transformations in $\E$. Analogous to previous example $V(\inProf\E) = \inCat\E$, that is the $2$-category of internal categories, functors and transformations in $\E$; the latter in the usual sense of e.g.\ Section 1 of \cite{Street74b}.
	\end{example}
	
	The following example forms the main motivation for extending the notion of virtual double category to that of augmented virtual double category.
	\begin{example} \label{(Set, Set')-Prof}
		Taking $\V = \Set'$ in \exref{enriched profunctors}, we write $\enProf{(\Set, \Set')} \subset \enProf{\Set'}$ for the sub-augmented virtual double category that is obtained by restricting to small profunctors. In detail: $\enProf{(\Set, \Set')}$ consists of all (possibly large) categories and functors, only those profunctors $\hmap JAB$ with small sets $J(x, y)$ for all \mbox{$(x, y) \in A \times B$}, and all transformations between such profunctors (including the nullary ones).
		
		Thus we have an ascending chain of augmented virtual double categories
		\begin{displaymath}
			 \enProf\Set \subset \enProf{(\Set, \Set')} \subset \enProf{\Set'},
		\end{displaymath}
		and we take the view that the classical theory of locally small categories is best considered in $\enProf{(\Set, \Set')}$. As motiviation for this view, on one hand remember from \cite{Freyd-Street95} that, for a locally small category $A$, the category $\Set^{\op A}$ of small presheaves on $A$ is locally small too precisely if $A$ is essentially small, so that small presheaves on $A$ need not form an object in $\enProf\Set$, whereas they do in $\enProf{(\Set, \Set')}$. On the other hand, writing $\map\yon A{\Set^{\op A}}$ for the yoneda embedding, Yoneda's lemma supplies, for each small profunctor $\hmap JAB$ in $\enProf{(\Set, \Set')}$, a functor $\map{\cur J}B{\Set^{\op A}}$ equipped with a natural isomorphism of small profunctors $J \iso \Set^{\op A}(\yon, \cur J)$\footnote{Indeed we take $\cur J(y) \dfn J(\dash, y)$.}; of course such a $\cur J$ does not exist for the properly large profunctors $J$ that exist in $\enProf{\Set'}$.
		
		For an advantage of working in the augmented virtual double category $\enProf{(\Set, \Set')}$ rather than the virtual double category $U\pars{\enProf{(\Set, \Set')}}$ that it contains notice that, for any two functors $f$ and $\map gAC$ into a large category $C$, the natural transformations $\nat\phi fg$ can be considered in the former but not the latter. Indeed in $\enProf{(\Set, \Set')}$ they form the vertical cells $\cell\phi fg$, which are removed from $U\pars{\enProf{(\Set, \Set')}}$. On the other hand such natural transformations correspond to cells in $\enProf{\Set'}$ of the form below, where $I_C$ is the `unit profunctor' given by the large hom-sets $I_C(x,y) = C(x,y)$ (see \exref{horizontal composites in (V, V')-Prof}), but these do not exist in either $\enProf{(\Set, \Set')}$ or $U\pars{\enProf{(\Set, \Set')}}$.
		\begin{displaymath}
			\begin{tikzpicture}
				\matrix(m)[math35, column sep={1.75em,between origins}]{& A & \\ C & & C \\};
				\path[map]	(m-1-2) edge node[left] {$f$} (m-2-1)
														edge node[right] {$g$} (m-2-3)
										(m-2-1) edge[barred] node[below] {$I_C$} (m-2-3);
				\path				(m-1-2) edge[cell, transform canvas={yshift=-0.25em}] node[right, inner sep=2.5pt] {$\phi$} (m-2-2);
			\end{tikzpicture}
		\end{displaymath}
	\end{example}
	
	As a variation on $\enProf{(\Set, \Set')}$ the following example describes the augmented virtual double category $\enProf{(\Set, \Set')}^\Ss$ of small profunctors indexed by a category $\Ss$. The case where $\Ss$ is the category $\Gg_1 = \pars{1 \rightrightarrows 0 }$ will be important to us, as the `free strict double category'-monad is defined on $\enProf{(\Set, \Set')}^{\Gg_1}$; see \exref{free strict double category-monad}.
	\begin{example} \label{indexed profunctors}
		Let $\Ss$ be a small category. We first describe the augmented virtual double category $\enProf{\Set'}^\Ss \dfn \inProf{\Set'^\Ss}$ of profunctors internal to the functor category $\Set'^\Ss$ (\exref{internal profunctors}). Its objects are \emph{large $\Ss$-indexed categories}, i.e.\ functors \mbox{$\map A\Ss{\Cat(\Set')}$}, while its vertical morphisms $\map fAC$ are the \emph{$\Ss$"/indexed functors} between them, that is natural transformations $A \Rar C$. For $s \in \Ss$ we write $A_s \dfn A(s)$; likewise $\map{A_u \dfn A(u)}{A_s}{A_t}$ for each map $\map ust \in \Ss$. A horizontal morphism $\hmap JAB$ in $\enProf{\Set'}^\Ss$ is a \emph{large $\Ss$-indexed profunctor} between $A$ and $B$, consisting of a family of profunctors $\hmap{J_s}{A_s}{B_s}$, one for each $s \in \Ss$, that is equipped with $(1,1)$-ary cells $J_u$ in $\enProf{\Set'}$, as on the left below, one for each $\map ust \in \Ss$. The assignment $u \mapsto J_u$ is required to be natural, that is $J_{v \of u} = J_v \of J_u$ and $J_{\id_s} = \id_{J_s}$.
		
		An \emph{$\Ss$-indexed cell} $\phi$ in $\enProf{\Set'}^\Ss$, of the form as in \eqref{cell}, consists of a family of cells $\phi_s$ as on the right below, one for each $s \in \Ss$, that are natural in $s$, in the sense that $\phi_t \of (J_{1u}, \dotsc, J_{nu}) = K_{1u} \of \phi_s$ if $\lns{\ul K} = 1$, and $\phi_t \of (J_{1u}, \dotsc, J_{nu}) = C_u \of \phi_s$ if $\lns{\ul K} = 0$. Finally, compositions and identities in $\enProf{\Set'}^\Ss$ are simply given indexwise: for instance $\pars{\psi \of (\phi_1, \dotsc, \phi_n)}_s = \psi_s \of (\phi_{1s}, \dotsc, \phi_{ns})$, and so on.
		\begin{displaymath}
			\begin{tikzpicture}[baseline]
					\matrix(m)[math35]{A_s & B_s \\ A_t & B_t \\};
					\path[map]	(m-1-1) edge[barred] node[above] {$J_s$} (m-1-2)
															edge node[left] {$A_u$} (m-2-1)
											(m-1-2) edge node[right] {$B_u$} (m-2-2)
											(m-2-1) edge[barred] node[below] {$J_t$} (m-2-2);
					\path[transform canvas={xshift=1.75em}]	(m-1-1) edge[cell] node[right] {$J_u$} (m-2-1);
				\end{tikzpicture} \qquad\qquad\qquad\qquad \begin{tikzpicture}[baseline]
				\matrix(m)[math35]{A_{0s} & A_{1s} & A_{n's} & A_{ns} \\ C_s & & & D_s \\};
				\path[map]	(m-1-1) edge[barred] node[above] {$J_{1s}$} (m-1-2)
														edge node[left] {$f_s$} (m-2-1)
										(m-1-3) edge[barred] node[above] {$J_{ns}$} (m-1-4)
										(m-1-4) edge node[right] {$g_s$} (m-2-4)
										(m-2-1) edge[barred] node[below] {$\ul K_s$} (m-2-4);
				\path[transform canvas={xshift=1.75em}]	(m-1-2) edge[cell] node[right, inner sep=2.5pt] {$\phi_s$} (m-2-2);
				\draw				($(m-1-2)!0.5!(m-1-3)$) node {$\dotsb$};
			\end{tikzpicture}
		\end{displaymath}
		
		Analogous to the definition of $\enProf{(\Set, \Set')}$ as a sub-augmented virtual double category of $\enProf{\Set'}$, in the previous example, we define $\enProf{(\Set, \Set')}^\Ss$ as the sub-augmented virtual double category of $\enProf{\Set'}^\Ss$ that is obtained by restricting to \emph{small $\Ss$-indexed profunctors} $\hmap JAB$, i.e.\ those with $J_s(x, y) \in \Set$ for all $s \in \Ss$, $x \in A_s$ and $y \in B_s$.
	\end{example}
	
	\subsection{Universe enlargements}
	Notice that, analogous to \exref{(Set, Set')-Prof}, we can consider sub-augmented virtual double categories $\enProf{(\mathsf{Ab}, \mathsf{Ab}')} \subset \enProf{\mathsf{Ab}'}$, $\enProf{(\Cat, \Cat')} \subset \enProf{\Cat'}$, etc., where $\mathsf{Ab} \subset \mathsf{Ab}'$, $\Cat \subset \Cat'$, etc., are the embeddings obtained by considering abelian groups, categories, etc., in the categories of sets $\Set$ and $\Set'$ respectively. Again we prefer to work in e.g.\ $\enProf{(\mathsf{Ab}, \mathsf{Ab}')}$ instead of $\enProf{\mathsf{Ab}}$ or $\enProf{\mathsf{Ab'}}$, for reasons similar to the ones given in \exref{(Set, Set')-Prof}.
	
	More generally, we will follow Kelly in Section 3.11 of \cite{Kelly82} and enrich both in a monoidal category $\V$ as well as in a `universe enlargement' $\V'$ of $\V$, as follows. We call a category $\mathcal C$ \emph{small (resp.\ large) (co-)complete} if the (co-)limit of any diagram $\map DS\mathcal C$ exists as soon as the set of morphisms $S_1$ of $S$ is small (resp.\ large); that is $S_1 \in \Set$ (resp.\ $S_1 \in \Set'$). By a \emph{closed} monoidal category $\V$ we mean a monoidal category $\V = (\V, \tens, I)$ that is equipped with, for every object $X \in \V$, a right adjoint $\brks{X, \dash}$ to the endofunctor $X \tens \dash$ on $\V$. These right adjoints combine to form an \emph{internal hom} functor $\map{\brks{\dash, \dash}}{\V \times \V}\V$.
	
	Remember that a \emph{normal monoidal} functor $\map F\V\V'$ between monoidal categories $\V$ and $\V' = (\V', \tens', I')$ is equipped with coherent natural isomorphisms $\bar F\colon FX \tens' FY \xrar\iso F(X \tens Y)$, for any pair $X, Y \in \V$, while it preserves the unit $FI = I'$. If $\V$ and $\V'$ are symmetric monoidal, with symmetries denoted $\mathfrak s$ and $\mathfrak s'$, then a monoidal functor $\map F\V\V'$ is called \emph{symmetric} monoidal whenever its coherence isomorphisms $\bar F$ make the diagrams below commute.
	\begin{displaymath}
		\begin{tikzpicture}
			\matrix(m)[math35, column sep=2em]{FX \tens' FY & F(X \tens Y) \\ FY \tens' FX & F(Y \tens X) \\};
			\path[map]	(m-1-1) edge node[above] {$\bar F$} (m-1-2)
													edge node[left] {$\mathfrak s'$} (m-2-1)
									(m-1-2) edge node[right] {$F\mathfrak s$} (m-2-2)
									(m-2-1) edge node[below] {$\bar F$} (m-2-2);
		\end{tikzpicture}
	\end{displaymath}
	Finally if $\V$ and $\V'$ are both closed, with internal hom functors $\brks{\dash, \dash}$ and $\brks{\dash, \dash}'$, then a (symmetric) monoidal functor $\map F\V\V'$ is called \emph{closed} as soon as the canonical maps $F\brks{X, Y} \to \brks{FX, FY}'$, that are adjunct to the composites below, are invertible.
	\begin{displaymath}
		FX \tens' F\brks{X, Y} \xrar{\bar F} F(X \tens \brks{X, Y}) \xrar{F\textup{ev}} FY
	\end{displaymath}
	\begin{definition} \label{universe enlargement}
		Let $\V$ be a (closed) monoidal category whose set of morphisms is large. By a \emph{universe enlargement} of $\V$ we mean a closed monoidal category $\V'$ that is equipped with a full and faithful (closed) normal monoidal functor $\V \to \V'$ satisfying the following axioms.
		\begin{enumerate}[label=(\alph*)]
			\item $\V'$ is locally large, that is $\V'(X,Y) \in \Set'$ for all $X, Y \in \V'$; 
			\item $\V'$ is large cocomplete and large complete;
			\item $\V \to \V'$ preserves all limits;
			\item $\V \to \V'$ preserves large colimits.
		\end{enumerate}
		A universe enlargement $\V \to \V'$ is called \emph{symmetric} whenever both $\V$ and $\V'$ are symmetric monoidal categories, while $\V \to \V'$ is a symmetric monoidal functor.
	\end{definition}
	We will see in \exref{equivalence induced by universe enlargement} below that any universe enlargement $\V \to \V'$ induces an `inclusion' $\enProf\V \to \enProf{\V'}$ of augmented virtual double categories, so that $\V$-categories, $\V$"/functors and $\V$-profunctors can be regarded as if enriched in $\V'$. We write \mbox{$R\V \subset \V'$} for the \emph{replete image} of the enlargement, that is the full subcategory consisting of $\V'$-objects isomorphic to objects in the image of $\V \to \V'$. Since $\V \to \V'$ is full and faithful it factors as an equivalence $\V \simeq R\V \subset \V'$, which we shall use to regard $\V$ as a subcategory of $\V'$; consequently, by a \emph{$\V$-object} we shall mean either an object in $\V$ or one in $R\V$.
	
	We shall also see that the factorisation $\V \to \V'$ through $R\V$ induces a factorisation
	\begin{displaymath}
		\begin{tikzpicture}
			\matrix(m)[math35]{ \enProf\V & & \enProf{\V'} \\ & \enProf{R\V} & \\ };
			\path[map]		(m-1-1) edge (m-1-3);
			\draw[white]	(m-1-1) edge node[sloped, black] {$\simeq$} (m-2-2)
										(m-2-2) edge node[sloped, black] {$\subset$} (m-1-3);
		\end{tikzpicture}
	\end{displaymath}
	as an `equivalence' of augmented virtual double categories, which we likewise use to identify $\V$"/categories, $\V$-functors and $\V$-profunctors with their $R\V$-enriched counterparts; using the term \emph{$\V$-category} (resp.\ \emph{$\V$-functor} and \emph{$\V$-profunctor}) in both cases.
	
	One can show that the embeddings $\Set \subset \Set'$, $\mathsf{Ab} \subset \mathsf{Ab}'$ and $\Cat \subset \Cat'$, of the categories of small sets, small abelian groups and small categories into the categories of their large counterparts, are symmetric universe enlargements in the above sense, as long as $\Set$ has infinite sets. More generally we have the following result, which summarises Sections 3.11 and 3.12 of \cite{Kelly82}.
	\begin{theorem}[Kelly]
		Let $\V$ be a (closed) (symmetric) monoidal category, with tensor product $\tens$, whose set of morphisms is large. The (symmetric) monoidal structure on $\V$ induces a closed (symmetric) monoidal structure on the category ${\Set'}^{\op\V}$ of large presheaves on $\V$, whose tensor product $\tens'$ is given by \emph{Day convolution} (see \cite{Day70} or \exref{Day convolution} below):
		\begin{displaymath}
			p \tens' q \dfn \int^{x, y \in \V} \V(\dash, x \tens y) \times px \times qy
		\end{displaymath}
		for presheaves $p$ and $\map q{\op\V}{\Set'}$. With respect to this structure the yoneda embedding $\map\yon \V{{\Set'}^{\op\V}}$ satisfies axioms \textup{(a)} to \textup{(c)} of the previous definition.
		
		Next let $\map{\yon'}\V{\V'}$ denote the factorisation of $\yon$ through the full subcategory \mbox{$\V' \subset \Set'^{\op \V}$} consisting of presheaves that preserve all large limits in $\op \V$. The induced closed (symmetric) monoidal structure on $\Set'^{\op\V}$ again induces such a structure on $\V'$. If $\V$ is closed (symmetric) monoidal then $\yon'$, with respect to the latter monoidal structure on $\V'$, forms a (symmetric) universe enlargement.
	\end{theorem}
	
	\begin{example} \label{(V, V')-Prof}
		Given a universe enlargement $\V \to \V'$ consider a $\V'$-profunctor $\hmap JAB$ in $\enProf{\V'}$ (see \exref{enriched profunctors}). We will call $J$ a \emph{$\V$-profunctor} whenever $J(x, y)$ is a $\V$-object for all pairs $x \in A$ and $y \in B$. Analogous to \exref{(Set, Set')-Prof}, we denote by $\enProf{(\V, \V')}$ the sub-augmented virtual double category of $\enProf{\V'}$ that consists of all $\V'$-categories and $\V'$-functors, as well as $\V$-profunctors and their transformations.
	\end{example}
	
	\begin{example} \label{universe enlargement of Lawvere quantale}
		A \emph{quantale} (see e.g.\ Section II.1.10 of \cite{Hofmann-Seal-Tholen14}) is a complete lattice equipped with a unital monoid structure $\tens$ that preserves suprema on both sides. Equivalently, a quantale can be thought of as a thin monoidal category $\V$ that is complete (and hence cocomplete) and that is equipped with a closed monoidal structure. Consider the quantale $\brks{-\infty, \infty}$, equipped with the reversed order $\geq$, whose monoid structure is the unique extension of addition $(+, 0)$ that preserves infima on both sides, while its closed structure similarly extends subtraction; for details see e.g.\ Section 2.2 of \cite{Willerton15}.
		
		The quantale $\brks{-\infty, \infty}$ forms a symmetric universe enlargement of the \emph{Lawvere quantale} $\brks{0, \infty}$, which it contains as a subquantale. Since the closed structure on the latter is given by \emph{truncated} subtraction the monoidal inclusion $\brks{0, \infty} \into \brks{-\infty, \infty}$ is not closed. Categories enriched in $\brks{0, \infty}$ form a fundamental example of enriched category theory: it was Lawvere who showed in \cite{Lawvere73} that they can be regarded as \emph{generalised metric spaces}: sets $A$ equipped with a (non-symmetric) distance function $A \times A \to \brks{0, \infty}$, which here we again denote by $A$. Enriching in $\brks{-\infty, \infty}$ instead allows for distance functions $\map A{A \times A}{\brks{-\infty, \infty}}$ that take negative distances; see e.g.\ Section 2.3 of \cite{Willerton15}. Both vertical morphisms $\map fAC$ and horizontal morphisms $\hmap JAB$ between generalised metric spaces are required to be \emph{non"/expanding}, that is $A(x, y) \geq C(fx, fy)$ for all $x, y \in A$ and 
		\begin{displaymath}
			A(x_1, x_2) + J(x_2, z) \geq J(x_1, z) \qquad \text{and} \qquad J(x, z_1) + B(z_1, z_2) \geq J(x, z_2)
		\end{displaymath}
		for all $x_1, x_2, x \in A$ and $z, z_1, z_2 \in B$.
	\end{example}
	
	\subsection{The \texorpdfstring{$2$}{2}-category of augmented virtual double categories}
	Having introduced the notion of augmented virtual double categories we next consider the functors between them, as well as their transformations.
	\begin{definition}
		A \emph{functor} $\map F\K\L$ between augmented virtual double categories consists of a functor $\map {F_\textup v}{\K_\textup v}{\L_\textup v}$ (which will be denoted $F$) as well as assignments mapping the horizontal morphisms and cells of $\K$ to those of $\L$, as shown below, in a way that preserves compositions and identities strictly.
		\begin{align*}
			\hmap JAB \qquad&\mapsto\qquad \hmap{FJ}{FA}{FB} \\

		\end{align*}
	\end{definition}
	
	\begin{definition} \label{transformation}
		A \emph{transformation} $\nat\xi FG$ of functors $F$, $\map G\K\L$ between augmented virtual double categories consists of a natural transformation $\nat\xi{F_\textup v}{G_\textup v}$ as well as a family of $(1,1)$-ary cells
		\begin{displaymath}
			%
		\end{displaymath}
		in $\L$, one for each $\hmap JAB \in \K$, that satisfies the \emph{naturality axiom}
		\begin{displaymath}
			G\phi \of \xi_{\ul J} = \xi_{\ul K} \of F\phi
		\end{displaymath}
		whenever this makes sense, where $\xi_{\ul J} \dfn (\xi_{J_1}, \dotsc, \xi_{J_n})$ if $\ul J = (J_1, \dotsc, J_n)$ and $\xi_{\ul J} \dfn \xi_A$ if $\ul J = (A)$.
	\end{definition}
	
	Recall that the definition of augmented virtual double category reduces to that of virtual double category by restricting it to unary cells, together with removing the references to vertical identity cells; see the discussion following \lemref{horizontal composition}. Likewise the definitions above reduce to that of functor and transformation for virtual double categories as given in Section 3 of \cite{Cruttwell-Shulman10}. The latter combine into a $2$-category of virtual double categories which we denote $\VirtDblCat$.	Remember that every augmented virtual double category $\K$ contains a $2$-category $V(\K)$ and a virtual double category $U(\K)$. In the following, which is easily checked, $\twoCat$ denotes the $2$-category of $2$-categories, strict $2$-functors and $2$-natural transformations.
	\begin{proposition} \label{2-category of (fc, bl)-multicategories}
		Augmented virtual double categories, their functors and the transformations between them form a $2$-category $\AugVirtDblCat$. Both the assignments \mbox{$\K \mapsto V(\K)$} and $\K \mapsto U(\K)$ extend to strict $2$-functors
		\begin{displaymath}
			\map V\AugVirtDblCat\twoCat \quad \text{and} \quad \map U\AugVirtDblCat\VirtDblCat. 
		\end{displaymath}
	\end{proposition}
	
	\begin{example}
		Every lax monoidal functor $\map F\V{\catvar W}$ between monoidal categories induces a functor $\map{\Mat F}{\Mat\V}{\Mat{\catvar W}}$ between the virtual double categories of matrices in $\V$ and $\catvar W$ (\exref{enriched profunctors}) in the evident way. Likewise the components of any monoidal transformation $\nat\xi FG$ form the cell-components of an induced transformation $\nat{\Mat\xi}{\Mat F}{\Mat G}$. In fact the assignments $F \mapsto \Mat F$ and $\xi \mapsto \Mat\xi$ combine to form a strict $2$-functor $\map{\Mat{(\dash)}}{\MonCat_\textup l}{\VirtDblCat}$, where $\MonCat_\textup l$ denotes the $2$-category of monoidal categories, lax monoidal functors and monoidal transformations.
	\end{example}
	
	\begin{example} \label{Span is a 2-functor}
		Similarly any pullback-preserving functor $\map F\D\E$, between categories with pullbacks, induces a functor $\map{\Span F}{\Span\D}{\Span\E}$ between the virtual double categories of spans in $\D$ and $\E$ (see \exref{internal profunctors}). This too extends to a strict $2$-functor $\map{\Span{\dash}}{\Cat_\textup{pb}}{\VirtDblCat}$, where $\Cat_\textup{pb}$ denotes the $2$"/category of categories with pullbacks, pullback-preserving functors and all natural transformations between them.
	\end{example}
	
	\begin{proposition} \label{Mod is a 2-functor}
		The assignment $\K \mapsto \Mod(\K)$ of \propref{(fc, bl)-multicategory of monoids} extends to a strict $2$-functor $\map\Mod\VirtDblCat\AugVirtDblCat$.
	\end{proposition}
	\begin{proof}[Sketch of the proof.]
		The image $\map{\Mod F}{\Mod \K}{\Mod \L}$ of a functor $\map F\K\L$ between virtual double categories is simply given by applying $F$ indexwise; e.g.\ it maps a monoid $A = (A, \alpha, \bar \alpha, \tilde \alpha)$ in $\K$ to the monoid $(\Mod F)(A) \dfn (FA, F\alpha, F\bar \alpha, F\tilde\alpha)$ in $\L$. The image $\map{\Mod\xi}{\Mod F}{\Mod G}$ of a transformation $\nat\xi FG$ has as components the monoid morphisms $\map{(\Mod\xi)_A \dfn (\xi_A, \xi_\alpha)}{FA}{GA}$, one for each monoid $A$ in $\K$, as well as the cells of bimodules $\cell{(\Mod\xi)_J \dfn \xi_J}{FJ}{GJ}$, one for each bimodule $\hmap{J = (J, \lambda, \rho)}AB$ in $\K$.
	\end{proof}
	
	\subsection{Equivalence of augmented virtual double categories}
	Now that we have a $2$-category of augmented virtual double categories we can consider (adjoint) equivalences between such double categories.
	\begin{definition}
		An \emph{equivalence} $\K \simeq \L$ of augmented virtual double categories $\K$ and $\L$ is an equivalence in the $2$-category $\AugVirtDblCat$. That is, it is given by pair of functors \mbox{$\map F\K\L$} and $\map G\L\K$ that is equipped with invertible transformations \mbox{$\eta\colon\id_\K \iso GF$} and $\eps\colon FG \iso \id_\L$.
		
		If $\eta$ and $\eps$ satisfy the triangle identities, that is they define $F$ and $G$ as adjoint morphisms in $\AugVirtDblCat$, then the quadruple $(F, G, \eta, \eps)$ is called an \emph{adjoint equivalence} $\K \simeq \L$.
	\end{definition}
	Remember that, in any $2$-category, an equivalence $(F, G, \eta, \eps)\colon\K \simeq \L$ can be turned into an adjoint equivalence $(F, G, \eta, \eps')\colon\K \simeq \L$; see for instance Proposition~1.5.7 of \cite{Leinster04}.
	
	\begin{example} \label{equivalence induced by universe enlargement}
		Any universe enlargement $\V \to \V'$, in the sense of \defref{universe enlargement}, induces an equivalence of augmented virtual double categories as follows. Firstly applying the composite $2$-functor $\enProf{(\dash)} \dfn \Mod \of \Mat{(\dash)}$, which extends the assignment $\V \mapsto \enProf\V$ of \exref{enriched profunctors}, to $\V \to \V'$ gives a functor $\enProf\V \to \enProf{\V'}$ of augmented virtual double categories, as in the top of the diagram below.
		\begin{displaymath}
			\begin{tikzpicture}
				\matrix(m)[math35, row sep={4em,between origins}, column sep={6.5em,between origins}]
					{	\enProf\V & \enProf{\V'} \\
						\enProf{R\V} & \enProf{(\V, \V')} \\ };
				\draw[map]	(m-1-1) edge (m-1-2)
														edge (m-2-2);
				\draw[white]	(m-1-1) edge node[sloped, black] {$\simeq$} (m-2-1)
											(m-2-2)	edge node[black, sloped] {$\subset$} (m-1-2)
											(m-2-1) edge node[black, sloped] {$\subset$} (m-2-2);
			\end{tikzpicture}
		\end{displaymath}
		Next consider the factorisation $\V \simeq R\V \subset \V'$ of $\V \to \V'$ through the replete image $R\V$. Because the enlargement, and hence its factorisation $\V \simeq R\V$, is a monoidal functor, it follows from Kelly's result on `doctrinal adjunctions' \cite{Kelly74} (also see \thmref{creating adjunctions} below) that $\V \simeq R\V$ is in fact an equivalence of monoidal categories. Consequently its image under $\enProf{(\dash)}$, as on the left in the diagram above, is an equivalence of augmented virtual double categories as well.
		
		Finally notice that $\enProf{R\V} \subset \enProf{\V'}$ factors through the augmented virtual double category $\enProf{(\V, \V')}$ of $\V$-profunctors between $\V'$-categories, that was described in \exref{(V, V')-Prof}. Since each of the functors $\enProf\V \simeq \enProf{R\V} \subset \enProf{(\V, \V')} \subset \enProf{\V'}$ is full and faithful, in the sense below, it follows that both $\enProf\V \to \enProf{(\V, \V')}$ and $\enProf\V \to \enProf{\V'}$ are full and faithful too.
	\end{example}
	
	The goal of this section is to prove that, just like in classical category theory (see e.g.\ Section IV.4 of \cite{MacLane98}),  giving an equivalence $\K \simeq \L$ is the same as giving a functor $\map F\K\L$ that is `full, faithful and essentially surjective'. The following definitions generalise those for functors between double categories, as given in Section~7 of \cite{Shulman08}.
	
	We start with the notions full and faithful. Let $\map F\K\L$ be a functor between augmented virtual double categories. Its restriction $J \mapsto FJ$ to horizontal morphisms preserves sources and targets, so that it extends to an assignment $\ul J = (J_1, \dotsc, J_n) \mapsto F\ul J \dfn (FJ_1, \dotsc, FJ_n)$ on paths. Next, for each pair of morphisms $\map f{A_0}C$ and \mbox{$\map g{A_n}D$} in $\K$, and paths $\hmap{\ul J}{A_0}{A_n}$ and $\hmap{\ul K}CD$, where $\lns K \leq 1$, the functor $F$ restricts to the following assignment, between classes of cells of the forms as shown:
	\begin{displaymath}
		\Bigl\lbrace\begin{tikzpicture}[textbaseline]
			\matrix(m)[math35]{A_0 & A_n \\ C & D \\};
			\path[map]	(m-1-1) edge[barred] node[above] {$\ul J$} (m-1-2)
													edge node[left] {$f$} (m-2-1)
									(m-1-2) edge node[right] {$g$} (m-2-2)
									(m-2-1) edge[barred] node[below] {$\ul K$} (m-2-2);
			\path[transform canvas={xshift=1.75em}]	(m-1-1) edge[cell] node[right] {$\phi$} (m-2-1);
		\end{tikzpicture}\Bigr\rbrace \quad \xrar F \quad \Bigl\lbrace\begin{tikzpicture}[textbaseline]
			\matrix(m)[math35]{FA_0 & FA_n \\ FC & FD \\};
			\path[map]	(m-1-1) edge[barred] node[above] {$F\ul J$} (m-1-2)
													edge node[left] {$Ff$} (m-2-1)
									(m-1-2) edge node[right] {$Fg$} (m-2-2)
									(m-2-1) edge[barred] node[below] {$F\ul K$} (m-2-2);
			\path[transform canvas={xshift=1.75em}]	(m-1-1) edge[cell] node[right] {$\psi$} (m-2-1);
		\end{tikzpicture}\Bigr\rbrace
	\end{displaymath}
	\begin{definition} \label{full and faithful functor}
		A functor $\map F\K\L$ between augmented virtual double categories is called \emph{locally faithful} (resp.\ \emph{locally full}) if, for each $\map f{A_0}C$, $\map g{A_n}D$, $\hmap{\ul J}{A_0}{A_n}$ and $\hmap{\ul K}CD$ in $\K$, the assignment above is injective (resp.\ surjective). If moreover the restriction $\map F{\K_\textup v}{\L_\textup v}$, to the vertical categories, is faithful (resp.\ full), then $F$ is called \emph{faithful} (resp.\ \emph{full}).
	\end{definition}
	
	The following is a direct translation of Definition 7.6 of \cite{Shulman08} to the setting of augmented virtual double categories.
	\begin{definition} \label{essentially surjective}
		A functor $\map F\K\L$ of augmented virtual double categories is called \emph{essentially surjective} if we can simultaneously make the following choices:
		\begin{enumerate}[label=-]
			\item an object $A_\K \in \K$ for each object $A \in \L$, together with an isomorphism $\sigma_A\colon FA_\K \iso A$ in $\L$;
			\item a horizontal morphism $\hmap{J_\K}{A_\K}{B_\K}$ for each horizontal morphism $\hmap JAB$ in $\L$, together with an invertible cell
			\begin{displaymath}
				\begin{tikzpicture}
					\matrix(m)[math35]{FA_\K & FB_\K \\ A & B. \\};
					\path[map]	(m-1-1) edge[barred] node[above] {$FJ_\K$} (m-1-2)
															edge node[left] {$\sigma_A$} (m-2-1)
											(m-1-2) edge node[right] {$\sigma_B$} (m-2-2)
											(m-2-1) edge[barred] node[below] {$J$} (m-2-2);
					\path[transform canvas={xshift=1.75em}]	(m-1-1) edge[cell] node[right] {$\sigma_J$} (m-2-1);
				\end{tikzpicture}
			\end{displaymath}
		\end{enumerate}
	\end{definition}
	
	\begin{proposition} \label{equivalences}
		A functor $\map F\K\L$ between augmented virtual double categories is part of an adjoint equivalence $\K \simeq \L$ if and only if it is full, faithful and essentially surjective.
	\end{proposition}
	\begin{proof}[Sketch of the proof]
		The `only if'-part is straightforward; we will sketch the `if'"/part. First, because $F$ is essentially surjective, we can choose objects $A_\K \in \K$, for each $A \in \L$, and horizontal morphisms $\hmap{J_\K}{A_\K}{B_\K} \in \K$, for each $\hmap JAB \in \L$, as in the definition above, together with isomorphisms $\sigma_A\colon FA_\K \iso A$ and $\sigma_J\colon FJ_\K \iso J$. Using the full and faithfulness of $F$ these choices can be extended to a functor $\map{(\dash)_\K}\L\K$ as follows: for each vertical morphism $\map fAC$ in $\L$ we define $\map{f_\K}{A_\K}{C_\K}$ to be the unique map in $\K$ such that $Ff = \inv\sigma_C \of f \of \sigma_A$, and for each cell $\cell\phi{\ul J}{\ul K}$ in $\L$ we define $\cell{\phi_\K}{\ul J_\K}{\ul K_\K}$ to be the unique cell in $\K$ such that $F\phi_\K = \inv\sigma_{\ul K} \of \phi \of \sigma_{\ul J}$, where the notation $\sigma_{\ul J}$ is as in \defref{transformation}. Using that $F$ is faithful it is easily checked that these assignments preserve the composition and identities of $\L$.
		
		Now the isomorphisms $(\sigma_A)_{A \in \L}$ and $(\sigma_J)_{J \in \L}$ combine to form a transformation $\sigma \colon F \of (\dash)_\K \iso \id_\L$. Conversely, a transformation $\eta\colon\id_K\iso (\dash)_\K \of F$ is obtained by defining $\eta_A$, where $A \in \K$, to be unique with $F\eta_A = \inv\sigma_{FA}$ and defining $\eta_J$, where $\hmap JAB$ in $\K$, such that $F\eta_J = \inv\sigma_{FJ}$. Checking that $\eta$ and $\sigma$ satisfy the triangle identities is easy.
	\end{proof}
	
	\section{Restriction of horizontal morphisms}\label{restriction section}
	Two notions that are invaluable to the theory of double categories and their generalisations are that of restricting horizontal morphisms and that of composing horizontal morphisms. In this section we consider a variation of the former notion, to one that is appropriate for augmented virtual double categories, following for a large part Section 7 of \cite{Cruttwell-Shulman10}; in the next section we consider the latter.
	
	\subsection{Cartesian cells}
	Restrictions of horizontal morphisms are defined by `cartesian cells', as follows.
	
	\begin{definition} \label{cartesian cells}
		A cell $\cell\psi{\ul J}{\ul K}$ with $\lns{\ul J} \leq 1$, as in the right-hand side below, is called \emph{cartesian} if any cell $\chi$, as on the left-hand side, factors uniquely through $\psi$ as a cell $\phi$ as shown.
		\begin{displaymath}
			\begin{tikzpicture}[textbaseline]
				\matrix(m)[math35]{X_0 & X_1 & X_{n'} & X_n \\ A & & & B \\ C & & & D \\};
				\path[map]	(m-1-1) edge[barred] node[above] {$H_1$} (m-1-2)
														edge node[left] {$h$} (m-2-1)
										(m-1-3) edge[barred] node[above] {$H_n$} (m-1-4)
										(m-1-4) edge node[right] {$k$} (m-2-4)
										(m-2-1) edge node[left] {$f$} (m-3-1)
										(m-2-4) edge node[right] {$g$} (m-3-4)
										(m-3-1) edge[barred] node[below] {$\ul K$} (m-3-4);
				\draw				($(m-1-2)!0.5!(m-1-3)$) node {$\dotsc$};
				\path[transform canvas={yshift=-1.625em}]	($(m-1-1.south)!0.5!(m-1-4.south)$) edge[cell] node[right] {$\chi$} ($(m-2-1.north)!0.5!(m-2-4.north)$);
			\end{tikzpicture} = \begin{tikzpicture}[textbaseline]
				\matrix(m)[math35]{X_0 & X_1 & X_{n'} & X_n \\ A & & & B \\ C & & & D \\};
				\path[map]	(m-1-1) edge[barred] node[above] {$H_1$} (m-1-2)
														edge node[left] {$h$} (m-2-1)
										(m-1-3) edge[barred] node[above] {$H_n$} (m-1-4)
										(m-1-4) edge node[right] {$k$} (m-2-4)
										(m-2-1) edge[barred] node[below] {$\ul J$} (m-2-4)
														edge node[left] {$f$} (m-3-1)
										(m-2-4) edge node[right] {$g$} (m-3-4)
										(m-3-1) edge[barred] node[below] {$\ul K$} (m-3-4);
				\draw				($(m-1-2)!0.5!(m-1-3)$) node {$\dotsc$};
				\path				($(m-1-1.south)!0.5!(m-1-4.south)$) edge[cell] node[right] {$\phi$} ($(m-2-1.north)!0.5!(m-2-4.north)$)
										($(m-2-1.south)!0.5!(m-2-4.south)$) edge[cell, transform canvas={yshift=-2pt}] node[right] {$\psi$} ($(m-3-1.north)!0.5!(m-3-4.north)$);
			\end{tikzpicture}
		\end{displaymath}
		Vertically dual, provided that the cell $\phi$ in the right-hand side above is unary, it is called \emph{weakly cocartesian} if any cell $\chi$ factors uniquely through $\phi$ as shown.
	\end{definition}
	
	If an $(1,n)$-ary cartesian cell $\psi$ like above exists then its horizontal source \mbox{$\hmap JAB$} is called the \emph{restriction} of $\hmap{\ul K}CD$ along $f$ and $g$, and denoted $\ul K(f, g) \dfn J$. If $\ul K = (C \xbrar K D)$ is of length $n = 1$ then we will call $K(f, g)$ \emph{unary}; in the case that $\ul K = (C)$ we call $C(f, g)$ \emph{nullary}. By their universal property any two cartesian cells defining the same restriction factor through each other as invertible horizontal cells. We will often not name cartesian cells, but simply depict them as
	\begin{displaymath}
		\begin{tikzpicture}
			\matrix(m)[math35]{A & B \\ C & D. \\};
			\path[map]	(m-1-1) edge[barred] node[above] {$\ul J$} (m-1-2)
													edge node[left] {$f$} (m-2-1)
									(m-1-2) edge node[right] {$g$} (m-2-2)
									(m-2-1) edge[barred] node[below] {$\ul K$} (m-2-2);
			\draw				($(m-1-1)!0.5!(m-2-2)$) node[font=\scriptsize] {$\cart$};
		\end{tikzpicture}
	\end{displaymath}
	
	\begin{example} \label{restrictions of V-profunctors}
		In the augmented virtual double category $\enProf{(\V, \V')}$ of $\V$-profunctors between $\V'$"/categories (\exref{(V, V')-Prof}) all unary restrictions $K(f, g)$ exist: they are simply the $\V$-profunctors given by the family of $\V$-objects $K(fx, gy)$, for all $x \in A$ and $y \in B$, that is equipped with actions induced by those of $K$. The cartesian cell $K(f, g) \Rar K$ simply consists of the identities on the $\V$"/objects $K(fx, gy)$.
		
		On the other hand, the nullary restriction $C(f, g)$ of two $\V'$-functors $f$ and $g$ does not exist in general, but it does whenever all of the hom-objects $C(fx, gy)$ are $\V$-objects: in that case the cartesian cell $C(f, g) \Rar C$ consists of the identities on these hom-objects. Specifically, in $\enProf{(\Set, \Set')}$ all nullary restrictions $C(f, g)$ exist as soon as $C$ is locally small. We will see in \exref{necessary condition for the existence of companions and conjoints} below that, in the case that either $f$ or $g$ is an identity $\V'$-functor, the previous condition is necessary for the existence of $C(f, g)$ as well.
		
		Similarly, the augmented virtual double category $\enProf\V$ of $\V$-profunctors (\exref{enriched profunctors}) admits all (both unary and nullary) restrictions.
	\end{example}
	
	\begin{example} \label{restrictions of indexed profunctors}
		Analogous to the situation for $\enProf{(\Set, \Set')}$, in the augmented virtual double category $\enProf{(\Set, \Set')}^\Ss$ of small $\Ss$-indexed profunctors (\exref{indexed profunctors}) all unary restrictions exist, while the nullary restriction $\hmap{C(f, g)}AB$ exists as soon as the hom-sets $C_s(f_sx, g_sy)$ are small for all $s \in \Ss$, $x \in A_s$ and $y \in B_s$. In either case the restrictions can be defined indexwise, by setting $\ul K(f, g)_s \dfn \ul{K_s}(f_s, g_s)$ for all $s \in \Ss$.
	\end{example}
	
	\begin{example} \label{isomorphisms have cartesian identity cells}
		For any isomorphism $\map fAC$ the vertical identity cell $\id_f$ is cartesian.
	\end{example}
	
	If the weakly cocartesian cell on the left below exists then we call its horizontal target $K$ the \emph{extension} of $\ul J$ along $f$ and $g$. Like restrictions, extensions are unique up to isomorphism. When considering the above notion of weakly cocartesian cell in a virtual double category, by restricting the factorisations to unary cells $\chi$, we recover the notion of weakly cocartesian cell that was considered in Remark 5.8 of \cite{Cruttwell-Shulman10}.
	\begin{displaymath}
		\begin{tikzpicture}[baseline]
			\matrix(m)[math35]{A & B \\ C & D \\};
			\path[map]	(m-1-1) edge[barred] node[above] {$\ul J$} (m-1-2)
													edge node[left] {$f$} (m-2-1)
									(m-1-2) edge node[right] {$g$} (m-2-2)
									(m-2-1) edge[barred] node[below] {$K$} (m-2-2);
			\draw				($(m-1-1)!0.5!(m-2-2)$) node[font=\scriptsize] {$\cocart$};
		\end{tikzpicture}
	\end{displaymath}
	
	Where it is often easy to give examples of restrictions, giving examples of extensions is usually harder. Fortunately, as we shall see in the next section (\corref{extensions and composites}), the extension of $J$ along $f$ and $g$ above coincides with the `horizontal composite' \mbox{$(C(\id, f) \hc J \hc D(g, \id))$} whenever it exists, where $\hmap{C(\id, f)}C{A_0}$ and $\hmap{C(g, \id)}{A_n}D$ are nullary restrictions in the above sense. Analogously, in \lemref{restrictions and composites} we will see that the restriction of $K$ along $f$ and $g$ coincides with the composite $(C(f, \id) \hc K \hc D(\id, g))$. Thus most results concerning weakly cocartesian cells are left to the next section, except for a characterisation of certain such cells in $\enProf{(\V, \V')}$, which is given at the end of this subsection.
	
	Cartesian cells satisfy the following pasting lemma. As a consequence, taking restrictions is `pseudofunctorial' in the sense that $\ul K(f, g)(h, k) \iso \ul K(f \of h, g \of k)$ and $K(\id, \id) \iso K$.
	\begin{lemma}[Pasting lemma] \label{pasting lemma for cartesian cells}
		If the cell $\phi$ in the composite below is cartesian then the full composite $\phi \of \psi$ is cartesian if and only if $\psi$ is.
		\begin{displaymath}

		\end{displaymath}
	\end{lemma}
	
	Horizontal composition with the (co"/)unit of an adjunction preserves nullary cartesian cells as follows.
	\begin{lemma} \label{horizontal composition with (co-)units preserves nullary cartesian cells}
		In an augmented virtual double category $\K$ consider the composite below. If $\eta$ is the unit of the adjunction $f \ladj g$ (in $V(\K)$) then $\phi$ is cartesian precisely if the full composite is. A horizontally dual result holds for the counit.
		\begin{displaymath}
			%
  	\end{displaymath}
  	In particular any invertible cell $\eta\colon\id_C \iso s$, where $\map sCC$, is cartesian.
  \end{lemma}
  \begin{proof}
  	Consider the commutative diagram of assignments below, between collections of cells in $\K$ that are of the form as drawn.
  	\begin{displaymath}
  		%
		\end{displaymath}
		Notice that the $\phi$ is cartesian precisely if the top left assignment is a bijection, and that the composite $(\eta \of h) \hc (g \of \phi)$ is cartesian precisely if the top right assignment is a bijection. The proof follows form the fact that the bottom assignment is a bijection: its inverse is given by composing the cells $\xi$ on the right with the counit of $f \ladj g$.
		
		For the final assertion notice that $\eta\colon \id_C \iso s$ forms the unit of $\id_C \ladj s$, with $\inv\eta$ the counit. Clearly the identity cell $\id_C \dfn \id_{\id_C}$ on $\id_C$ is cartesian so that, by the above, the composite $\eta \hc (s \of \id_C) = \eta$ is cartesian too.
	\end{proof}
	
	It is clear from, for instance, \cite{Cruttwell-Shulman10} and \cite{Koudenburg14a} that the notion of a (generalised) double category admitting all restrictions is a useful one. In Section~7 of the former virtual double categories having all restrictions and all `horizontal units' (which we consider in the next section) are called `virtual equipments'. Virtual equipments that have all `horizontal composites' as well can be regarded as `bicategories equipped with proarrows', in the original sense of Wood \cite{Wood82}, by combining the results in Appendix C of \cite{Shulman08} with \propref{pseudo double categories} below; whence the term `equipment'. Because important augmented virtual double categories such as $\enProf{(\V, \V')}$ do not have all units nor all nullary restrictions, but do have all unary restrictions, we have chosen the following generalisation of `equipment' as appropriate for augmented virtual double categories.
	\begin{definition} \label{augmented virtual equipment}
		An \emph{augmented virtual equipment} is an augmented virtual double category that has all unary restrictions $K(f, g)$.
	\end{definition}
	
	In the augmented virtual double category $\enProf{(\V, \V')}$, of $\V$-profunctors between $\V'$-categories, full and faithfulness of $\V'$-functors is related to cartesianness as follows.
	\begin{proposition}
		In $\enProf{(\V, \V')}$ let $\map fAC$ be a $\V'$-functor such that the hom"/objects $C(fx, fy)$ are $\V$-objects for all $x, y \in A$. The identity cell $\id_f$ is cartesian if and only if $f$ is full and faithful.
	\end{proposition}
	\begin{proof}
		The `if'-part is easy: if the actions $\map{\bar f}{A(x, y)}{C(fx, fy)}$ of $f$ on hom"/objects are invertible then the unique factorsation of a cell $\cell\chi{(H_1, \dotsc, H_n)}C$ through $\id_f$, as in \defref{cartesian cells}, is given by composing the components of $\chi$ with the inverses of $\bar f$.
		
		For the converse, assume that $\id_f$ is cartesian and remember that, by the assumption on the hom-objects $C(fx, fy)$, the restriction $C(f, f)$ exists; see \exref{restrictions of V-profunctors}.
		\begin{displaymath}

		\end{displaymath}
		 Consider the unique factorisations $\phi$ and $\psi$ in the identities above: the components of $\phi$ are simply the actions $\map{\bar f}{A(x, y)}{C(fx, fy)}$, and we claim that the components $C(fx, fy) \to A(x,y)$ of $\psi$ form their inverses. This claim is a straightforward consequence of the identities $\phi \of \psi = \id_{C(f, f)}$ and $\psi \of \phi = \id_A$, which themselves follow from the equations $\cart \of \phi \of \psi = \id_f \of \psi = \cart$ and $\id_f \of \psi \of \phi = \cart \of \phi = \id_f$, as well as the uniqueness of factorisations through $\cart$ and $\id_f$. We conclude that $f$ is full and faithful, completing the proof.
	\end{proof}
	
	In view of the above result we make the following definition.
	\begin{definition} \label{full and faithful morphism}
		A vertical morphism $\map fAC$ is called \emph{full and faithful} if both its identity cell $\id_f$ is cartesian and the restriction $C(f, f)$ exists.
	\end{definition}
	It is clear that a full and faithful morphism $\map fAC$, in the above sense, is full and faithful in the vertical $2$-category $V(\K)$, of objects, vertical morphisms and vertical cells in $\K$, in the classical sense; that is, for each object $X \in \K$ the functor $\map{V(\K)(X,f)}{V(\K)(X,A)}{V(\K)(X,C)}$, given by postcomposition with $f$, is full and faithful (see e.g.\ Example 2.18 of \cite{Weber07}). The converse to this holds as soon as $\K$ has `cocartesian tabulations', as we shall see in \propref{cells corresponding to vertical cells}.
	
	Closing this subsection we characterise weakly cocartesian cells  that are of the form as on the left below, in the augmented virtual double categories $\enProf\V$ and $\enProf{(\V, \V')}$. Here $I$ denotes the unit $\V$-category consisting of a single object $*$ and hom-object $I(*, *) = I$, the unit of $\V$. Because the universe enlargement $\V \to \V'$ preserves the monoidal unit strictly, we can regard $I$ as the unit $\V'$-category as well. We identify $\V$-functors $I \to A$ with objects in $A$ and $\V$-profunctors $I \brar I$ with $\V$-objects; cells between such profunctors are identified with $\V$-maps.
	\begin{equation} \label{horizontal weakly cocartesian cell}

	\end{equation}
	
	First consider a general path $\ul J = (A_0 \xbrar{J_1} A_1, \dotsc, A_{n'} \xbrar{J_n} A_n)$ of $\V$-profunctors. For any pair of objects $x \in A_0$ and $y \in A_n$ we denote by $\ul J^\S(x, y)$ the embedding into $\V$ of its subcategory consisting of the spans of the form below where, for each $i = 1, \dotsc, n'$, both $v_i$ and $w_i$ range over all objects in $A_i$.
	\begin{equation} \label{internal equivariance span}
		%
	\end{equation}
	If $\V$ is closed symmetric monoidal, so that each $\hmap{J_i}{A_{i'}}{A_i}$ can be identified with a $\V$-functor $\map{J_i}{\op{A_{i'}} \tens A_i}{\V}$, then colimit of $\ul J^\S(x, y)$, if it exists, is easily checked to coincide with the \emph{coend}\footnote{For the definition of the dual notion \emph{end} see Section 2.1 of \cite{Kelly82}.} $\int^{u_1 \in A_1, \dotsc, u_{n'} \in A_{n'}} J_1(x, u_1) \tens \dotsb \tens J_n(u_{n'}, y)$. Consequently we will use this coend notation for the colimit of $\ul J^\S(x,y)$, regardless of $\V$ being closed symmetric monoidal.
	
	Returning to a path of $\V$-profunctors $\hmap{\ul J}II$, as on the left of \eqref{horizontal weakly cocartesian cell} above, notice that for any cocone $\ul J^\S(*, *) \Rar \Delta X$, with $X \in \V$, its naturality with respect to the spans above coincides with the internal equivariance axioms satisfied by the cells in $\enProf\V$ with source $\ul J$. We conclude that giving a unary cell $\psi$ as on the right of \eqref{horizontal weakly cocartesian cell}, where $s \in C$ and $t \in D$, is the same as giving a cocone $\nat\gamma{\ul J^\S}{\Delta \ul L(s, t)}$. The following characterisation of weakly cocartesian cells in $\enProf\V$ is now straightforward.
	\begin{proposition} \label{weakly cocartesian cell in (V, V')-Prof}
		A cell $\phi$ in $\enProf\V$, of the form as on the left of \eqref{horizontal weakly cocartesian cell}, is weakly cocartesian precisely if its corresponding cocone $\ul J^\S(*, *) \Rar \Delta K$ is colimiting; that is, it defines the $\V$-object $K$ as the coend $\int^{u_1 \in A_1, \dotsc, u_{n'} \in A_{n'}} J_1(*, u_1) \tens \dotsb \tens J_n(u_{n'}, *)$.
		
		Furthermore, given a universe enlargement $\V \to \V'$, the inclusions
		\begin{displaymath}
			\enProf\V \to \enProf{(\V, \V')} \to \enProf{\V'}
		\end{displaymath}
		both preserve and reflect weakly cocartesian cells of the form as on the left in \eqref{horizontal weakly cocartesian cell}.
	\end{proposition}
	\begin{proof}[Sketch of the proof.]
		Reflection along $\enProf{(\V, \V')} \to \enProf{\V'}$ is clear; for preservation notice that, for any $\V'$-object $X$, cocones $\ul J^\S(*, *) \Rar \Delta X$ correspond to nullary cells $\psi$ as on the right of \eqref{horizontal weakly cocartesian cell}, with $C$ the $\V'$"/category consisting of objects $0$ and $1$, together with hom-objects $C(0, 0) = I' = C(1, 1)$ and $C(0,1) = X$, while $\psi$ has vertical morphisms $s = 0$ and $t = 1$.
		
		For reflection and preservation along $\enProf\V \to \enProf{\V'}$ notice that the $\V'$"/cocone $\ul J^\S(*, *) \Rar \V'$ corresponding to a cell $\cell\phi{\ul J}K$ in $\enProf\V$ is isomorphic to the composition of its corresponding $\V$-cocone $\ul J^\S(*, *) \Rar \V$ with $\V \to \V'$. Use that universe enlargements preserve large colimits and are full and faithful.
	\end{proof}
	
	\subsection{Companions and conjoints}
	Here we consider nullary restrictions of the forms $C(f, \id)$ and $C(\id, f)$, where $\map fAC$ is a vertical morphism. These have been called respectively `horizontal companions' and `horizontal adjoints' in the setting of double categories \cite{Grandis-Pare04}; we follow \cite{Cruttwell-Shulman10} in calling them `companions' and `conjoints'. As foreshadowed in the discussion preceding \lemref{pasting lemma for cartesian cells}, companions and conjoints can be regarded as building blocks out of which restrictions and extensions can be built, as will be explained in next section.
	
	\begin{definition}
		Consider a vertical morphism $\map fAC$ in an augmented virtual double category. The nullary restriction $\hmap{C(f, \id)}AC$, if it exists, is called the \emph{companion} of $f$ and denoted $f_*$, while $\hmap{C(\id, f)}CA$, if it exists, is called its \emph{conjoint} and denoted $f^*$.
	\end{definition}
	
	Notice that the notions of companion and conjoint are swapped when moving from $\K$ to its horizontal dual $\co\K$.
	
	Although companions and conjoints are defined as nullary restrictions, the following lemma and its horizontal dual show that they can equivalently be defined as extensions along $f$. More precisely it gives a bijective correspondence between the cartesian cells $\phi$ defining a horizontal morphism $\hmap JAC$ as the companion of $f$ and the weakly cocartesian cells $\psi$ defining $J$ as the extension of $(A)$ along $\id_A$ and $f$, in such a way that each corresponding pair $(\psi, \phi)$ satisfies the identities below, which are called the \emph{companion identities}. Analogous identities are satisfied by corresponding pairs of a cartesian and weakly cocartesian cell defining a conjoint; these are called the \emph{conjoint identities}.
	\begin{lemma} \label{companion identities lemma}
		Consider a factorisation of the identity cell $\id_f$ of $\map fAC$ as on the left below. The following conditions are equivalent: $\psi$ is cartesian; the identity on the right holds; $\phi$ is weakly cocartesian.
	 	\begin{equation} \label{companion identities}

	  \end{equation}
	\end{lemma}
	\begin{proof}
		We will show that both $\psi$ being cartesian, as well as $\phi$ being weakly cocartesian, implies the identity on the right, while the latter implies the (co-)cartesianness of $\psi$ and $\phi$. For the first implication, notice that the identity on the left above implies that composing the left-hand side of the identity on the right either with $\psi$ or $\phi$ results in $\psi$ or $\phi$ again, respectively. By the uniqueness of factorisations through (co-)cartesian cells, it follows that the identity on the right holds as soon as $\psi$ is cartesian or $\phi$ is weakly cocartesian.
		
		For the converse assume that both identities above hold; we will show that $\psi$ is cartesian and that $\phi$ is weakly cocartesian. For the first consider a cell $\chi$ as on the left below; we have to show that it factors uniquely through $\psi$. That it factors through $\psi$ follows from the identity on the left above as follows, where the last identity is one of the interchange axioms.
		\begin{displaymath}
			\chi = (\id_f \of \id_h) \hc \chi = (\psi \of \phi \of h) \hc \chi = \psi \of \bigpars{(\phi \of h) \hc \chi}.
		\end{displaymath}
		To show the uniqueness of this factorisation consider a second factorisation $\chi = \psi \of \chi'$. Using the identity on the right above we have
		\begin{displaymath}
			(\phi \of h) \hc \chi = (\phi \of h) \hc (\psi \of \chi') = (\phi \hc \psi) \of \chi' = \chi',
		\end{displaymath}
		showing that the factorisation obtained before coincides with $\chi'$. This concludes the proof of $\psi$ being cartesian.
		\begin{displaymath}
			\begin{tikzpicture}[baseline]
				\matrix(m)[math35, column sep={1.75em,between origins}]
					{X_0 & & X_1 & \dotsb & X_{n'} & & X_n \\ & & & E & & & \\};
				\path[map]	(m-1-1) edge[barred] node[above] {$H_1$} (m-1-3)
														edge node[below left] {$f \of h$} (m-2-4)
										(m-1-5) edge[barred] node[above] {$H_n$} (m-1-7)
										(m-1-7) edge node[below right] {$k$} (m-2-4);
				\path				(m-1-4) edge[cell] node[right] {$\chi$} (m-2-4);
			\end{tikzpicture} \qquad\qquad\qquad \begin{tikzpicture}[baseline]
					\matrix(m)[math35, column sep={1.75em,between origins}]{& A & \\ E & & F \\};
					\path[map]	(m-1-2) edge node[left] {$h$} (m-2-1)
															edge node[right] {$k \of f$} (m-2-3)
											(m-2-1) edge[barred] node[below] {$L$} (m-2-3);
					\path				(m-1-2) edge[cell, transform canvas={yshift=-0.25em}] node[right, inner sep=2.5pt] {$\xi$} (m-2-2);
				\end{tikzpicture}
		\end{displaymath}
		
		To show that $\phi$ is weakly cocartesian under the identities \eqref{companion identities} consider any unary cell $\xi$ as on the right above. That it factors through $\phi$ is shown by
		\begin{displaymath}
			\xi = \xi \hc \id_{k \of f} = \xi \hc (k \of \psi \of \phi) = \bigpars{\xi \hc (k \of \psi)} \of \phi.
		\end{displaymath}
		To see that this factorisation is unique, suppose that $\xi = \xi' \of \phi$ as well. Then
		\begin{displaymath}
			\xi \hc (k \of \psi) = (\xi' \of \phi) \hc (k \of \psi) = \xi' \of (\phi \hc \psi) = \xi',
		\end{displaymath}
		which concludes the proof that $\phi$ is weakly cocartesian.
	\end{proof}
	
	As an immediate consequence we find that functors of augmented virtual double categories preserve companions and conjoints.
	\begin{corollary} \label{functors preserve companions and conjoints}
		Any functor of augmented virtual double categories preserves cartesian and weakly cocartesian cells that define companions and conjoints.
	\end{corollary}
	\begin{proof}
		This follows immediately from the fact that functors preserve vertical composition strictly, so that the companion and conjoint identities of (the horizontal dual of) the previous lemma are preserved.
	\end{proof}
	
	The lemma above can also be used to show that the sufficient conditions for the existence of nullary restrictions $C(f, g)$ in the augmented virtual equipments $\enProf{(\V, \V')}$ and $\enProf{(\Set, \Set')}^\Ss$, that were given in \exref{restrictions of V-profunctors} and \exref{restrictions of indexed profunctors} above, are necessary in the case of companions and conjoints, as follows.
	\begin{example} \label{necessary condition for the existence of companions and conjoints}
		In \exref{restrictions of V-profunctors} we saw that the companion $f_*$ of a $\V'$-functor \mbox{$\map fAC$} exists in $\enProf{(\V, \V')}$ as soon as the hom-objects $C(fx, y)$ are $\V$-objects, for all $x \in A$ and $y \in C$. For the converse consider cells $\cell\psi JC$ and $\cell\phi AJ$ as in the lemma above: it is straightforward to check that the companion identities for $\phi$ and $\psi$ imply that the composites below are inverses for the components $J(x, y) \to C(fx, y)$ of $\psi$, showing that $C(fx, y) \iso J(x, y)$.
		\begin{displaymath}
			C(fx, y) \xrar{\phi_x \tens' \id} J(x, fx) \tens' C(fx, y) \xrar\rho J(x, y)
		\end{displaymath}
		
		Given an indexing category $\Ss$, the previous argument can be applied at each index $s \in \Ss$ to show that the companion $f_*$ of an $\Ss$-indexed functor $\map fAC$ exists in $\enProf{(\Set, \Set')}^\Ss$ precisely if the hom-sets $C_s(f_sx, y)$ are small for all $s \in \Ss$, $x \in A_s$ and $y \in C_s$.
	\end{example}
	
	The remainder of this subsection records some useful properties of companions.
	\begin{lemma} \label{companion of a composite}
		Let $\map fAC$ and $\map hCE$ be morphisms such that the companion $\hmap{h_*}CE$ exists. The companion $(h \of f)_*$ exists if and only if the restriction $h_*(f, \id)$ does, and in that case they are isomorphic.
		\begin{displaymath}
			\begin{tikzpicture}[textbaseline]
  			\matrix(m)[math35, column sep={0.875em,between origins}]
  				{	A & & & & E \\
  					& C & & & \\
  					& & E & & \\ };
  			\path[map]	(m-1-1) edge[barred] node[above] {$J$} (m-1-5)
  													edge[transform canvas={xshift=-1pt}] node[left] {$f$} (m-2-2)
  									(m-2-2) edge[transform canvas={xshift=-1pt}] node[left] {$h$} (m-3-3);
  			\path				(m-1-3) edge[cell, transform canvas={yshift=-0.5em}] node[right] {$\phi$} (m-2-3)
  									(m-1-5) edge[eq, transform canvas={xshift=1pt}] (m-3-3);
  		\end{tikzpicture} = \begin{tikzpicture}[textbaseline]
  			\matrix(m)[math35, column sep={1.875em,between origins}]{A & & E \\ C & & E \\ & E & \\};
  			\path[map]	(m-1-1) edge[barred] node[above] {$J$} (m-1-3)
  													edge node[left] {$f$} (m-2-1)
  									(m-2-1) edge[barred] node[below] {$h_*$} (m-2-3)
  													edge[transform canvas={xshift=-2pt}] node[left] {$h$} (m-3-2);
  			\path				(m-1-3) edge[eq] (m-2-3)
  									(m-2-3) edge[eq, transform canvas={xshift=2pt}] (m-3-2)
  									(m-1-2) edge[cell] node[right] {$\psi$} (m-2-2);
        \draw[font=\scriptsize]	([yshift=0.25em]$(m-2-2)!0.5!(m-3-2)$) node {$\cart$};
  		\end{tikzpicture}
		\end{displaymath}
		In detail, in the factorisation above the cell $\phi$ is cartesian if and only if the cell $\psi$ is.
		
		Dually if the conjoint $h^*$ exists then the conjoint $(h \of f)^*$ exists precisely if the restriction $h^*(\id, f)$ does.
	\end{lemma}
	\begin{proof}
		This follows immediately from applying the pasting lemma (\lemref{pasting lemma for cartesian cells}) to the factorisation above.
	\end{proof}
	
	\begin{lemma}
	  Let $\map fAC$, $\map gBC$ and $\map hCE$ be vertical morphisms. If $h$ is full and faithful then the restriction $C(f, g)$ exists if and only if the restriction $\hmap{E(h \of f, h \of g)}AC$ does, and in that case they are isomorphic.
		\begin{displaymath}
			\begin{tikzpicture}[textbaseline]
  			\matrix(m)[math35, column sep={0.875em,between origins}]
  				{	A & & & & B \\
  					& C & & C & \\
  					& & E & & \\ };
  			\path[map]	(m-1-1) edge[barred] node[above] {$J$} (m-1-5)
  													edge[transform canvas={xshift=-1pt}] node[left] {$f$} (m-2-2)
  									(m-2-2) edge[transform canvas={xshift=-1pt}] node[left] {$h$} (m-3-3)
  									(m-1-5) edge[transform canvas={xshift=1pt}] node[right] {$g$} (m-2-4)
  									(m-2-4) edge[transform canvas={xshift=1pt}] node[right] {$h$} (m-3-3);
  			\path				(m-1-3) edge[cell, transform canvas={yshift=-0.5em}] node[right] {$\phi$} (m-2-3);
  		\end{tikzpicture} = \begin{tikzpicture}[textbaseline]
				\matrix(m)[math35, column sep={1.875em,between origins}]{A & & B \\ & C & \\ & E & \\};
				\path[map]	(m-1-1) edge[barred] node[above] {$J$} (m-1-3)
														edge[transform canvas={xshift=-2pt}] node[left] {$f$} (m-2-2)
										(m-1-3) edge[transform canvas={xshift=2pt}] node[right] {$g$} (m-2-2)
										(m-2-2) edge[bend right=45] node[left] {$h$} (m-3-2)
														edge[bend left=45] node[right] {$h$} (m-3-2);
				\path				(m-1-2) edge[cell, transform canvas={yshift=0.25em}] node[right] {$\psi$} (m-2-2)
										(m-2-2) edge[cell, transform canvas={xshift=-0.35em}] node[right] {$\id$} (m-3-2);
			\end{tikzpicture}
		\end{displaymath}
		In detail, in the factorisation above the cell $\phi$ is cartesian if and only if the cell $\psi$ is.
	\end{lemma}
	\begin{proof}
		Because $h$ is full and faithful its identity cell is cartesian by \defref{full and faithful morphism}. The proof follows immediately from applying the pasting lemma (\lemref{pasting lemma for cartesian cells}) to the factorisation above.  
	\end{proof}
	
	\begin{lemma} \label{companions of morphisms composed with an isomorphism}
		In an augmented virtual equipment consider a composite $g \of f$ of vertical morphisms. If $g$ is an isomorphism then $(g \of f)_*$ exists as soon as $f_*$ does.
	\end{lemma}
	\begin{proof}
		Assuming that $f_*$ exists, we consider the composite on the left-hand side below, whose top cartesian cell defines the restriction $C(f, \inv g)$ of $f_*$ along $\inv g$.
		\begin{displaymath}
			\begin{tikzpicture}[textbaseline]
  			\matrix(m)[math35, column sep={1.875em,between origins}]{A & & E \\ A & & C \\ & C & \\};
  			\path[map]	(m-1-1) edge[barred] node[above] {$C(f, \inv g)$} (m-1-3)
  									(m-1-3) edge node[right] {$\inv g$} (m-2-3)
  									(m-2-1) edge[barred] node[below] {$f_*$} (m-2-3)
  													edge[transform canvas={xshift=-2pt}] node[left] {$f$} (m-3-2);
  			\path				(m-1-1) edge[eq] (m-2-1)
  									(m-2-3) edge[eq, transform canvas={xshift=2pt}] (m-3-2);
        \draw[font=\scriptsize]	($(m-1-2)!0.5!(m-2-2)$) node {$\cart$}
        						([yshift=0.25em]$(m-2-2)!0.5!(m-3-2)$) node {$\cart$};
  		\end{tikzpicture} = \begin{tikzpicture}[textbaseline]
				\matrix(m)[math35, column sep={1.875em,between origins}]{A & & E \\ & E & \\ & C & \\};
				\path[map]	(m-1-1) edge[barred] node[above] {$C(f, \inv g)$} (m-1-3)
														edge[transform canvas={xshift=-2pt}] node[left] {$g \of f$} (m-2-2)
										(m-2-2) edge[bend right=60] node[left] {$\inv g$} (m-3-2)
														edge[bend left=60] node[right] {$\inv g$} (m-3-2);
				\path				(m-1-3) edge[eq, transform canvas={xshift=2pt}] (m-2-2)
										(m-1-2) edge[cell, transform canvas={yshift=0.25em}] node[right] {$\phi$} (m-2-2)
										(m-2-2) edge[cell, transform canvas={xshift=-0.9em}] node[right, inner sep=2.5pt] {$\id_{\inv g}$} (m-3-2);
			\end{tikzpicture}
		\end{displaymath}
		Since the vertical identity cell for $\inv g$ is cartesian (\exref{isomorphisms have cartesian identity cells}) the left-hand side factors uniquely as a cell $\phi$ as shown. Applying the pasting lemma (\lemref{pasting lemma for cartesian cells}) we find that $\phi$ is cartesian, thus defining $C(f, \inv g)$ as the companion of $g \of f$.
	\end{proof}
	
	Recall that the objects, vertical morphisms and vertical cells of any augmented virtual double category $\K$ form a $2$-category $V(\K)$. The next lemma reformulates the notion of adjunction in $V(\K)$ in terms of companions and conjoints in $\K$.
	\begin{lemma} \label{adjunctions}
		In an augmented virtual double category $\K$ let $\map fAC$ be a vertical morphism whose companion $f_*$ exists. Consider vertical cells $\eta$ and $\eps$ below as well as their factorisations through $f_*$, as shown.
		\begin{displaymath}

		\end{displaymath}
		The following are equivalent:
		\begin{enumerate}[label=\textup{(\alph*)}]
			\item	$(\eta, \eps)$ defines an adjunction $f \ladj g$ in $V(\K)$;
			\item $(\eta', \eps')$ forms a pair that defines $f_*$ as the conjoint of $g$ in $\K$.
		\end{enumerate}
	\end{lemma}
	\begin{proof}
		Notice that condition (b) is equivalent to $\eta'$ and $\eps'$ satisfying the conjoint identities $\eta' \of \eps' = \id_g$ and $\eta' \hc \eps' = \id_{f_*}$, by the horizontal dual of the previous lemma; we claim that these correspond to the two triangle identities for $\eta$ and $\eps$. Indeed we have
		\begin{align*}
			(f \of \eta) \hc (\eps \of f) = \id_f \; &\Leftrightarrow \; (f \of \eta' \of \cocart) \hc (\cart \of \eps' \of f) = \id_f \\
			&\Leftrightarrow \; \cart \of (\eta' \hc \eps') \of \cocart = \id_f \; \Leftrightarrow \; \eta' \hc \eps' = \id_{f_*},
		\end{align*}
		where the second equivalence follows from one of the interchange axioms, and the third from the vertical companion identity $\cart \of \id_{f_*} \of \cocart = \id_f$ together with the fact that factorisations through (co-)cartesian cells are unique. Likewise
		\begin{align*}
			(\eta \of g) \hc (g \of \eps) = \id_g \; &\Leftrightarrow \; (\eta' \of \cocart \of g) \hc (g \of \cart \of \eps') = \id_g \\
			&\Leftrightarrow \; \eta' \of (\cocart \hc \cart) \of \eps' = \id_g \; \Leftrightarrow \; \eta' \of \eps' = \id_g,
		\end{align*}
		where we used the horizontal companion identity.
	\end{proof}
		
	\subsection{Representability}
	In this final subsection we consider the (op-)representability of horizontal morphisms, in the sense below. Our aim is to characterise the sub-augmented virtual double categories of representable and oprepresentable horizontal morphisms, that are contained in any augmented virtual double category $\K$, in terms of strict double category $(Q \of V)(\K)$ of `quintets' in the vertical $2$-category $V(\K)$.
	
	\begin{definition} \label{representable horizontal morphism}
		A vertical morphism $\map jAB$ is said to \emph{represent} the horizontal morphism $\hmap JAB$ if there exists a cartesian cell as on the left below, that is $J$ forms the companion of $j$; in this case we say that $J$ is \emph{representable}. Horizontally dual, $J$ is called \emph{oprepresentable} whenever there exists a cartesian cell as on the right.
		\begin{displaymath}
			\begin{tikzpicture}[baseline]
				\matrix(m)[math35, column sep={1.75em,between origins}]{A & & B \\ & B & \\};
				\path[map]	(m-1-1) edge[barred] node[above] {$J$} (m-1-3)
														edge[transform canvas={xshift=-1pt}] node[left] {$j$} (m-2-2);
				\path				(m-1-3) edge[transform canvas={xshift=2pt}, eq] (m-2-2);
				\draw[font=\scriptsize]	([yshift=0.333em]$(m-1-2)!0.5!(m-2-2)$) node {$\cart$};
			\end{tikzpicture} \qquad\qquad\qquad\qquad \begin{tikzpicture}[baseline]
				\matrix(m)[math35, column sep={1.75em,between origins}]{A & & B \\ & A & \\};
				\path[map]	(m-1-1) edge[barred] node[above] {$J$} (m-1-3)
										(m-1-3)	edge[transform canvas={xshift=1pt}] node[right] {$h$} (m-2-2);
				\path				(m-1-1) edge[eq, transform canvas={xshift=-2pt}] (m-2-2);
				\draw				([yshift=0.333em]$(m-1-2)!0.5!(m-2-2)$) node[font=\scriptsize] {$\cart$};
			\end{tikzpicture}
    \end{displaymath}
	\end{definition}
	
	For an augmented virtual double category $\K$ we write $\Rep(\K) \subseteq \K$ for the sub"/augmented virtual double category that consists of all objects, all vertical morphisms, the representable horizontal morphisms, and all cells between those. The subcategory $\opRep(\K)$ generated by the oprepresentable horizontal morphisms is defined analogously; notice that $\opRep(\K) = \co{\pars{\Rep(\co\K)}}$. Because functors of augmented virtual double categories preserve companions and conjoints (\corref{functors preserve companions and conjoints}), they preserve (op-)representable horizontal morphisms as well; whence the following.
	\begin{proposition} \label{2-functor Rep}
		The assignments $\K \mapsto \Rep(\K)$ and $\K \mapsto \opRep(\K)$ extend to strict $2$-endofunctors $\Rep$ and $\opRep$ on $\AugVirtDblCat$.
	\end{proposition}
	
	In \cite{Ehresmann63} Ehresmann defined, inside any $2$-category $\mathcal C$, a \emph{quintet} to be a cell of the form as on the left below.
	
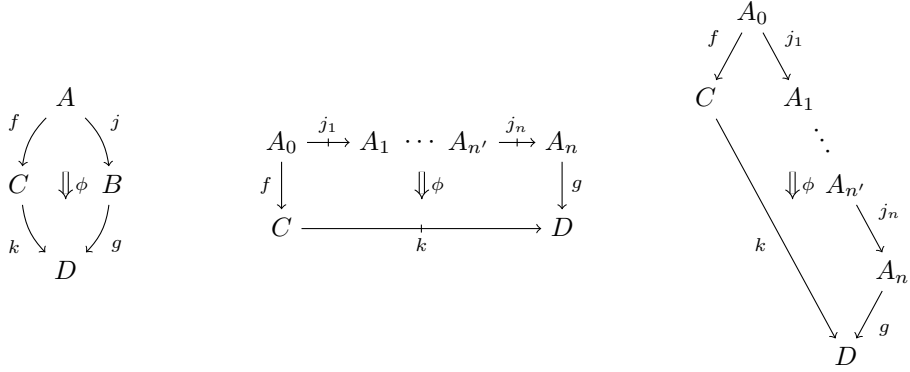
\begin{displaymath}
		\begin{tikzpicture}[baseline]
			\matrix(m)[math35, column sep={1.75em,between origins}]{& A & \\ C & & B \\ & D & \\};
			\path[map]	(m-1-2) edge[bend right=18] node[above left] {$f$} (m-2-1)
													edge[bend left=18] node[above right] {$j$} (m-2-3)
									(m-2-1) edge[bend right=18] node[below left] {$k$} (m-3-2)
									(m-2-3) edge[bend left=18] node[below right] {$g$} (m-3-2);
			\path				(m-1-2) edge[cell, transform canvas={yshift=-1.625em}] node[right] {$\phi$} (m-2-2);
		\end{tikzpicture}	\qquad\qquad \begin{tikzpicture}[baseline]
				\matrix(m)[math35]{A_0 & A_1 & A_{n'} & A_n \\ C & & & D \\};
				\path[map]	(m-1-1) edge[barred] node[above] {$j_1$} (m-1-2)
														edge node[left] {$f$} (m-2-1)
										(m-1-3) edge[barred] node[above] {$j_n$} (m-1-4)
										(m-1-4) edge node[right] {$g$} (m-2-4)
										(m-2-1) edge[barred] node[below] {$k$} (m-2-4);
				\path[transform canvas={xshift=1.75em}]	(m-1-2) edge[cell] node[right] {$\phi$} (m-2-2);
				\draw				($(m-1-2)!0.5!(m-1-3)$) node {$\dotsb$};
			\end{tikzpicture} \qquad\quad \begin{tikzpicture}[baseline]
				\matrix(m)[math35, column sep={1.75em,between origins}]
					{ & A_0 & & & \\
						C & & A_1 & & \\
						& & & A_{n'} &\\
						& & & & A_n \\
						& & & D & \\ };
				\path[map]	(m-1-2) edge node[above left] {$f$} (m-2-1)
														edge node[above right] {$j_1$} (m-2-3)
										(m-2-1) edge node[below left] {$k$} (m-5-4)
										(m-3-4) edge node[above right] {$j_n$} (m-4-5)
										(m-4-5) edge node[below right] {$g$} (m-5-4);
				\path				(m-2-3) edge[cell, transform canvas={shift={(-0.25em,-1.625em)}}] node[right] {$\phi$} (m-3-3)
										(m-2-3)	edge[white] node[sloped, black] {$\dotsb$} (m-3-4);
			\end{tikzpicture}
	\end{displaymath}
	While the quintets of $\mathcal C$ most naturally arrange as a `strict double category' (see \exref{strict double category of quintets} below), for now we will think of them as forming an augmented virtual double category $Q(\mathcal C)$ as follows.
	\begin{definition} \label{quintets}
		Let $\mathcal C$ be a $2$-category. The augmented virtual double category $Q(\mathcal C)$ of \emph{quintets in $\mathcal C$} has as objects those of $\mathcal C$, while both its vertical and horizontal morphisms are morphisms in $\mathcal C$. A unary cell $\phi$ in $Q(\mathcal C)$, as in the middle above, is a cell $\phi$ in $\mathcal C$ as on the right, while the nullary cells of $Q(\mathcal C)$ are cells in $\mathcal C$ as on the right but with $k = \id_C$. Composition in $Q(\mathcal C)$ is induced by that of $\mathcal C$ in the evident way.
			
		We abbreviate $\co Q(\mathcal C) \dfn \co{(Q(\co{\mathcal C}))}$. Thus, to each morphism $\map jAB$ in $\mathcal C$ there is a horizontal morphism $\hmap{\co j}BA$ in $\co Q(\mathcal C)$, and to each cell $\phi$ as on the left below there is a unary cell $\co\phi$ in $\co Q(\mathcal C)$ as on the right.
		
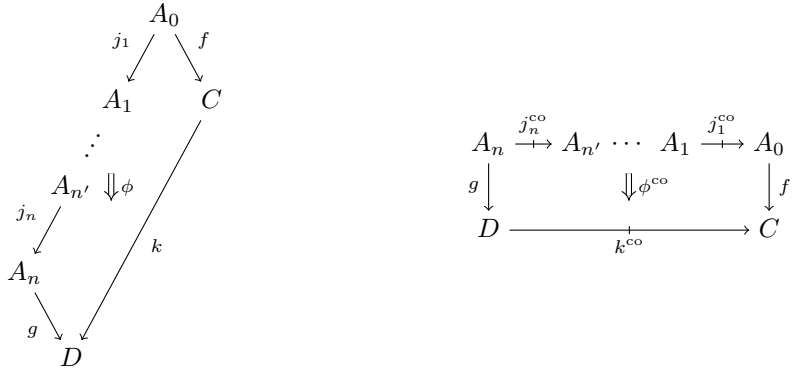
\begin{displaymath}
			\begin{tikzpicture}[baseline]
				\matrix(m)[math35, column sep={1.75em,between origins}]
					{ & & & A_0 & \\
						& & A_1 & & C \\
						& A_{n'} & & & \\
						A_n & & & & \\
						& D & & & \\ };
				\path[map]	(m-1-4) edge node[above right] {$f$} (m-2-5)
														edge node[above left] {$j_1$} (m-2-3)
										(m-2-5) edge node[below right] {$k$} (m-5-2)
										(m-3-2) edge node[above left] {$j_n$} (m-4-1)
										(m-4-1) edge node[below left] {$g$} (m-5-2);
				\path				(m-2-3) edge[cell, transform canvas={shift={(-0.25em,-1.625em)}}] node[right] {$\phi$} (m-3-3)
										(m-2-3)	edge[white] node[sloped, black] {$\dotsb$} (m-3-2);
			\end{tikzpicture} \qquad\qquad\qquad\qquad \begin{tikzpicture}[baseline]
				\matrix(m)[math35]{A_n & A_{n'} & A_1 & A_0 \\ D & & & C \\};
				\path[map]	(m-1-1) edge[barred] node[above] {$\co j_n$} (m-1-2)
														edge node[left] {$g$} (m-2-1)
										(m-1-3) edge[barred] node[above] {$\co j_1$} (m-1-4)
										(m-1-4) edge node[right] {$f$} (m-2-4)
										(m-2-1) edge[barred] node[below] {$\co k$} (m-2-4);
				\path[transform canvas={xshift=1.75em}]	(m-1-2) edge[cell] node[right] {$\co\phi$} (m-2-2);
				\draw				($(m-1-2)!0.5!(m-1-3)$) node {$\dotsb$};
			\end{tikzpicture}
		\end{displaymath}
	\end{definition}
	
	\begin{proposition} \label{2-functor Q}
		The assignments $\mathcal C \mapsto Q(\mathcal C)$ and $\mathcal C \mapsto \co Q(\mathcal C)$ above extend to strict $2$-functors $\map Q{\twoCat}{\AugVirtDblCat}$ and $\map{\co Q}{\twoCat}{\AugVirtDblCat}$.
	\end{proposition}
	\begin{proof}
		The image $\map{QF}{Q\mathcal C}{Q\mathcal D}$ of a strict $2$-functor $\map F{\mathcal C}{\mathcal D}$ is simply given by letting $F$ act on objects, morphisms and cells. The image $\nat{Q\xi}{QF}{QG}$ of a $2$"/natural transformation $\nat\xi FG$ is given by $(Q\xi)_A \dfn \xi_A$ on objects, while the cell $\cell{(Q\xi)_j}{Fj}{Gj}$ in $Q(\mathcal D)$, where $\map jAB$, is the quintet given by the naturality square $Gj \of \xi_A = \xi_B \of Fj$. Finally $\mathcal C \mapsto \co Q(\mathcal C)$ is extended by the composite of strict $2$-functors $\co Q \dfn \co{(\dash)} \of Q \of \co{(\dash)}$.
	\end{proof}
	
	Remember that any augmented virtual double category $\K$ contains a $2$-category $V(\K)$ of vertical morphisms and vertical cells. We denote by $(Q \of V)_*(\K) \subseteq (Q \of V)(\K)$ the sub-augmented virtual double category generated by all vertical morphisms, the horizontal morphisms $\hmap jAB$ that admit companions in $\K$, and all quintets between them. Notice that this extends to a sub-2-endofunctor $(Q \of V)_* \subseteq Q \of V$ on $\AugVirtDblCat$, because functors between augmented virtual double categories preserve cartesian cells that define companions (\corref{functors preserve companions and conjoints}). The sub-2-endofunctor $(\co Q \of V)^*$ is defined likewise, by mapping each $\K$ to the sub-augmented virtual double category $(\co Q \of V)^*(\K) \subseteq (\co Q \of V)(\K)$ that is generated by horizontal morphisms $\hmap{\co j}BA$ that correspond to vertical morphism $\map jAB$ that admit conjoints in $\K$.
	\begin{theorem} \label{lower star}
		Let $\K$ be an augmented virtual double category. Choosing, for each \mbox{$\hmap jAB$} in $(Q \of V)_*(\K)$, a cartesian cell $\eps_j$ that defines the companion of $j$ in $\K$, induces an equivalence 
		\begin{displaymath}
			(\dash)_*\colon (Q \of V)_*(\K) \simeq \Rep(\K)
		\end{displaymath}
		of augmented virtual double categories as follows. Restricting to the identity on objects and vertical morphisms, it maps each horizontal morphism $\hmap jAB$ in $(Q \of V)_*(K)$ to its chosen companion $j_*$, while a cell $\phi$ of $(Q \of V)_*(K)$, as in the left-hand side below, is mapped to the unique factorisation $\phi_*$ as shown; here $\ul{\eps_k} \dfn \eps_k$ if $\phi$ is unary and $\ul{\eps_k} \dfn \id_C$ otherwise.
		\begin{equation} \label{lower star on quintets}

		\end{equation}
		Letting $\K$ vary, these functors combine to form a pseudonatural transformation $\nat{(\dash)_*}{(Q \of V)_*}\Rep$ of strict $2$-endofunctors on $\AugVirtDblCat$.
		
		Analogously, choosing cartesian cells that define conjoints induces an equivalence $(\co Q \of V)^*(\K)\simeq\opRep(\K)$. Their underlying functors too combine to form a pseudonatural transformation $\nat{(\dash)^*}{(\co Q \of V)^*}\opRep$.
	\end{theorem}
	\begin{proof}
		We will construct the functors $\map{(\dash)_*}{(Q \of V)_*(\K)}{\Rep(\K)}$; show that they are full, faithful and essentially surjective, so that they are part of equivalences by \propref{equivalences}; and prove that they are pseudonatural in $\K$. The functors $\map{(\dash)^*}{(\co Q \of V)^*(\K)}{\opRep(\K)}$ can then be defined as the composites $(\dash)^* \dfn \co{(\dash)} \of (\dash)_* \of \co{(\dash)}$, where we use that companions in $\co\K$ correspond to conjoints in $\K$, so that $\co{\pars{(Q \of V)_*(\co\K)}} = (\co Q \of V)^*(\K)$ and $\co{\pars{\Rep(\co\K)}} = \opRep(\K)$.
		
		It is clear that $\phi \mapsto \phi_*$ preserves identities. To see that it preserves composites $\psi \of (\phi_1, \dotsc, \phi_n)$ too consider the following equation, where we denoted all cartesian cells that define the chosen companions simpy by `$\eps$'. Its identities follow from the identity above for $\psi$, as well as for $\phi_1, \dotsc, \phi_n$, and the definition of composition in $(Q \of V)(\K)$. We conclude that $\psi_* \of (\phi_{1*}, \dotsc, \phi_{n*})$ and $\bigpars{\psi \of (\phi_1, \dotsc, \phi_n)}_*$ coincide after composition with the cartesian cell that defines the horizontal target of $\psi_*$. By uniqueness of factorisations through cartesian cells we conclude that $\psi \of (\phi_1, \dotsc, \phi_n)$ is preserved by $(\dash)_*$.
		\begin{multline*}

		\end{multline*}
		
		To prove that $(\dash)_*$ is part of an equivalence we will show that it is full, faithfull and essentially surjective, and then apply \propref{equivalences}. That it is essentially surjective and full and faithful on vertical morphisms is clear. It remains to show that it is locally full and faithful, that is full and faithful on cells. To see this we denote, for each $\hmap jAB$ in $(Q \of V)_*(\K)$, by $\eta_j$ the weakly cocartesian cell that corresponds to $\eps_j$, such that the pair $(\eps_j, \eta_j)$ satisfies the companion identities; see \lemref{companion identities lemma}. To show faithfulness, consider cells $\phi$ and $\cell\psi{\ul j}{\ul k}$ in $(Q \of V)_*(\K)$ such that $\phi_* = \psi_*$. It follows that the left-hand sides of \eqref{lower star on quintets} coincide for $\phi$ and $\psi$ so that, by precomposing both with $(\eta_{j_1}, \dotsc, \eta_{j_n})$, $\phi = \psi$ follows from the vertical companion identities. To show fullness on unary cells, consider \mbox{$\cell\psi{(j_{1*}, \dotsc, j_{n*})}{k_*}$} in \mbox{$(Q \of V)_*(\K)$}. We claim that the composite
		\begin{displaymath}
			\phi \dfn \eps_k \of \psi \of (\eta_{j_1}', \dotsc, \eta_{j_n}'),
		\end{displaymath}
		where $\eta_{j_i}' \dfn \eta_{j_i} \of j_{i'} \of \dotsb \of j_1$ for each $i = 1, \dotsc, n$, is mapped to $\psi$ by $(\dash)_*$. Indeed, plugging $\phi$ into the left-hand side of \eqref{lower star on quintets} we find $\eps_k \of \psi = \eps_k \of \phi_*$, by using the horizontal companion identities, so that $\psi = \phi_*$ follows. The case of $\psi$ nullary is similar; simply take $\phi \dfn \psi \of (\eta_{j_1}', \dotsc, \eta_{j_n}')$ instead.
		
		We now turn to proving that the functors $(\dash)_*$ combine to form a pseudonatural transformation $(Q \of V)_* \Rar \Rep$, of strict $2$-endofunctors on $\AugVirtDblCat$. This means that we have to give an invertible transformation $\nu_F$ as on the left below, for each functor $\map F\K\L$ of augmented virtual double categories. We take $\nu_F$ to consist of identities $(\nu_F)_A = \id_{FA}$ on objects and, for each $\hmap jAB$ in $(Q \of V)_*(\K)$, the unique factorisation $\cell{(\nu_F)_j}{F(j_*)}{(Fj)_*}$ as on the right below. The latter is invertible since $F\eps_j$, on the left-hand side, is cartesian by \corref{functors preserve companions and conjoints}.
		\begin{displaymath}
			\begin{tikzpicture}[textbaseline]
				\matrix(m)[math35, column sep=1.75em]{(Q \of V)_* (\K) & \Rep(\K) \\ (Q \of V)_* (\L) & \Rep(\L) \\};
				\path[map]	(m-1-1) edge node[above] {$(\dash)_*$} (m-1-2)
														edge node[left] {$(Q \of V)_*(F)$} (m-2-1)
										(m-1-2) edge node[right] {$\Rep(F)$} (m-2-2)
										(m-2-1) edge node[below] {$(\dash)_*$} (m-2-2);
				\path				(m-1-2) edge[cell, shorten >= 1.5em, shorten <= 1.5em] node[below right] {$\nu_F$} (m-2-1);				
			\end{tikzpicture} \qquad \begin{tikzpicture}[textbaseline]
				\matrix(m)[math35, column sep={1.75em,between origins}]{FA & & FB \\ & FB & \\};
				\path[map]	(m-1-1) edge[barred] node[above] {$F(j_*)$} (m-1-3)
														edge[transform canvas={xshift=-2pt}] node[left] {$Fj$} (m-2-2);
				\path				(m-1-3) edge[eq, transform canvas={xshift=1pt}] (m-2-2);
				\path[transform canvas={shift={(-0.75em,0.25em)}}]	(m-1-2) edge[cell] node[right, inner sep=2.5pt] {$F\eps_j$} (m-2-2);
			\end{tikzpicture} = \begin{tikzpicture}[textbaseline]
  			\matrix(m)[math35, column sep={1.75em,between origins}]{FA & & FB \\ FA & & FB \\ & FB & \\};
  			\path[map]	(m-1-1) edge[barred] node[above] {$F(j_*)$} (m-1-3)
  									(m-2-1) edge[barred] node[above] {$(Fj)_*$} (m-2-3)
  													edge[transform canvas={xshift=-2pt}] node[left] {$Fj$} (m-3-2);
  			\path				(m-1-1) edge[eq] (m-2-1)
  									(m-1-3) edge[eq] (m-2-3)
  									(m-2-3) edge[eq, transform canvas={xshift=1pt}] (m-3-2);
  			\path				(m-1-2) edge[cell, transform canvas={shift={(-0.95em,0.25em)}}] node[right] {$(\nu_F)_j$} (m-2-2)
  									(m-2-2) edge[cell, transform canvas={shift={(-0.75em,0.25em)}}] node[right] {$\eps_{Fj}$} (m-3-2);
  		\end{tikzpicture}
		\end{displaymath}
		We have to show that the components of $\nu_F$ are natural with respect to the cells of $(Q \of V)_*(\K)$, in the sense of \defref{transformation}. We will do so in case of a unary cell $\cell\phi{(j_1, \dotsc, j_n)}k$; the case of nullary cells is similar. Consider the following equation, where $\eps_{Fj_i}' \dfn Fg \of Fj_n \of \dotsb \of Fj_{i+1} \of \eps_{Fj_i}$ and $\eps_{j_i}' = g \of j_n \of \dotsb \of j_{i+1} \of \eps_{j_i}$ for each $i = 1, \dotsc, n$, as in the left-hand side of \eqref{lower star on quintets}. The identities follow from \eqref{lower star on quintets} for $F\phi$, the identity above, $F$ preserves composition, the $F$-image of \eqref{lower star on quintets} for $\phi$ and the identity above again. Since factorisations through $\eps_{Fk}$, in the left and right-hand side below, are unique, we conclude that the components of $\nu_F$ are natural with respect to $\phi$. This completes the definition of the transformation $\nu_F$.
		\begin{align*}
			\eps_{Fk} \of (&F\phi)_* \of \bigpars{(\nu_F)_{j_1}, \dotsc, (\nu_F)_{j_n}} \\
			&= (F\phi \hc \eps_{Fj_1}' \hc \dotsb \hc \eps_{Fj_n}') \of \bigpars{(\nu_F)_{j_1}, \dotsc, (\nu_F)_{j_n}} = F\phi \hc F\eps_{j_1}' \dotsb \hc F\eps_{j_n}' \\
			&= F(\phi \hc \eps_{j_1}' \hc \dotsb \hc \eps_{j_n}') = F(\eps_k \of \phi_*) = \eps_{Fk} \of (\nu_F)_k \of F(\phi_*)
		\end{align*}
		
		Finally we have to show that the transformations $\nu_F$ are natural with respect to the transformations $\nat\xi FG$ in $\AugVirtDblCat$, and that they are compatible with compositions and identities, that is $\nu_{\id} = \id$ and $\nu_GF \of G\nu_F = \nu_{G \of F}$. Since the latter is straightforward to prove, we will only prove the former. Thus, for each $\hmap jAB$ in $(Q \of V)_*(\K)$, we have to show that $(\nu_G)_j \of \xi_{(j_*)} = (\xi_j)_* \of (\nu_F)_j$. Consider the equation
		\begin{displaymath}
			\eps_{Gj} \of (\nu_G)_j \of \xi_{(j_*)} = G\eps_j \of \xi_{(j_*)} = \xi_B \of F\eps_j = \xi_B \of \eps_{Fj} \of (\nu_F)_j = \eps_{Gj} \of (\xi_j)_* \of (\nu_F)_j,
		\end{displaymath}
		where we have used the defining identities for $(\nu_G)_j$ and $(\nu_F)_j$, the naturality of $\xi$, identity \eqref{lower star on quintets} for $\xi_j$, and the fact that the latter is simply the quintet given by the naturality square $Gj \of \xi_A = \xi_B \of Fj$; see the proof of \propref{2-functor Q}. Using the cartesianess of $\eps_{Gj}$ we conclude that $(\nu_G)_j \of \xi_{(j_*)} = (\xi_j)_* \of (\nu_F)_j$, proving the naturality of the transformations $\nu_F$. This concludes the proof.
	\end{proof}
	
	\section{Composition of horizontal morphisms}\label{composition section}
	We now turn to compositions of horizontal morphisms in augmented virtual double categories, as well as units for such compositions. Analogous to the situation for virtual double categories (see Section 2 of \cite{Dawson-Pare-Pronk06} or Section 5 of \cite{Cruttwell-Shulman10}), these are defined by horizontal cells that satisfy a stronger variant of the universal property for weakly cocartesian cells.
	
	\subsection{Cocartesian paths of cells}
	We start with both strengthening, as well as extending to paths, \defref{cartesian cells} of weakly cocartesian cell.
	\begin{definition} \label{cocartesian paths}
		A path of unary cells $(\phi_1, \dotsc, \phi_n)$, as in the right-hand side below, is called \emph{weakly cocartesian} if any cell $\psi$, as on the left-hand side, factors uniquely through $(\phi_1, \dotsc, \phi_n)$ as shown.
		\begin{multline*}

		\end{multline*}
		If moreover the following conditions are satisfied then the path $(\phi_1, \dotsc, \phi_n)$ is called \emph{cocartesian}.
		\begin{enumerate}[label=\textup{(\alph*)}]
			\item The restrictions $J'(\id, f_0)$ and $J''(f_n, \id)$ exist for all horizontal morphisms $\hmap{J'}{A'}{C_0}$ and $\hmap{J''}{C_n}{A''}$;
			\item any path of the form below, where $p, q \geq 1$, is weakly cocartesian.
		\end{enumerate}
		\begin{displaymath}
			%
		\end{displaymath}
	\end{definition}
	If a cocartesian horizontal cell of the form
	\begin{displaymath}
		%
	\end{displaymath}
	exists with $n \geq 1$ then we call $K$ the \emph{(horizontal) composite} of $(J_1, \dotsc, J_n)$ and write $(J_1 \hc \dotsb \hc J_n) \dfn K$; in the case that $n = 0$ we call $K$ the \emph{(horizontal) unit} of $A_0$ and write $I_{A_0} \dfn K$, while we call $A_0$ \emph{unital}. By their universal property any two cocartesian horizontal cells defining the same composite or unit factor through each other as invertible horizontal cells. Like weakly cocartesian cells, we shall denote single cocartesian cells in our drawings simply by ``cocart''.
	
	\begin{example} \label{horizontal composites in V-Prof}
		Let $\V'$ be a monoidal category whose tensor product $\tens'$ preserves large colimits on both sides. Then the composite of a path $\hmap{(J_1, \dotsc, J_n)}{A_0}{A_n}$ of $\V'$-profunctors exists in $\enProf{\V'}$ as soon as for each $x \in A_0$ and $y \in A_n$ the coend, on the right-hand side below and in the sense of \propref{weakly cocartesian cell in (V, V')-Prof}, exists. 
		\begin{displaymath}
			(J_1 \hc \dotsb \hc J_n)(x, y) \dfn \int^{u_1 \in A_1, \dotsc, u_{n'} \in A_{n'}} J_1(x, u_1) \tens' \dotsb \tens' J_n(u_{n'}, y)
		\end{displaymath}
		In that case, using the assumption on $\tens'$, there is exactly one way of extending the assignment above into a $\V'$-profunctor $\hmap{(J_1 \hc \dotsb \hc J_n)}{A_0}{A_n}$ such that colimiting cocones combine into a horizontal cell $(J_1, \dotsc, J_n) \Rar (J_1 \hc \dotsb \hc J_n)$. In \propref{cocartesian cells in (V, V')-Prof} below we will see that this cell is cocartesian; in fact it will be shown that it defines $(J_1 \hc \dotsb \hc J_n)$ as a `pointwise' horizontal composite, in the sense of \secref{pointwise horizontal composites section}. The notion of pointwise horizontal composite will be important when we consider Kan extensions.
	\end{example}

	\begin{example} \label{horizontal composites in (V, V')-Prof}
		Let $\V \to \V'$ be a universe enlargement (\defref{universe enlargement}) such that the tensor product $\tens'$ of $\V'$ preserves large colimits on both sides. Remember that $\V'$ is assumed to be large cocomplete $\enProf{\V'}$ so that it has all horizontal composites by the previous example. We will see in \propref{cocartesian cells in (V, V')-Prof} below that the inclusion $\enProf{(\V, \V')} \to \enProf{\V'}$ reflects all cocartesian cells so that, for any path $\hmap{(J_1, \dotsc, J_n)}{A_0}{A_n}$ of $\V$-profunctors between $\V'$-categories, the horizontal composite $(J_1 \hc \dotsb \hc J_n)$ exists in $\enProf{(\V, \V')}$ as soon as the $\V'$-coends above are $\V$-objects, for each $x \in A_0$ and $y \in A_n$. This is the case, for instance, when the $\V'$-categories $A_1, \dotsc, A_{n'}$ are small $\V$-categories and $\V$ is small cocomplete.
		
		In \lemref{unit identities} below it is shown that, in general, the unit $I_A$ of an object $A$ coincides with the nullary restriction $A(\id, \id)$ (see \defref{cartesian cells}). Specifically, as an easy consequence of that lemma, a $\V'$-category $A$ is unital in $\enProf{(\V, \V')}$ if and only if it is a $\V$-category; in that case the cocartesian cell $A \Rar I_A$ consists of the unit $\V$-maps $\map{\tilde A}I{A(x, x)}$ of $A$.
	\end{example}
	
	\begin{example} \label{composites of non-expanding metric relations}
		In the case of $\V' = \brks{-\infty, \infty}$ (see \exref{universe enlargement of Lawvere quantale}) the horizontal composite $(J_1 \hc \dotsb \hc J_n)$ of \exref{horizontal composites in V-Prof} is given by the infima
		\begin{displaymath}
			(J_1 \hc \dotsb \hc J_n)(x, y) = \inf_{u_1 \in A_1, \dotsc, u_{n'} \in A_{n'}} J_1(x, u_1) + \dotsb + J_n(u_{n'}, y).
		\end{displaymath}
		Notice that the composite of $\brks{0, \infty}$"/profunctors is again a $\brks{0, \infty}$"/profunctor, so that in the augmented virtual double category $\enProf{(\brks{0, \infty}, \brks{-\infty, \infty})}$ all horizontal compositions exist. A generalised metric space in $\enProf{(\brks{0, \infty}, \brks{-\infty, \infty})}$ is unital precisely if its distance function takes only non"/negative values.
	\end{example}

	\begin{example}
		If $\E$ has reflexive coequalisers preserved by pullback then the augmented virtual double category $\inProf\E$ of profunctors internal to $\E$ (\exref{internal profunctors}) has all horizontal composites. The composite of internal profunctors is an ``internal coend''.
	\end{example}
	
	\begin{example}
		Given a path of small $\Ss$-indexed profunctors \mbox{$A_0 \xbrar{J_1} A_1 \dotsb A_{n'} \xbrar{J_n} A_n$} in $\enProf{(\Set, \Set')}^\Ss$ (see \exref{indexed profunctors}), it is easily checked that their composite \mbox{$(J_1 \hc \dotsb \hc J_n)$} can be defined indexwise; that is it exists whenever, for each $s \in \Ss$, the composite $(J_{1s} \hc \dotsb \hc J_{ns})$ exists in $\enProf{(\Set, \Set')}$. Combining \exref{necessary condition for the existence of companions and conjoints} and \lemref{unit identities} shows that an $\Ss$-indexed category $A$ is unital precisely if $A_s$ is locally small for each $s \in \Ss$.
	\end{example}
	
	Because the notion of weakly cocartesian cell in augmented virtual double categories restricts to the corresponding notion for virtual double categories, so do the notions of horizontal composite and horizontal unit restrict to the corresponding notions for virtual double categories.
	\begin{example} \label{monoids have units}
		Let $\K$ be a virtual double category. It was shown in Proposition 5.5 of \cite{Cruttwell-Shulman10} that the virtual double category $\Mod(\K)$ of monoids and bimodules in $\K$ (see \defref{monoids and bimodules}) has all units. More precisely, the unit bimodule $\hmap{I_A}AA$ of a monoid $A = (A, \alpha, \bar \alpha, \tilde \alpha)$ is simply $I_A \dfn \alpha$ equipped with left and right actions given by the multiplication $\cell{\bar \alpha}{(\alpha, \alpha)}\alpha$ of $A$, while the unit cell $\tilde\alpha$ forms the cocartesian cell $A \Rar I_A$ in $\Mod(\K)$.
	\end{example}
	
	(Weakly) cocartesian paths, like cartesian cells, satisfy a pasting lemma as follows. For a proof of the cocartesian case use the pasting lemma for cartesian cells (\lemref{pasting lemma for cartesian cells}).
	\begin{lemma}[Pasting lemma] \label{pasting lemma for cocartesian paths}
		Consider a configuration of unary cells below, where the vertical source of $\psi_1$ is denoted by $\map{h_0}{C_{10}}{E_0}$ and the vertical target of $\psi_n$ by $\map{h_n}{C_{nm_n}}{E_n}$.
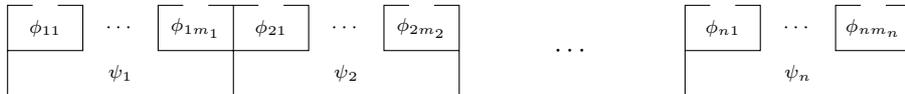
		\begin{displaymath}
			\begin{tikzpicture}[x=0.33cm, y=0.6cm, font=\scriptsize]
				\draw	(1,2) -- (0,2) -- (0,0) -- (18,0) -- (18,2) -- (17,2)
							(2,2) -- (3,2) -- (3,1) -- (0,1)
							(7,2) -- (6,2) -- (6,1) -- (12,1) -- (12,2) -- (11,2)
							(8,2) -- (10,2)
							(9,2) -- (9,0)
							(16,2) -- (15,2) -- (15,1) -- (18,1)
							(28,2) -- (27,2) -- (27,0) -- (36,0) -- (36,2) -- (35,2)
							(29,2) -- (30,2) -- (30,1) -- (27,1)
							(34,2) -- (33,2) -- (33,1) -- (36,1);
				\draw[shift={(0.165cm, 0.3cm)}]	(1,1) node {$\phi_{11}$}
							(7,1) node {$\phi_{1m_1}$}
							(10,1) node {$\phi_{21}$}
							(16,1) node {$\phi_{2m_2}$}
							(28,1) node {$\phi_{n1}$}
							(34,1) node {$\phi_{nm_n}$}
							(4,0) node {$\psi_1$}
							(13,0) node {$\psi_2$}
							(31,0) node {$\psi_n$}
							(4,1) node {$\dotsb$}
							(13,1) node {$\dotsb$}
							(31,1) node {$\dotsb$};
				\draw[font=]	(22.5,1) node {$\dotsb$};
			\end{tikzpicture}
		\end{displaymath}
		First assume that the path $(\phi_{11}, \dotsc, \phi_{nm_n})$ is weakly cocartesian. Then the path $\bigpars{\psi_1 \of (\phi_{11}, \dotsc, \phi_{1m_1}), \dotsc, \psi_n \of (\phi_{n1}, \dotsc, \phi_{nm_n})}$ is weakly cocartesian if and only if $(\psi_1, \dotsc, \psi_n)$ is.
		
		Secondly assume that the path $(\phi_{11}, \dotsc, \phi_{nm_n})$ is cocartesian and that, for any horizontal morphisms $\hmap{J'}{A'}{E_0}$ and $\hmap{J''}{E_n}{A''}$, the restrictions $J'(\id, h_0)$ and $J''(h_n, \id)$ exist. Then the path $\bigpars{\psi_1 \of (\phi_{11}, \dotsc, \phi_{1m_1}), \dotsc, \psi_n \of (\phi_{n1}, \dotsc, \phi_{nm_n})}$ is cocartesian if and only if $(\psi_1, \dotsc, \psi_n)$ is.
	\end{lemma}
	
	By applying the pasting lemma to compositions $\psi \of (\phi_1, \dotsc, \phi_n)$ of horizontal cells we find that the collection of horizontal composites and horizontal units in any augmented virtual double category is coherent, in the sense that if all of the composites $(J_{11} \hc \dotsb \hc J_{1m_1})$, \dots, $(J_{n1} \hc \dotsb \hc J_{nm_n})$ exist, then the composite $(J_{11} \hc \dotsb \hc J_{nm_n})$ exists if and only if
	\begin{displaymath}
		\bigpars{(J_{11} \hc \dotsb \hc J_{1m_1}) \hc \dotsb \hc (J_{n1} \hc \dotsb \hc J_{nm_n})}
	\end{displaymath}
	does, in which case they are canonically isomorphic. Notice that this also includes isomorphisms of the form $(I_A \hc J) \iso J \iso (J \hc I_B)$, for any $\hmap JAB$, and similar.	
	
	In \secref{unital fc-multicategories section} below we will see that the notions of augmented virtual double category and virtual double category coincide as soon as all horizontal units exist. Consequently we do not distinguish between these notions and refer to either as a \emph{unital virtual double category}. Furthermore, in \corref{restrictions in terms of units} below it is shown that any unital virtual double category that is also an augmented virtual equipment (\defref{augmented virtual equipment}) admits all (both unary and nullary) restrictions; such a double category, e.g.\ $\enProf\V$ (see \exref{restrictions of V-profunctors}), we will call a \emph{unital virtual equipment}. The latter notion coincides with that of `virtual equipment', that was studied in \cite{Cruttwell-Shulman10}.
	
	We denote by $\VirtDblCat_\textup u$ the full sub-2-category of $\AugVirtDblCat$ consisting of unital virtual double categories; in \corref{functors preserve horizontal units} below we will see that, unlike functors between unital virtual double categories, any functor between unital augmented virtual double categories preserves horizontal units. We follow \cite{Cruttwell-Shulman10} in calling a functor $\map F\K\L$ \emph{strong} if it preserves horizontal composites too; that is, its image of any horizontal cocartesian cell is again cocartesian.
	
	The following is Proposition 5.14 of \cite{Cruttwell-Shulman10}.
	\begin{proposition}[Cruttwell-Shulman] \label{Mod as right pseudo-adjoint}
		Let $\K$ be a virtual double category. Mon\-oids in $\K$, the morphisms and bimodules between them, as well as the unary cells between those, in the sense of \defref{monoids and bimodules}, combine to form a unital virtual double category $\Mod(\K)$. Moreover, the assignment $\K \mapsto \Mod(\K)$ extends to a strict $2$-functor $\map \Mod\VirtDblCat{\VirtDblCat_\textup u}$ that is right pseudo-adjoint to the forgetful $2$"/functor $\VirtDblCat_\textup u \to \VirtDblCat$.
	\end{proposition}
	
	To complete the picture, we will briefly describe the classical notion of `pseudo double category', as introduced by Grandis and Par\'e in the Appendix to \cite{Grandis-Pare99}; see also Section 2 of \cite{Shulman08}. In our terms, a \emph{pseudo double category} is a virtual double category containing $(1,1)$-ary cells only, which is equipped with a horizontal composition
	\begin{align*}
		A \xbrar J B \xbrar H E \mspace{-4.25mu}\qquad&\mapsto\mspace{-4.25mu}\qquad A \xbrar{J \hc H} E; \\
		\begin{tikzpicture}[textbaseline, ampersand replacement=\&]
			\matrix(m)[math35]{A \& B \& E \\ C \& D \& F \\};
			\path[map]	(m-1-1) edge[barred] node[above] {$J$} (m-1-2)
													edge node[left] {$f$} (m-2-1)
									(m-1-2) edge[barred] node[above] {$H$} (m-1-3)
													edge node[right] {$g$} (m-2-2)
									(m-1-3) edge node[right] {$h$} (m-2-3)
									(m-2-1) edge[barred] node[below] {$K$} (m-2-2)
									(m-2-2) edge[barred] node[below] {$L$} (m-2-3);
			\path[transform canvas={xshift=1.75em}]	(m-1-1) edge[cell] node[right] {$\phi$} (m-2-1)
									(m-1-2) edge[cell] node[right] {$\psi$} (m-2-2);
		\end{tikzpicture} \quad&\mapsto\quad \begin{tikzpicture}[textbaseline, ampersand replacement=\&]
			\matrix(m)[math35, column sep={5em,between origins}]{A \& E \\ C \& F, \\};
			\path[map]	(m-1-1) edge[barred] node[above] {$J \hc H$} (m-1-2)
													edge node[left] {$f$} (m-2-1)
									(m-1-2) edge node[right] {$h$} (m-2-2)
									(m-2-1) edge[barred] node[below] {$K \hc L$} (m-2-2);
			\path[transform canvas={xshift=1.5em}]	(m-1-1) edge[cell] node[right] {$\phi \hc \psi$} (m-2-1);
		\end{tikzpicture}
	\end{align*}
	as well as horizontal units $\hmap{I_A}AA$, that come with horizontal coherence cells of the forms $(J \hc H) \hc M \iso J \hc (H \hc M)$, $I_A \hc J \iso J$ and $J \hc I_B \iso J$. A pseudo double category with identity cells as coherence cells is called a \emph{strict} double category.
	
	Any pseudo double category gives rise to a virtual double category which has the same objects and morphisms, while its cells $(J_1, \dotsc, J_n) \Rar K$ are cells $J_1 \hc \dotsb \hc J_n \Rar K$ in the double category. The following result, which is Proposition 2.8 of \cite{Dawson-Pare-Pronk06} and Theorem 5.2 of \cite{Cruttwell-Shulman10} characterises the virtual double categories obtained in this way as unital virtual double categories with all horizontal composites.
	\begin{proposition}[MacG. Dawson, Par\'e and Pronk] \label{pseudo double categories}
		A virtual double category is induced by a pseudo double category if and only if it has all horizontal composites and units.
	\end{proposition}
	
	In view of the proposition above, by a \emph{double category} we shall mean either an augmented virtual double category that has all horizontal composites and units or, equivalently, a pseudo double category in the classical sense. Finally, following \cite{Cruttwell-Shulman10} again, by an \emph{equipment} we shall mean a double category that has all restrictions.
	
	We close this subsection with two examples of strict double categories.
	\begin{example} \label{strict double category of quintets}
		The augmented virtual double category $Q(\catvar C)$ of quintets in a $2$-category $\catvar C$ (\defref{quintets}) is clearly a strict double category: the horizontal composite $(j \hc k)$ of two composable morphisms in $\catvar C$ is simply their composite $k \of j$.
	\end{example}

	\begin{example} \label{relations}
		Let $\2$ denote the `walking arrow'-category $(\bot \to \top)$, and notice that it admits all limits and colimits. Given sets $A$ and $B$, a $\2$"/matrix $\hmap JAB$ (\exref{enriched profunctors}) is simply a relation $J \subseteq A \times B$. When working with relations we will write $xJy$ to mean $(x, y) \in J$. Assuming the cartesian structure $(\wedge, \top)$ on $\2$, where $\wedge$ denotes conjunction, relations form a strict double category $\Rel \coloneqq \Mat\2$ as follows. It consists of sets, functions, relations, and cells that are uniquely determined by their sources and targets: a cell
		\begin{displaymath}
			\begin{tikzpicture}[textbaseline]
				\matrix(m)[math35]{A & B \\ C & D \\};
				\path[map]	(m-1-1) edge[barred] node[above] {$J$} (m-1-2)
														edge node[left] {$f$} (m-2-1)
										(m-1-2) edge node[right] {$g$} (m-2-2)
										(m-2-1) edge[barred] node[below] {$K$} (m-2-2);
				\path				(m-1-1) edge[transform canvas={xshift=1.75em}, cell]	(m-2-1);
			\end{tikzpicture}
		\end{displaymath}
		precisely if $(fx)K(gy)$ for all $xJy$. The composite $J \hc H$ of modular relations $\hmap JAB$ and $\hmap HBE$ is given by the usual composition of relations:
		\begin{displaymath}
			x(J \hc H)z \defeq \bigvee_{y \in B} (xJy \wedge yHz),
		\end{displaymath}
		where $x \in A$, $z \in E$ and $\bigvee$ denotes disjunction. Clearly $(J, H) \mapsto J \hc H$ is associative, while the relations $\hmap{I_A}AA$, defined by $x(I_A)y \defeq x = y$, form units for $\hc$.
		
		Modules in $\Rel$ form again a strict double category, which we denote $\ModRel \coloneqq \enProf\2$. Its objects are \emph{preordered sets} $(X, \leq)$, with reflexive and transitive ordering $\leq$, while its vertical morphisms are monotone functions. A horizontal morphism $\hmap JAB$ in $\enProf\2$ is a \emph{modular relation}, that is $J \subseteq A \times B$ such that $x_1 \leq x_2$, $x_2Jy_1$ and $y_1 \leq y_2$ implies $x_1Jy_2$. Cells and horizontal composites in $\ModRel$ are simply those of $\Rel$, while the unit modular relation $\hmap{I_A}AA$ is defined by $x(I_A)y \defeq x \leq y$.
	\end{example}

	\subsection{Units in terms of restrictions}
	The following lemma implies that, for any object $A$, unit $I_A$ coincides with the restriction $A(\id, \id)$.
	\begin{lemma} \label{unit identities}
		Consider cells $\phi$ and $\psi$ as in the identities below, and assume that either identity holds. The following conditions are equivalent: \textup{(a)} $\psi$ is cartesian; \textup{(b)} both identities hold; \textup{(c)} $\phi$ is weakly cocartesian; \textup{(d)} $\phi$ is cocartesian; \textup{(e)} $\phi$ is cartesian.
		\begin{displaymath}

  	\end{displaymath}
  	Consequently a cell $\phi$, of the form as in the above, is cocartesian precisely if it is cartesian.
	\end{lemma}
	\begin{proof}
		We claim that, under the assumption of the identity (A), the implications (a) $\Leftrightarrow$ (b) $\Leftrightarrow$ (c) $\Leftrightarrow$ (d) $\Rightarrow$ (e) hold. To see this first notice that the interchange axioms (\lemref{horizontal composition}) imply $\phi \of \psi = \phi \hc \psi$, so that (a), (b) and (c) are equivalent by \lemref{companion identities lemma}. Clearly (d) $\Rightarrow$ (c), while (a) $\Rightarrow$ (e) follows from applying the pasting lemma to (A). Thus it suffices to show that (b) $\Rightarrow$ (d). To do so consider any cell $\cell \chi{\ul J}{\ul K}$, where $\ul J = (A_0 \xbrar{J_1} A_1 \dotsb A_{i''} \xbrar{J_{i'}} A \xbrar{J_i} A_i \dotsb A_{n'} \xbrar{J_n} A_n)$; we have to show that it factors uniquely through $\phi$ as in \defref{cocartesian paths}. Assuming (b), it is easily seen that this factorisation is given by $\chi' \dfn \chi \of (\id_{J_1}, \dotsc, \id_{J_{i'}}, \psi, \id_{J_i}, \dotsc, \id_{J_n})$: indeed, that $\chi'$ composed with $\phi$ gives back $\chi$ follows from (A), while uniqueness of $\chi'$ follows from (J).
		
		Next we show that under (e) the identities (A) and (J) are equivalent. If (e) holds, that is $\phi$ is cartesian, then there exists a unique cell $\psi'$ such that $\id_J = \phi \of \psi'$. Because $\phi \of \psi' \of \phi = \phi$ and $\phi$ is cartesian, $\psi' \of \phi = \id_A$ follows. If (A) holds then $\psi = \psi \of \phi \of \psi' = \psi'$ follows, so that $\id_J = \phi \of \psi' = \phi \of \psi$. On the other hand if (J) then $\psi = \psi' \of \phi \of \psi = \psi'$, so that $\id_A = \psi' \of \phi = \psi \of \phi$. We have shown that (A)~$\Leftrightarrow$~(J) under assumption of (e).
		
		From the above we conclude that all five conditions are equivalent as soon as (A) holds. We complete the proof by showing that (J) $\Rightarrow$ (A) under each condition. Assume (J). If (a) holds then $\id_A$ factors as $\id_A = \psi \of \phi'$; hence $\phi = \phi \of \psi \of \phi' = \phi'$ so that (A) follows. If (c) or (d) holds then $\id_A$ factors as $\id_A = \psi' \of \phi$; hence $\psi = \psi' \of \phi \of \psi = \psi'$ so that (A) follows. Of course (b) implies (A), and we have already shown that (A) $\Leftrightarrow$ (J) under (e). This completes the proof of the main assertion.
		
		For the final assertion notice that if $\phi$ is cocartesian then we can obtain a factorisation of the form (A), while if it is cartesian then we can obtain one of the form (J), so that the equivalence follows from applying the main assertion.
	\end{proof}
	
	The following  corollary shows that an augmented virtual double category has all restrictions whenever it has all unary restrictions and horizontal units.
	\begin{corollary} \label{restrictions in terms of units}
		Consider a unital object $C$ and let $\id_C'$ denote the factorisation of $\id_C$ through the cocartesian cell defining its unit, as in the previous lemma. A nullary cell $\phi$, as on the left-hand side below, is cartesian if and only if its factorisation $\phi'$ through $\id_C'$ is cartesian.
		\begin{displaymath}
  		\begin{tikzpicture}[textbaseline]
				\matrix(m)[math35, column sep={1.75em,between origins}]{A & & B \\ & C & \\};
				\path[map]	(m-1-1) edge[barred] node[above] {$J$} (m-1-3)
														edge node[left] {$f$} (m-2-2)
										(m-1-3) edge node[right] {$g$} (m-2-2);
				\path[transform canvas={yshift=0.25em}]	(m-1-2) edge[cell] node[right, inner sep=3pt] {$\phi$} (m-2-2);
			\end{tikzpicture} = \begin{tikzpicture}[textbaseline]
  			\matrix(m)[math35, column sep={1.75em,between origins}]{A & & B \\ C & & C \\ & C & \\};
  			\path[map]	(m-1-1) edge[barred] node[above] {$J$} (m-1-3)
  													edge node[left] {$f$} (m-2-1)
  									(m-1-3) edge node[right] {$g$} (m-2-3)
  									(m-2-1) edge[barred] node[below, inner sep=1.5pt] {$I_C$} (m-2-3);
  			\path				(m-2-1)	edge[eq, transform canvas={xshift=-2pt}] (m-3-2)
  									(m-2-3) edge[eq, transform canvas={xshift=2pt}] (m-3-2);
  			\path[transform canvas={xshift=1.75em}]
										(m-1-1) edge[cell] node[right] {$\phi'$} (m-2-1);
				\path				(m-2-2) edge[transform canvas={shift={(-0.575em,0.325em)}}, cell] node[right, inner sep=3pt] {$\id_C'$} (m-3-2);
  		\end{tikzpicture}
  	\end{displaymath}
  	Consequently, if the horizontal unit $I_C$ exists in an augmented virtual equipment then so do all nullary restrictions of the form $C(f, g)$. In particular, in that case the companion and conjoint of any morphism $\map fAC$ exist.
	\end{corollary}
	\begin{proof}
		Since the factorisation $\id_C'$ is cartesian by the previous lemma, this follows immediately from the pasting lemma.
	\end{proof}
	
	Unlike functors between virtual double categories, functors between augmented virtual double categories necessarily preserve horizontal units, as follows.
	\begin{corollary} \label{functors preserve horizontal units}
		Any functor between augmented virtual double categories preserves cocartesian cells that define horizontal units.
	\end{corollary}
	\begin{proof}
		Consider a cocartesian cell $\psi$ defining the horizontal unit of an object $A$. By the lemma above the vertical identity cell $\id_A$ factors uniquely through $\psi$ as a cartesian cell $\phi$ which, by \corref{functors preserve companions and conjoints}, is preserved by any functor $F$. Since $F$ preserves the factorisation of $\id_A$, mapping it to $\id_{FA} = F\psi \of F\phi$, where now $F\phi$ is cartesian, it follows from the lemma above again that $F\psi$ is cocartesian.
	\end{proof}
	
	Recall that a vertical morphism $\map fAC$ is called full and faithful whenever the restriction $C(f, f)$ exists and the vertical identity $\id_f$ is cartesian.
	\begin{corollary} \label{unit in terms of full and faithful map}
		Consider a vertical morphism $\map fAC$. If the restriction $C(f, f)$ exists then $f$ is full and faithful precisely if the factorisation of $\id_f$, as shown below, is cocartesian.
		\begin{displaymath}
	  	\begin{tikzpicture}[textbaseline]
						\matrix(m)[math35]{A \\ C \\};
						\path[map]	(m-1-1) edge[bend right=45] node[left] {$f$} (m-2-1)
																edge[bend left=45] node[right] {$f$} (m-2-1);
						\path[transform canvas={xshift=-0.5em}]	(m-1-1) edge[cell] node[right] {$\id_f$} (m-2-1);
			\end{tikzpicture} = \begin{tikzpicture}[textbaseline]
    		\matrix(m)[math35, column sep={1.75em,between origins}]{& A & \\ A & & A \\ & C & \\};
    		\path[map]	(m-2-3) edge node[right] {$f$} (m-3-2)
    								(m-2-1) edge[barred] node[below, inner sep=2pt] {$C(f, f)$} (m-2-3)
    												edge node[left] {$f$} (m-3-2);
    		\path				(m-1-2) edge[eq, transform canvas={xshift=-1pt}] (m-2-1)
    												edge[eq, transform canvas={xshift=1pt}] (m-2-3);
    		\path				(m-1-2) edge[cell, transform canvas={shift={(-0.5em,-0.5em)}}] node[right] {$\id_f'$} (m-2-2);
    		\draw				([yshift=0.25em]$(m-2-2)!0.5!(m-3-2)$) node[font=\scriptsize] {$\cart$};
  		\end{tikzpicture}
  	\end{displaymath}
	\end{corollary}
	\begin{proof}
		This follows from
		\begin{displaymath}
			\text{$\id_f$ is cartesian} \Leftrightarrow \text{$\id_f'$ is cartesian} \Leftrightarrow \text{$\id_f'$ is cocartesian}
		\end{displaymath}
		where the equivalences follow from the pasting lemma and \lemref{unit identities}.
	\end{proof}
	
	\subsection{Pointwise horizontal composites} \label{pointwise horizontal composites section}
	Here we return to \exref{horizontal composites in V-Prof} to show that, under the condition considered there, the horizontal composite $(J_1 \hc \dotsb \hc J_n)$ of a path $\hmap{(J_1, \dotsc, J_n)}{A_0}{A_n}$ in $\enProf{\V'}$, of $\V'$-profunctors between $\V'$-categories, exists as soon as the coends on the right below do, for each pair $x \in A_0$ and $y \in A_n$.
		\begin{displaymath}
			(J_1 \hc \dotsb \hc J_n)(x, y) \dfn \int^{u_1 \in A_1, \dotsc, u_{n'} \in A_{n'}} J_1(x, u_1) \tens' \dotsb \tens' J_n(u_{n'}, y)
		\end{displaymath}
	In fact we first capture, in the definition below, the sense in which the assignment above defines $(J_1 \hc \dotsb \hc J_n)$ as a horizontal composite that is `pointwise'. This notion, which can be thought of informally as ``any restriction of $(J_1 \hc \dotsb \hc J_n)$ is again a composite'', will be used often in the next section, where we study Kan extensions. While the definition is stated in the general terms of a path $(\phi_1, \dotsc, \phi_n)$ of unary cells, to treat the example above we will apply it to the just the single horizontal cell $(J_1, \dotsc, J_n) \Rar (J_1 \hc \dotsb \hc J_n)$ that consists of the colimiting cocones defining the coends above.
	
	\begin{definition} \label{pointwise cocartesian path}
		Consider a path $\ul \phi = (\phi_1, \dotsc, \phi_n)$ of unary cells and of length $n \geq 1$, and assume that $\phi_n$ has non-nullary horizontal source as well as trivial vertical target, as shown in the composite in the left-hand side below. The path $\ul \phi$ is called \emph{right pointwise cocartesian} if for any morphism $\map fB{A_{nm_n}}$ the following holds: the restriction $J_{nm_n}(\id, f)$ exists if and only if $K_n(\id, f)$ does, and in that case the path $(\phi_1, \dotsc, \phi_n')$ is cocartesian, where $\phi_n'$ is the unique factorisation in
		\begin{equation} \label{pointwise cocartesian factorisation}
			\begin{tikzpicture}[textbaseline]
				\matrix(m)[math35, column sep={2em,between origins}]
					{ A_{n0} & & A_{n1} & & A_{n(m_n)'} &[0.25em] & B \\
						A_{n0} & & A_{n1} & & A_{n(m_n)'} & & A_{nm_n} \\
						& C_{n'} & & & & A_{nm_n} & \\ };
				\path[map]	(m-1-1) edge[barred] node[above] {$J_{n1}$} (m-1-3)
										(m-1-5) edge[barred] node[above, xshift=-8pt] {$J_{nm_n}(\id, f)$} (m-1-7)
										(m-1-7) edge node[right] {$f$} (m-2-7)
										(m-2-1) edge[barred] node[below] {$J_{n1}$} (m-2-3)
														edge node[left] {$f_{n'}$} (m-3-2)
										(m-2-5) edge[barred] node[below] {$J_{nm_n}$} (m-2-7)
										(m-3-2) edge[barred] node[below] {$K_n$} (m-3-6);
				\path				(m-1-3) edge[eq] (m-2-3)
										(m-1-5) edge[eq] (m-2-5)
										(m-1-1) edge[eq] (m-2-1)
										(m-2-7) edge[eq] (m-3-6)
										(m-2-4) edge[cell, transform canvas={xshift=0.125em}] node[right] {$\phi_n$} (m-3-4);
				\draw[font=\scriptsize]	($(m-1-6)!0.5!(m-2-6)$) node {$\cart$};
				\draw				($(m-1-4)!0.5!(m-2-4)$) node {$\dotsb$};
			\end{tikzpicture} = \begin{tikzpicture}[textbaseline]
				\matrix(m)[math35, column sep={2em,between origins}]
					{ A_{n0} & & A_{n1} & & A_{n(m_n)'} & & B \\
						C_{n'} & & & & & & B \\
						& C_{n'} & & & & A_{nm_n} & \\ };
				\path[map]	(m-1-1) edge[barred] node[above] {$J_1$} (m-1-3)
														edge node[left] {$f_{n'}$} (m-2-1)
										(m-1-5) edge[barred] node[above, xshift=-8pt] {$J_{nm_n}(\id, f)$} (m-1-7)
										(m-2-1) edge[barred] node[below] {$K_n(\id, f)$} (m-2-7)
										(m-2-7) edge node[right] {$f$} (m-3-6)
										(m-3-2) edge[barred] node[below] {$K_n$} (m-3-6);
				\path				(m-2-1) edge[eq] (m-3-2)
										(m-1-7) edge[eq] (m-2-7)
										(m-1-4) edge[cell] node[right] {$\phi_n'$} (m-2-4);
				\draw[font=\scriptsize]	([yshift=-0.25em]$(m-2-4)!0.5!(m-3-4)$) node {$\cart$};
				\draw				(m-1-4) node[xshift=-3pt] {$\dotsb$};
			\end{tikzpicture}.
		\end{equation}
		
		The notion of \emph{left pointwise cocartesian} path is horizontally dual. A path of unary cells that is both left and right pointwise cocartesian is called \emph{pointwise cocartesian}.
	\end{definition}
	Notice that any right pointwise, or left pointwise, cocartesian path is cocartesian in particular, by taking $f = \id_{A_{nm_n}}$ in the above. Of course a single horizontal cell $\cell\phi{(J_1, \dotsc, J_n)}K$, where $n \geq 1$, is called \emph{pointwise cocartesian} whenever the singleton path $(\phi)$ is pointwise cocartesian. In that case we call $K$ the \emph{pointwise composite} of $J_1, \dotsc, J_n$.
	
	Recall from \eqref{internal equivariance span} the construction of a diagram $\ul J^\S(x,y)$ in $\V$, for any path \mbox{$\hmap{\ul J}{A_0}{A_n}$} of $\V'$-profunctors, as well as objects $x \in A_0$ and $y \in A_n$.
	\begin{proposition} \label{cocartesian cells in (V, V')-Prof}
		Let $\V'$ be a monoidal category whose tensor product $\tens'$ preserves large coends on both sides; let $\hmap{\ul J}{A_0}{A_n}$ be a path of $\V'$-profunctors of length $n \geq 1$. A horizontal cell $\cell\phi{\ul J}K$ in $\enProf{\V'}$ is pointwise cocartesian precisely if, for each $x \in A_0$ and $y \in A_n$, its restriction $\bigpars{J_1(x, \id), \dotsc, J_n(\id, y)}\Rar K(x, y)$ defines $K(x, y)$ as the coend $\int^{u_1 \in A_1, \dotsc, u_{n'} \in A_{n'}} J_1(x, u_1) \tens' \dotsb \tens' J_n(u_{n'}, y)$, in the sense of \propref{weakly cocartesian cell in (V, V')-Prof}.
		\begin{displaymath}
			\begin{tikzpicture}
				\matrix(m)[math35]{A_0 & A_1 & A_{n'} & A_n \\ C & & & D \\};
				\path[map]	(m-1-1) edge[barred] node[above] {$J_1$} (m-1-2)
														edge node[left] {$f$} (m-2-1)
										(m-1-3) edge[barred] node[above] {$J_n$} (m-1-4)
										(m-1-4) edge node[right] {$g$} (m-2-4)
										(m-2-1) edge[barred] node[below] {$K$} (m-2-4);
				\path[transform canvas={xshift=1.75em}]	(m-1-2) edge[cell] node[right] {$\psi$} (m-2-2);
				\draw				($(m-1-2)!0.5!(m-1-3)$) node {$\dotsb$};
			\end{tikzpicture}
		\end{displaymath}
		Moreover, given a universe enlargement $\V \to \V'$ with $\V'$ as above, the inclusion $\enProf{(\V, \V')} \to \enProf{\V'}$ reflects all (pointwise) cocartesian cells while preserving the horizontal pointwise ones. In particular a cell $\psi$ as above is cocartesian in $\enProf{(\V, \V')}$ if the induced $\V'$-cocones $(f^*, J_1, \dotsc, J_n, g_*)^\S(x,y) \Rar \Delta K(x,y)$ define the $\V$-objects $K(x, y)$ as the $\V'$-coends below, for any $x \in C$ and $y \in D$. In that case if $g = \id_{A_n}$ then $\psi$ is right pointwise cocartesian.
		\begin{displaymath}
			\int^{u_0 \in A_0, \dotsc, u_n \in A_n} C(x, fu_0) \tens' J_1(u_0, u_1) \tens' \dotsb \tens' J_n(u_{n'}, u_n) \tens' D(gu_n, y)
		\end{displaymath}
	\end{proposition}
	\begin{proof}
		The `precisely'-part follows immediately from the previous definition and \propref{weakly cocartesian cell in (V, V')-Prof}, as the restrictions $\bigpars{J_1(x, \id), \dotsc, J_n(\id, y)} \Rar K(x,y)$ are obtained from $\phi$ by composing it with the cartesian cells defining $J_1(x, \id)$ and $J_n(\id, y)$ and then factorising the result through the cartesian cell that defines $K(x, y)$.
		
		For the `if'-part consider $\V'$-functors $\map f{A'}{A_0}$ and $\map g{A''}{A_n}$; we have to show that any cell $\chi$ in $\enProf{\V'}$, of the form below, factors uniquely through the restriction $\cell{\phi'}{\bigpars{J_1(f, \id), \dotsc, J_n(\id, g)}}{K(f, g)}$ of $\phi$.
		\begin{displaymath}
			\begin{tikzpicture}
				\matrix(m)[math35]
					{ A'_0 & A'_1 & A'_{p'} & A' & A_1 & A_{n'} & A'' & A''_1 & A''_{q'} & A''_q \\
						C & & & & & & & & & D \\ };
				\path[map]	(m-1-1) edge[barred] node[above] {$J'_1$} (m-1-2)
														edge node[left] {$h$} (m-2-1)
										(m-1-3) edge[barred] node[above] {$J'_p$} (m-1-4)
										(m-1-4) edge[barred] node[above, xshift=2pt] {$J_1(f, \id)$} (m-1-5)
										(m-1-6) edge[barred] node[above] {$J_n(\id, g)$} (m-1-7)
										(m-1-7) edge[barred] node[above] {$J''_1$} (m-1-8)
										(m-1-9) edge[barred] node[above] {$J''_q$} (m-1-10)
										(m-1-10) edge node[right] {$k$} (m-2-10)
										(m-2-1) edge[barred] node[below] {$\ul L$} (m-2-10);
				\path				($(m-1-1.south)!0.5!(m-1-10.south)$) edge[cell] node[right] {$\chi$} ($(m-2-1.north)!0.5!(m-2-10.north)$);
				\draw				($(m-1-2)!0.5!(m-1-3)$) node {$\dotsb$}
										($(m-1-5)!0.5!(m-1-6)$) node {$\dotsb$}
										($(m-1-8)!0.5!(m-1-9)$) node {$\dotsb$};
			\end{tikzpicture}
		\end{displaymath}
		To do so notice that $\chi$ restricts, for any sequences of objects $\ul x' \in \ob A'_0 \times \dotsb \times \ob A'$ and $\ul x'' \in \ob A'' \times \dotsb \times \ob A''_q$, as a family of cocones below, where $\ul J^\S(fx', gx'')$ is the diagram described in \eqref{internal equivariance span}.
		\begin{multline*}
			J'_1(x'_0, x'_1) \tens' \dotsb \tens' J'_p(x'_{p'}, x') \tens' \ul J^\S(fx', gx'') \\
			\tens' J''_1(x'', x''_1) \tens' \dotsb \tens' J''_q(x''_{q'}, x''_q) \Rar \Delta\ul L(hx'_0, kx''_q)
		\end{multline*}
		By the assumptions, this cocone factors uniquely through
		\begin{displaymath}
			 J'_1(x'_0, x'_1) \tens' \dotsb \tens' J'_p(x'_{p'}, x') \tens' \phi'_{(x', x'')} \tens' J''_1(x'', x''_1) \tens' \dotsb \tens' J''_q(x''_{q'}, x''_q),
		\end{displaymath}
		where $\phi'_{(x', x'')}$ denotes the restriction of $\phi'$ to $x' \in A'$ and $x'' \in A''$, as a $\V'$-map
		\begin{multline*}
			\chi'_{\ul x' \conc \ul x''}\colon J'_1(x'_0, x'_1) \tens' \dotsb \tens' J'_p(x'_{p'}, x') \tens' K(fx', gx'') \\
			\tens' J''_1(x'', x''_1) \tens' \dotsb \tens' J''_q(x''_{q'}, x''_q) \to \ul L(hx'_0, kx''_q).
		\end{multline*}
		It remains to check that these $\V'$-maps combined satisfy the equivariance axioms of \defref{monoids and bimodules}, so that they form a cell $\cell{\chi'}{\bigpars{J'_1, \dotsc, J'_p, K(f, g), J''_1, \dotsc, J''_q}}{\ul L}$, which is the factorisation of $\chi$ through $\phi'$. This is straightforward: it follows from the uniqueness of factorisations through the restrictions $\phi'_{(x', x'')}$, together with the equivariance axioms for $\phi'$ and $\chi$.
		
		Because $\enProf{(\V, \V')} \to \enProf{\V'}$ is full and faithful reflection of cocartesian cells follows, while its preservation of horizontal cocartesian cells follows from the first assertion combined with its preservation of horizontal weakly cocartesian cells (\propref{weakly cocartesian cell in (V, V')-Prof}). We conclude that a cell $\psi$, as in the statement, is cocartesian in $\enProf{(\V, \V')}$ as soon as it is cocartesian in $\enProf{\V'}$. Using \corref{extensions and composites}, the latter is implied by $\cart \hc \psi \hc \cart$, with the cartesian cells defining $f^*$ and $g_*$, being cocartesian in $\enProf{\V'}$, from which the final assertion follows. That, in the case $g = \id_{A_n}$, any restriction of $\psi$ along $J_n(\id, h)$, where $\map hB{A_n}$ and in the sense of \defref{pointwise cocartesian path}, is again cocartesian is clear, so that $\psi$ is right pointwise cocartesian in this case. This completes the proof.
	\end{proof}
		
	Pointwise cocartesian paths are coherent in the following sense.
	\begin{lemma} \label{coherence of pointwise cocartesian paths}
		If the path $\ul \phi = (\phi_1, \dotsc, \phi_n)$ is right pointwise cocartesian then any path of the form $(\phi_1, \dotsc, \phi_n')$, as in \defref{pointwise cocartesian path}, is again right pointwise cocartesian.
	\end{lemma}
	\begin{proof}
		Consider $\map fB{A_{nm_n}}$ as in \defref{pointwise cocartesian path}, and let $\phi_n'$ be the factorisation in \eqref{pointwise cocartesian factorisation}. Then, for any $\map hCB$ the following are equivalent: $J_{nm_n}(\id, f)(\id, h)$ exists; $J_{nm_n}(\id, h \of f)$ exists; $K_n(\id, h \of f)$ exists; $K_n(\id, f)(\id, h)$ exists, by using the pasting lemma for cartesian cells and the fact that $\ul \phi$ is pointwise cocartesian. This shows that the first assertion of \defref{pointwise cocartesian path} holds. Next consider the unique factorisation $\phi_n''$ in $\phi'_n \of (\id, \dotsc, \id, \cart) = \cart \of \phi_n''$, where the cartesian cells define $J_{nm_n}(\id, f)(\id, h)$ and $K_n(\id, f)(\id, h)$ respectively; we have to show that $(\phi_1, \dotsc, \phi_n'')$ is cocartesian. To see this consider the following equation, where the identities follow form the definitions of $\phi_n''$ and $\phi_n'$ respectively.
		\begin{displaymath}
			\begin{tikzpicture}[textbaseline, x=1.5em, y=1.5em, yshift=-2.25em, font=\scriptsize]
				\draw	(1,3) -- (0,3) -- (0,1) -- (0.5,0) -- (2.5,0) -- (3,1) -- (3,3) -- (2,3) (0,1) -- (3,1) (0,2) -- (3,2);
				\draw[shift={(0.5,0.5)}]	(1,2.5) node[xshift=0.5pt] {$\dotsb$}
							(1,2) node {$\phi_n''$}
							(1,1) node {c}
							(1,0) node {c};
			\end{tikzpicture} \mspace{9mu} = \mspace{9mu} \begin{tikzpicture}[textbaseline, x=1.5em, y=1.5em, yshift=-2.25em, font=\scriptsize]
				\draw	(0,2) -- (1,2) -- (1,3) -- (0,3) -- (0,1) -- (0.5,0) -- (2.5,0) -- (3,1) -- (3,3) -- (2,3) -- (2,2) -- (3,2) (0,1) -- (3,1);
				\draw[shift={(0.5,0.5)}]	(1,2) node[xshift=0.5pt] {$\dotsb$}
							(2,2) node {c}
							(1,1) node {$\phi_n'$}
							(1,0) node {c};
			\end{tikzpicture} \mspace{9mu} = \mspace{9mu} \begin{tikzpicture}[textbaseline, x=1.5em, y=1.5em, yshift=-2.25em, font=\scriptsize]
				\draw	(0,1) -- (1,1) -- (1,3) -- (0,3) -- (0,1) -- (0.5,0) -- (2.5,0) -- (3,1) -- (3,3) -- (2,3) -- (2,1) -- (3,1) (0,2) -- (1,2) (2,2) -- (3,2);
				\draw[shift={(0.5,0.5)}]	(1,2) node[xshift=0.5pt] {$\dotsb$}
							(2,2) node {c}
							(1,1) node[xshift=0.5pt] {$\dotsb$}
							(1,0) node {$\phi_n$}
							(2,1) node {c};
			\end{tikzpicture}
		\end{displaymath}
		Since the composites of cartesian cells in the left-hand and right-hand sides are again cartesian by the pasting lemma, we conclude that $(\phi_1, \dotsc, \phi_n'')$ is cocartesian, as $(\phi_1, \dotsc, \phi_n)$ is pointwise cocartesian. This concludes the proof.
	\end{proof}
	
	The pasting lemma for cocartesian paths (\lemref{pasting lemma for cocartesian paths}) induces one for pointwise cocartesian paths, as follows.
	\begin{lemma}[Pasting lemma] \label{pasting lemma for right pointwise cocartesian paths}
		Consider a configuration of unary cells as below and assume that the path $(\phi_{11}, \dotsc, \phi_{nm_n})$ is right pointwise cocartesian. The path $\bigpars{\psi_1 \of (\phi_{11}, \dotsc, \phi_{1m_1}), \dotsc, \psi_n \of (\phi_{n1}, \dotsc, \phi_{nm_n})}$ is right pointwise cocartesian if and only if $(\psi_1, \dotsc, \psi_n)$ is.
		\begin{displaymath}
			\begin{tikzpicture}[x=0.33cm, y=0.6cm, font=\scriptsize]
				\draw	(1,2) -- (0,2) -- (0,0) -- (18,0) -- (18,2) -- (17,2)
							(2,2) -- (3,2) -- (3,1) -- (0,1)
							(7,2) -- (6,2) -- (6,1) -- (12,1) -- (12,2) -- (11,2)
							(8,2) -- (10,2)
							(9,2) -- (9,0)
							(16,2) -- (15,2) -- (15,1) -- (18,1)
							(28,2) -- (27,2) -- (27,0) -- (36,0) -- (36,2) -- (35,2)
							(29,2) -- (30,2) -- (30,1) -- (27,1)
							(34,2) -- (33,2) -- (33,1) -- (36,1);
				\draw[shift={(0.165cm, 0.3cm)}]	(1,1) node {$\phi_{11}$}
							(7,1) node {$\phi_{1m_1}$}
							(10,1) node {$\phi_{21}$}
							(16,1) node {$\phi_{2m_2}$}
							(28,1) node {$\phi_{n1}$}
							(34,1) node {$\phi_{nm_n}$}
							(4,0) node {$\psi_1$}
							(13,0) node {$\psi_2$}
							(31,0) node {$\psi_n$}
							(4,1) node {$\dotsb$}
							(13,1) node {$\dotsb$}
							(31,1) node {$\dotsb$};
				\draw[font=]	(22.5,1) node {$\dotsb$};
			\end{tikzpicture}
		\end{displaymath}
	\end{lemma}
	\begin{proof}
		That both the path $(\phi_{11}, \dotsc, \phi_{nm_n})$ as well as one of the paths $(\psi_1, \dotsc, \psi_n)$ or $\bigpars{\psi_1 \of (\phi_{11}, \dotsc, \phi_{1m_1}), \dotsc, \psi_n \of (\phi_{n1}, \dotsc, \phi_{nm_n})}$ are right pointwise cocartesian means that the vertical targets of $\phi_{nm_n}$ and $\psi_n$ are both the identity on a single object, say, $A$. It also means that, for any $\map fBA$, that the existence of the following restrictions along $f$ are equivalent: that along the horizontal target of $\psi_n$; that along the horizontal target of $\phi_{nm_n}$; that along the last horizontal source of $\phi_{nm_n}$. In the case that these restrictions exist we can obtain factorisations $\psi_n'$ and $\phi_{nm_n}'$, as in \defref{pointwise cocartesian path}, such that the equations
		\begin{displaymath}
			\begin{tikzpicture}[textbaseline, x=1.125em, y=1.25em, yshift=-1.875em, font=\scriptsize]
				\draw (0,2) -- (1,2) -- (1,3) -- (0,3) -- (0,0) -- (9,0) -- (9,3) -- (8,3) -- (8,2) -- (9,2) (3,2) -- (2,2) -- (2,3) -- (3,3) -- (3,1) -- (0,1) (6,2) -- (7,2) -- (7,3) -- (6,3) -- (6,1) -- (9,1);
				\draw	(1.5,2.5) node[xshift=0.75pt] {$\dotsb$}
							(1.5,1.5) node {$\phi_{n1}$}
							(4.5,2) node[font=] {$\dotsb$}
							(4.5,0.5) node {$\psi_n$}
							(7.5,2.5) node[xshift=0.75pt] {$\dotsb$}
							(7.5,1.5) node {$\phi_{nm_n}$}
							(8.5,2.5) node {c};
			\end{tikzpicture} \mspace{6mu} = \mspace{6mu} \begin{tikzpicture}[textbaseline, x=1.125em, y=1.25em, yshift=-1.875em, font=\scriptsize]
				\draw (1,3) -- (0,3) -- (0,0) -- (9,0) -- (9,3) -- (8,3) (2,3) -- (3,3) -- (3,1) -- (0,1) (7,3) -- (6,3) -- (6,1) -- (9,1) (0,2) -- (3,2) (6,2) -- (9,2);
				\draw	(1.5,3) node[xshift=0.75pt] {$\dotsb$}
							(1.5,2.5) node {$\phi_{n1}$}
							(4.5,2.5) node[font=] {$\dotsb$}
							(4.5,1.5) node[font=] {$\dotsb$}
							(4.5,0.5) node {$\psi_n$}
							(7.5,3) node[xshift=0.75pt] {$\dotsb$}
							(7.5,2.5) node {$\phi_{nm_n}'$}
							(7.5,1.5) node {c};
			\end{tikzpicture} \mspace{6mu} = \mspace{6mu} \begin{tikzpicture}[textbaseline, x=1.125em, y=1.25em, yshift=-1.875em, font=\scriptsize]
				\draw (1,3) -- (0,3) -- (0,0) -- (9,0) -- (9,3) -- (8,3) (2,3) -- (3,3) -- (3,2) -- (0,2) (7,3) -- (6,3) -- (6,2) -- (9,2) (0,1) -- (9,1);
				\draw	(1.5,3) node[xshift=0.75pt] {$\dotsb$}
							(1.5,2.5) node {$\phi_{n1}$}
							(4.5,2.5) node[font=] {$\dotsb$}
							(4.5,1.5) node {$\psi_n'$}
							(7.5,3) node[xshift=0.75pt] {$\dotsb$}
							(7.5,2.5) node {$\phi_{nm_n}'$}
							(4.5,0.5) node {c};
			\end{tikzpicture}
		\end{displaymath}
		hold, where `c' denotes any of the three cartesian cells defining the restrictions along $f$. From this we conclude that the unique factorisation that corresponds to $\psi_n \of (\phi_{n1}, \dotsc, \phi_{nm_n})$, as in \defref{pointwise cocartesian path} and with respect to the restrictions along $f$, equals $\psi_n' \of (\phi_{n1}, \dotsc, \phi_{nm_n}')$. The proof now follows from the fact that, assuming that $(\phi_{11}, \dotsc, \phi_{nm_n}')$ is cocartesian, the cocartesianness of $(\psi_1, \dotsc, \psi_n')$ is equivalent to that of $\bigpars{\psi_1 \of (\phi_{11}, \dotsc, \phi_{1m_1}), \dotsc, \psi_n' \of (\phi_{n1}, \dotsc, \phi_{nm_n}')}$, by the pasting lemma (\lemref{pasting lemma for cocartesian paths}).
	\end{proof}
	
	\subsection{Restrictions and extensions in terms of companions and conjoints}
	Here we make precise the fact that restrictions and extensions can be defined by composing horizontal morphisms with companions and conjoints, as was described in the discussion following \exref{restrictions of V-profunctors}.
	
	We start with restrictions. In the setting of virtual double categories the `only if'-part of the following lemma has been proved as Theorem 7.16 of \cite{Cruttwell-Shulman10}. In the next section we will see that, when considered in an augmented virtual equipment, the composite of $f_* \conc \ul K \conc g^*$ considered below is in fact a `pointwise' composite.
	\begin{lemma} \label{restrictions and composites}
		In an augmented virtual double category $\K$ assume that the companion $\hmap{f_*}AC$ and the conjoint $\hmap{g^*}DB$ exist. Then, for each path $\hmap{\ul K}CD$ of length $\leq 1$, the restriction $\ul K(f, g)$ exists if and only if the horizontal composite of the path $f_* \conc \ul K \conc g^*$ does, and in that case they are isomorphic.
		\begin{equation} \label{first restriction identity}

		\end{equation}
		In detail, for a factorisation as above (where the empty cell is the vertical identity cell $\id_C$ if $\lns{\ul K} = 0$) the following are equivalent: $\psi$ is cartesian; $\phi$ is cocartesian; the identity below holds.
		\begin{displaymath}
			%
  	\end{displaymath}
  	
  	Analogous assertions hold for one-sided restrictions: $K(f, \id)$ exists precisely if $f_* \hc K$ does, while $K(\id, g)$ exists if and only if $K \hc g^*$ does. Finally if $\K$ is an augmented virtual equipment then the cocartesian cell $\cell\phi{f_* \conc \ul K \conc g^*}J$ above is pointwise cocartesian; the same holds for cocartesian cells of the forms $(f_*, K) \Rar K(f, \id)$ and $(K, g_*) \Rar K(\id, g)$.
	\end{lemma}
	\begin{proof}
		Assuming that \eqref{first restriction identity} holds, it follows from the companion identities and the conjoint identities (see \lemref{companion identities lemma} and its horizontal dual) that precomposing the composite on the left-hand side above with $\phi$ results in $\phi$, while postcomposing it with $\psi$ gives back $\psi$. Using the uniqueness of factorisations through (co-)cartesian cells, we conclude that either $\psi$ or $\phi$ being (co-)cartesian implies the identity above.
		
		Conversely, assume that identities above holds; we will prove that $\psi$ is cartesian and $\phi$ cocartesian. For the first it suffices to show that the following assignment of cells is invertible. To see that it is notice that, as a consequence of the assumed identities, the inverse can be given as $\chi \mapsto \phi \of (\cocart \of h, \chi, \cocart \of k)$, where the weakly cocartesian cells define $f_*$ and $g^*$ respectively.
		\begin{displaymath}

		\end{displaymath}
		To prove that $\phi$ is cocartesian is to show that, for any paths $\hmap{\ul J'}{A'_0}{A'_p = A}$ and $\hmap{\ul J''}{B= B'_0}{B'_q}$, the assignment
		\begin{displaymath}
			%
		\end{displaymath}
		is a bijection. That it is follows from the assumed identities: the inverse is given by $\chi \mapsto \chi \of (\ul\id, \cocart, \psi, \cocart, \ul \id)$. This completes the proof of the main assertion.
		
		For the final assertion, assume that $\K$ is an augmented virtual equipment. We will prove that the cocartesian cells of the forms $(f_*, K) \Rar K$ are pointwise cocartesian; the other two cases can then be easily derived. Thus we consider morphisms $\map pXA$ and $\map qYD$; we have to show that the composite of the top two rows in the left-hand side below factors as a cocartesian cell through $K(f \of p, q)$. To see this consider the equation below, whose identities follow from applying \eqref{first restriction identity} to the bottom two cells of the left-hand side; applying \eqref{first restriction identity} to the top two cells of the third composite; factorising the composite of the two cartesian cells in the third composite through the bottom cartesian cell in the right-hand side which, by the pasting lemma, results in a cartesian cell as shown.
		\begin{displaymath}

		\end{displaymath}
		By the uniqueness of factorisations through the bottom cartesian cell in the left-hand side and right-hand side above, we conclude that the composite of the top two rows in the former factors as a cocartesian cell followed by a cartesian one, as required.
		
		A horizontal dual argument applies to cocartesian cells $(K, g^*) \Rar K(\id, g)$. The proof for a cocartesian cell of the general form $(f_*, K, g^*) \Rar K(f, g)$ then follows by writing it as a composite of the cocartesian cells $(f_*, K) \Rar K(f, \id)$ and $\bigpars{K(f, \id), g^*} \Rar K(f, g)$, followed by applying the pasting lemma (\lemref{pasting lemma for right pointwise cocartesian paths}). This completes the proof.
	\end{proof}
	
	As corollaries we find that functors of augmented virtual double categories behave well with respect to restrictions and full and faithful morphisms, as follows. The first of these is a variation on the corresponding result for functors between double categories; see Proposition 6.8 of \cite{Shulman08}.
	\begin{corollary} \label{functors preserving cartesian cells}
		Let $\map F\K\L$ be a functor between augmented virtual double categories. Consider morphisms $\map fAC$ and $\map gBD$ in $\K$ and let $\hmap{\ul K}CD$ be a path of length $\leq 1$. If the companion $\hmap{f_*}AC$ and the conjoint $\hmap{g^*}DB$ exist then $F$ preserves both the cartesian cell defining the restriction $\ul K(f, g)$ as well as the cocartesian cell defining the composite of the path $f_* \conc \ul K \conc g^*$.
		
		Under the same conditions the cartesian cells defining the restrictions of the form $K(f, \id)$ and $K(\id, g)$, as well as the cocartesian ones defining the composites of the form $(f_* \hc K)$ and $(K \hc g^*)$, are preserved by $F$.
	\end{corollary}
	\begin{proof}
		This follows from the fact that $F$ preserves the identities of the previous lemma, as well as the (weakly co-)cartesian cells that they contain; the latter by \corref{functors preserve companions and conjoints}.
	\end{proof}
	
	\begin{corollary}
		A full and faithful morphism $\map fAC$ is preserved by any functor as soon as its companion $f_*$ and conjoint $f^*$ exist.
	\end{corollary}
	\begin{proof}
		This follows from the fact that $F$ preserves the factorisation of $\id_f$ that is considered in \corref{unit in terms of full and faithful map}. Indeed by the previous corollary it preserves the cartesian cell defining $C(f, f)$ while, by \corref{functors preserve horizontal units}, it preserves the cocartesian cell that defines $C(f,f)$ as the horizontal unit too.
	\end{proof}
	
	Next we turn to describing (weakly) cocartesian cells in terms of companions and conjoints.
	\begin{lemma} \label{cocartesian cells for companions and conjoint}
		Let $\map fAC$ be a vertical morphism. If the restriction $J''(f, \id)$ exists for every $\hmap{J''}C{A''}$ then the weakly cocartesian cell defining the companion $\hmap{f_*}AC$ is cocartesian. A horizontal dual result holds for the weakly cocartesian cell defining the conjoint $\hmap{f^*}CA$.
	\end{lemma}
	\begin{proof}
		Consider paths $\hmap{\ul J'}{A_0}A$ and $\hmap{\ul J''}C{A''_q}$; we have to show that the bottom path in
		\begin{displaymath}

		\end{displaymath}
		is weakly cocartesian. By the pasting lemma we may equivalently show that the full path above is weakly cocartesian, where the top weakly cocartesian cell corresponds to the cartesian cell as in the previous lemma; it defines $J''_1(f, \id)$ as the horizontal composite of $f_*$ and $J''_1$. But this follows easily from the fact that the latter two cells compose as the cartesian cell defining $f_*$ and the identity on $J''_1$ (see \lemref{restrictions and composites}), and the horizontal companion identity for $f_*$ (see \lemref{companion identities lemma}).
	\end{proof}
	
	Together with the pasting lemma for cocartesian paths (\lemref{pasting lemma for cocartesian paths}) the previous result implies the following corollary, whose first part is a variation of Theorem 7.20 of \cite{Cruttwell-Shulman10} for virtual double categories.
	\begin{corollary} \label{extensions and composites}
		If the conjoint of $\map f{A_0}C$ and the composite $(f^* \hc J_1 \hc \dotsb \hc J_n)$ exist then the composite on the left below is weakly cocartesian.
		\begin{displaymath}
			\begin{tikzpicture}[baseline]
				\matrix(m)[math35, column sep={1.75em,between origins}]{& A_0 & & A_1 & & A_{n'} & & A_n & \\ C & & A_0 & & A_1 & & A_{n'} & & A_n \\ & C & & & & & & A_n & \\};
				\path[map]	(m-1-2) edge[barred] node[above] {$J_1$} (m-1-4)
														edge[transform canvas={xshift=-2pt}] node[above left] {$f$} (m-2-1)
										(m-1-6) edge[barred] node[above] {$J_n$} (m-1-8)
										(m-2-1) edge[barred] node[below] {$f^*$} (m-2-3)
										(m-2-3) edge[barred] node[above] {$J_1$} (m-2-5)
										(m-2-7) edge[barred] node[above] {$J_n$} (m-2-9)
										(m-3-2) edge[barred] node[below] {$(f^* \hc J_1 \hc \dotsb \hc J_n)$} (m-3-8);
				\path				(m-1-2) edge[eq, transform canvas={xshift=2pt}] (m-2-3)
										(m-1-4) edge[eq, transform canvas={xshift=1pt}] (m-2-5)
										(m-1-6) edge[eq] (m-2-7)
										(m-1-8) edge[eq] (m-2-9)
										(m-2-1) edge[eq] (m-3-2)
										(m-2-9) edge[eq] (m-3-8);
				\draw				($(m-1-6)!0.5!(m-2-5)$) node {$\dotsb$};
				\draw[font=\scriptsize]	([xshift=-1pt]$(m-1-2)!0.666!(m-2-2)$) node {$\cocart$}
										($(m-2-1)!0.5!(m-3-9)$) node {$\cocart$};
			\end{tikzpicture} \qquad \begin{tikzpicture}[baseline]
				\matrix(m)[math35]{A_0 & A_1 & A_{n'} & A_n \\ C & & & A_n \\};
				\path[map]	(m-1-1) edge[barred] node[above] {$J_1$} (m-1-2)
														edge node[left] {$f$} (m-2-1)
										(m-1-3) edge[barred] node[above] {$J_n$} (m-1-4)
										(m-2-1) edge[barred] node[below] {$K$} (m-2-4);
				\path       (m-1-4) edge[eq] (m-2-4)
				            (m-1-2) edge[cell, transform canvas={xshift=1.75em}] node[right] {$\phi$} (m-2-2);
				\draw				($(m-1-2)!0.5!(m-1-3)$) node {$\dotsb$};
			\end{tikzpicture}
		\end{displaymath}
		If for any horizontal morphisms $\hmap{J'}{A'}C$ and $\hmap{J''}D{A''}$ the restrictions $J'(\id, f)$ and $J''(g, \id)$ exist then it is cocartesian as well, while it is right pointwise cocartesian as soon as the composite $(f^* \hc J_1 \hc \dotsb \hc J_n)$ is; in that case any weakly cocartesian cell as on the right above is (right pointwise) cocartesian.
		
		Analogous results hold for composites of the forms $(J_1 \hc \dotsb \hc J_n \hc g_*)$ and $(f^* \hc J_1 \hc \dotsb \hc J_n \hc g_*)$, where $\map g{A_n} D$.
	\end{corollary}
	
	\subsection{The equivalence of unital augmented virtual double categories and unital virtual double categories}\label{unital fc-multicategories section}
	Closing this section, here we will show the equivalence of the notions of augmented virtual double category and virtual double category in the case that all horizontal units exist. We shall denote by $\VirtDblCat_\textup u \subset \VirtDblCat$ the locally full sub-$2$-category consisting of virtual double categories that have all horizontal units, and the \emph{normal} functors---that preserve the cocartesian cells defining horizontal units---between them, and all transformations between those. Likewise $\AugVirtDblCat_\textup u \subset \AugVirtDblCat$ denotes the full sub-$2$-category consisting of augmented virtual double categories that have all horizontal units; remember that any functor between augmented virtual double categories preserves horizontal units (\corref{functors preserve horizontal units}).
	
	Remember the strict $2$-functor $\map U{\AugVirtDblCat}{\VirtDblCat}$ of \propref{2-category of (fc, bl)-multicategories}, that maps any augmented virtual double category $\K$ to its underlying virtual double category $U(\K)$ consisting of its unary cells. Clearly the cocartesian cells in an augmented virtual double category $\K$ form again cocartesian cells in $U(\K)$, so that $U$ restricts to a strict $2$-functor $\map U{\AugVirtDblCat_\textup u}{\VirtDblCat_\textup u}$.	
	\begin{theorem} \label{unital fc-multicategories}
		Let $\K$ be a virtual double category that has all horizontal units and suppose that, for each object $A$ in $\K$, a cocartesian cell $\eta_A$ that defines its unit $I_A$ has been chosen. Consider to each unary cell $\phi$ of $\K$, as on the left below, a new unary cell $\bar\phi$ as on the right, that is of the same form, and to each unary cell $\psi$ as on the left a new nullary cell $\bar\psi$ as on the right.
		\begin{displaymath}

		\end{displaymath}
		The objects and morphisms of $\K$, together with the unary cells $\bar\phi$ and the nullary cells $\bar\psi$, form an augmented virtual double category $N(\K)$ that has all horizontal units. Composition of cells in $N(\K)$ is given by
		\begin{equation} \label{composition in N(K)}
			\bar\psi \of (\bar\phi_1, \dotsc, \bar\phi_n) \dfn \overline{\psi' \of (\phi_1, \dotsc, \phi_n)}
		\end{equation}
		where $\psi'$ is the unique factorisation of $\psi$ through the cocartesian path of cells $(\eta_{\bar\phi_1}, \dotsc, \eta_{\bar\phi_n})$, where $\eta_{\bar\phi_i} \dfn \eta_{C_{i'}}$ if $\bar\phi_i$ is nullary with horizontal target $C_{i'}$ and $\eta_{\bar\phi_i} \dfn \id_{K_i}$ if $\bar\phi_i$ is unary with horizontal target $\hmap{K_i}{C_{i'}}{C_i}$. The identity cells in $\K$, for morphisms $\hmap JAB$ and $\map fAC$, are given by
		\begin{displaymath}
			\id_J \dfn \overline{\id_J} \qquad \text{and} \qquad \id_f \dfn \overline{\eta_C \of f}.
		\end{displaymath}
		
		The strict $2$-functor $\map U{\AugVirtDblCat_\textup u}{\VirtDblCat_\textup u}$ and the assignment $\K \mapsto N(\K)$ extend to form a strict $2$-equivalence $\AugVirtDblCat_\textup u \simeq \VirtDblCat_\textup u$.
	\end{theorem}
	
	\begin{example} \label{unital fc-multicategory of monoids}
		Let $\K$ be any virtual double category. As we have seen in \exref{monoids have units}, the monoids and bimodules in $\K$ form a unital virtual double category $\Mod(\K)$. The corresponding unital augmented virtual double category was described in \propref{(fc, bl)-multicategory of monoids}.
	\end{example}
		
	\begin{proof}[Proof of \thmref{unital fc-multicategories}]
		That the composition for $N(\K)$ defined above satisfies the associativity and unit axioms is a straightforward consequence of the those axioms in $\K$, together with the uniqueness of factorisations $\psi'$ through cocartesian paths.
		
		To show that $N(\K)$ has all horizontal units let $A$ be any object in $N(\K)$; we claim that $\cell{\bar{\eta_A}}A{I_A}$ defines $I_A$ as the horizontal unit of $A$. To see this, consider the identity of $I_A$ as a nullary cell $\cell{\overline{\id_{I_A}}}{I_A}A$ in $N(\K)$; we will show that $\bar{\eta_A}$ and $\overline{\id_{I_A}}$ satisfy the identities of \lemref{unit identities}. Indeed, we have $\overline{\id_{I_A}} \of \bar{\eta_A} = \overline{(\id_{I_A} \of \eta_A)} = \bar{\eta_A} = \id_A$ (the identity cell for $A$ in $N(\K)$). On the other hand we have
		\begin{displaymath}
			\bar\eta_A \hc \overline{\id_{I_A}} = \overline{\id_{I_A}' \of (\eta_A, \id_{I_A})} = \overline{(\id_{I_A})} = \id_{I_A},
		\end{displaymath}
		where the right-hand side denotes the identity cell for $I_A$ in $N(\K)$ and where $\id_{I_A}'$ is the unique factorisation $\id_{I_A}' \of (\id_{I_A}, \eta_A) = \id_{I_A}$ in $\K$, as before. The first identity above follows from the definition $\hc$ (see the discussion preceding \lemref{horizontal composition}), while the second identity follows from the fact that
		\begin{displaymath}
			 \id_{I_A}' \of (\eta_A, \id_{I_A}) \of \eta_A = \id_{I_A}' \of (\id_{I_A}, \eta_A) \of \eta_A = \id_{I_A} \of \eta_A
		\end{displaymath}
		in $\K$, so that $\id_{I_A}' \of (\eta_A, \id_{I_A}) = \id_{I_A}$ by the uniqueness of factorisations through $\eta_A$.
		
		We conclude that $N(\K)$ forms a well-defined augmented virtual double category that has all horizontal units.	Next we extend the assignment $\K \mapsto N(\K)$ to a strict $2$-functor $\map N{\VirtDblCat_\textup u}{\AugVirtDblCat_\textup u}$. For the action of $N$ on morphisms consider a normal functor $\map F\K\L$ between unital virtual double categories. Since $F$ preserves the cocartesian cells $\eta_A$ of $\K$ we can obtain, for each object $A \in K$, an invertible horizontal cell $\cell{(F_I)_A}{FI_A}{I_{FA}}$ in $\L$ that is the unique factorisation in
		\begin{equation} \label{unitors}
			\begin{tikzpicture}[textbaseline]
				\matrix(m)[math35, column sep={1.75em,between origins}]{& FA & \\ FA & & FA \\};
				\path[map]	(m-2-1) edge[barred] node[below] {$I_{FA}$} (m-2-3);
				\path				(m-1-2) edge[transform canvas={xshift=-2pt}, eq] (m-2-1)
														edge[transform canvas={xshift=2pt},eq] (m-2-3)
														edge[transform canvas={shift={(-0.7em,-0.5em)}}, cell] node[right] {$\eta_{FA}$} (m-2-2);
			\end{tikzpicture} = \begin{tikzpicture}[textbaseline]
				\matrix(m)[math35, column sep={1.75em,between origins}]{& FA & \\ FA & & FA \\ FA & & FA \\};
				\path[map]	(m-2-1) edge[barred] node[below] {$FI_A$} (m-2-3)
										(m-3-1) edge[barred] node[below] {$I_{FA}$} (m-3-3);
				\path				(m-1-2) edge[transform canvas={xshift=-2pt}, eq] (m-2-1)
														edge[transform canvas={xshift=2pt},eq] (m-2-3)
														edge[transform canvas={shift={(-0.7em,-0.5em)}}, cell] node[right] {$F\eta_A$} (m-2-2)
										(m-2-1) edge[eq] (m-3-1)
										(m-2-2) edge[transform canvas={shift={(-1.05em,-0.15em)}}, cell] node[right] {$(F_I)_A$} (m-3-2)
										(m-2-3) edge[eq] (m-3-3);
			\end{tikzpicture}.
		\end{equation}
		We define $\map{NF}{N(\K)}{N(\L)}$ as follows. On objects and morphisms it simply acts like $F$. To define its action on cells we define, for each $\bar\phi$ in $N(\K)$, the cell $\delta_{\bar\phi}$ in $\L$ by $\delta_{\bar\phi} = (F_I)_C$ if $\bar\phi$ is nullary, and $\delta_{\bar\phi} = \id_{FK}$ otherwise, and set $(NF)(\bar\phi) = \overline{(\delta_{\bar\phi} \of F\phi)}$. That this assignment preserves identity cells is easily checked; that it preserves any composition $\bar\psi \of (\bar\phi_1, \dotsc, \bar\phi_n)$ in $N(\K)$, as in \eqref{composition in N(K)}, is shown by
		\begin{align*}
			(NF)(\bar\psi) \of \bigl(&(NF)(\bar\phi_1), \dotsc, (NF)(\bar\phi_n)\bigr) \\
			& = \overline{\delta_{\bar\psi} \of F\psi} \of \pars{\overline{\delta_{\bar\phi_1} \of F\phi_i}, \dotsc, \overline{\delta_{\bar\phi_n} \of F\phi_n}} \\
			& = \overline{\delta_{\bar\psi} \of (F\psi)' \of (\delta_{\bar\phi_1} \of F\phi_1, \dotsc, \delta_{\bar\phi_n} \of F\phi_n)} \\ 
			& = \overline{\delta_{\bar\psi} \of F(\psi') \of (F\phi_1, \dotsc, F\phi_n)} = \overline{\delta_{\bar\psi} \of F\bigpars{\psi' \of (\phi_1, \dotsc, \phi_n)}} \\
			& = (NF)\bigpars{\overline{\psi' \of (\phi_1, \dotsc, \phi_n)}} = (NF)\bigpars{\bar\psi \of (\bar\phi_1, \dotsc, \bar\phi_n)},
		\end{align*}
		where the third identity is shown as follows. The cells $(F\psi)'$ and $\psi'$, on either side, are the factorisations in $F\psi = (F\psi)' \of (\eta_{(NF)(\bar\phi_1)}, \dotsc, \eta_{(NF)(\bar\phi_n)})$ and $\psi = \psi' \of (\eta_{\bar\phi_1}, \dotsc, \eta_{\bar\phi_n})$ respectively. The identity follows from the fact that
		\begin{align*}
			(F\psi)' \of (\delta_{\bar\phi_1}, \dotsc, \delta_{\bar\phi_n}) \of (&F\eta_{\bar\phi_1}, \dotsc, F\eta_{\bar\phi_n}) = (F\psi)' \of \bigpars{\eta_{(NF)(\bar\phi_1)}, \dotsc, \eta_{(NF)(\bar\phi_n)}} \\
			&= F\psi = F(\psi') \of (F\eta_{\bar\phi_1}, \dotsc, F\eta_{\bar\phi_n})
		\end{align*}
		together with the uniqueness of factorisations through the path $(F\eta_{\bar\phi_1}, \dotsc, F\eta_{\bar\phi_n})$, which is cocartesian because $F$ is normal. This concludes the definition of $N$ on morphisms.
		
		Next consider a transformation $\nat\xi FG$ between normal functors $F$ and $\map G\K\L$ of unital virtual double categories. We claim that the components of $\xi$ again form a transformation $NF \Rar NG$, which we take to be the image $N\xi$. For instance, to see that the components of $\xi$ are natural with respect to a nullary cell $\cell{\bar\phi}{\ul J}C$ in $N(\K)$, notice that
		\begin{align*}
			(NG)(\bar\phi&) \of (\bar\xi_{J_1}, \dotsc, \bar\xi_{J_n}) = \overline{(G_I)_C \of G\phi} \of (\bar\xi_{J_1}, \dotsc, \bar\xi_{J_n}) \\
			& = \overline{(G_I)_C \of G\phi \of (\xi_{J_1}, \dotsc, \xi_{J_n})} = \overline{(G_I)_C \of \xi_{I_C} \of F\phi} \\
			& = \overline{(\eta_{GC} \of  \xi_C)' \of (F_I)_C \of F\phi} = \xi_C \of (NF)(\bar\phi)
		\end{align*}
		where the last identity follows from the definition of $\id_{\xi_C}$ in $N(\L)$, while the penultimate one follows from the fact that
		\begin{align*}
			(G_I)_C \of \xi_{I_C} \of F\eta_C &= (G_I)_C \of G\eta_C \of \xi_C = \eta_{GC} \of \xi_C \\
			&= (\eta_{GC} \of \xi_C)' \of \eta_{FC} = (\eta_{GC} \of \xi_C)' \of (F_I)_C \of F\eta_C,
		\end{align*}
		by using that $F\eta_C$ is cocartesian.
		
		That the assignments $\K \mapsto N(\K)$, $F \mapsto NF$ and $\xi \mapsto N\xi$ combine the form a strict $2$-functor $\map N{\VirtDblCat_\textup u}{\AugVirtDblCat_\textup u}$ follows easily from the uniqueness of the factorisations \eqref{unitors}. It is also clear that the identity $(U \of N)(\K) = \K$ extends to an identity of strict $2$-functors $U \of N = \id$. Thus, it remains to show the existence of an invertible $2$-natural transformation $\tau\colon \id \iso N \of U$. Given a unital augmented virtual double category $\K$ we define the functor $\map{\tau_\K}\K{(N \of U)(\K)}$ as follows. It is the identity on objects and morphisms, it is given by $\phi \mapsto \bar\phi$ on unary cells and by $\psi \mapsto \overline{\eta_C \of \psi}$ on nullary cells $\cell\psi{\ul J}C$. That these assignments preserve composites and identity cells is easily checked; that the family $\tau = (\tau_\K)_{\K}$ is $2$-natural is clear. Finally, an inverse functor $\map\sigma{(N \of U)(\K)}\K$ can be given as the identity on objects and morphisms, as $\bar\phi \mapsto \phi$ on unary cells and as $\bar\psi \mapsto \eps_C \of \psi$ on nullary cells $\cell{\bar\psi}{\ul J}C$, where $\eps_C$ is the nullary cartesian cell $I_C \Rar C$ that corresponds to $\cell{\eta_C}C{I_C}$ as in \lemref{unit identities}. This completes the proof.
	\end{proof}
	
	\section{Kan extensions}\label{Kan extension section}
	Now that we have most of the necessary terminology of augmented virtual double categories in place, we can begin studying `formal category theory' within such double categories. In this section we generalise the notion of Kan extension in double categories to the setting of an augmented virtual double category $\K$, and consider three strengthenings of this notion. The first of these generalises the notion of `Kan extension along enriched functors', as introduced by Dubuc in \cite{Dubuc70} (see also Section~4 of \cite{Kelly82}), while the second generalises the notion of `pointwise Kan extension in a $2$-category', that was introduced by Street in \cite{Street74b}; the third then combines the latter two strengthenings. We will compare these strengthenings among each other, as well as compare them with the classical such notions for the vertical $2$-category $V(\K)$ that is contained in $\K$. In the next section we consider the notion of yoneda embedding.
	
	\subsection{Weak Kan extensions}
	We start with the notion of `weak' left Kan extension, which generalises the notion of Kan extension for double categories that was given in Section 2 of \cite{Grandis-Pare08}.
	\begin{definition} \label{weak left Kan extension}
		The nullary cell $\eta$, in the composite on the right-hand side below, is said to define $\map l{A_n}M$ as the \emph{weak left Kan extension} of $\map d{A_0}M$ along $(J_1, \dotsc, J_n)$ if any nullary cell $\phi$, as on the left-hand side, factors uniquely through $\eta$ as a vertical cell $\phi'$, as shown.
		\begin{displaymath}
			\begin{tikzpicture}[textbaseline]
				\matrix(m)[math35, yshift=1.625em]{A_0 & A_1 & A_{n'} & A_n \\};
				\draw	([yshift=-3.25em]$(m-1-1)!0.5!(m-1-4)$) node (M) {$M$};
				\path[map]	(m-1-1) edge[barred] node[above] {$J_1$} (m-1-2)
														edge[transform canvas={yshift=-2pt}] node[below left] {$d$} (M)
										(m-1-3) edge[barred] node[above] {$J_n$} (m-1-4)
										(m-1-4) edge[transform canvas={yshift=-2pt}] node[below right] {$k$} (M);
				\path				($(m-1-2.south)!0.5!(m-1-3.south)$) edge[cell] node[right] {$\phi$} (M);
				\draw				($(m-1-2)!0.5!(m-1-3)$) node {$\dotsb$};
										
			\end{tikzpicture} = \begin{tikzpicture}[textbaseline]
				\matrix(m)[math35]{A_0 & A_1 & A_{n'} & A_n \\ & & & M \\};
				\path[map]	(m-1-1) edge[barred] node[above] {$J_1$} (m-1-2)
														edge[transform canvas={shift={(-2pt,-2pt)}}] node[below left] {$d$} (m-2-4)
										(m-1-3) edge[barred] node[above] {$J_n$} (m-1-4)
										(m-1-4) edge[bend right=45] node[left] {$l$} (m-2-4)
														edge[bend left=45] node[right] {$k$} (m-2-4);
				\path				($(m-1-1.south)!0.5!(m-1-4.south)$) edge[transform canvas={shift={(0.3em, 0.5em)}}, cell] node[right] {$\eta$} ($(m-2-1.north)!0.5!(m-2-4.north)$)
										(m-1-4) edge[cell, transform canvas={xshift=-0.25em}] node[right] {$\phi'$} (m-2-4);
				\draw				($(m-1-2)!0.5!(m-1-3)$) node {$\dotsb$};
										
			\end{tikzpicture}
		\end{displaymath}
	\end{definition}
	As usual any two nullary cells defining the same weak left Kan extension factor through each other as invertible vertical cells. 
	
	\begin{example} \label{Kan extensions in quintets}
		In $Q(\mathcal C)$, the double category of quintets in a $2$"/category $\mathcal C$ (see \defref{quintets}, the notion of weak Kan extension coincides with the usual $2$"/categorical notion of Kan extension given in Section~2 of \cite{Street72} (or see e.g.\ Section~2 of \cite{Weber07}).
	\end{example}
	
	The definition above generalises the classical notion of left Kan extension in a $2$-category in the following sense. Remember that any augmented virtual double category $\K$ contains a $2$-category $V(\K)$ consisting of its objects, vertical morphisms and vertical cells.
	
	\begin{proposition} \label{weak left Kan extensions along companions}
		In an augmented virtual double category $\K$ consider a vertical cell $\eta$, as on the left-hand side below, and assume that the companion $j_*$ of $\map jAB$ exists.
		\begin{displaymath}

  	\end{displaymath}
		The vertical cell $\eta$ defines $l$ as the left Kan extension of $d$ along $j$ in $V(\K)$ precisely if its factorisation $\eta'$, as shown, defines $l$ as the weak left Kan extension of $d$ along $j_*$ in $\K$.
	\end{proposition}
	\begin{proof}
		This is a simple consequence of the universal property of the weakly cocartesian cell defining $j_*$ as follows. Consider the diagram of assignments below, between collections of cells that are of the form as drawn. The identity above implies that the diagram commutes.
		\begin{displaymath}
			%
		\end{displaymath}
		Now the cell $\eta'$ defines $l$ as a weak left Kan extension in $\K$ precisely if the top assignment on the left is bijective, while $\eta$ defines $l$ as a left Kan extension in $V(\K)$ precisely if the top assignment on the right is bijective. The proof follows from the fact that the bottom assignment is bijective, by the universal property of the weakly cocartesian cell that defines $j_*$.
	\end{proof}
	
	\subsection{(Pointwise) Kan extensions}
	\defref{weak left Kan extension} is strengthened to give a notion of left Kan extension in augmented virtual double categories as follows. This generalises the corresponding notion for double categories, that was given in Section 3 of \cite{Koudenburg14a} under the name `pointwise left Kan extension'.
	\begin{definition} \label{left Kan extension}
		The nullary cell $\eta$, in the composite on the right-hand side below, is said to define $\map l{A_n}M$ as the \emph{left Kan extension} of $\map d{A_0}M$ along $(J_1, \dotsc, J_n)$ if any nullary cell $\phi$, this time of the more general form as on the left-hand side below, factors uniquely through $\eta$ as a nullary cell $\phi'$, as shown.
		\begin{displaymath}

		\end{displaymath}
	\end{definition}
	Clearly every left Kan extension is a weak left Kan extension, by restricting the universal property above to cells $\phi$ with $m = 0$. The following result, which is an immediate consequence of \propref{left Kan extensions preserved by restriction} below, leads us to the definition of pointwise left Kan extensions. Recall that an object $A$ is called unital whenever its horizontal unit $\hmap{I_A}AA$ exists.
	\begin{corollary} \label{left Kan extensions with unital sources}
		In an augmented virtual equipment consider a nullary cell $\eta$ as in the composite below, where $n \geq 1$, and suppose that the object $A_n$ is unital. The cell $\eta$ defines $l$ as the left Kan extension of $d$ along $(J_1, \dotsc, J_n)$ precisely if, for each vertical morphism $\map fB{A_n}$, the full composite defines $l \of f$ as the left Kan extension of $d$ along $\bigpars{J_1, \dotsc, J_n(\id, f)}$.
		\begin{displaymath}
			\begin{tikzpicture}
				\matrix(m)[math35]{A_0 & A_1 & A_{n'} & B \\ A_0 & A_1 & A_{n'} & A_n \\};
				\draw	([yshift=-6.5em]$(m-1-1)!0.5!(m-1-4)$) node (M) {$M$};
				\path[map]	(m-1-1) edge[barred] node[above] {$J_1$} (m-1-2)
														
										(m-1-3) edge[barred] node[above, xshift=-3pt] {$J_n(\id, f)$} (m-1-4)
										(m-1-4) edge node[right] {$f$} (m-2-4)
										(m-2-4) edge[transform canvas={yshift=-2pt}] node[below right] {$l$} (M)
										(m-2-1) edge[barred] node[below, inner sep=2.5pt] {$J_1$} (m-2-2)
														edge[transform canvas={yshift=-2pt}] node[below left] {$d$} (M)
										(m-2-3) edge[barred] node[below, inner sep=2.5pt] {$J_n$} (m-2-4);
				\path				(m-1-1) edge[eq] (m-2-1)
										(m-1-2) edge[eq] (m-2-2)
										(m-1-3) edge[eq, transform canvas={xshift=-2pt}] (m-2-3);
				\path				($(m-2-2.south)!0.5!(m-2-3.south)$) edge[cell] node[right] {$\eta$} (M);
				\draw				($(m-1-2)!0.5!(m-2-3)$) node {$\dotsb$}
										($(m-1-3)!0.5!(m-2-4)$) node[font=\scriptsize] {$\cart$};
			\end{tikzpicture}
		\end{displaymath}
	\end{corollary}
	\begin{proof}
		For the `if'-part take $f = \id_{A_n}$ and use that $J_n(\id, \id) \iso J_n$. For the `only if'-part remember $A_n$ being unital ensures that all conjoints $\hmap{f^*}{A_n}B$ exist, by \corref{restrictions in terms of units}, and apply \propref{left Kan extensions preserved by restriction} below.
	\end{proof}
	
	\begin{definition} \label{pointwise left Kan extension}
		A nullary cell $\eta$ as in the composite below, where $n \geq 1$, is said to define $l$ as the \emph{pointwise} (weak) left Kan extension of $d$ along $(J_1, \dotsc, J_n)$ if, for any $\map fB{A_n}$ such that $J_n(\id, f)$ exists, the full composite defines $l \of f$ as the (weak) left Kan extension of $d$ along $\bigpars{J_1, \dotsc, J_n(\id, f)}$.
		\begin{displaymath}
			\begin{tikzpicture}
				\matrix(m)[math35]{A_0 & A_1 & A_{n'} & B \\ A_0 & A_1 & A_{n'} & A_n \\};
				\draw	([yshift=-6.5em]$(m-1-1)!0.5!(m-1-4)$) node (M) {$M$};
				\path[map]	(m-1-1) edge[barred] node[above] {$J_1$} (m-1-2)
														
										(m-1-3) edge[barred] node[above, xshift=-3pt] {$J_n(\id, f)$} (m-1-4)
										(m-1-4) edge node[right] {$f$} (m-2-4)
										(m-2-4) edge[transform canvas={yshift=-2pt}] node[below right] {$l$} (M)
										(m-2-1) edge[barred] node[below, inner sep=2.5pt] {$J_1$} (m-2-2)
														edge[transform canvas={yshift=-2pt}] node[below left] {$d$} (M)
										(m-2-3) edge[barred] node[below, inner sep=2.5pt] {$J_n$} (m-2-4);
				\path				(m-1-1) edge[eq] (m-2-1)
										(m-1-2) edge[eq] (m-2-2)
										(m-1-3) edge[eq, transform canvas={xshift=-2pt}] (m-2-3);
				\path				($(m-2-2.south)!0.5!(m-2-3.south)$) edge[cell] node[right] {$\eta$} (M);
				\draw				($(m-1-2)!0.5!(m-2-3)$) node {$\dotsb$}
										($(m-1-3)!0.5!(m-2-4)$) node[font=\scriptsize] {$\cart$};
			\end{tikzpicture}
		\end{displaymath}
	\end{definition}
	Clearly every pointwise left Kan extension is a pointwise weak left Kan extension; that it is a left Kan extension too follows from the following lemma. As a consequence of \corref{left Kan extensions with unital sources}, all left Kan extensions in a unital virtual equipment, such as $\enProf\V$ of $\V$-profunctors (\exref{enriched profunctors}), are pointwise left Kan extensions. In \secref{pointwise Kan extensions section} we will see that all pointwise weak left Kan extensions in an augmented virtual double category $\K$ are pointwise left Kan extensions as soon as $\K$ has `cocartesian tabulations'; e.g.\ $\K = \enProf{(\Set, \Set')}$ of \exref{(V, V')-Prof}. It follows that in a unital virtual equipment admitting cocartesian tabulations, such as $\enProf\2$ (see \exref{relations} and \exref{tabulations of 2-profunctors}), the notions of left Kan extension, pointwise weak left Kan extension, and pointwise left Kan extension coincide.
	
	\begin{lemma} \label{properties of pointwise left Kan extensions}
		Consider a nullary cell $\eta$ as in the composite above, that defines $l$ as the pointwise (weak) left Kan extension of $d$ along $(J_1, \dotsc, J_n)$. The following hold:
		\begin{enumerate}[label=\textup{(\alph*)}]
			\item $\eta$ defines $l$ as the (weak) left Kan extension of $d$ along $(J_1, \dotsc, J_n)$;
			\item for each $\map fB{A_n}$ the composite above defines $l \of f$ as the pointwise (weak) left Kan extension of $d$ along $\bigpars{J_1, \dotsc, J_n(\id, f)}$.
		\end{enumerate}
	\end{lemma}
	\begin{proof}
		For (a) take $f = \id_{A_n}$ in the previous definition and use that $J_n(\id, \id) \iso J_n$. For (b) remember that $J(\id, f)(\id, g) \iso J(\id, g \of f)$ for composable $\map gCB$ and $\map fB{A_n}$.
	\end{proof}
	
	\begin{example}
		For a Kan extension that fails to be a pointwise weak Kan extension consider the augmented virtual double category $\enProf{(\brks{0, \infty}, \brks{-\infty, \infty})}$ described in \exref{universe enlargement of Lawvere quantale}. Let $M$ be the generalised metric space with two points as pictured on the left below and $I$ be the `unit $\brks{-\infty, \infty}$"/category', consisting of a single point $*$ with $I(*, *) = 0$. Let $\hmap JIM$ be given by $J(*, d) = 3$ and $J(*, u) = 2$; it is straightforward to check the existence of the cell $\eta$ in the middle below.
		\begin{displaymath}
			\begin{tikzpicture}[textbaseline]
				\matrix(m)[math35, column sep=2.5em]{ & \\};
				\path[map]	(m-1-1) edge[bend left=40] node[above] {$-1$} (m-1-2)
										(m-1-2)	edge[bend left=40] node[below] {$1$} (m-1-1);
				\draw[map, transform canvas={yshift=3pt}]	(m-1-1) arc(-20:280:0.15) node[above, inner sep=10pt] {$0$};
				\draw[map, transform canvas={yshift=3pt}]	(m-1-2) arc(200:-100:0.15) node[above, inner sep=10pt] {$0$};
				\fill	(m-1-1) circle (1pt) node[below left] {$d$}
							(m-1-2) circle (1pt) node[below right, xshift=-2pt] {$u$};
			\end{tikzpicture} \qquad\qquad\qquad\qquad \begin{tikzpicture}[textbaseline]
				\matrix(m)[math35, column sep={1.75em,between origins}]{I & & M \\ & M & \\};
				\path[map]	(m-1-1) edge[barred] node[above] {$J$} (m-1-3)
														edge[transform canvas={xshift=-2pt}] node[left] {$d$} (m-2-2);
				\path[map, transform canvas={yshift=0.25em}]	(m-1-2) edge[cell] node[right] {$\eta$} (m-2-2);
				\path				(m-1-3) edge[eq, transform canvas={xshift=2pt}] (m-2-2);
			\end{tikzpicture} \qquad\qquad\qquad\qquad \begin{tikzpicture}[textbaseline]
				\matrix(m)[math35, column sep={1.75em,between origins}]{M & & B \\ & M & \\};
				\path[map]	(m-1-1) edge[barred] node[above] {$H$} (m-1-3)
										(m-1-3)	edge[transform canvas={xshift=2pt}] node[right] {$k$} (m-2-2);
				\path[map, transform canvas={yshift=0.25em}]	(m-1-2) edge[cell] node[right] {$\phi'$} (m-2-2);
				\path				(m-1-1) edge[eq, transform canvas={xshift=-2pt}] (m-2-2);
			\end{tikzpicture}
		\end{displaymath}
		
		The cell $\eta$ fails to define $\id_M$ as a pointwise weak left Kan extension. Indeed the universal property of its restriction along the map $I \to M\colon * \mapsto u$ reduces to
		\begin{displaymath}
			2 \geq M(d, z) \quad \implies \quad 0 \geq M(u, z)
		\end{displaymath}
		for all $z \in M$, but this fails for $z = d$. To see that $\eta$ does define $\id_M$ as a left Kan extension first notice that in the universal property of \defref{left Kan extension} we may restrict to cells $\phi$ with $m = 0$ or $m = 1$: similar to the argument that proves \propref{left Kan extensions in double categories} below, this is a consequence of the fact that $\enProf{(\brks{0, \infty}, \brks{-\infty, \infty})}$ has all horizontal composites (\exref{composites of non-expanding metric relations}). Hence to prove the universal property it suffices to show that all vertical cells $\id_M \Rightarrow h$, where $\map hMM$, exist in $\enProf{(\brks{0, \infty}, \brks{-\infty, \infty})}$ as well as all cells $\phi'$ as on the right above. That they do exist is easily seen, by using the fact that the morphisms $\map hMM$ and $\hmap HMB$ are required to be non"/expanding (see \exref{universe enlargement of Lawvere quantale}).
	\end{example}
	
	To close this subsection we give an important example of pointwise left Kan extension, that of the restriction along a vertical morphism, as follows. In fact, such Kan extension is `absolute', in the following sense: a nullary cell $\eta$, as in \defref{weak left Kan extension}, is said to define $l$ as an \emph{absolute} (pointwise) (weak) left Kan extension if, for any vertical morphism $\map gMN$, the composite $g \of \eta$ defines $g \of l$ as the (pointwise) (weak) left Kan extension of $g \of d$ along $(J_1, \dotsc, J_n)$.
	\begin{proposition} \label{restriction as left Kan extension}
		Let $\map jBA$ be a vertical morphism. Any nullary cartesian cell
		\begin{displaymath}

		\end{displaymath}
		that defines the conjoint of $j$, defines $j$ as the absolute pointwise left Kan extension of $\id_A$ along $j^*$.
	\end{proposition}
	\begin{proof}
		To prove that the cartesian cell above defines $j$ as an absolute left Kan extension we have to show that, for all $\map gAN$, the assignment below, between collections of cells that are of the form as shown, is a bijection.
		\begin{displaymath}
			%
		\end{displaymath}
		To see that it is, we claim that an inverse can be given by $\phi \mapsto \phi \of (\cocart, \id, \dotsc, \id)$, where $\cocart$ denotes the weakly cocartesian cell that corresponds to the cartesian cell defining $j^*$, as described before \lemref{companion identities lemma}. The proof of the claim is a straightforward calculation using the conjoint identities satisfied by $\cart$ and $\cocart$ (horizontally dual to those in \lemref{companion identities lemma}) as well as the interchange axioms (\lemref{horizontal composition}). We conclude that $j$ forms the absolute left Kan extension of $\id_A$ along $j^*$. That it is in fact an absolute pointwise left Kan extension follows from the simple fact that $j^*(\id, f) \iso (j \of f)^*$ for any $\map fCB$.
	\end{proof}

	\subsection{Pasting lemmas}	
	In this subsection we describe two important pasting lemmas for Kan extension as well as their consequences, which will be used throughout. We begin with the horizontal pasting lemma for Kan extensions which, in view of \propref{enriched left Kan extensions in terms of weighted colimits} below, generalises a classical result on the iteration of Kan extensions of enriched functors; see page~42 of \cite{Dubuc70} or Theorem 4.47 of \cite{Kelly82}.
	\begin{lemma}[Horizontal pasting lemma] \label{horizontal pasting lemma}
		In an augmented virtual double category consider horizontally composable nullary cells
		\begin{displaymath}
			\begin{tikzpicture}[textbaseline]
				\matrix(m)[math35]
					{	A_0 & A_1 & A_{n'} & A_n & B_1 & B_{m'} & B_m \\
						& & & M & & & \\};
				\path[map]	(m-1-1) edge[barred] node[above] {$J_1$} (m-1-2)
														edge[transform canvas={yshift=-2pt}] node[below left] {$d$} (m-2-4)
										(m-1-3) edge[barred] node[above] {$J_n$} (m-1-4)
										(m-1-4) edge[barred] node[above] {$H_1$} (m-1-5)
														edge node[right] {$l$} (m-2-4)
										(m-1-6) edge[barred] node[above] {$H_m$} (m-1-7)
										(m-1-7) edge[transform canvas={yshift=-2pt}] node[below right] {$k$} (m-2-4);
				\draw[transform canvas={xshift=-0.5pt}]	($(m-1-2)!0.5!(m-1-3)$) node {$\dotsb$}
										($(m-1-5)!0.5!(m-1-6)$) node {$\dotsb$};
				\path				($(m-1-1.south)!0.5!(m-1-4.south)$)	edge[cell, transform canvas={shift={(1em,0.333em)}}] node[right] {$\eta$} ($(m-2-1.north)!0.5!(m-2-4.north)$)
										($(m-1-4.south)!0.5!(m-1-7.south)$) edge[cell, transform canvas={shift={(-1em,0.333em)}}] node[right] {$\zeta$} ($(m-2-4.north)!0.5!(m-2-7.north)$);
			\end{tikzpicture}
		\end{displaymath}
		and suppose that $\eta$ defines $l$ as the left Kan extension of $d$ along $(J_1, \dotsc, J_n)$. Then $\eta \hc \zeta$ defines $k$ as the (weak) left Kan extension of $d$ along $(J_1, \dotsc, J_n, H_1, \dotsc, H_m)$ precisely if $\zeta$ defines $k$ as the (weak) left Kan extension of $l$ along $(H_1, \dotsc, H_m)$.
		
		Analogously $\eta \hc \zeta$ defines $k$ as the pointwise (weak) left Kan extension precisely if $\zeta$ does.
	\end{lemma}
	This lemma restricts to the classical pasting lemma for left Kan extensions in a $2$"/category $\mathcal C$ (see e.g.\ Proposition~1 of \cite{Street-Walters78}) when applied to the double category of quintets $Q(\mathcal C)$ (see \exref{Kan extensions in quintets}).
	\begin{proof}
		We will consider the case of left Kan extensions; the pointwise/weak cases are analogous. Given any path $\hmap{\ul K}{B_m}{C_p}$ and any morphism $\map h{C_p}M$ consider the following commutative diagram of assignments between collections of cells, that are of the form as shown.
		\begin{displaymath}
			\begin{tikzpicture}
				\matrix(n)[math35, column sep={2em,between origins}]{A_n & & C_p \\ & M & \\};
				\path[map]	(n-1-1) edge[barred] node[above] {$\ul H \conc \ul K$} (n-1-3)
														edge[transform canvas={xshift=-1pt}] node[left] {$l$} (n-2-2)
										(n-1-3) edge[transform canvas={xshift=1pt}] node[right] {$h$} (n-2-2);
				\path[transform canvas={yshift=0.25em}]	(n-1-2) edge[cell] node[right] {$\phi'$} (n-2-2);
				
				\matrix(n)[math35, column sep={1.75em,between origins}, xshift=-11em]{B_m & & C_p \\ & M & \\};
				\path[map]	(n-1-1) edge[barred] node[above] {$\ul K$} (n-1-3)
														edge[transform canvas={xshift=-2pt}] node[left] {$k$} (n-2-2)
										(n-1-3) edge[transform canvas={xshift=2pt}] node[right] {$h$} (n-2-2);
				\path[transform canvas={yshift=0.25em}]	(n-1-2) edge[cell] node[right, inner sep=3pt] {$\phi''$} (n-2-2);
				
				\matrix(n)[math35, column sep={2.5em,between origins}, xshift=11.6em]{A_0 & & C_p \\ & M & \\};
				\path[map]	(n-1-1) edge[barred] node[above] {$\ul J \conc \ul H \conc \ul K$} (n-1-3)
														edge node[left] {$d$} (n-2-2)
										(n-1-3) edge node[right] {$h$} (n-2-2);
				\path[transform canvas={yshift=0.25em}]	(n-1-2) edge[cell] node[right] {$\phi$} (n-2-2);
				
				\draw[font=\Large]	(-2.9em,0) node {$\lbrace$}
										(2.8em,0) node {$\rbrace$}
										(-13.75em,0) node {$\lbrace$}
										(-8.3em,0) node {$\rbrace$}
										(8.2em,0) node {$\lbrace$}
										(14.9em,0) node {$\rbrace$};
				\path[map]	(-7.3em,0) edge node[above] {$\zeta \hc \dash$} (-3.9em,0)
										(3.8em,0) edge node[above] {$\eta \hc \dash$} (7.2em,0)
										(-8.5em,-2.5em) edge[bend right=20] node[below] {$\eta \hc \zeta \hc \dash$} (9em,-2.5em);
			\end{tikzpicture}
		\end{displaymath}
		The proof now follows from the fact that, by definition, the cell $\eta$ (resp.\ $\zeta$ or \mbox{$\eta \hc \zeta$}) defines $l$ (resp.\ $k$) as a left Kan extension whenever the assignment given by horizontal composition with $\eta$ (resp. $\zeta$ or $\eta \hc \zeta$) above is a bijection.
	\end{proof}
	
	Next is the vertical pasting lemma.
	\begin{lemma}[Vertical pasting lemma] \label{vertical pasting lemma}
		Assume that, in the composite below, the path $(\phi_1, \dotsc, \phi_n)$ is cocartesian. The nullary cell $\eta$ defines $l$ as the (weak) left Kan extension of $d$ along $(K_1, \dotsc, K_n)$ precisely if the full composite defines $l$ as the (weak) left Kan extension of $d \of f_0$ along $(J_{11}, \dotsc, J_{nm_n})$.
		\begin{displaymath}
			\begin{tikzpicture}
				\matrix(m)[math35, column sep={4em,between origins}]
					{	A_{10} & A_{11} & A_{1m'_1} & A_{1m_1} &[4em] A_{n0} & A_{n1} & A_{nm'_n} & A_{nm_n} \\
						C_0 & & & C_1 & C_{n'} & & & A_{nm_n} \\};
				\draw				([yshift=-3.25em]$(m-2-1)!0.5!(m-2-8)$) node (M) {$M$};
				\path[map]	(m-1-1) edge[barred] node[above] {$J_{11}$} (m-1-2)
														edge node[left] {$f_0$} (m-2-1)
										(m-1-3) edge[barred] node[above] {$J_{1m_1}$} (m-1-4)
										(m-1-4) edge node[right] {$f_1$} (m-2-4)
										(m-1-5)	edge[barred] node[above] {$J_{n1}$} (m-1-6)
														edge node[left] {$f_{n'}$} (m-2-5)
										(m-1-7) 	edge[barred] node[above] {$J_{nm_n}$} (m-1-8)
										(m-2-1) edge[barred] node[below] {$K_1$} (m-2-4)
														edge[transform canvas={yshift=-2pt}] node[below left] {$d$} (M)
										(m-2-5)	edge[barred] node[below] {$K_n$} (m-2-8)
										(m-2-8)	edge[transform canvas={yshift=-2pt}] node[below right] {$l$} (M);
				\path				($(m-1-1.south)!0.5!(m-1-4.south)$) edge[cell] node[right] {$\phi_1$} ($(m-2-1.north)!0.5!(m-2-4.north)$)
										($(m-1-5.south)!0.5!(m-1-8.south)$) edge[cell] node[right] {$\phi_n$} ($(m-2-5.north)!0.5!(m-2-8.north)$)
										($(m-2-1.south)!0.5!(m-2-8.south)$) edge[cell] node[right] {$\eta$} (M)
										(m-1-8)	edge[eq] (m-2-8);
				\draw[transform canvas={xshift=-1pt}]	($(m-1-2)!0.5!(m-1-3)$) node {$\dotsb$}
										($(m-1-6)!0.5!(m-1-7)$) node {$\dotsb$};
				\draw				($(m-1-4)!0.5!(m-2-5)$) node {$\dotsb$};
			\end{tikzpicture}
		\end{displaymath}
		
		Analogously, if $(\phi_1, \dotsc, \phi_n)$ is right pointwise cocartesian (\defref{pointwise cocartesian path}) then $\eta$ defines $l$ as a pointwise (weak) left Kan extension of $d$ along $(K_1, \dotsc, K_n)$ precisely if the full composite defines $l$ as the pointwise (weak) left Kan extension of $d \of f_0$ along $(J_{11}, \dotsc, J_{nm_n})$.
	\end{lemma}
	\begin{proof}
		We will consider in the case of pointwise left Kan extensions; the other cases follow by, in the following, either taking $f = \id_{A_{nm_n}}$ or choosing the path $\ul H$ to be empty. Hence we assume that $(\phi_1, \dotsc, \phi_n)$ is right pointwise cocartesian. Let $\map fB{A_{nm_n}}$ be any morphism and remember from \defref{pointwise cocartesian path} that the restriction $J_{nm_n}(\id, f)$ exists if and only if $K_n(\id, f)$ does. In that case $(\phi_1, \dotsc, \phi_n')$ is cocartesian, where $\phi_n'$ is obtained from $\phi_n$ as in the factorsation \eqref{pointwise cocartesian factorisation}. We consider the following assignments between collections of cells, that are of the form as shown, where $\zeta \dfn \eta \of (\id, \dotsc, \id, \cart)$, with $\cart$ the cartesian cell defining $K_n(\id, f)$, and $\theta \dfn \eta \of (\phi_1, \dotsc, \phi_n) \of (\id, \dotsc, \id, \cart)$, with $\cart$ the cartesian cell defining $J_{nm_n}(\id, f)$. It follows from \eqref{pointwise cocartesian factorisation} that the diagram commutes.
		\begin{displaymath}

		\end{displaymath}
		Now notice that $\eta$ defines $l$ as a pointwise left Kan extension precisely if, for every $\map fB{A_{nm_n}}$, the top assignment is a bijection, while $\eta \of (\phi_1, \dotsc, \phi_n)$ does so precisely if the left assignment is a bijection, for every $\map fB{A_{nm_n}}$. Thus the proof follows from the fact that $(\phi_1, \dotsc, \phi'_n)$ is cocartesian, so that the assignment on the right is a bijection.
	\end{proof}
	
	Applying the vertical pasting lemma to a single horizontal cocartesian cell we obtain the following corollary.
		\begin{corollary} \label{Kan extensions and composites}
		In an augmented virtual double category $\K$ consider a nullary cell $\eta$ as on the left-hand side below.
		\begin{displaymath}
			\begin{tikzpicture}[textbaseline]
				\matrix(m)[math35, yshift=1.625em]{A_0 & A_1 & A_{n'} & A_n \\};
				\draw	([yshift=-3.25em]$(m-1-1)!0.5!(m-1-4)$) node (M) {$M$};
				\path[map]	(m-1-1) edge[barred] node[above] {$J_1$} (m-1-2)
														edge[transform canvas={yshift=-2pt}] node[below left] {$d$} (M)
										(m-1-3) edge[barred] node[above] {$J_n$} (m-1-4)
										(m-1-4) edge[transform canvas={yshift=-2pt}] node[below right] {$l$} (M);
				\path				($(m-1-2.south)!0.5!(m-1-3.south)$) edge[cell] node[right] {$\eta$} (M);
				\draw				($(m-1-2)!0.5!(m-1-3)$) node {$\dotsb$};
										
			\end{tikzpicture} = \begin{tikzpicture}[textbaseline]
				\matrix(m)[math35, yshift=1.625em]{A_0 & A_1 & A_{n'} & A_n \\ A_0 & & & A_n \\};
				\draw	([yshift=-6.5em]$(m-1-1)!0.5!(m-1-4)$) node (M) {$M$};
				\path[map]	(m-1-1) edge[barred] node[above] {$J_1$} (m-1-2)
														
										(m-1-3) edge[barred] node[above] {$J_n$} (m-1-4)
										(m-2-4) edge[transform canvas={yshift=-2pt}] node[below right] {$l$} (M)
										(m-2-1) edge[barred] node[below] {$(J_1 \hc \dotsb \hc J_n)$} (m-2-4)
														edge[transform canvas={yshift=-2pt}] node[below left] {$d$} (M);
				\path				(m-1-1) edge[eq] (m-2-1)
										(m-1-4) edge[eq] (m-2-4);
				\path				($(m-2-2.south)!0.5!(m-2-3.south)$) edge[cell, transform canvas={yshift=-0.25em}] node[right] {$\eta'$} (M);
				\draw				($(m-1-2)!0.5!(m-1-3)$) node {$\dotsb$}
										($(m-1-1)!0.5!(m-2-4)$) node[font=\scriptsize] {$\cocart$};
			\end{tikzpicture}
		\end{displaymath}
		If the composite $(J_1 \hc \dotsb \hc J_n)$ exists then $\eta$ defines $l$ as the (weak) left Kan extension of $d$ along $(J_1, \dotsc, J_n)$ precisely if its factorisation $\eta'$, as shown, defines $l$ as the (weak) left Kan extension of $d$ along $(J_1 \hc \dotsb \hc J_n)$.
		
		Moreover if the cocartesian cell above is right pointwise (\defref{pointwise cocartesian path}), then the analogous result for cells $\eta$ defining $l$ as a pointwise (weak) left Kan extension holds.
	\end{corollary}
	
	The following result was used in the proof of \corref{left Kan extensions with unital sources}.
	\begin{proposition} \label{left Kan extensions preserved by restriction}
		Assume that the nullary cell $\eta$ in the composites below defines $l$ as the left Kan extension of $d$ along $(J_1, \dotsc, J_n)$ and let $\map fB{A_n}$ be a morphism whose conjoint exists. The composite on the left defines $l \of f$ as the left Kan extension of $d$ along $(J_1, \dotsc, J_n, f^*)$.
		\begin{displaymath}

		\end{displaymath}
		If the restriction $J_n(\id, f)$ exists as well then the composite on the right defines $l \of f$ as the left Kan extension of $d$ along $\bigpars{J_1, \dotsc, J_n(\id, f)}$.
	\end{proposition}
	\begin{proof}
		The composite on the left above can be rewritten as $\eta \hc (l \of \cart)$ by the axioms for horizontal composition (\lemref{horizontal composition}). Since both $\eta$ and $l \of \cart$ define left Kan extensions, the first by assumption and the latter by \propref{restriction as left Kan extension}, so does their composite by the horizontal pasting lemma.
		
		For the second assertion assume that the restriction $J_n(\id, f)$ exists and is defined by the cartesian cell in the composite above. By \lemref{restrictions and composites} we know that $J_n(\id, f)$ forms the horizontal composite of $J_n$ and $f^*$, whose defining cocartesian cell $\cell\eps{(J_n, f^*)}{J_n(\id, f)}$ is the unique factorisation of the composite $\id_{J_n} \hc \cart$, in the composite on the left, through the cartesian cell defining $J_n(\id, f)$, in the composite on the right. We conclude that the composite on the left factors through that on the right as the path of cells $(\id_{J_1}, \dotsc, \id_{J_{n'}}, \eps)$. Since $\eps$ is cocartesian it follows from the vertical pasting lemma that either composite defines $l \of f$ as a left Kan extension if the other does, from which the proof follows.
	\end{proof}
	
	Similar to the previous result is the following.
	\begin{proposition} \label{Kan extension along conjoints}
		Let $\map fAC$ be a morphism whose conjoint exists. The nullary cell $\eta$ in the composite below defines $l$ as the (pointwise) (weak) left Kan extension of $d \of f$ along $(J_1, \dotsc, J_n)$ precisely if the full composite defines $l$ as the (pointwise) (weak) left Kan extension of $d$ along $(f^*, J_1, \dotsc, J_n)$.
		\begin{displaymath}
			\begin{tikzpicture}
				\matrix(m)[math35, column sep={1.75em,between origins}]
					{ C & & A & & A_1 & & A_{n'} & & A_n \\
						& C & & & & & & & \\};
				\draw				([yshift=-6.5em]$(m-1-1)!0.5!(m-1-9)$)	node (M) {$M$};
				\path[map]	(m-1-1) edge[barred] node[above] {$f^*$} (m-1-3)
										(m-1-3) edge[barred] node[above] {$J_1$} (m-1-5)
														edge[transform canvas={xshift=2pt}]	node[right] {$f$} (m-2-2)
										(m-1-7) edge[barred] node[above] {$J_n$} (m-1-9)
										(m-1-9) edge node[below right] {$l$} (M)
										(m-2-2) edge node[below left] {$d$} (M);
				\path				(m-1-1) edge[eq, transform canvas={xshift=-2pt}]	(m-2-2)
										($(m-1-1.south)!0.5!(m-1-9.south)$) edge[cell, transform canvas={yshift=-1.75em}] node[right] {$\eta$} ($(m-2-1.north)!0.5!(m-2-9.north)$);
				\draw				($(m-1-5)!0.5!(m-1-7)$) node {$\dotsb$}
										([yshift=0.333em]$(m-1-2)!0.5!(m-2-2)$) node[font=\scriptsize] {$\cart$};
			\end{tikzpicture}
		\end{displaymath}
	\end{proposition}
	\begin{proof}
		The composite above can be rewritten as $(d \of \cart) \hc \eta$ by the axioms for horizontal composition, so that the proof follows from \propref{restriction as left Kan extension} and the horizontal pasting lemma.
	\end{proof}

	The following three propositions generalise well known results for $2$-categories and enriched categories. The first of these generalises Proposition~22 of \cite{Street74b} (for 2-categories) and Proposition~4.23 of \cite{Kelly82} (for enriched categories); see also Proposition~14 of \cite{Wood82} for the analogous result in proarrow equipments. Remember that a vertical morphism $\map fAC$ is called full and faithful (\defref{full and faithful morphism}) whenever the restriction $C(f, f)$ exists and the identity cell $\id_f$ is cartesian.
	\begin{proposition} \label{pointwise left Kan extension along full and faithful map}
		Let $\map f{A_n}B$ be a full and faithful morphism whose companion $f_*$ exists and assume that the nullary cell $\eta$ in the composite below defines $l$ as the pointwise (weak) left Kan extension of $d$ along $(J_1, \dotsc, J_n, f_*)$. The full composite defines $l \of f$ as the (weak) left Kan extension of $d$ along $(J_1, \dotsc, J_n)$; in particular if $n = 0$ then the composite, in that case a vertical cell, is invertible.
		\begin{displaymath}
			\begin{tikzpicture}
				\matrix(m)[math35, column sep={1.75em,between origins}]
					{	& A_0 & & A_1 & & A_{n'} & & A_n & \\
						A_0 & & A_1 & & A_{n'} & & A_n & & B \\};
				\draw				([yshift=-3.25em]$(m-2-1)!0.5!(m-2-9)$) node (M) {$M$};
				\path[map]	(m-1-2) edge[barred] node[above] {$J_1$} (m-1-4)
										(m-1-6) edge[barred] node[above] {$J_n$} (m-1-8)
										(m-1-8) edge[transform canvas={xshift=2pt}]	node[right] {$f$} (m-2-9)
										(m-2-1) edge[barred] node[above] {$J_1$} (m-2-3)
														edge[transform canvas={yshift=-2pt}] node[below left] {$d$} (M)
										(m-2-5) edge[barred] node[above] {$J_n$} (m-2-7)
										(m-2-7) edge[barred] node[below, inner sep=2pt] {$f_*$} (m-2-9)
										(m-2-9) edge[transform canvas={yshift=-2pt}] node[below right] {$l$} (M);
				\draw				([xshift=-2pt]$(m-1-4)!0.5!(m-2-5)$) node {$\dotsb$}
										([yshift=-0.5em]$(m-1-8)!0.5!(m-2-8)$) node[font=\scriptsize] {$\cocart$};
				\path				($(m-2-1.south)!0.5!(m-2-9.south)$) edge[cell] node[right] {$\eta$} (M)
										(m-1-2) edge[eq, transform canvas={xshift=-2pt}] (m-2-1)
										(m-1-4) edge[eq, transform canvas={xshift=-2pt}] (m-2-3)
										(m-1-6) edge[eq, transform canvas={xshift=-2pt}] (m-2-5)
										(m-1-8)	edge[eq, transform canvas={xshift=-2pt}] (m-2-7);
			\end{tikzpicture}
		\end{displaymath}
	\end{proposition}
	\begin{proof}
		By applying \corref{unit in terms of full and faithful map} to $\map f{A_n}B$,  and by factorising the identity considered there through the cartesian cell that defines $f_*$, we find that the weakly cocartesian cell in the composite above equals the composite below where $\id_f'$, which is the factorisation of $\id_f$ through $B(f, f)$, is cocartesian. That $\eta$ composed with the cartesian cell below defines $l \of f$ as the (weak) left Kan extension of $d$ along $\bigpars{J_1, \dotsc, J_n, B(f, f)}$ is by definition (\defref{pointwise left Kan extension}). That the resulting composite $\eta \of (\id_{J_1}, \dotsc, \id_{J_n}, \cart)$, after being composed with $\id_f'$, again defines $l \of f$ as a (weak) left Kan extension follows from the vertical pasting lemma. Since the resulting composite coincides with the composite above, this proves the main assertion.
		\begin{displaymath}
			\begin{tikzpicture}
				\matrix(m)[math35, column sep={1.75em,between origins}]{& A_n & \\ A_n & & A_n \\ A_n & & B \\};
				\path[map]	(m-2-1) edge[barred] node[below] {$B(f, f)$} (m-2-3)
										(m-2-3) edge node[right] {$f$} (m-3-3)
										(m-3-1) edge[barred] node[below] {$f_*$} (m-3-3);
				\path				(m-1-2) edge[transform canvas={xshift=-2pt}, eq] (m-2-1)
														edge[transform canvas={xshift=2pt},eq] (m-2-3)
														edge[transform canvas={shift={(-0.5em,-0.5em)}}, cell] node[right, inner sep=2pt] {$\id_f'$} (m-2-2)
										(m-2-1) edge[eq] (m-3-1);
				\draw[font=\scriptsize]	([yshift=-0.25em]$(m-2-1)!0.5!(m-3-3)$) node {$\cart$};
			\end{tikzpicture}
		\end{displaymath}
		The final assertion, for the case $n = 0$, follows immediately from the easy lemma below.
	\end{proof}
	
	\begin{lemma} \label{vertical cells defining left Kan extensions}
		A vertical cell defines a (weak) left Kan extension if and only if it is invertible.
	\end{lemma}
	
	When applied in the augmented virtual double categories $Q(\mathcal C)$, of quintets (see \defref{quintets}), or $\enProf\V$, of $\V$"/profunctors (see \exref{enriched profunctors}), the following result reduces to the classical description of (enriched) right adjoints as (enriched) left Kan extensions, see Proposition~2 of \cite{Street-Walters78} and Theorem~4.81 of \cite{Kelly82} respectively. Consider a nullary cell $\eta$, as in \defref{weak left Kan extension}, that defines a morphism $\map lBM$ as a (pointwise) (weak) left Kan extension. We say that $\map gMN$ \emph{preserves $l$} if the composite $g \of \eta$ defines $g \of l$ as a (pointwise) (weak) left Kan extension.
	\begin{proposition} \label{adjunctions in terms of left Kan extensions}
		Let $\map fAC$ be a morphism whose companion $f_*$ exists. For a factorisation 
		\begin{displaymath}
			\begin{tikzpicture}[textbaseline]
				\matrix(m)[math35, column sep={1.75em,between origins}]{& A & \\ & & C \\ & A & \\};
				\path[map]	(m-1-2) edge[bend left = 18] node[right] {$f$} (m-2-3)
										(m-2-3) edge[bend left = 18] node[right] {$g$} (m-3-2);
				\path				(m-1-2) edge[bend right = 45, eq] (m-3-2);
				\path[transform canvas={yshift=-1.625em}]	(m-1-2) edge[cell] node[right] {$\eta$} (m-2-2);
			\end{tikzpicture} = \begin{tikzpicture}[textbaseline]
    		\matrix(m)[math35, column sep={1.75em,between origins}]{& A & \\ A & & C \\ & A & \\};
    		\path[map]	(m-1-2) edge[transform canvas={xshift=2pt}] node[right] {$f$} (m-2-3)
    								(m-2-1) edge[barred] node[below, inner sep=2pt] {$f_*$} (m-2-3)
    								(m-2-3)	edge[transform canvas={xshift=2pt}] node[right] {$g$} (m-3-2);
    		\path				(m-1-2) edge[eq, transform canvas={xshift=-2pt}] (m-2-1)
    								(m-2-1) edge[eq, transform canvas={xshift=-2pt}] (m-3-2);
    		\path				(m-2-2) edge[cell, transform canvas={yshift=0.1em}]	node[right, inner sep=3pt] {$\eta'$} (m-3-2);
    		\draw				([yshift=-0.5em]$(m-1-2)!0.5!(m-2-2)$) node[font=\scriptsize] {$\cocart$};
  		\end{tikzpicture}
		\end{displaymath}
		the following are equivalent:
		\begin{enumerate}[label=\textup{(\alph*)}]
			\item $\eta$ is the unit of an adjunction $f \ladj g$;
			\item $\eta'$ is cartesian;
			\item $\eta'$ defines $g$ as the weak left Kan extension of $f$ along $\id_A$, which is preserved by $f$;
			\item $\eta'$ defines $g$ as the absolute pointwise left Kan extension of $f$ along $\id_A$.
		\end{enumerate}
		Furthermore, under these conditions and assuming that the restriction $C(f, f)$ exists, $f$ is full and faithful if and only if $\eta$ is invertible.
	\end{proposition}
	\begin{proof}
		We have (a) $\Rightarrow$ (b) by \lemref{adjunctions} and (b) $\Rightarrow$ (d) by \propref{restriction as left Kan extension} while clearly (d) $\Rightarrow$ (c), so that (c) $\Rightarrow$ (a) remains. Assuming (c) we consider the factorisation $\eps$ below.
		\begin{displaymath}
			\begin{tikzpicture}[textbaseline]
				\matrix(m)[math35, column sep={1.75em,between origins}]{A & & C \\ & C & \\};
				\path[map]	(m-1-1) edge[barred] node[above] {$f_*$} (m-1-3)
														edge[transform canvas={xshift=-1pt}] node[left] {$f$} (m-2-2);
				\path				(m-1-3) edge[transform canvas={xshift=2pt}, eq] (m-2-2);
				\draw[font=\scriptsize]	([yshift=0.333em]$(m-1-2)!0.5!(m-2-2)$) node {$\cart$};
			\end{tikzpicture} = \begin{tikzpicture}[textbaseline]
				\matrix(m)[math35, column sep={1.75em,between origins}]{A & & C & \\ & A & & \\ & & C & \\ };
				\path[map]	(m-1-1) edge[barred] node[above] {$f_*$} (m-1-3)
										(m-1-3) edge[transform canvas={xshift=2pt}] node[right] {$g$} (m-2-2)
										(m-2-2) edge[bend right=18] node[below left] {$f$} (m-3-3);
				\path				(m-1-1) edge[eq, transform canvas={xshift=-2pt}] (m-2-2)
										(m-1-3) edge[eq, bend left=45] (m-3-3)
										(m-1-3) edge[cell, transform canvas={yshift=-1.625em}] node[right] {$\eps$} (m-2-3)
										(m-1-2) edge[cell, transform canvas={yshift=0.1em}]	node[right, inner sep=3pt] {$\eta'$} (m-2-2);
			\end{tikzpicture}
		\end{displaymath}
		We claim that $\eta$ and $\eps$ satisfy the triangle identities, thus proving (a). The identity $(f \of \eta) \hc (\eps \of f) = \id_f$ is obtained by precomposing the identity above with the cocartesian cell defining $f_*$. The other triangle identity follows from
		\begin{displaymath}
			\eta' \hc (\eta \of g) \hc (g \of \eps) = \eta \hc (f \of \eta') \hc \eps = (\eta' \of \cocart) \hc (g \of \cart) = \eta'
		\end{displaymath}
		together with the fact that factorisations through $\eta'$ are unique.
		
		In the final assertion `only if' follows from \propref{pointwise left Kan extension along full and faithful map}. The `if'-part follows from \lemref{horizontal composition with (co-)units preserves nullary cartesian cells}: the invertible cell $\eta$ is cartesian so that, composing with the counit $\cell\eps{f \of g}{\id_C}$, the identity cell $\id_f = (f \of \eta) \hc (\eps \of f)$ is cartesian too.
	\end{proof}

	The following proposition too is well known for $2$-categories (see e.g.\ Proposition~2.19(1) of \cite{Weber07}) and enriched functors (see e.g.\ Section 4.1 of \cite{Kelly82}). We say that a vertical morphism $\map fMN$ is \emph{cocontinuous} if, for any nullary cell
	\begin{displaymath}
		\begin{tikzpicture}
			\matrix(m)[math35]{A_0 & A_1 & A_{n'} & A_n \\};
			\draw	([yshift=-3.25em]$(m-1-1)!0.5!(m-1-4)$) node (M) {$M$};
			\path[map]	(m-1-1) edge[barred] node[above] {$J_1$} (m-1-2)
													edge[transform canvas={yshift=-2pt}] node[below left] {$d$} (M)
									(m-1-3) edge[barred] node[above] {$J_n$} (m-1-4)
									(m-1-4) edge[transform canvas={yshift=-2pt}] node[below right] {$l$} (M);
			\path				($(m-1-2.south)!0.5!(m-1-3.south)$) edge[cell] node[right] {$\zeta$} (M);
			\draw				($(m-1-2)!0.5!(m-1-3)$) node {$\dotsb$};
		\end{tikzpicture}
	\end{displaymath}
	that defines $l$ as the left Kan extension of $d$ along $(J_1, \dotsc, J_n)$, the composite $f \of \zeta$ defines $f \of l$ as the left Kan extension of $f \of d$ along $(J_1, \dotsc, J_n)$. Notice that if $l$ is a pointwise left Kan extension then $f \of l$ is again pointwise.
	\begin{proposition} \label{left adjoints preserve pointwise left Kan extensions}
		Left adjoints are cocontinuous.
	\end{proposition}
	\begin{proof}
		Let $\map fMN$ be left adjoint to $\map gNM$, with vertical cells $\cell\eta{\id_M}{g\of f}$ and $\cell\eps{f \of g}{\id_N}$ forming the unit and counit. Consider any nullary cell $\zeta$ as above, that defines $l$ as a left Kan extension; we have to show that $f \of \zeta$ defines $f \of l$ as a left Kan extension. To this end consider the commuting diagram of assignments between collections of cells below.
		\begin{displaymath}

		\end{displaymath}
		The vertically drawn assignments above are bijections. Indeed, the triangle identities for $\eta$ and $\eps$ imply that their inverses are given by composition with $\eps$ on the right. The bottom assignment is a bijection as well, because $\zeta$ defines $l$ as a left Kan extension. We conclude that the top assignment is a bijection too, showing that $f \of \zeta$ defines $f \of l$ as left Kan extension as required.
	\end{proof}
	
	The following result shows that the condition of \defref{left Kan extension} can be simplified in an double category. Combined with \corref{Kan extensions and composites} it follows that the latter, when considered in a double category, coincides with Definition 3.10 of \cite{Koudenburg14a}.
	\begin{proposition}	\label{left Kan extensions in double categories}
		When considered in a double category the condition of \defref{left Kan extension} can be reduced to the existence of unique factorisations of nullary cells of the form $\cell\phi{(J_1, \dotsc, J_n, H)}M$.
	\end{proposition}
	
	We shall prove a more general result instead, in the form of the following lemma; applying it to the cocartesian cells that define composites of paths $(H_1, \dotsc, H_m)$, as considered in \defref{left Kan extension}, proves the proposition above.
	\begin{lemma} \label{unique factorisations factorised through a cocartesian cell on the right}
		Consider a path of cells
		\begin{displaymath}

		\end{displaymath}
		with $n, m \geq 1$ and assume that the subpath $(\phi_1, \dotsc, \phi_m)$ is cocartesian. Remember that, for each object $M$, precomposition with the full path gives a bijective correspondence between nullary cells of the form
		\begin{displaymath}
			%
		\end{displaymath}
		and those of the form
		\begin{displaymath}
			%
		\end{displaymath}
		see \defref{cocartesian paths}. Under this correspondence a unique factorisation of the cells $\psi$ through the nullary cell $\eta$ on the left below corresponds to a unique factorisation of the cells $\chi$ through the composite on the right.
		\begin{displaymath}
			%
		\end{displaymath}
	\end{lemma}
	\begin{proof}
		Consider the following assignments between collections of nullary cells, that are of the form as shown. Notice that the diagram commutes.
		\begin{displaymath}
			%
		\end{displaymath}
		Now the vertical assignments are bijective, because the path $(\phi_1, \dotsc, \phi_m)$ is cocartesian, so that the proof follows from the fact that unique factorisation through the cell $\eta$ is, by definition, the top assignment being bijective, while unique factorisation through the composite $\eta \of (\id, \dotsc, \cart)$ is the bottom one being bijective.
	\end{proof}
	
	\subsection{Pointwise Kan extensions in \texorpdfstring{$\enProf{(\V, \V')}$}{(V, V')-Prof}}\label{Kan extensions in (V, V')-Prof}
	As an example, here we study pointwise left Kan extensions in the augmented virtual equipment $\enProf{(\V, \V')}$, of $\V$-profunctors between $\V'$-categories (see \exref{(V, V')-Prof}), and show that they can be described in terms of (generalised) `$\V$-weighted colimits', as one would expect. The results of this section are straightforward generalisations of those in Section 2.21 of \cite{Koudenburg15a}.
	
	In the case of a closed symmetric monoidal category $\V$ the notion of Kan extension along $\V$-functors evolved as follows. Firstly Day and Kelly gave a definition for $\V$-functors into a cocomplete $\V$-category in \cite{Day-Kelly69}, which was extended to $\V$"/functors into a copowered $\V$-category by Dubuc \cite{Dubuc70}, and later to $\V$-functors into a general $\V$-category by Borceux and Kelly \cite{Borceux-Kelly75}. Following this, \cite{Kelly82} describes how to remove the assumption of a closed structure on $\V$, by considering a universe enlargement $\V \to \V'$ with $\V'$ instead being closed symmetric monoidal.

	Throughout this subsection let $\V \to \V'$ be a universe enlargement (\defref{universe enlargement}); remember that we assume only a monoidal structure on $\V$, and a closed monoidal structure on $\V'$. Recall that we write $I$ for the unit $\V$-category, with the unit $I = I(*, *)$ of $\V$ as its single hom-object, and that we regard $I$ as the unit $\V'$"/category as well. We identify $\V'$-functors $I \to A$ with objects in $A$ and $\V$-profunctors $I \brar I$ with $\V$-objects; cells between such profunctors are identified with $\V$-maps. Notice that the horizontal composite $(H_1 \hc \dotsb \hc H_n)$ of $\V$-profunctors $\hmap{H_i}II$ always exists; it corresponds to the tensor product of $\V$-objects $H_1 \tens' \dotsb \tens' H_n$.
	
	By a \emph{$\V$-weight} on a $\V$-category $A$ we mean a path $(A \xbrar{J_1} A_1, \dotsc, A_{n'} \xbrar{J_n} I)$ of $\V$-profunctors between $\V$-categories. It is called \emph{small} if each of the $\V$-categories $A, A_1, \dotsc, A_{n'}$ is small. As usual we denote a singleton $\V$-weight $(J)$ on $A$ simply by $J$.
	\begin{definition} \label{weighted colimit}
		Let $(J_1, \dotsc, J_n)$ a $\V$-weight on a $\V$-category $A$, $\map dAM$ be a $\V'$-functor, and $l$ an object of $M$. A cell $\eta$ in $\enProf{(\V, \V')}$, as in the right-hand side below, is said to define $l$ as the \emph{$(J_1, \dotsc, J_n)$-weighted colimit of $d$} if every cell $\phi$ below factors uniquely through $\eta$ as shown.
		\begin{displaymath}
			\begin{tikzpicture}[textbaseline]
				\matrix(m)[math35, yshift=1.625em]{A & A_1 & A_{n'} & I & I \\};
				\draw				([yshift=-3.25em]$(m-1-1)!0.5!(m-1-5)$) node (M) {$M$};
				\path[map]	(m-1-1) edge[barred] node[above] {$J_1$} (m-1-2)
														edge[transform canvas={yshift=-2pt}] node[below left] {$d$} (M)
										(m-1-3) edge[barred] node[above] {$J_n$} (m-1-4)
										(m-1-4) edge[barred] node[above] {$H$} (m-1-5)
										(m-1-5) edge[transform canvas={yshift=-2pt}] node[below right] {$k$} (M);
				\path				(m-1-3) edge[cell] node[right] {$\phi$} (M);
				\draw				($(m-1-2)!0.5!(m-1-3)$) node {$\dotsb$};
			\end{tikzpicture} = \begin{tikzpicture}[textbaseline]
				\matrix(m)[math35]{A & A_1 & A_{n'} & I & I \\ & & & M & \\};
				\path[map]	(m-1-1) edge[barred] node[above] {$J_1$} (m-1-2)
														edge[transform canvas={yshift=-2pt}] node[below left] {$d$} (m-2-4)
										(m-1-3) edge[barred] node[above] {$J_n$} (m-1-4)
										(m-1-4) edge[barred] node[above] {$H$} (m-1-5)
														edge node[left] {$l$} (m-2-4)
										(m-1-5) edge[transform canvas={yshift=-2pt}] node[below right] {$k$} (m-2-4);
				\path				(m-1-2) edge[cell, transform canvas={shift={(3em,0.333em)}}] node[right] {$\eta$} (m-2-2)
										(m-1-4) edge[cell, transform canvas={shift={(1em,0.333em)}}] node[right] {$\phi'$} (m-2-4);
				\draw				($(m-1-2)!0.5!(m-1-3)$) node {$\dotsb$};
			\end{tikzpicture}
  	\end{displaymath}
	\end{definition}
	Since $\V \to \V'$ induces a full and faithful inclusion $\enProf\V \to \enProf{(\V, \V')}$ (see \exref{equivalence induced by universe enlargement}) it follows that, in the case $\eta$ is contained in $\enProf\V$, the existence of the factorisations above reduces to the existence of such factorisations in $\enProf\V$.
	
	When reduced to the case of $(1,0)$-ary cells $\cell\eta{J_1}M$ in $\enProf\V$, the definition above extends the classical definition of weighted colimit as follows. Note that if $\V$ is closed symmetric monoidal then singleton $\V$-weights $J$ on $A$ can be identified with $\V$-functors $\map J{\op A}\V$, using the isomorphism $\op A \tens I \iso \op A$ (see \exref{enriched profunctors}). This recovers the classical definition of $\V$-weight given in e.g.\ Section 3.1 of \cite{Kelly82}, where such $\V$-functors are called `indexing types'. Using this identification the cell $\cell\eta{J_1}M$ can be regarded as a $\V$-natural transformation between the $\V$-functors $\map {J_1}{\op A}\V$ and $\map{M(d,l)}{\op A}\V$.
	\begin{proposition} \label{classical definition of weighted colimit}
		Let $d$, $J_1$, $l$ and $\eta$ be as above. If $\V$ is closed symmetric monoidal then the cell $\eta$ defines $l$ as the $J$-weighted colimit of $d$, in the above sense, precisely if the pair $(l, \eta)$, where $\eta$ is regarded as a $\V$"/natural transformation $J_1 \Rar M(d, l)$ of $\V$-functors $\op A \to \V$, forms the `colimit of $d$ indexed by $J_1$' in the sense of Section 3.1 of \cite{Kelly82}.
	\end{proposition}
	\begin{proof}
		Since the factorisations of the definition above reduce to factorisations in $\enProf\V$, as discussed above, the horizontal dual of the proof of Proposition 2.23 of \cite{Koudenburg15a} applies verbatim; the only difference is that there $A$ is assumed to be small.
	\end{proof}
	
	Next we describe Kan extensions in $\enProf{(\V, \V')}$ in terms of weighted colimits and, as a consequence, obtain an `enriched variant' of \propref{weak left Kan extensions along companions}. Remember that the vertical cells of $\enProf\V$ are $\V$-natural transformations between functors, in the classical sense.
	\begin{proposition}\hspace{-0.25cm}\footnote{\textbf{This result is false.} See Example~1.10 of \cite{Koudenburg22} instead.} \label{enriched left Kan extensions in terms of weighted colimits}
		Consider a cell $\eta$ in $\enProf{(\V, \V')}$ as in the composite on the left below, where $A_0, \dotsc, A_{n'}$ are $\V$-categories. It defines $l$ as the pointwise left Kan extension of $d$ along $(J_1, \dotsc, J_n)$ if and only if, for each $y \in A_n$, the full composite defines $ly$ as the $\bigpars{J_1, \dotsc, J_n(\id, y)}$-weighted colimit of $d$. In particular, pointwise left Kan extensions along a $\V$-weight $\hmap JAI$ coincide with $J$-weighted colimits.
		\begin{equation} \label{pointwise Kan extension at a point}

		\end{equation}
		
		Finally assume that $\V$ is closed symmetric monoidal, and consider a factorsiation in $\enProf\V$ as on the right above. The cell $\zeta'$ defines $l$ as a pointwise left Kan extension in $\enProf{(\V, \V')}$ precisely if the $\V$-natural transformation $\nat\zeta d{l \of j}$ `exhibits $l$ as the left Kan extension of $d$ along $j$', in the sense of Section 4.1 of \cite{Kelly82}.
	\end{proposition}
	\begin{proof}
		The `only if'-part of the first assertion follows immediately from \defref{pointwise left Kan extension}, by restricting the universal property of \defref{left Kan extension}, for the composite on the left above, to cells of the form $(J_1, \dotsc, J_n(\id, y), H) \Rar M$, where $\hmap HII$.
		
		For the `if'-part consider any $\V'$-functor $\map fB{A_n}$ as well as any cell $\phi$ as in the composite on the left below; we have to show that it factors uniquely through the composition on the right, as a cell $\cell{\phi'}{(H_1, \dotsc, H_m)}M$.
		\begin{displaymath}

		\end{displaymath}
		In order to obtain the $(x, x_1, \dotsc, x_m)$-component of $\phi'$, where $x \in B$ and each $x_i \in B_i$, consider the full composite on the left, and notice that it can be regarded as a cell $\bigpars{J_1, \dotsc, J_n(\id, fx), H_1(x, x_1) \tens' \dotsb \tens' H_m(x_{m'}, x_m)} \Rar M$. By assumption on $\eta$, the latter cells factor uniquely through the composite on the left of \eqref{pointwise Kan extension at a point}, with $y = fx$, as $\V'$-maps
		\begin{displaymath}
			\map{\phi'_{(x, x_1, \dotsc, x_m)}}{H_1(x, x_1) \tens' \dotsb \tens' H_m(x_{m'}, x_m)}{M(lfx, kx_m)}
		\end{displaymath}
		which, we claim, combine to form a cell $\cell{\phi'}{(H_1, \dotsc, H_m)}M$. The unique factorisations above then guarantee that $\phi'$ is unique such that $\phi = \eta \hc \phi'$. The proof of the claim consists of showing that the $\V'$-maps $\phi'_{(x, x_1, \dotsc, x_m)}$ are $\V'$-natural in each of the coordinates $x$, $x_1$, \dots, $x_m$. This is a straightforward consequence of the $\V'$-naturality of the cells $\eta$ and $\phi$, together with the uniqueness of the factorisations above, which we leave to the interested reader (see also the proof of Proposition~2.24 of \cite{Koudenburg15a} where this claim was proved in the horizontal dual case for cells $\eta$ and $\phi$ with $n$, $m = 1$). This completes the proof of the first assertion.
		
		To prove the second assertion notice that $\zeta'$ is given by the $\V$-maps
		\begin{displaymath}
			B(jx, y) \xrar l M(ljx, ly) \xrar{M(\zeta_x, ly)} M(dx, ly),
		\end{displaymath}
		for pairs $x \in A$ and $y \in B$. The proof follows from applying to $\zeta$ the first assertion above, \propref{classical definition of weighted colimit} and the horizontal dual of condition (ii) of Theorem 4.6 of \cite{Kelly82}, which lists equivalent conditions defining a $\V$-natural transformation to `exhibit a right Kan extension'.
	\end{proof}
	
	As is to be expected the pointwise left Kan extension along a $\V$-profunctor $\hmap JAB$ can be constructed out of weighted colimits whenever $B$ is a $\V$-category, as follows.
	\begin{proposition}
		In $\enProf{(\V, \V')}$ consider a $\V$-profunctor $\hmap JAB$ between $\V$"/categories $A$ and $B$, as well as a $\V'$-functor $\map dAM$. Assume given, for each $y \in B$, a cell $\eta_y$ as on the left below, that defines the $J(\id, y)$-weighted colimit of $d$.
		\begin{displaymath}

		\end{displaymath}
		There is exactly one way of extending the assignment $l\colon y \mapsto ly$ to a $\V'$-functor $\map lBM$, such that the cells $\eta_y$ combine to form a cell $\eta$ as on the right. Moreover $\eta$ defines $l$ as the pointwise left Kan extension of $d$ along $J$.
	\end{proposition}
	\begin{proof}
		For each $y, z \in B$ we define the action of $l$ on the hom-object $B(y, z)$ to be the unique factorisation below, where the top cell in the composite on the left-hand side is given by the action $\map\rho{J(x,y) \tens' B(y,z)}{J(x,z)}$ of $B$ on $J$.
		\begin{displaymath}
			%
		\end{displaymath}
		That these factorisations combine to form a $\V'$-functor $\map lBM$, that is they preserve composition and units, follows easily from the fact that the action of $B$ on $J$ is associative and unital, together with the uniqueness of the factorisations through the cells $\eta_y$. Having constructed $\map lBM$, notice that the factorisations above imply that the cells $\eta_y$ combine to form a nullary cell $\cell\eta JM$ as asserted; in fact, their uniqueness implies that $l$ is unique with this property. That $\eta$ defines $l$ as a pointwise left Kan extension follows from the previous proposition and the fact that, for each $y \in B$, it restricts along $J(\id, y)$ to a cell that defines a weighted colimit.
	\end{proof}
	
	We record the following simple lemma for later use. Remember that the $2$"/category $\enCat{\V'}$ of $\V'$-categories admits a monoidal structure when $\V'$ is symmetric monoidal; see for instance Section 1.4 of \cite{Kelly82}. It is straightforward to show that this structure extends to a `monoidal structure' on the augmented virtual double category $\enProf{\V'}$ of $\V'$-profunctors, in a sense that is the evident adaptation of that of `monoidal double category', as given in Definition 9.1 of \cite{Shulman08}. While we shall not make this precise, we will use its underlying tensor product $\map\tens{\enProf{\V'} \times \enProf{\V'}}{\enProf{\V'}}$. On $\V'$-categories and $\V'$-functors it is given as usual, while the tensor product $\hmap{J \tens H}{A \tens C}{B \tens D}$ of $\V'$-profunctors is given by $(J \tens H)\bigpars{(x,z)(y,w)} \dfn J(x,y) \tens H(z,w)$. Its action on a pair of cells in $\enProf{\V'}$ is likewise given by tensoring their components, as well as using the symmetric structure of $\V'$. Given a symmetric universe enlargement $\V \to \V'$, notice that the tensor product on $\enProf{\V'}$ restricts to one on $\enProf{(\V, \V')}$.
	\begin{lemma}	\label{enriched left Kan extensions along tensor products}
		Let $\V \to \V'$ be a symmetric universe enlargement, $A$ and $B$ be $\V$"/categories, and $J$ a $\V$-weight on $B$. A cell $\eta$ in $\enProf{(\V, \V')}$, as in the composite below, defines $l$ as a pointwise left Kan extension of $d$ along $I_A \tens J$ precisely if, for each $u \in A$, the full composite, whose top cell is induced by the unit \mbox{$\map{\tilde A_u}{I'}{A(u, u)}$}, defines $lu$ as the $J$-weighted colimit of $d \of (u \tens \id)$
		\begin{displaymath}
			\begin{tikzpicture}
				\matrix(m)[math35, column sep={2.25em,between origins}]{B & & I' \\ A \tens B & & A \\ & M & \\};
  			\path[map]	(m-1-1) edge[barred] node[above] {$J$} (m-1-3)
  													edge node[left] {$u \tens \id$} (m-2-1)	
  									(m-1-3) edge node[right] {$u$} (m-2-3)
  									(m-2-1) edge[barred] node[below, inner sep=2pt] {$I_A \tens J$} (m-2-3)
  													edge[transform canvas={xshift=-2pt}] node[below left] {$d$} (m-3-2)
  									(m-2-3) edge[transform canvas={xshift=2pt}] node[below right] {$l$} (m-3-2);
  			\path				(m-1-1) edge[cell, transform canvas={xshift=0.7em}] node[right] {$\tilde A_u \tens \id_J$} (m-2-1)
  									(m-2-2) edge[cell] node[right] {$\eta$} (m-3-2);
  		\end{tikzpicture}
		\end{displaymath}
	\end{lemma}
	\begin{proof}
		By \propref{enriched left Kan extensions in terms of weighted colimits} the cell $\eta$ defines $l$ as a pointwise left Kan extension precisely if, for all $u \in A$, its restriction along $(I_A \tens J)(\id, u) \iso u^* \tens J$ defines $lu$ as a weighted colimit. Using \propref{cocartesian cells in (V, V')-Prof} it is easy to see that $\tilde A_u \tens \id_J$ factors through $u^* \tens J$ as a right pointwise cocartesian cell, so that the proof follows from the vertical pasting lemma (\lemref{vertical pasting lemma}).
	\end{proof}
	
	\subsection{Tabulations}
	Here the notion of tabulation, that was introduced by Grandis and Par\'e in \cite{Grandis-Pare99} in the setting of double categories, is generalised to the setting of augmented virtual double categories. In the next subsection we show that in augmented virtual double categories that have all `cocartesian tabulations', all pointwise Kan weak extensions are pointwise Kan extensions (see \defref{pointwise left Kan extension}).
	
	\begin{definition} \label{tabulation}
		Given a horizontal morphism $\hmap JAB$ in an augmented virtual double category $\K$, the \emph{tabulation} $\tab J$ of $J$ consists of an object $\tab J$ equipped with a unary cell $\pi$ as on the left below, satisfying the following $1$-dimensional and $2$"/dimensional universal properties.
		\begin{displaymath}
			\begin{tikzpicture}[baseline]
  			\matrix(m)[math35, column sep={1.75em,between origins}]{& \tab J & \\ A & & B \\};
				\path[map]	(m-1-2) edge[transform canvas={xshift=-2pt}] node[left] {$\pi_A$} (m-2-1)
														edge[transform canvas={xshift=2pt}] node[right] {$\pi_B$} (m-2-3)
										(m-2-1) edge[barred] node[below] {$J$} (m-2-3);
				\path				(m-1-2) edge[cell, transform canvas={yshift=-0.25em}] node[right, inner sep=2.5pt] {$\pi$} (m-2-2);
			\end{tikzpicture} \qquad\qquad\qquad\qquad\qquad \begin{tikzpicture}[baseline]
  			\matrix(m)[math35, column sep={1.75em,between origins}]{& X & \\ A & & B \\};
				\path[map]	(m-1-2) edge[transform canvas={xshift=-2pt}] node[left] {$\phi_A$} (m-2-1)
														edge[transform canvas={xshift=2pt}] node[right] {$\phi_B$} (m-2-3)
										(m-2-1) edge[barred] node[below] {$J$} (m-2-3);
				\path				(m-1-2) edge[cell, transform canvas={yshift=-0.25em}] node[right, inner sep=2.5pt] {$\phi$} (m-2-2);
			\end{tikzpicture}
		\end{displaymath}
		Given another unary cell $\phi$ as on the right above, the $1$-dimensional property states that there exists a unique vertical morphism $\map{\phi'}X{\tab J}$ such that $\pi \of \id_{\phi'} = \phi$.
		
		The $2$-dimensional property is the following. Suppose we are given a further unary cell $\psi$ as in the identity below, which factors through $\pi$ as $\map{\psi'}Y{\tab J}$, like $\phi$ factors as $\phi'$. Then for any pair of cells $\xi_A$ and $\xi_B$ as below, so that the identity holds, there exists a unique cell $\xi'$ as on the right below such that $\pi_A \of \xi' = \xi_A$ and $\pi_B \of \xi' = \xi_B$.
		\begin{displaymath}

		\end{displaymath}
		We call $\tab J$ \emph{cocartesian} whenever its defining cell $\pi$ is cocartesian.
		
		We shall briefly use the vertical dual notion as well: the \emph{cotabulation} $\brks J$ of $J$ is defined by a nullary cell $\sigma$ below, satisfying 1-dimensional and 2-dimensional universal properties that are vertical dual to those for $\tab J$. We call $\brks J$ \emph{cartesian} whenever $\sigma$ is cartesian.
		\begin{displaymath}
			%
		\end{displaymath}
	\end{definition}
	
	\begin{example} \label{tabulations of profunctors}
		In the augmented virtual equipment $\enProf{(\Set, \Set')}$, of small profunctors between large categories, the tabulation $\tab J$ is the \emph{graph} of $\map J{\op A \times B}\Set$ as follows. It has triples $(x, u, y)$ as objects, where \mbox{$(x, y) \in A \times B$} and $\map uxy$ in $J$, while a map $(x, u, y) \to (x', u', y')$ is a pair $\map{(s, t)}{(x,y)}{(x', y')}$ in $A \times B$ making the diagram
	\begin{displaymath}
		%
	\end{displaymath}
	commute in $J$. The functors $\pi_A$ and $\pi_B$ are the projections while the cell $\nat \pi{\tab J}J$ maps $(x, u, y)$ to $u \in J(x, y)$. It is straightforward to check that $\pi$ satisfies the universal properties above, and that it is cocartesian.
	
		Vertically dual, the cotabulation $\brks J$ is \emph{cograph} of $J$, as follows. Its set of objects is the disjoint union $\ob\brks J \dfn \ob A \djunion \ob B$ while its hom-sets are given by
		\begin{displaymath}
			\brks J(x, y) \dfn \begin{cases}
				A(x, y) & \text{if $x, y \in A$;} \\
				J(x, y) & \text{if $x \in A$ and $y \in B$;} \\
				B(x, y) & \text{if $x, y \in B$;} \\
				\emptyset & \text{otherwise.}
			\end{cases}
		\end{displaymath}
		The functors $\sigma_A$ and $\sigma_B$ are the embeddings of $A$ and $B$ into $\brks J$, while the cell $\nat\sigma J{\brks J}$ consists of the identities on the sets $J(x, y)$. Again it is straightforward to check that $\sigma$ satisfies the universal properties and that it is cartesian.
	\end{example}
	
	\begin{example} \label{tabulations of indexed profunctors}
		The augmented virtual equipment $\enProf{(\Set, \Set')}^\Ss$ of $\Ss$-indexed profunctors (see \exref{indexed profunctors}) has all cocartesian tabulations as well as cartesian cotabulations: in each index they can be constructed as in the previous example.
	\end{example}	
	
	\begin{example} \label{tabulations of modular relations}
		In the strict equipment $\enProf\2$ of modular relations, described in \exref{relations}, the tabulation $\tab J$ of $\hmap JAB$ is the set $J \subseteq A \times B$ itself, equipped with the product order
		\begin{displaymath}
			(x_1, y_1) \leq (x_2, y_2) \defeq (x_1 \leq x_2) \wedge (y_1 \leq y_2).
		\end{displaymath}
		Taking the monotone maps $\pi_A$ and $\pi_B$ to be the projections, it is straightforward to check that the evident cell $\tab J \Rar J$ satisfies the universal properties and is weakly cocartesian; applying \corref{extensions and composites} we conclude that it is cocartesian.
	\end{example}
	
	\begin{example} \label{tabulations of 2-profunctors}
		In the augmented virtual equipment $\enProf{(\Cat, \Cat')}$ the tabulation $\tab J$ of a (small) $2$-profunctor $\map J{\op A \times B}\Cat$, where $A$ and $B$ are (possibly large) $2$"/categories, has as underlying category the graph of the profunctor underlying $J$, while its cells $(s, t) \Rar (s', t')$ are pairs $(\delta, \eps)$ of cells $\cell\delta s{s'}$ in $A$ and $\cell\eps t{t'}$ in $B$ that make the diagram on the left below commute in $J$.
		\begin{displaymath}
			\begin{tikzpicture}[textbaseline]
				\matrix(m)[math35]{x & y \\ x' & y' \\};
				\path[map]	(m-1-1) edge node[above] {$u$} (m-1-2)
														edge[bend right=35] node[left] {$s$} (m-2-1)
														edge[bend left=35] node[right, inner sep=2pt] {$s'$} (m-2-1)
										(m-1-2) edge[bend right=35] node[left, inner sep=2pt] {$t$} (m-2-2)
														edge[bend left=35] node[right] {$t'$} (m-2-2)
										(m-2-1) edge node[below] {$u'$} (m-2-2);
				\path				([xshift=-0.8em]$(m-1-1)!0.5!(m-2-1)$) edge[cell] node[above] {$\delta$} ([xshift=0.9em]$(m-1-1)!0.5!(m-2-1)$)
										([xshift=-0.8em]$(m-1-2)!0.5!(m-2-2)$) edge[cell] node[above] {$\eps$} ([xshift=0.9em]$(m-1-2)!0.5!(m-2-2)$);
			\end{tikzpicture} \qquad J(*,*) = \bigpars{\begin{tikzpicture}[textbaseline, font=\scriptsize, ampersand replacement=\&]
				\matrix(m)[math35, column sep=2.5em]{ \& \\};
				\path[map]	(m-1-1) edge[bend left=40] node[above] {$u$} (m-1-2)
														edge[bend right=40] node[below] {$v$} (m-1-2);
				\fill	(m-1-1) circle (1pt) node[left] {$x$}
							(m-1-2) circle (1pt) node[right] {$y$};
				\path	($(m-1-1)!0.5!(m-1-2)+(0,0.8em)$) edge[cell] ($(m-1-1)!0.5!(m-1-2)-(0,0.8em)$);
			\end{tikzpicture}} \qquad K(*,*) = \bigpars{\begin{tikzpicture}[textbaseline, font=\scriptsize, ampersand replacement=\&]
				\matrix(m)[math35, column sep=2.5em]{ \& \\};
				\path[map]	(m-1-1) edge[bend left=40] node[above] {$u'$} (m-1-2)
														edge[bend right=40] node[below] {$v'$} (m-1-2);
				\fill	(m-1-1) circle (1pt) node[left] {$x'$}
							(m-1-2) circle (1pt) node[right] {$y'$};
			\end{tikzpicture}}
		\end{displaymath}

		Tabulations of $2$-profunctors fail to be cocartesian in general. As an example consider two $2$-profunctors $J$ and $\hmap K11$, where $1$ denotes the terminal $2$-category with single object $*$, whose single images $J(*,*)$ and $K(*,*)$ are the `free living' cell and parallel pair of arrows respectively, as shown above. The tabulation $\tab J$ is discrete with objects $(*, u, *)$ and $(*, v, *)$, so that the assignments $(*, u, *) \mapsto u'$ and $(*, v, *) \mapsto v'$ define a cell $\cell\phi{\tab J}K$; it is easily checked that $\phi$ does not factor through $\cell\pi{\tab J}J$.
	\end{example}
	
	Consider an augmented virtual double category $\K$ that has all tabulations. Choosing a tabulation $\tab J$ for every $\hmap JAB$ in $\K$ gives an assignment $J \mapsto \tab J$ which we extend to paths $\ul J = (A_0 \xbrar{J_1} A_1, \dotsc, A_{n'} \xbrar{J_n} A_n)$ by mapping $\ul J$ to the top $\tab{J_1, \dotsc, J_n}$ of the `piramid' below, where each $\pi_i$ denotes the chosen tabulation of $\hmap{J_i(\pi_{A_{i'}}, \id)}{\tab{J_1, \dotsc, J_{i-1}}}{A_i}$. For each $1 \leq i \leq n$ we denote the composite $\tab{J_1, \dotsc, J_n} \to \tab{J_1, \dotsc, J_{n'}} \to \dotsb \to \tab{J_1, \dotsc, J_i} \xrar{\pi_{A_i}} A_i$ again by $\pi_{A_i}$.
	\begin{equation} \label{piramid}

	\end{equation}
	Importantly, if the tabulations of $\K$ are cocartesian then so is each of the rows above and hence, by the pasting lemma (\lemref{pasting lemma for cocartesian paths}), so is the path consisting of the columns that make up the piramid.
	
	\begin{example}
		For profunctors $\hmap{J_1}{A_0}{A_1}$, $\dotsc,$ $\hmap{J_n}{A_{n'}}{A_n}$, the tabulation $\tab{J_1, \dotsc, J_n}$ in $\enProf{(\Set, \Set')}$ is (isomorphic to) the category whose objects are given by paths $(x_0 \xrar{u_1} x_1, \dotsc, x_{n'} \xrar{u_n} x_n)$, where each $x_i \in A_i$ and each $u_i \in J_i$, while its morphisms are sequences of commutative squares
		\begin{displaymath}
			\begin{tikzpicture}
				\matrix(m)[math35]{x_0 & x_1 &[1.75em] x_{n'} & x_n \\ x'_0 & x'_1 & x'_{n'} & x'_n, \\};
				\path[map]	(m-1-1) edge node[above] {$u_1$} (m-1-2)
														edge node[left] {$s_0$} (m-2-1)
										(m-1-2) edge node[right] {$s_1$} (m-2-2)
										(m-1-3) edge node[above] {$u_n$} (m-1-4)
														edge node[left] {$s_{n'}$} (m-2-3)
										(m-1-4) edge node[right] {$s_n$} (m-2-4)
										(m-2-1) edge node[below] {$u'_1$} (m-2-2)
										(m-2-3) edge node[below] {$u'_n$} (m-2-4);
				\draw				($(m-1-2)!0.5!(m-2-3)$) node {$\dotsb$};
			\end{tikzpicture}
		\end{displaymath}
		where each $s_i \in A_i$.
	\end{example}
	
	Besides horizontal composites, the tabulations $\tab{J_1, \dotsc, J_n}$ can also be used to reduce $(n, m)$-ary cells to $(1, m)$-ary ones, as follows. Similar to \corref{Kan extensions and composites}, this reduction preserves cells that define Kan extensions, as shown below.
	\begin{proposition} \label{cells corresponding to vertical cells}
		Consider an augmented virtual double category $\K$ that has cocartesian tabulations and assume that a tabulation $\tab{J_1, \dotsc, J_n}$ has been chosen for each path of horizontal morphisms $(J_1, \dotsc, J_n)$, as in \eqref{piramid}. The following hold.
		\begin{enumerate}[label=\textup{(\alph*)}]
			\item There exists a bijective correspondence between nullary cells $\phi$, of the form as on the left below, and nullary cells $\psi$, of the form as in the middle, which preserves cells that define $k$ as a (weak) left Kan extension.
			\item If $A_n$ is unital then the cells $\psi$ in turn are in bijective correspondence with nullary cells $\chi$, that are of the form as on the right below; this too preserves cells defining (weak) left Kan extensions.
			\item Given a vertical morphism $\map fMN$ such that the restriction $N(f, f)$ exists, $f$ is full and faithful in $\K$ precisely if it is so in the vertical $2$-category $V(\K)$ (see \defref{full and faithful morphism}).
		\end{enumerate}
		\begin{displaymath}
			\begin{tikzpicture}[baseline]
				\matrix(m)[math35]{A_0 & A_1 & A_{n'} & A_n \\};
				\draw	([yshift=-3.25em]$(m-1-1)!0.5!(m-1-4)$) node (M) {$M$};
				\path[map]	(m-1-1) edge[barred] node[above] {$J_1$} (m-1-2)
														edge[transform canvas={yshift=-2pt}] node[below left] {$d$} (M)
										(m-1-3) edge[barred] node[above] {$J_n$} (m-1-4)
										(m-1-4) edge[transform canvas={yshift=-2pt}] node[below right] {$k$} (M);
				\path				($(m-1-1.south)!0.5!(m-1-4.south)$) edge[cell] node[right] {$\phi$} (M);
				\draw				($(m-1-2)!0.5!(m-1-3)$) node {$\dotsb$};
			\end{tikzpicture} \quad \begin{tikzpicture}[baseline]
  			\matrix(m)[math35, column sep={5.5em,between origins}]{ \tab{J_1, \dotsc, J_{n'}} & A_n \\};
  			\draw				([yshift=-3.25em]$(m-1-1)!0.5!(m-1-2)$) node (M) {$M$};
  			\path[map]  (m-1-1) edge[barred] node[above, inner sep=4pt] {$J_n(\pi_{A_{n'}}, \id)$} (m-1-2)
														edge[transform canvas={xshift=-2pt}] node[left] {$d \of \pi_{A_0}$} (M)
										(m-1-2)	edge[transform canvas={xshift=2pt}] node[right] {$k$} (M);
				\path				($(m-1-1.south)!0.5!(m-1-2.south)$) edge[cell] node[right] {$\psi$} (M);
			\end{tikzpicture} \quad \begin{tikzpicture}[baseline]
  			\matrix(m)[math35, column sep={5em,between origins}]{ \tab{J_1, \dotsc, J_n} & A_n \\};
  			\draw				([yshift=-3.25em]$(m-1-1)!0.5!(m-1-2)$) node (M) {$M$};
  			\path[map]  (m-1-1) edge[barred] node[above, inner sep=4pt] {$\pi_{A_n*}$} (m-1-2)
														edge[transform canvas={xshift=-2pt}] node[left] {$d \of \pi_{A_0}$} (M)
										(m-1-2)	edge[transform canvas={xshift=2pt}] node[right] {$k$} (M);
				\path				($(m-1-1.south)!0.5!(m-1-2.south)$) edge[cell] node[right] {$\chi$} (M);
			\end{tikzpicture}
		\end{displaymath}		
	\end{proposition}
	\begin{proof}
		The correspondence of (a) is given by precomposing the nullary cells $\phi$ with the cocartesian path consisting of all cells that make up the piramid \eqref{piramid}, except the top one. That this preserves cells defining (weak) left Kan extensions follows immediately from the vertical pasting lemma (\lemref{vertical pasting lemma}).
		
		For the correspondence of (b), between the nullary cells $\psi$ and $\chi$, assume that the horizontal unit $I_{A_n}$ exists. The correspondence is given by the assignment on the left below, where $\psi'$ is the unique factorisation of $\psi$ through the cocartesian cell that defines $I_{A_n}$. That this assignment is bijective follows from the fact that both the latter cell and the path $(\pi_n, \cart)$ are cocartesian. That it preserves cells defining (weak) left Kan extensions follows from the same fact, using the vertical pasting lemma.
		\begin{displaymath}
			\psi \quad \mapsto \quad 

		\end{displaymath}
		
		The `only if'-part of (c) is clear. For the `if'-part, let $\map fMN$ be a vertical morphism and assume that it is full and faithful in $V(\K)$. Consider the commutative diagram of assignments between classes of nullary cells below, where the cells $\phi$ are as in the statement of the proposition; the cells $\xi$ are of the form as on the left above, and the unlabelled assignments are given by precompostion with the path of cells \eqref{piramid}.
		\begin{displaymath}
			%
		\end{displaymath}
		Now the assumption on $f$ means that the assignment on the right is a bijection and, because the path \eqref{piramid} is cocartesian, so are the horizontal maps. We conclude that the assignment on the left is a bijection too, so that $\id_f$ is cartesian; that is $f$ is full and faithful in $\K$. This completes the proof.
	\end{proof}
	
	Closing this subsection, the following proposition describes the relation between representable nullary restrictions in an augmented virtual double category $\K$ and \emph{absolute left liftings} (as introduced in Section 1 of \cite{Street-Walters78}; or see Section 2.4 of \cite{Weber07}) in its vertical $2$-category $V(\K)$.
	\begin{proposition} \label{restrictions and absolute left liftings}
		In an augmented virtual double category $\K$ consider the factorisation below. The vertical cell $\psi$ defines $j$ as an absolute left lifting of $f$ along $g$ in $V(\K)$ whenever its factorisation $\psi'$ is cartesian in $\K$. The converse holds as soon as $\K$ has all cocartesian tabulations.
		\begin{displaymath}
			\begin{tikzpicture}[textbaseline]
				\matrix(m)[math35, column sep={1.75em,between origins}]{& A & \\ & & B \\ & C & \\};
				\path[map]	(m-1-2) edge[bend left = 18] node[above right] {$j$} (m-2-3)
														edge[bend right = 45] node[left] {$f$} (m-3-2)
										(m-2-3) edge[bend left = 18] node[below right] {$g$} (m-3-2);
				\path[transform canvas={yshift=-1.625em}]	(m-1-2) edge[cell] node[right] {$\psi$} (m-2-2);
			\end{tikzpicture} = \begin{tikzpicture}[textbaseline]
    		\matrix(m)[math35, column sep={1.75em,between origins}]{& A & \\ A & & B \\ & C & \\};
    		\path[map]	(m-1-2) edge[transform canvas={xshift=2pt}] node[right] {$j$} (m-2-3)
    								(m-2-1) edge[barred] node[below, inner sep=2pt] {$j_*$} (m-2-3)
    												edge[transform canvas={xshift=-2pt}] node[left] {$f$} (m-3-2)
    								(m-2-3)	edge[transform canvas={xshift=2pt}] node[right] {$g$} (m-3-2);
    		\path				(m-1-2) edge[eq, transform canvas={xshift=-2pt}] (m-2-1);
    		\path				(m-2-2) edge[cell, transform canvas={yshift=0.1em}]	node[right, inner sep=3pt] {$\psi'$} (m-3-2);
    		\draw				([yshift=-0.5em]$(m-1-2)!0.5!(m-2-2)$) node[font=\scriptsize] {$\cocart$};
  		\end{tikzpicture}
  	\end{displaymath}
	\end{proposition}
	
	For an example of a cell $\psi$ defining an absolute left lifting whose factorisation $\psi'$, as above, is not cartesian see \exref{no good yoneda structure on (Cat, Cat')-Prof} below.
	
	\begin{proof}
		For any vertical morphisms $\map pXA$ and $\map qXB$ consider the commuting diagram of assignments between collections of cells in $\K$ on the left below, where the assignment on the left is given by composing with the cartesian cell that defines $j_*$. By definition, the vertical cell $\psi$ defines $j$ as the absolute left lifting of $f$ along $g$ in $V(\K)$ when the bottom assignment is a bijection, so that the proof of the first assertion follows from the fact that, assuming that $\psi'$ is cartesian, both the assignment on the left and that at the top are bijections.
		\begin{displaymath}

		\end{displaymath}
		For the converse assume that $\K$ has cocartesian tabulations and that $\psi$ defines $j$ as an absolute left lifting. It follows that, for any vertical morphisms $p$ an $q$, the top assignment in the diagram on the left above is a bijection; that is, any vertical cell $\chi$ as in the collection on the right factors uniquely through $\psi'$. To prove that $\psi'$ is cartesian in $\K$ we have to show that any nullary cell $\chi'$ as on the right above factors uniquely through $\psi'$ as well. But this follows the fact that any such $\chi'$ corresponds to a cell of the form $\chi$, by composing it with the piramid-shaped cocartesian path that defines the tabulation $\tab{H_1, \dotsc, H_m}$, as in \eqref{piramid}; under this correspondence the factorisation of $\chi'$ through $\psi'$ corresponds to that of $\chi$. This completes the proof.
	\end{proof}
	
	\subsection{Pointwise Kan extensions in terms of pointwise weak Kan extensions} \label{pointwise Kan extensions section}
	 In the theorem below we prove that the notions of pointwise weak Kan extension and pointwise Kan extension coincide in augmented virtual double categories that have cocartesian trabulations. This is analogous to the situation for double categories, as described in Section~5 of \cite{Koudenburg14a}, and reminiscent of Street's definition of pointwise Kan extension in $2$-categories (see \cite{Street74b}), which is given in terms of ordinary Kan extensions.
	\begin{theorem} \label{pointwise Kan extensions in terms of pointwise weak Kan extensions}
		In an augmented virtual double category that has cocartesian tabulations, all pointwise weak left Kan extensions are pointwise left Kan extensions.
	\end{theorem}
	\begin{proof}
		Consider a nullary cell $\eta$ as on the left below and assume that it defines $l$ as the pointwise weak left Kan extension of $d$ along $(J_1, \dotsc, J_n)$. Below we prove that $\eta$ defines $l$ as a left Kan extension; thus, from applying the proof to composites of $\eta$ with a cartesian cell defining restrictions $J_n(\id, f)$, which again define weak left Kan extensions by \lemref{properties of pointwise left Kan extensions}(b), we conclude that $\eta$ defines $l$ as a pointwise left Kan extension.
		\begin{displaymath}
			\begin{tikzpicture}[baseline]
				\matrix(m)[math35]{A_0 & A_1 & A_{n'} & A_n \\};
				\draw	([yshift=-3.25em]$(m-1-1)!0.5!(m-1-4)$) node (M) {$M$};
				\path[map]	(m-1-1) edge[barred] node[above] {$J_1$} (m-1-2)
														edge[transform canvas={yshift=-2pt}] node[below left] {$d$} (M)
										(m-1-3) edge[barred] node[above] {$J_n$} (m-1-4)
										(m-1-4) edge[transform canvas={yshift=-2pt}] node[below right] {$l$} (M);
				\path				($(m-1-1.south)!0.5!(m-1-4.south)$) edge[cell] node[right] {$\eta$} (M);
				\draw				($(m-1-2)!0.5!(m-1-3)$) node {$\dotsb$};
			\end{tikzpicture}\begin{tikzpicture}[baseline]
				\matrix(m)[math35]{A_0 & A_1 &[-0.5em] A_{n'} & A_n & B_1 &[-0.5em] B_{m'} & B_m \\};
				\draw				([yshift=-3.25em]$(m-1-1)!0.5!(m-1-7)$) node (M) {$M$};
				\path[map]	(m-1-1) edge[barred] node[above] {$J_1$} (m-1-2)
														edge[transform canvas={yshift=-2pt}] node[below left] {$d$} (M)
										(m-1-3) edge[barred] node[above] {$J_n$} (m-1-4)
										(m-1-4) edge[barred] node[above] {$H_1$} (m-1-5)
										(m-1-6) edge[barred] node[above] {$H_m$} (m-1-7)
										(m-1-7) edge[transform canvas={yshift=-2pt}] node[below right] {$k$} (M);
				\draw				($(m-1-2)!0.5!(m-1-3)$) node {$\dotsb$}
										($(m-1-5)!0.5!(m-1-6)$) node {$\dotsb$};
				\path				($(m-1-1.south)!0.5!(m-1-7.south)$) edge[cell] node[right] {$\phi$} (M);
			\end{tikzpicture}
		\end{displaymath}

		Hence consider a cell $\phi$ as on the right above; we have to show that it factors uniquely through $\eta$. To do so we consider the tabulation $\tab{H_1, \dotsc, H_m}$ analogous to \eqref{piramid}; remember that, since the tabulations used in constructing $\tab{H_1, \dotsc, H_m}$ are cocartesian, they form a piramid-shaped cocartesian path, as was discussed following \eqref{piramid}. It follows that, under precomposition with this path, cells of the form $\phi$ on the right above correspond bijectively to cells $\psi$ as on the left below. Using \lemref{unique factorisations factorised through a cocartesian cell on the right} we conclude that, under this correspondence, a unique factorsation of the cells $\phi$ through the composite on the left above corresponds to a unique factorisation of the cells $\psi$ through the composite on the right below, as vertical cells of the form $l \of \pi_{A_n} \Rar k \of \pi_{B_m}$.
		\begin{displaymath}
			\begin{tikzpicture}[baseline]
				\matrix(m)[math35, yshift=1.625em]{A_0 & A_1 & A_{n'} &[2em] \tab{H_1, \dotsc, H_m} \\};
				\draw				([yshift=-3.25em]$(m-1-1)!0.5!(m-1-4)$) node (M) {$M$};
				\path[map]	(m-1-1) edge[barred] node[above] {$J_1$} (m-1-2)
														edge[transform canvas={yshift=-2pt}] node[below left] {$d$} (M)
										(m-1-3) edge[barred] node[above, shift={(-3pt,2pt)}] {$J_n(\id, \pi_{A_n})$} (m-1-4)
										(m-1-4) edge[transform canvas={yshift=-2pt}] node[below right] {$k\of \pi_{B_m}$} (M);
				\draw				($(m-1-2)!0.5!(m-1-3)$) node {$\dotsb$};
				\path				($(m-1-1.south)!0.5!(m-1-4.south)$) edge[cell] node[right] {$\psi$} (M);
			\end{tikzpicture} \quad\quad\mspace{-11mu} \begin{tikzpicture}[baseline, yshift=1.625em]
				\matrix(m)[math35]{A_0 & A_1 & A_{n'} &[2em] \tab{H_1, \dotsc, H_m} \\ A_0 & A_1 & A_{n'} & A_n \\};
				\draw	([yshift=-6.5em]$(m-1-1)!0.5!(m-1-4)$) node (M) {$M$};
				\path[map]	(m-1-1) edge[barred] node[above] {$J_1$} (m-1-2)
														
										(m-1-3) edge[barred] node[above, shift={(-3pt,2pt)}] {$J_n(\id, \pi_{A_n})$} (m-1-4)
										(m-1-4) edge node[right] {$\pi_{A_n}$} (m-2-4)
										(m-2-4) edge[transform canvas={yshift=-2pt}] node[below right] {$l$} (M)
										(m-2-1) edge[barred] node[below, inner sep=2.5pt] {$J_1$} (m-2-2)
														edge[transform canvas={yshift=-2pt}] node[below left] {$d$} (M)
										(m-2-3) edge[barred] node[below, inner sep=2.5pt] {$J_n$} (m-2-4);
				\path				(m-1-1) edge[eq] (m-2-1)
										(m-1-2) edge[eq] (m-2-2)
										(m-1-3) edge[eq] (m-2-3);
				\path				($(m-2-1.south)!0.5!(m-2-4.south)$) edge[cell] node[right] {$\eta$} (M);
				\draw				($(m-1-2)!0.5!(m-2-3)$) node {$\dotsb$}
										($(m-1-3)!0.5!(m-2-4)$) node[font=\scriptsize] {$\cart$};
			\end{tikzpicture}
		\end{displaymath}
		That the cells $\psi$ do factorise uniquely through the composite on the right follows, by definition, from the assumption that $\eta$ defines $l$ as a pointwise weak left Kan extension. This completes the proof.
	\end{proof}
	
	In closing this section we show how \thmref{pointwise Kan extensions in terms of pointwise weak Kan extensions} can be used to extend \propref{weak left Kan extensions along companions} to the pointwise case, in the setting of augmented virtual equipments (\defref{augmented virtual equipment}) that have cocartesian tabulations. This generalises the corresponding result for double categories given in \cite{Koudenburg14a}.
	\begin{proposition} \label{pointwise left Kan extensions along companions}
		Consider the following factorisation in an augmented virtual equipment that has cocartesian tabulations.
		\begin{displaymath}
			\begin{tikzpicture}[textbaseline]
				\matrix(m)[math35, column sep={1.75em,between origins}]{& A & \\ & & B \\ & M & \\};
				\path[map]	(m-1-2) edge[bend left = 18] node[above right] {$j$} (m-2-3)
														edge[bend right = 45] node[left] {$d$} (m-3-2)
										(m-2-3) edge[bend left = 18] node[below right] {$l$} (m-3-2);
				\path[transform canvas={yshift=-1.625em}]	(m-1-2) edge[cell] node[right] {$\eta$} (m-2-2);
			\end{tikzpicture} = \begin{tikzpicture}[textbaseline]
    		\matrix(m)[math35, column sep={1.75em,between origins}]{& A & \\ A & & B \\ & M & \\};
    		\path[map]	(m-1-2) edge[transform canvas={xshift=2pt}] node[right] {$j$} (m-2-3)
    								(m-2-1) edge[barred] node[below, inner sep=2pt] {$j_*$} (m-2-3)
    												edge[transform canvas={xshift=-2pt}] node[left] {$d$} (m-3-2)
    								(m-2-3)	edge[transform canvas={xshift=2pt}] node[right] {$l$} (m-3-2);
    		\path				(m-1-2) edge[eq, transform canvas={xshift=-2pt}] (m-2-1);
    		\path				(m-2-2) edge[cell, transform canvas={yshift=0.1em}]	node[right, inner sep=3pt] {$\eta'$} (m-3-2);
    		\draw				([yshift=-0.5em]$(m-1-2)!0.5!(m-2-2)$) node[font=\scriptsize] {$\cocart$};
  		\end{tikzpicture}
  	\end{displaymath}
		The vertical cell $\eta$ defines $l$ as the pointwise left Kan extension of $d$ along $j$ in the $2$"/category $V(\K)$, in the sense of Section 4 of \cite{Street74b} (or see Section 2.4 of \cite{Weber07}), precisely if its factorisation $\eta'$ defines $l$ as the pointwise left Kan extension of $d$ along $j_*$ in $\K$.
	\end{proposition}
	Street's notion of pointwise Kan extension in a $2$-category uses the well known notion of \emph{comma object}; see \cite{Street74b}, or \cite{Weber07} where it is called lax pullback. Instead of recalling the definition of comma object we record the following lemma, which relates it to that of tabulation.
	\begin{lemma}
		Consider vertical morphisms $\map fAC$ and $\map gBC$ in an augmented virtual double category $\K$. If both the nullary cartesian cell and the tabulation below exist then their composite defines $\tab{C(f, g)}$ as the comma object $f \slash g$ of $f$ and $g$ in $V(\K)$.
		\begin{displaymath}
			\begin{tikzpicture}
    		\matrix(m)[math35, column sep={1.75em,between origins}]{& \tab{C(f,g)} & \\ A & & B \\ & C & \\};
    		\path[map]	(m-1-2)	edge[transform canvas={xshift=-2pt}] node[left] {$\pi_A$} (m-2-1)
    												edge[transform canvas={xshift=2pt}] node[right] {$\pi_B$} (m-2-3)
    								(m-2-1) edge[barred] node[below, inner sep=2pt] {$C(f,g)$} (m-2-3)
    												edge[transform canvas={xshift=-2pt}] node[below left] {$f$} (m-3-2)
    								(m-2-3)	edge[transform canvas={xshift=2pt}] node[right] {$g$} (m-3-2);
    		\path				(m-1-2) edge[cell, transform canvas={yshift=-0.25em}]	node[right] {$\pi$} (m-2-2);
    		\draw				([yshift=0.1em]$(m-2-2)!0.5!(m-3-2)$) node[font=\scriptsize] {$\cart$};
  		\end{tikzpicture}
		\end{displaymath}
	\end{lemma}
	\begin{proof}
		This follows immediately from the fact that the universal properties defining the comma object of $f$ and $g$ in $V(\K)$ correspond precisely to universal properties satisfied by the tabulation of $C(f,g)$ in $\K$, by factorising them through the nullary cartesian cell that defines $C(f, g)$.
	\end{proof}
	\begin{proof}[Proof of \propref{pointwise left Kan extensions along companions}]
		Given any $\map fCB$ we consider the composite on the left below, where $\eta = \eta' \of \cocart$ and where $\cart \of \pi$ defines the comma object $j \slash f$, by the previous lemma. Hence $\eta$ defines $l$ as the pointwise left Kan extension in the $2$-category $V(\K)$ if, for any morphism $f$, the composite on the left defines $l \of f$ as a left Kan extension. The latter in turn means that the top assignment in the commutative diagram on the right, of assignments between collections of cells in $\K$ as shown, is a bijection.
		\begin{displaymath}

		\end{displaymath}
		The bottom assignment on the other hand is a bijection, for every $f$, precisely when $\eta'$ defines $l$ as the pointwise left Kan extension of $d$ along $j_*$: this is a consequence of \thmref{pointwise Kan extensions in terms of pointwise weak Kan extensions} and the fact that the composite $\cell{\cocart \hc \cart}{B(j, f)}{j_*}$, in the composite on the left, is cartesian. The proof now follows from the fact that the right assignment above is a bijection, as $\pi$ is cocartesian.
	\end{proof}
	
	The following is Example 2.24 of \cite{Koudenburg13}.
	\begin{example}
		Recall from \exref{tabulations of 2-profunctors} that tabulations of $2$-profunctors are in general not cocartesian. As a consequence the equivalence described in \propref{pointwise left Kan extensions along companions} above fails to hold in the augmented virtual equipment $\enProf{(\Cat, \Cat')}$. For a counter example consider the $2$-natural transformation of $2$-functors below, where $1$ is the terminal $2$-category and the collapsing $2$-functor on the left has the `free living' cell as source and the free living parallel pair of arrows as target. It is straightforward to check that this transformation defines the collapsing $2$-functor as the enriched left Kan extension of $x'$ along $x$, in the sense of e.g.\ Section 4.1 of \cite{Kelly82}, while the pointwise left Kan extension of $x'$ along $x$ in the $2$-category $\twoCat$, of $2$-categories, $2$-functors and $2$-natural transformations, and in the sense of Section 4 of \cite{Street74b}, does not exist.
		\begin{displaymath}
			\begin{tikzpicture}
				\matrix(m)[math35, column sep=2.5em, xshift=5em]{ & \\};
				\path[map]	(m-1-1) edge[bend left=40] (m-1-2)
														edge[bend right=40] (m-1-2);
				\fill[font=\scriptsize]	(m-1-1) circle (1pt) node[left, inner sep=2pt] {$x$}
							(m-1-2) circle (1pt) node[right, inner sep=2pt] {$y$};
				\path	($(m-1-1)!0.5!(m-1-2)+(0,0.8em)$) edge[cell] ($(m-1-1)!0.5!(m-1-2)-(0,0.8em)$);
				
				\matrix(m)[math35, column sep=2.5em, yshift=-5em]{ & \\};
				\path[map]	(m-1-1) edge[bend left=40] (m-1-2)
														edge[bend right=40] (m-1-2);
				\fill[font=\scriptsize]	(m-1-1) circle (1pt) node[left, inner sep=1pt] {$x'$}
							(m-1-2) circle (1pt) node[right, inner sep=2pt] {$y'$};
				
				\draw	(-5em,0) node {$1$};
				
				\draw[font=\Large]	(2.5em,0) node {$($}
							(7.5em,0) node {$)$}
							(-2.75em,-5em) node {$($}
							(2.75em,-5em) node {$)$};
				
				\path[map]	(-3em, 0) edge node[above] {$x$} (0.5em, 0)
										(-4.5em, -1.25em) edge node[below left] {$x'$} (-2.75em, -3.75em)
										(4.5em, -1.25em) edge node[below right] {(collapse onto $x'$)} (2.75em, -3.75em);
				
				\path				(-1em,-2.5em) edge[cell] node[above] {(identity at $x'$)} (1em,-2.5em);
			\end{tikzpicture}
		\end{displaymath}
	\end{example}

	\section{Yoneda embeddings}\label{yoneda embeddings section}
	We are now ready to introduce the main notion of this paper, that of yoneda embedding in augmented virtual double categories. Informally a yoneda embedding is a vertical morphism that is `dense' (as defined below) and that satifies an axiom that ``formally captures the Yoneda's lemma''. These two conditions are closely related to the axioms satisfied by morphisms that make up a `yoneda structure' on a $2$-category, as introduced by Street and Walters in \cite{Street-Walters78}; see also \cite{Weber07}. Our position of regarding horizontal morphisms as ``being small in size'' enables us to give a relatively simple definition of a single yoneda embedding. This in contrast to the definition of yoneda structure on a $2$-category, which consists of a collection of yoneda embeddings that satisfy a ``formal Yoneda's lemma'' with respect to a specified collection of `admissible' morphisms (informally these are to be thought of as ``small in size'').
	
	\subsection{Definition of yoneda embedding}
	We start with the notion of density. Recall from Section 5.1 of \cite{Kelly82} that one way of defining a $\V$-functor $\map fAM$ to be \emph{dense} is to require that its identity $\V$"/natural transformation $\nat{\id_f}ff$ defines $\id_M$ as the (enriched) left Kan extension of $f$ along itself. As we have seen in \propref{enriched left Kan extensions in terms of weighted colimits}, the latter is equivalent to asking that the cartesian cell defining the companion $\hmap{f_*}AM$, in $\enProf{(\V, \V')}$, defines $\id_M$ as the pointwise left Kan extension of $f$ along $f_*$. In general, for a vertical morphism $\map fAM$ in any augmented virtual double category, further equivalent conditions are given by the following lemma. We call $f$ \emph{dense} if these equivalent conditions are satisfied.
	\begin{lemma} \label{density axioms}
		For a vertical morphism $\map fAM$ in an augmented virtual double category the following conditions are equivalent:
		\begin{enumerate}[label=\textup{(\alph*)}]
			\item	if a cell $\eta$, as on the left below, is cartesian then it defines $l$ as the left Kan extension of $f$ along $J$;
			\item if a cell $\eta$ as below is cartesian then it defines $l$ as the pointwise left Kan extension of $f$ along $J$.
		\end{enumerate}
		If the companion $\hmap{f_*}AM$ exists then the following condition is equivalent too:
		\begin{enumerate}
			\item[\textup{(c)}]	the nullary cartesian cell on the right below, that defines the companion $f_*$, defines $\id_M$ as the pointwise left Kan extension of $f$ along $f_*$.
		\end{enumerate}
			\begin{displaymath}

	    \end{displaymath}
	\end{lemma}
	\begin{proof}
		(b) $\Rightarrow$ (a) is clear. For the converse, assume that (a) holds and consider any cartesian cell as in the composite on the left hand side below. Since $\eta$ is cartesian the full composite is cartesian too by the pasting lemma. Hence, applying (a) we find that the composite defines $l \of g$ as a left Kan extension, showing that (b) holds.
		\begin{displaymath}
			%
	  \end{displaymath}
		
		(b) $\Rightarrow$ (c) is clear. For the converse consider a cartesian cell $\eta$ as on the right above and let $\eta'$ be its factorisation as shown; it is cartesian because $\eta$ is, by the pasting lemma. Assuming (c) we may apply \lemref{properties of pointwise left Kan extensions}(b) to find that $\eta$ defines $l$ as a pointwise left Kan extension; we conclude that (c) $\Rightarrow$ (b).
	\end{proof}
	
	Following the definition of dense morphism we define yoneda embeddings.
	\begin{definition} \label{yoneda embedding}
		A dense vertical morphism $\map\yon A{\ps A}$ in an augmented virtual double category is called a \emph{yoneda embedding} if it satisfies the \emph{yoneda axiom}: for every horizontal morphism $\hmap JAB$ there exists a cartesian cell
		\begin{displaymath}
			\begin{tikzpicture}
				\matrix(m)[math35, column sep={1.75em,between origins}]{A & & B \\ & \ps A. & \\};
				\path[map]	(m-1-1) edge[barred] node[above] {$J$} (m-1-3)
														edge[transform canvas={xshift=-2pt}] node[left] {$\yon$} (m-2-2)
										(m-1-3) edge[transform canvas={xshift=2pt}] node[right] {$\cur J$} (m-2-2);
				\draw				([yshift=0.333em]$(m-1-2)!0.5!(m-2-2)$) node[font=\scriptsize] {$\cart$};
			\end{tikzpicture}
	  \end{displaymath}
	\end{definition}
	
	Notice that the density of $\map\yon A{\ps A}$ implies that the vertical morphism $\cur J$ above is unique up to vertical isomorphism. If the yoneda embedding above exists then we call $\ps A$ the \emph{object of presheaves on $A$}, or \emph{presheaf object} for short. Yoneda embeddings $\map\yon A{\ps A}$ such that all restrictions $\ps A(\yon, f)$ exist, for any $\map fB{\ps A}$, are especially pleasant to work with; we call them \emph{good} yoneda embeddings. Notice that, when considered in an augmented virtual equipment, the latter condition reduces to the existence of the companion $\hmap{\yon_*}A{\ps A}$, since $\ps A(\yon, f) \iso \yon_*(\id, f)$.
	
	Writing $H(\K)(A, B)$ for the category of horizontal morphism $A \brar B$ and the horizontal cells between them, we remark that choosing a restriction $\ps A(\yon, g)$ for each $\map gB{\ps A}$ induces a functor $\map{\ps A(\yon, \dash)}{V(\K)(B, \ps A)}{H(\K)(A, B)}$ whose action $\phi \mapsto \ps A(\yon, \phi)$ on cells is defined by the unique factorisations
	\begin{displaymath}
		\begin{tikzpicture}[textbaseline]
			\matrix(m)[math35, column sep={1.75em,between origins}]{A & & B \\ & \ps A & \\};
				\path[map]	(m-1-1) edge[barred] node[above] {$\ps A(\yon, f)$} (m-1-3)
														edge[transform canvas={xshift=-2pt}] node[left] {$\yon$} (m-2-2)
										(m-1-3) edge[transform canvas={xshift=3pt}] node[above left, inner sep=-0.75pt] {$f$} (m-2-2)
										(m-1-3) edge[bend left=65] node[below right] {$g$} (m-2-2);
				\draw				([yshift=0.25em]$(m-1-2)!0.5!(m-2-2)$) node[font=\scriptsize] {$\cart$}
										(m-1-2) edge[cell, transform canvas={shift={(1.5em,-0.35em)}}] node[right, inner sep=3pt] {$\phi$} (m-2-2);
		\end{tikzpicture} = \begin{tikzpicture}[textbaseline]
			\matrix(m)[math35, column sep={1.75em,between origins}]{A & & B \\ A & & B. \\ & \ps A & \\};
			\path[map]	(m-1-1) edge[barred] node[above] {$\ps A(\yon, f)$} (m-1-3)
									(m-2-3) edge[ps, transform canvas={xshift=1pt}] node[right] {$g$} (m-3-2)
									(m-2-1) edge[barred] node[below, inner sep=1.5pt] {$\ps A(\yon, g)$} (m-2-3)
													edge[ps, transform canvas={xshift=-1pt}] node[left] {$\yon$} (m-3-2);
			\path				(m-1-2) edge[cell, transform canvas={xshift=-1.275em}] node[right] {$\ps A(\yon, \phi)$} (m-2-2)
									(m-1-1) edge[eq] (m-2-1)
									(m-1-3) edge[eq] (m-2-3);
			\draw				([yshift=0.2em]$(m-2-2)!0.5!(m-3-2)$) node[font=\scriptsize] {$\cart$};
		\end{tikzpicture}
	\end{displaymath}
	Notice that the density of $\yon$ ensures that $\ps A(\yon, \dash)$ is full and faithful. Given a vertical morphism $\map\yon A{\ps A}$ in a `bicategory equipped with proarrows' (see the remarks preceding \defref{augmented virtual equipment}), Melli\`es and Tabareau consider in \cite{Mellies-Tabareau08} a functor analogous to $\ps A(\yon, \dash)$ above; they define $\yon$ to be a `yoneda situation' if both $\yon$ and this functor are full and faithful.
	
	At the end of this subsection we will show that a $\V'$-category $A$ admits a good yoneda embedding in $\enProf{(\V, \V')}$ (\exref{(V, V')-Prof}) precisely if $A$ is a $\V$-category. First we record some basic properties of yoneda embeddings; the following are variations of Lemma 3.2 and Corollary 3.5 of \cite{Weber07}.
	\begin{lemma} \label{left Kan extensions along yoneda embeddings}
		Let $\map\yon A{\ps A}$ be a yoneda embedding. If the nullary cell below defines $l$ as the (weak) left Kan extension of $J$ along $\yon$ then it is cartesian and, moreover, it defines $l$ as a pointwise left Kan extension.
		\begin{displaymath}
			\begin{tikzpicture}
				\matrix(m)[math35, column sep={1.75em,between origins}]{A & & B \\ & \ps A & \\};
				\path[map]	(m-1-1) edge[barred] node[above] {$J$} (m-1-3)
														edge[transform canvas={xshift=-2pt}] node[left] {$\yon$} (m-2-2)
										(m-1-3) edge[transform canvas={xshift=2pt}] node[right] {$l$} (m-2-2);
				\path[transform canvas={yshift=0.25em}]	(m-1-2) edge[cell] node[right, inner sep=3pt] {$\eta$} (m-2-2);
			\end{tikzpicture}
	  \end{displaymath}
	\end{lemma}
	\begin{proof}
		By the yoneda axiom there exists a morphism $\map{\cur J}B{\ps A}$ such that $J$ is the nullary restriction of $\ps A$ along $\yon$ and $g$, and the cartesian cell $\eps$ defining $J$ as this restriction defines $\cur J$ as a pointwise left Kan extension of $\yon_A$ along $J$ by \lemref{density axioms}(b). Since $\eta$ defines the same left Kan extension we conclude that it factors through $\eps$ as an invertible vertical cell $g \iso l$; it follows that, like $\eps$, $\eta$ is cartesian and defines $l$ as a pointwise left Kan extension.
	\end{proof}
	
	\begin{lemma}\hspace{-0.25cm}\footnote{\textbf{This result is false.} See Lemma~4.3 and Proposition~4.16 of \cite{Koudenburg22} instead.} \label{yoneda embedding full and faithful}
		For a yoneda embedding $\map\yon A{\ps A}$ the following hold:
		\begin{enumerate}[label=\textup{(\alph*)}]
			\item if $\yon$ is a good yoneda embedding then it is full and faithful, $A$ is unital, while $\cur{I_A} \iso \yon$;
			\item any $\map fAC$, such that the companion $f_*$ and the restriction $C(f, f)$ exist, is full and faithful precisely if the composite
			\begin{displaymath}
				\begin{tikzpicture}
	    		\matrix(m)[math35, column sep={1.75em,between origins}]{& A & \\ A & & C \\ & \ps A & \\};
  	  		\path[map]	(m-1-2) edge[transform canvas={xshift=2pt}] node[right] {$f$} (m-2-3)
  	  								(m-2-1) edge[barred] node[below, inner sep=2pt] {$f_*$} (m-2-3)
  	  												edge[transform canvas={xshift=-2pt}] node[left] {$\yon$} (m-3-2)
  	  								(m-2-3)	edge[transform canvas={xshift=2pt}] node[right] {$\cur{f_*}$} (m-3-2);
  	  		\path				(m-1-2) edge[eq, transform canvas={xshift=-2pt}] (m-2-1);
  	  		\draw				([yshift=-0.5em]$(m-1-2)!0.5!(m-2-2)$) node[font=\scriptsize] {$\cocart$}
  	  								([yshift=0.25em]$(m-2-2)!0.5!(m-3-2)$) node[font=\scriptsize] {$\cart$};
  			\end{tikzpicture}
			\end{displaymath}
			is invertible, where the cartesian cell exists by the yoneda axiom.
		\end{enumerate}
	\end{lemma}
	\begin{proof}
		(a). That $\yon$ is full and faithful follows from applying part (b) to the vertical companion identity of $\yon$ (see \lemref{companion identities lemma}); here we use that $\yon$ is a good yoneda embedding, so that the restriction $\ps A(\yon, \yon)$ exists. That $A$ is unital then follows from \corref{unit in terms of full and faithful map}, which also implies that $I_A \iso \ps A(\yon, \yon)$. Composing the latter with the cell defining $\ps A(\yon, \yon)$ we obtain a cartesian cell that defines $\yon$ as the left Kan extension of $I_A$ along $\yon$, as follows from the density of $\yon$. We conclude $\cur{I_A} \iso \yon$ by the uniqueness of Kan extensions.
		
		(b). First notice that the composite above can be rewritten as the composite below, where the cell $\id_f'$ is the factorisation of $\id_f$ through the cartesian cell defining $C(f, f)$, as in \corref{unit in terms of full and faithful map}. Indeed this follows from the fact that, when composed with the cartesian cell defining $f_*$, both the weakly cocartesian cell above as well as the composite of the two top cells below equal $\id_f$.
		\begin{displaymath}
			\begin{tikzpicture}
	    		\matrix(m)[math35, column sep={1.75em,between origins}]{& A & \\ A & & A \\ A & & C \\ & \ps A & \\};
  	  		\path[map]	(m-2-1) edge[barred] node[below] {$C(f, f)$} (m-2-3)
  	  								(m-2-3) edge node[right] {$f$} (m-3-3)
  	  								(m-3-1) edge[barred] node[below, inner sep=2pt] {$f_*$} (m-3-3)
  	  												edge[transform canvas={xshift=-2pt}] node[left] {$\yon$} (m-4-2)
  	  								(m-3-3)	edge[transform canvas={xshift=2pt}] node[right] {$\cur{f_*}$} (m-4-2);
  	  		\path				(m-1-2) edge[cell, transform canvas={shift={(-0.5em,-0.5em)}}] node[right] {$\id_f'$} (m-2-2)
  	  								(m-1-2) edge[eq, transform canvas={xshift=-2pt}] (m-2-1)
  	  												edge[eq, transform canvas={xshift=2pt}] (m-2-3)
  	  								(m-2-1) edge[eq] (m-3-1);
  	  		\draw				([yshift=-0.25em]$(m-2-2)!0.5!(m-3-2)$) node[font=\scriptsize] {$\cart$}
  	  								([yshift=0.25em]$(m-3-2)!0.5!(m-4-2)$) node[font=\scriptsize] {$\cart$};
  			\end{tikzpicture}
		\end{displaymath}
		The proof follows by noticing that the following conditions are equivalent: $f$ is full and faithful; $\id_f'$ is cocartesian (by \corref{unit in terms of full and faithful map}); $\id_f'$ is cartesian (by \lemref{unit identities}); the composite above is cartesian (by the pasting lemma); the composite defines a left Kan extension (because $\yon$ is dense and the previous lemma); the composite is invertible (by \lemref{vertical cells defining left Kan extensions}).
	\end{proof}
	
	As an example we now consider yoneda embeddings in the augmented virtual double category $\enProf{(\V, \V')}$ of $\V$-profunctors between $\V'$-categories, that is associated to a universe enlargement $\V \to \V'$ as described in \exref{(V, V')-Prof}. In the case that $\V'$ is closed symmetric monoidal the construction of such yoneda embeddings is of course classical, see for instance Section 2 of \cite{Kelly82}. The nuts and bolts description below, which leaves many easy details to check, shows that the classical construction extends to the case in which $\V'$ is merely closed monoidal.
	
	Let $A$ be a (possibly large) $\V'$-category; by a \emph{$\V$-presheaf} $p$ on $A$ we mean a $\V$"/profunctor $\hmap pAI$. Thus $p$ consists of $\V$-objects $px$, one for each $x \in A$, equipped with an action of $A$ given by $\V'$-maps
	\begin{displaymath}
		\map\lambda{A(x,y) \tens' py}{px},
	\end{displaymath}
	that is associative and unital. If both $\V$ and $\V'$ are closed symmetric monoidal then $\V$-presheaves on a $\V$-category $A$ can be identified with $\V$-functors $\op A \to \V$.
	
	The $\V$-presheaves on $A$ arrange into a $\V'$-category $\ps A$ as follows. To define the hom-object $\ps A(p, q)$ of $\V$-presheaves $p$ and $q$, we write $\brks{p, q}'$ for the inclusion into $\V'$ of its subcategory consisting of the cospans
	\begin{equation} \label{cospans for V'-category of presheaves}
		\begin{tikzpicture}[textbaseline]
			\matrix(m)[math35, column sep={4.5em,between origins}]
				{	\brks{px, qx}' & & \brks{py, qy}', \\ & \brks{A(x, y) \tens' py, qx}' & \\ };
			\path[map]	(m-1-1) edge (m-2-2)
									(m-1-3) edge (m-2-2);
		\end{tikzpicture}
	\end{equation}
	where $x$ and $y$ range over the objects in $A$, whose legs are adjunct to the $\V'$-maps
	\begin{flalign*}
		&& A(x,y)& \tens' py \tens' \brks{px, qx}' \xrar{\lambda \tens' \id} px \tens' \brks{px, qx}' \xrar{\textup{ev}} qx& \\
		\text{and} && A(x,y)& \tens' py \tens' \brks{py, qy}' \xrar{\id \tens' \textup{ev}} A(x,y) \tens' qy \xrar\lambda qx,&
	\end{flalign*}
	under the hom-tensor adjunction $X \tens' \dash \ladj \brks{X, \dash}'$ of $\V'$. Since $A$ is large the limit of $\brks{p, q}'$ exists in $\V'$, which we take as the hom-object $\ps A(p, q)$. In the case that $\V'$ is closed symmetric monoidal it easy to check that $\ps A(p, q)$ forms the end $\int_{x \in A} \brks{px, qx}'$, in the sense of Section 2.1 of \cite{Kelly82}. Composition $\ps A(p, q) \tens' \ps A(q, r) \to \ps A(p, r)$ is the factorisation of the cone $\Delta\bigpars{\ps A(p, q) \tens' \ps A(q, r)}\Rar\brks{p, r}'$ that uniquely extends the family of $\V'$-maps adjunct to
	\begin{displaymath}
		px \tens' \ps A(p,q) \tens' \ps A(q, r) \to px \tens' \brks{px, qx}' \tens' \brks{qx, rx}' \xrar{\textup{ev} \tens' \id} qx \tens' \brks{qx, rx}' \xrar{\textup{ev}} rx
	\end{displaymath}
	where the first map is given by projections. Similarly the units $I' \to \ps A(p, p)$ are induced by the adjuncts to the unitors $px \tens' I' \xrar{\mathfrak r} px$ of $\V'$. This completes the definition of the $\V'$-category $\ps A$ of $\V$-presheaves on $A$; notice that, because $A$ and $\V$ are large while $\V'$ is locally large, $\ps A$ is again a large $\V'$-category.
	\begin{proposition} \label{enriched yoneda embeddings}
		Let $\V \to \V'$ be a universe enlargement. A large $\V'$-category $A$ admits a good yoneda embedding in $\enProf{(\V, \V')}$ if and only if it is a $\V$-category. More precisely, in that case $\map \yon A{\ps A}$ can be given with the $\V'$-category $\ps A$ as defined above and $\yon x = A(\dash, x)$ while, for each $\V$-profunctor $\hmap JAB$, the cartesian cell
		\begin{displaymath}
			\begin{tikzpicture}
				\matrix(m)[math35, column sep={1.75em,between origins}]{A & & B \\ & \ps A & \\};
				\path[map]	(m-1-1) edge[barred] node[above] {$J$} (m-1-3)
														edge[transform canvas={xshift=-2pt}] node[left] {$\yon$} (m-2-2)
										(m-1-3) edge[transform canvas={xshift=2pt}] node[right] {$\cur J$} (m-2-2);
				\draw				([yshift=0.333em]$(m-1-2)!0.5!(m-2-2)$) node[font=\scriptsize] {$\cart$};
			\end{tikzpicture}
	  \end{displaymath}
	  has $\cur J y = J(\dash, y)$ as vertical target.
	  
	  Finally also assume that $\V$ is closed monoidal, and that $\V \to \V'$ is a closed monoidal functor. For any $\V$-presheaves $p$ and $q$ on $A$ the diagram $\brks{p, q}'$ in $\V'$, as described above, factors (up to natural isomorphism) through $\V \to \V'$ as a diagram $\brks{p, q}$ in $\V$, that is obtained by replacing the inner-homs $\brks{\dash, \dash}'$ of $\V'$ in \eqref{cospans for V'-category of presheaves} by those $\brks{\dash, \dash}$ of $\V$. In this case $\ps A$ is a $\V$-category if and only if, for any pair of $\V$-presheaves $p$ and $q$ on $A$, the factorisation diagram $\brks{p, q}$ in $\V$ has a limit.
	\end{proposition}
	\begin{proof}
		For the `only if'-part notice that the existence of a good yoneda embedding implies that $A$ is unital by \lemref{yoneda embedding full and faithful}(a) and thus a $\V$-category as we saw in \exref{horizontal composites in (V, V')-Prof}.
		
		For the `if'-part assume that $A$ is a $\V$-category, so that $\yon$ can be defined on objects as described above. In detail: $\yon x$ is the $\V$-presheaf given by the hom-objects $(\yon x)(s) = A(s, x)$, equipped with actions $A(s, t) \tens' (\yon x)(t) \to (\yon x)(s)$ that are given by the composition of $A$. On hom-objects $\yon$ is given by the unique factorisations $\map{\bar\yon}{A(x, y)}{\ps A(\yon x, \yon y)}$ of the cones $\Delta A(x,y) \Rar \brks{\yon x, \yon y}'$ that are unique extensions of the adjuncts to $(\yon x)(s) \tens' A(x, y) \to (\yon y)(s)$, which in turn are again given by the composition of $A$. A straightforward calculation shows that $\bar\yon$ preserves composition and units.
		
		It remains to show that $\yon$ is dense, admits a companion $\hmap{\yon_*}A{\ps A}$, and satisfies the yoneda axiom. For the latter consider a $\V$-profunctor $\hmap JAB$; we have to supply a cartesian cell as above. As mentioned in the statement, the $\V$-functor $\map{\cur J}B{\ps A}$ maps $y \in B$ to the $\V$-presheaf given by $(\cur J y)(s) = J(s, y)$ on objects, while its action \mbox{$A(s, t) \tens' J(t, y) \to J(s, y)$} is simply that of $A$ on $J$. Its action on hom-objects $\map{\bar{\cur J}}{B(y, z)}{\ps A(\cur Jy, \cur Jz)}$ is induced by the adjuncts to the actions $J(x, y) \tens' B(y, z) \to J(x, z)$ of $B$ on $J$. Checking that $\bar{\cur J}$ preserves composition and units is again straightforward while, in order to show that the cartesian cell above exists, it suffices to show that $J(x, y) \iso \ps A(\yon x, \cur Jy)$, natural in $x$ and $y$. First notice that natural $\V'$-maps $J(x, y) \to \ps A(\yon x, \cur Jy)$ can be obtained from the adjuncts of the maps $(\yon x)(s) \tens' J(x, y) \to J(s, y)$ that form the action of $A$ on $J$; it is then not hard to show that the composites
		\begin{displaymath}
			\ps A(\yon x, \cur Jy) \xrar{\inv{\mathfrak l}} I' \tens' \ps A(\yon x, \cur Jy) \to A(x, x) \tens' \brks{A(x, x), J(x, y)}' \xrar{\textup{ev}} J(x,y)
		\end{displaymath}
		form inverses, where the middle map is the tensor product of the unit $I' \to A(x, x)$ and the projection $\ps A(\yon x, \cur Jy) \to \brks{A(x, x), J(x, y)}'$.
		
		Next, by applying the above to a $\V$-presheaf $\hmap pAI$ we obtain an isomorphism $px \iso \ps A(\yon x, p)$. We conclude that $\ps A(\yon x, p)$ is a $\V$-object for each $x \in A$ and $p \in \ps A$ so that, by \exref{restrictions of V-profunctors}, the companion $\hmap{\yon_*}A{\ps A}$ exists. Finally to prove that $\yon$ is dense we will show that, for any $\V$-profunctor $\hmap JAB$, the cartesian cell above defines $\map{\cur J}B{\ps A}$ as a pointwise left Kan extension. Remember that, by \propref{enriched left Kan extensions in terms of weighted colimits}, we may equivalently show that its restriction along \mbox{$\hmap{J(\id, y)}AI$} defines $\cur Jy$ as the $J(\id, y)$-weighted colimit of $\yon$. Hence it suffices to prove that, for every $\V$-presheaf $p$ on $A$, the cartesian cell in the composite on the right below, that consists of the isomorphisms $px \iso \ps A(\yon x, p)$, defines $\map pI{\ps A}$ as the $\ps A(\yon, p)$-weighted colimit of $\yon$.
		\begin{displaymath}
			\begin{tikzpicture}[textbaseline]
				\matrix(m)[math35]{A & I & I \\ & \ps A & \\};
				\path[map]	(m-1-1) edge[barred] node[above] {$p$} (m-1-2)
														edge node[below left] {$\yon$} (m-2-2)
										(m-1-2) edge[barred] node[above] {$H$} (m-1-3)
										(m-1-3) edge node[below right] {$q$} (m-2-2);
				\path				(m-1-2) edge[cell, transform canvas={yshift=0.25em}] node[right] {$\phi$} (m-2-2);
			\end{tikzpicture} = \begin{tikzpicture}[textbaseline]
				\matrix(m)[math35]{A & I & I \\ & \ps A & \\};
				\path[map]	(m-1-1) edge[barred] node[above] {$p$} (m-1-2)
														edge[transform canvas={xshift=-2pt}] node[below left] {$\yon$} node[sloped, above, inner sep=6.5pt, font=\scriptsize, pos=0.55] {$\cart$} (m-2-2)
										(m-1-2) edge[barred] node[above] {$H$} (m-1-3)
														edge[ps] node[right, inner sep=1pt] {$p$} (m-2-2)
										(m-1-3) edge[transform canvas={xshift=2pt}] node[below right] {$q$} (m-2-2);
				\path				(m-1-2) edge[cell, transform canvas={shift={(1.1em,0.333em)}}] node[right] {$\phi'$} (m-2-2);
			\end{tikzpicture}
		\end{displaymath}
		To see this consider any cell $\phi$ in $\enProf{(\V, \V')}$ that is of the form as on left above; we have to prove that it factors as shown. To see this remember that $\phi$ is given by a family $\V'$-maps $\map{\phi_x}{px \tens' H}{\ps A(\yon x, q)}$ that are natural in $x \in A$. After composition with $\ps A(\yon x, q) \iso qx$ it is easy to check that this family corresponds to a cone $\nat{\phi''}{\Delta H}{\brks{p, q}'}$, under the adjunctions $px \tens' \dash \ladj \brks{px, \dash}'$. The latter in turn factors uniquely through the limiting cone $\Delta \ps A(p, q) \Rar \brks{p, q}'$ as a $\V'$-map $\map{\phi'}H{\ps A(p, q)}$. A straightforward calculation will show that $\phi'$, when regarded as a cell as in the composite on the right above, forms the unique factorisation of $\phi$ as required. This completes the construction of the yoneda embedding $\map\yon A{\ps A}$ for a $\V$-category $A$.
		
		Briefly, on the final assertion: having constructed the diagram $\brks{p, q}$ in $\V$ as described, an isomorphism between its composite with $\V \to \V'$ and $\brks{p, q}'$ can be obtained from the assumption that $\V \to \V'$ is closed monoidal (see \defref{universe enlargement}). Now use the fact that $\V \to \V'$ is full and faithful, and that it preserves limits.
	\end{proof}
	
	Closing this subsection the following proposition compares, in the case of a suitable augmented virtual equipment $\K$, the collection of good yoneda embeddings in $\K$ to the notion of \emph{good yoneda structure} on the vertical $2$-category $V(\K)$; the latter, as introduced by Weber in \cite{Weber07}, is a strengthening of the original notion of yoneda structure given by Street and Walters in \cite{Street-Walters78}. In the statement and proof of the proposition below the terminology and notation of \cite{Weber07} are used; for the notion of (co-)tabulation see \defref{tabulation}.
	\begin{proposition}
		Consider the following types of structure on an augmented virtual equipment $\K$:
		\begin{enumerate}
			\item[\textup{(E)}] a collection of good yoneda embeddings $\map{\yon_A}A{\ps A}$, one for each unital object $A$ of $\K$;
			\item[\textup{(S)}] a good yoneda structure on $V(\K)$, in the sense of Section 3 of \cite{Weber07}, in which a vertical morphism $\map fAC$ is admissible precisely if its companion $\hmap{f_*}AC$ exists in $\K$.
		\end{enumerate}
		If $\K$ has all cocartesian tabulations then a structure of type \textup{(E)} induces one of type \textup{(S)}. Conversely a structure of type \textup{(S)} induces one of type \textup{(E)} whenever $\K$ has cartesian cotabulations $\brks J$ for all horizontal morphisms $\hmap JAB$ with $A$ unital as well, in a way such that the insertion $\map{\sigma_A}A{\brks J}$, as is part of the defining cell $\sigma$ below, admits a companion $\sigma_{A*}$.
		\begin{displaymath}
			\begin{tikzpicture}[textbaseline]
				\matrix(m)[math35, column sep={1.75em,between origins}]{A & & B \\ & \brks J & \\};
				\path[map]	(m-1-1) edge[barred] node[above] {$J$} (m-1-3)
														edge[transform canvas={xshift=-2pt}] node[left] {$\sigma_A$} (m-2-2)
										(m-1-3) edge[transform canvas={xshift=2pt}] node[right] {$\sigma_B$} (m-2-2);
				\path				(m-1-2) edge[cell, transform canvas={yshift=0.25em}] node[right] {$\sigma$} (m-2-2);
			\end{tikzpicture}
		\end{displaymath}
	\end{proposition}
	That the hypothesis of $\sigma_A$ having a companion does not depend on the choice of $\sigma$ and $\brks J$ is a consequence of \lemref{companions of morphisms composed with an isomorphism}.
	\begin{proof}
		Assume that $\K$ has cocartesian tabulations, and that a good yoneda embedding $\map{\yon_A}A{\ps A}$ is given for each unital object $A$ in $\K$; we define a good yoneda structure on $V(\K)$ as follows. Firstly we define a vertical morphism $\map fAC$ to be admissible if its companion $f_*$ exists, as described above; that these form a right ideal follows from the fact that $(g \of f)_* \iso g_*(f, \id)$ (a consequence of the pasting lemma) and the fact that $\K$ has all unary restrictions. Recall that an object $A$ is called admissible whenever its identity $\id_A$ is admissible which, by \lemref{unit identities}, is equivalent to $A$ being unital. Thus, secondly, for every admissible $A$ we can take the yoneda embedding $\map{\yon_A}A{\ps A}$ the chosen one; that it is admissible follows from the fact that it is a good yoneda embedding, see \defref{yoneda embedding}.
		
		Lastly we have to supply, for any vertical morphism $\map fAC$ with both $A$ and $f$ admissible, a morphism $\map{C(f, 1)}C{\ps A}$ that is equipped with a cell $\chi^f$: we take $C(f, 1) \dfn \cur{(f_*)}$, that is given by the yoneda axiom and equipped with the cartesian cell in the composite on the left below, while we set $\chi^f$ to be the full composite.
		\begin{displaymath}
		\chi^f \dfn \begin{tikzpicture}[textbaseline]
    		\matrix(m)[math35, column sep={1.75em,between origins}]{& A & \\ A & & C \\ & \ps A & \\};
    		\path[map]	(m-1-2) edge[transform canvas={xshift=2pt}] node[right] {$f$} (m-2-3)
    								(m-2-1) edge[barred] node[below, inner sep=2pt] {$f_*$} (m-2-3)
    												edge[transform canvas={xshift=-2pt}] node[left] {$\yon_A$} (m-3-2)
    								(m-2-3)	edge[transform canvas={xshift=2pt}] node[right] {$C(f, 1)$} (m-3-2);
    		\path				(m-1-2) edge[eq, transform canvas={xshift=-2pt}] (m-2-1);
    		\draw				([yshift=-0.5em]$(m-1-2)!0.5!(m-2-2)$) node[font=\scriptsize] {$\cocart$}
    								([yshift=0.25em]$(m-2-2)!0.5!(m-3-2)$) node[font=\scriptsize] {$\cart$};
  		\end{tikzpicture} \qquad\qquad\qquad \begin{tikzpicture}[textbaseline]
				\matrix(m)[math35, column sep={1.75em,between origins}]{& A & \\ & & C \\ & \ps A & \\};
				\path[map]	(m-1-2) edge[bend left = 18] node[above right] {$f$} (m-2-3)
														edge[bend right = 45] node[left] {$\yon_A$} (m-3-2)
										(m-2-3) edge[bend left = 18] node[below right] {$g$} (m-3-2);
				\path[transform canvas={yshift=-1.625em}]	(m-1-2) edge[cell] node[right] {$\phi$} (m-2-2);
			\end{tikzpicture} = \begin{tikzpicture}[textbaseline]
    		\matrix(m)[math35, column sep={1.75em,between origins}]{& A & \\ A & & C \\ & \ps A & \\};
    		\path[map]	(m-1-2) edge[transform canvas={xshift=2pt}] node[right] {$f$} (m-2-3)
    								(m-2-1) edge[barred] node[below, inner sep=2pt] {$f_*$} (m-2-3)
    												edge[transform canvas={xshift=-2pt}] node[left] {$\yon_A$} (m-3-2)
    								(m-2-3)	edge[transform canvas={xshift=2pt}] node[right] {$g$} (m-3-2);
    		\path				(m-1-2) edge[eq, transform canvas={xshift=-2pt}] (m-2-1);
    		\path				(m-2-2) edge[cell, transform canvas={yshift=0.1em}]	node[right, inner sep=3pt] {$\phi'$} (m-3-2);
    		\draw				([yshift=-0.5em]$(m-1-2)!0.5!(m-2-2)$) node[font=\scriptsize] {$\cocart$};
  		\end{tikzpicture}
		\end{displaymath}
		Having completed giving the data for the yoneda structure on $V(\K)$; it remains to check that it satisfies two axioms. The first of these states that the cells $\chi^f$ define $f$ as an absolute left lifting of $\yon_A$ along $C(f, 1)$, which follows from \propref{restrictions and absolute left liftings}. The second axiom states that, for any cell $\phi$ as on the right above with $A$ and $f$ admissible, defines $g$ as a pointwise left Kan extension of $\yon_A$ along $f$ as soon as it defines $f$ as an absolute left lifting of $\yon_A$ along $g$. Considering the factorisation $\phi'$ of $\phi$ through $f_*$ as shown, this is a consequence of combining \propref{restrictions and absolute left liftings}, the density of $\yon_A$ (\lemref{density axioms}(b)) and \propref{pointwise left Kan extensions along companions}.
		
		For the converse assume that $\K$ has all cartesian cotabulations $\brks J$ for $\hmap JAB$ with $A$ unital, in the sense described in the statement. Consider a good yoneda structure on $V(\K)$ in which $\map fAC$ is admissible precisely if $f_*$ exists in $\K$. Hence, for every unital $A$, an admissible morphism $\map{\yon_A}A{\ps A}$ is given, which we claim to be a good yoneda embedding in our sense. Indeed, to see that it satisfies the yoneda axiom consider any $\hmap JAB$ and let $\sigma$, as in the statement, be the cell defining its cotabulation $\brks J$; we have to provide the corresponding $\map{\cur J}B{\ps A}$ as well as its defining cartesian cell (see \defref{yoneda embedding}). We set $\cur J \dfn \brks J(\sigma_A, 1) \of \sigma_B$, where $\brks J(\sigma_A, 1)$ is supplied by the good yoneda structure; it is equipped with a vertical cell $\cell{\chi^f}{\yon_A}{\brks J(\sigma_A, 1) \of \sigma_A}$ that defines $\sigma_A$ as an absolute left lifting of $\yon_A$ through $\brks J(\sigma_A, 1)$. Hence, by \propref{restrictions and absolute left liftings}, the factorisation $\chi'$ of $\chi^f$ through $\sigma_{A*}$, as in the composite below, is cartesian. Since the defining cell $\sigma$ of $\brks J$ is assumed to be cartesian it factors through $\sigma_{A*}$ as a cartesian cell $\sigma'$ as well, forming the top cell below. Applying the pasting lemma we find that the full composite is cartesian, so that we can take it to be the cartesian cell defining $\cur J = \brks J(\sigma_A, 1) \of \sigma_B$.
		\begin{displaymath}
			\begin{tikzpicture}[textbaseline]
  			\matrix(m)[math35, column sep={1.75em,between origins}]{A & & B \\ A & & \brks J \\ & \ps A & \\};
  			\path[map]	(m-1-1) edge[barred] node[above] {$J$} (m-1-3)
  									(m-1-3) edge node[right] {$\sigma_B$} (m-2-3)
  									(m-2-1) edge[barred] node[below, inner sep=2pt] {$\sigma_{A*}$} (m-2-3)
  													edge[transform canvas={xshift=-2pt}] node[left] {$\yon_A$} (m-3-2)
  									(m-2-3) edge[transform canvas={xshift=2pt}] node[below right] {$\brks J(\sigma_A, 1)$} (m-3-2);
  			\path				(m-1-1) edge[eq] (m-2-1)
  									(m-1-2) edge[cell] node[right] {$\sigma'$} (m-2-2)
  									(m-2-2) edge[cell, transform canvas={yshift=0.25em}] node[right] {$\chi'$} (m-3-2);
  		\end{tikzpicture}
		\end{displaymath}

		It remains to prove that $\yon_A$ is dense. This follows easily from Proposition~3.4(1) of \cite{Weber07}, showing that the identity cell $\id_{\yon_A}$ defines $\id_{\ps A}$ as the pointwise left Kan extension of $\yon_A$ along $\yon_A$ in $V(\K)$, together with \propref{pointwise left Kan extensions along companions} and \lemref{density axioms}(c) above. This completes the proof.
	\end{proof}
	
	\begin{example} \label{no good yoneda structure on (Cat, Cat')-Prof}
		Because the augmented virtual equipment $\inProf{\Cat, \Cat'}$ of small $2$"/profunctors does not have cocartesian tabulations (see \exref{tabulations of 2-profunctors}), we cannot apply the previous proposition to obtain a good yoneda structure on $\inProf{\Cat, \Cat'}$, in the sense of \cite{Weber07}, out of the good yoneda embeddings obtained in \propref{enriched yoneda embeddings}. In fact the following argument, given in Remark 9 of \cite{Walker18}, shows that the good yoneda embedding $\map{\yon_1}1{\ps 1}$, for the terminal $2$"/category $1$, cannot be part of a good yoneda structure on the $2$"/category $\twoCat' = V(\inProf{\Cat, \Cat'})$ of (large) $2$"/categories.
		
		We consider the natural numbers monoid $\NN$ as a one object category and we again write $1$ and $\NN$ for the functors $1 \to \CAT$ that pick out $1$ and $\NN$ respectively. Also we identify $\ps 1$ with the $2$"/category $\CAT$ of small categories, functors and transformations, so that $\yon_1 = 1$. Consider the cell $!$ in $\twoCat'$ on the left below, that is uniquely determined by its source and target, as well as its factorisation $!'$ through the unit $2$"/profunctor $I_1$. It is easily seen that $!$ defines $\id_1$ as the absolute left lifting of $1$ along $\NN$, but that it does not define $\NN$ as a left Kan extension of $1$ along $\id_1$, hence contradicting axiom (2) in Definition 3.1 of \cite{Weber07}, where the notion of good yoneda structure is defined.
		\begin{displaymath}
			\begin{tikzpicture}[textbaseline]
				\matrix(m)[math35, column sep={1.75em,between origins}]{& 1 & \\ & & 1 \\ & \CAT & \\};
				\path[map]	(m-1-2) edge[bend left = 18] node[right] {$\id_1$} (m-2-3)
														edge[bend right = 45] node[left] {$1$} (m-3-2)
										(m-2-3) edge[bend left = 18] node[right] {$\NN$} (m-3-2);
				\path[transform canvas={yshift=-1.625em}]	(m-1-2) edge[cell] node[right] {$!$} (m-2-2);
			\end{tikzpicture} = \begin{tikzpicture}[textbaseline]
    		\matrix(m)[math35, column sep={1.75em,between origins}]{& 1 & \\ 1 & & 1 \\ & \CAT & \\};
    		\path[map]	(m-1-2) edge[transform canvas={xshift=2pt}] node[right] {$\id_1$} (m-2-3)
    								(m-2-1) edge[barred] node[below, inner sep=2pt] {$I_1$} (m-2-3)
    												edge[transform canvas={xshift=-2pt}] node[left] {$1$} (m-3-2)
    								(m-2-3)	edge[transform canvas={xshift=2pt}] node[right] {$\NN$} (m-3-2);
    		\path				(m-1-2) edge[eq, transform canvas={xshift=-2pt}] (m-2-1);
    		\path				(m-2-2) edge[cell, transform canvas={yshift=0.1em}]	node[right, inner sep=3pt] {$!'$} (m-3-2);
    		\draw				([yshift=-0.5em]$(m-1-2)!0.5!(m-2-2)$) node[font=\scriptsize] {$\cocart$};
  		\end{tikzpicture} \qquad\qquad\qquad\qquad \begin{tikzpicture}[textbaseline]
				\matrix(m)[math35, column sep={2em,between origins}]{1 & & 1 \\ & \CAT & \\};
				\path[map]	(m-1-1) edge[barred] node[above] {$\NN$} (m-1-3)
														edge[transform canvas={xshift=-2pt}] node[left] {$1$} (m-2-2)
										(m-1-3) edge[transform canvas={xshift=2pt}] node[right] {$\NN$} (m-2-2);
				\path				(m-1-2) edge[cell, transform canvas={yshift=0.25em}] node[right] {$\phi$} (m-2-2);
			\end{tikzpicture}
		\end{displaymath}
		
		This example does not contradict the fact that $\map{\yon_1 = 1}1{\CAT}$ is a good yoneda embedding in our sense, because the corresponding density axiom concerns cartesian cells with vertical source $1$ (see \lemref{density axioms}) instead of absolute left liftings of $1$. It is easily checked that the factorisation $!'$ of $!$ is not cartesian: writing $\hmap{\NN}11$ for the $2$"/profunctor defined by $\NN(*, *) = \NN$ the cell $\phi$ on the right above, that is given by the obvious isomorphism $\NN(*, *) \iso \CAT(1, \NN)$, does not factor through $!'$.
	\end{example}
	
	\subsection{Restriction of presheaves}
	Vertical morphisms $\map fAC$ induce morphisms $\map{\ps f}{\ps C}{\ps A}$ between presheaf objects given by ``restriction along $f$'' as follows. Given yoneda embeddings \mbox{$\map{\yon_A}A\ps A$} and $\map{\yon_C}C\ps C$, as well as any morphism $\map fAC$ such that the companion $\hmap{(\yon_C \of f)_*}A\ps C$ exists, we set $\ps f \dfn \cur{(\yon_C \of f)_*}$ as defined by the cartesian cell below that exists by the yoneda axiom.
	\begin{equation} \label{definition of ps f}
		\begin{tikzpicture}
				\matrix(m)[math35, column sep={2em,between origins}]{A & & \ps C \\ & \ps A & \\};
				\path[map]	(m-1-1) edge[barred] node[above] {$(\yon_C \of f)_*$} (m-1-3)
														edge[transform canvas={xshift=-2pt}] node[left] {$\yon_A$} (m-2-2)
										(m-1-3) edge[transform canvas={xshift=2pt}] node[right] {$\ps f$} (m-2-2);
				\draw				([yshift=0.333em]$(m-1-2)!0.5!(m-2-2)$) node[font=\scriptsize] {$\cart$};
			\end{tikzpicture} 
	\end{equation}
	Recall that $(\yon_C \of f)_* \iso \yon_{C*}(f, \id)$ by \lemref{companion of a composite}; in an augmented virtual equipment the latter exists as soon as the yoneda embbeding $\yon_C$ is good. Notice that the cell above defines $\ps f$ uniquely up to isomorphism, because $\yon_A$ is dense.
	\begin{example}
		For any $\V$-functor $\map fAC$ in $\enProf{(\V, \V')}$ the $\V'$-functor $\map{\ps f}{\ps C}{\ps A}$ exists. It is given by precomposition with $f$; more precisely, we have $\ps f(p) \iso p(f, \id)$ for every $\V$-presheaf $p \in \ps C$.
	\end{example}
	
	Combined with \lemref{yoneda embedding full and faithful}(a), the final assertion of the following proposition is analogous to the third axiom satisfied by a Yoneda structure in the sense of \cite{Street-Walters78}, see its Section~3, while the second assertion below is an adaptation of its Proposition~12.
	\begin{proposition}
		Consider yoneda embeddings $\map{\yon_A}A{\ps A}$ and $\map{\yon_C}C{\ps C}$ as well as morphisms $\map fAC$ and $\hmap KCD$ in an augmented virtual equipment. Assume that $\yon_C$ is a good yoneda embedding (\defref{yoneda embedding}), so that $\map{\ps f}{\ps C}{\ps A}$ can be defined as in \eqref{definition of ps f}. If the companion of $f$ exists then $\ps f \of \cur K \iso \cur{K(f, \id)}$ while the pointwise left Kan extension $\cur K$ of $\yon_C$ along $K$, as obtained by the yoneda axiom (\defref{yoneda embedding}) and defined by the cartesian cell below, is preserved by $\ps f$.
		\begin{displaymath}
			\begin{tikzpicture}
				\matrix(m)[math35, column sep={1.75em,between origins}]{C & & D \\ & \ps C & \\};
				\path[map]	(m-1-1) edge[barred] node[above] {$K$} (m-1-3)
														edge[transform canvas={xshift=-2pt}] node[left] {$\yon_C$} (m-2-2)
										(m-1-3) edge[transform canvas={xshift=2pt}] node[right] {$\cur K$} (m-2-2);
				\draw				([yshift=0.333em]$(m-1-2)!0.5!(m-2-2)$) node[font=\scriptsize] {$\cart$};
			\end{tikzpicture}
	  \end{displaymath}
		
		In particular $\ps f \of \yon_C \iso \cur{f_*}$, and $\ps f \of \cur{h_*} \iso \cur{(h \of f)_*}$ for any $\map hCE$ whose companion exists.
	\end{proposition}
	\begin{proof}
		First notice that, by \lemref{companion of a composite} and \lemref{restrictions and composites}, $(\yon_C \of f)_* \iso \yon_{C*}(f, \id) \iso (f_* \hc \yon_{C*})$ and $K(f, \id) \iso f_* \hc K$; the restrictions of $\yon_{C*}$ and $K$ here exist in any augmented virtual equipment. Consider the composite on the left-hand side below, where the bottom cartesian cell is that of \eqref{definition of ps f} and where the one denote $\cart'$ is the factorisation of the cartesian cell defining $\cur{h_*}$, as supplied by the yoneda axiom (\defref{yoneda embedding}), through the cartesian cell defining $\yon_{C*}$. The cocartesian cell defines $(\yon_C \of f)_*$ as the horizontal composite of $f_*$ and $\yon_{C*}$.
		\begin{displaymath}

		\end{displaymath}
		
		Because $\yon_A$ is dense (\lemref{density axioms}) the bottom cell in the composite on the left"/hand side above defines $\ps f$ as a pointwise left Kan extension. Since its cocartesian cell is pointwise (\lemref{restrictions and composites}), by \corref{Kan extensions and composites} the full composite defines $\ps f \of \cur K$ as a left Kan extension. Likewise, by factoring the left-hand side through the cocartesian cell $(f_*, K) \Rar K(f, \id)$, we obtain a cell that defines $\ps f \of \cur K$ as the left Kan extension of $\yon_A$ along $K(f, \id)$. Since $\cur{K(f, \id)}$ too is the left Kan extension of $\yon_A$ along $K(f, \id)$, the assertion $\ps f \of \cur K \iso \cur{K(f, \id)}$ now follows from the uniqueness of Kan extensions.
		
		For the preservation of the pointwise left Kan extension $\cur K$ by $\ps f$ notice that we can decompose $\cart' = \cocart \hc \cart$ as in the top rows of the composites in the identity above: this follows immediately from the fact that $\cart'$ is the factorisation of $\cart$ through $\yon_{A*}$. Because $\yon_A$ is full and faithful (\lemref{yoneda embedding full and faithful}(a)), the first column of the composite on the right-hand side defines $\ps f \of \yon_A$ as a left Kan extension by \propref{pointwise left Kan extension along full and faithful map}. Applying the horizontal pasting lemma (\lemref{horizontal pasting lemma}) we find that its second column $\ps f \of \cart$ defines $\ps f \of \cur K$ as a pointwise left Kan extension of $\ps f \of \yon_A$ along $K$.
		
		For the final assertions take $K = I_C$ and $K = h_*$ respectively, and use that $\cur{I_C} \iso \yon_C$, by \lemref{yoneda embedding full and faithful}(a), and that $h_*(f, \id) \iso (h \of f)_*$, by \lemref{companion of a composite}.
	\end{proof}
	
	Following \cite{Street-Walters78} and \cite{Weber07}, we pause to investigate the existence of left and right adjoints to $\ps f$; in particular \propref{right adjoint to ps f} and part of \propref{left adjoint to ps f} below are adaptations of Proposition~13 of \cite{Street-Walters78} and Theorem 3.20(2) of \cite{Weber07} respectively, to the setting of augmented virtual double categories. We start with the description of the right adjoint.
	\begin{proposition} \label{right adjoint to ps f}
		In an augmented virtual equipment $\K$ let $\map{\yon_A}A{\ps A}$ and $\map{\yon_C}C{\ps C}$ be yoneda embeddings and $\map fAC$ a morphism. Assume that $\yon_C$ is a good yoneda embedding (\defref{yoneda embedding}), so that $\map{\ps f}{\ps C}{\ps A}$ can be defined as in \eqref{definition of ps f}. If the companion of $\map{\cur{f_*}}C{\ps A}$ exists then $\ps f$ has a right adjoint $f^\flat \dfn \cur{(\cur{f_*})_*}$ that is defined by the cartesian cell
		\begin{displaymath}
			\begin{tikzpicture}
				\matrix(m)[math35, column sep={1.75em,between origins}]{C & & \ps A \\ & \ps C. & \\};
				\path[map]	(m-1-1) edge[barred] node[above] {$(\cur{f_*})_*$} (m-1-3)
														edge[transform canvas={xshift=-2pt}] node[left] {$\yon_C$} (m-2-2)
										(m-1-3) edge[transform canvas={xshift=2pt}] node[right] {$f^\flat$} (m-2-2);
				\draw				([yshift=0.333em]$(m-1-2)!0.5!(m-2-2)$) node[font=\scriptsize] {$\cart$};
			\end{tikzpicture}
		\end{displaymath}
	\end{proposition}
	\begin{proof}
		Notice that $\cur{f_*} \iso \ps f \of \yon_C$ by the previous proposition. Consider the identity below where the cocartesian cell is the factorisation of the cocartesian cell defining $(\cur{f_*})_*$ as the companion of $\ps f \of \yon_C$, through $\yon_{C*}$. Notice that the cartesian cells in either side define pointwise left Kan extensions because $\yon_C$ is dense (see \lemref{density axioms}); in particular the cell $\eta$ is the factorisation of the composite on the left-hand side through the cartesian cell defining $\yon_{C*}$.
		\begin{displaymath}
			\begin{tikzpicture}[textbaseline]
				\matrix(m)[math35, column sep={1.75em,between origins}]{C & & \ps C \\ C & & \ps A \\ & \ps C. & \\};
				\path[map]	(m-1-1) edge[barred] node[above] {$\yon_{C*}$} (m-1-3)
										(m-1-3) edge[ps] node[right] {$\ps f$} (m-2-3)
										(m-2-1) edge[barred] node[below, inner sep=1.5pt] {$(\cur{f_*})_*$} (m-2-3)
														edge[transform canvas={xshift=-2pt}] node[left] {$\yon_C$} (m-3-2)
										(m-2-3) edge[transform canvas={xshift=2pt}] node[right] {$f^\flat$} (m-3-2);
				\path				(m-1-1) edge[eq] (m-2-1);
				\draw[font=\scriptsize]	($(m-1-2)!0.5!(m-2-2)$) node {$\cocart$}
																([yshift=0.25em]$(m-2-2)!0.5!(m-3-2)$) node {$\cart$};
			\end{tikzpicture} = \begin{tikzpicture}[textbaseline]
				\matrix(m)[math35, column sep={1.75em,between origins}]{C & & \ps C & \\ & \ps C & & \ps A \\ & & \ps C & \\ };
				\path[map]	(m-1-1) edge[barred] node[above] {$\yon_{C*}$} (m-1-3)
														edge[ps, transform canvas={xshift=-1pt}] node[left] {$\yon_C$} (m-2-2)
										(m-1-3) edge[ps, bend left=18] node[above right] {$\ps f$} (m-2-4)
										(m-2-4) edge[bend left=18] node[below right] {$f^\flat$} (m-3-3);
				\path				(m-1-3) edge[eq, ps, transform canvas={xshift=1pt}] (m-2-2)
										(m-2-2) edge[eq, bend right=18] (m-3-3)
										(m-1-3) edge[cell, transform canvas={yshift=-1.625em}] node[right] {$\eta$} (m-2-3);
				\draw				([yshift=0.333em]$(m-1-2)!0.5!(m-2-2)$) node[font=\scriptsize] {$\cart$};
			\end{tikzpicture}
		\end{displaymath}
		Composing both sides with the cocartesian cell defining $\yon_{C*}$ we obtain, on the left"/hand side, a composite that defines $f^\flat \of \ps f$ as a left Kan extension in $V(\K)$, by \propref{weak left Kan extensions along companions}. Likewise on the right"/hand side we obtain the composite $\id_{\yon_C} \hc \eta$, in which the identity cell $\id_{\yon_C}$ defines $\id_{\ps C}$ as the left Kan extension of $\yon_C$ along $\yon_C$. Using the horizontal pasting lemma (\lemref{horizontal pasting lemma}) we conclude that $\eta$ defines $f^\flat$ as the left Kan extension of $\id_{\ps C}$ along $\ps f$ in $V(\K)$.
		
		Using \propref{adjunctions in terms of left Kan extensions} the proof now follows by showing that $\ps f \of \eta$ defines $\ps f \of f^\flat$ as a left Kan extension. But this follows directly from the previous proposition: $\ps f$ preserves both the pointwise Kan extensions defined by the composite on the left-hand side and the cartesian cell in the right-hand side so that, by combining \propref{weak left Kan extensions along companions} and \lemref{horizontal pasting lemma} once more, it preserves $f^\flat$. 
	\end{proof}

	The following is a variation on Corollary~14 of \cite{Street-Walters78}.
	\begin{corollary}
		In an augmented virtual equipment let $\map{\yon_A}A{\ps A}$ be a good yoneda embedding. The good yoneda embedding $\map{\yon_{\ps A}}{\ps A}{\ps {\ps{A\mspace{0 mu}}}}$, if it exists, forms the right adjoint to $\map{\ps{\yon_A}}{\ps{\ps{A\mspace{0 mu}}}}{\ps A}$.
	\end{corollary}
	\begin{proof}
		The cartesian cell defining $\yon_{A*}$ also defines $\id_{\ps A}$ as a left Kan extension because $\yon_A$ is dense (see \lemref{density axioms}); it follows that $\cur{(\yon_{A*})} \iso \id_{\ps A}$. Hence $\cur{(\yon_{A*})}$ has a companion so that, by the previous proposition, $\ps{\yon_A}$ has a right adjoint
		\begin{displaymath}
			\yon_A^\flat = \cur{\bigpars{\cur{(\yon_{A*})}}_*} \iso \cur{(\id_{\ps A*})} \iso \cur{I_{\ps A}} \iso \yon_{\ps A}
		\end{displaymath}
		where the last isomorphism follows from \lemref{yoneda embedding full and faithful}(a).
	\end{proof}
		
	The following proposition, which is a variation on Lemma 3.18 of \cite{Weber07}, allows us to describe the left adjoint to $\map{\ps f}{\ps A}{\ps C}$, in the proposition that follows.
	\begin{proposition} \label{pointwise left Kan extensions along yoneda embeddings are left adjoints}
		Let $\map\yon A{\ps A}$ be a good yoneda embedding. Any pointwise weak left Kan extension $\map l{\ps A}M$ of $\map dAM$ along $\yon_*$, such that the companion $d_*$ exists, has a right adjoint $\cur{d_*}$ defined by the cartesian cell
		\begin{displaymath}
			\begin{tikzpicture}
				\matrix(m)[math35, column sep={1.75em,between origins}]{A & & M \\ & \ps A. & \\};
				\path[map]	(m-1-1) edge[barred] node[above] {$d_*$} (m-1-3)
														edge[transform canvas={xshift=-2pt}] node[left] {$\yon$} (m-2-2)
										(m-1-3) edge[transform canvas={xshift=2pt}] node[right] {$\cur{d_*}$} (m-2-2);
				\draw				([yshift=0.333em]$(m-1-2)!0.5!(m-2-2)$) node[font=\scriptsize] {$\cart$};
			\end{tikzpicture}
		\end{displaymath}
	\end{proposition}
	
	The following proposition generalises Proposition 3.3 of \cite{Day-Lack07}.
	\begin{proposition} \label{left adjoint to ps f}
		Let $\map{\yon_A}A{\ps A}$ and $\map{\yon_C}C{\ps C}$ be good yoneda embeddings and consider a morphism $\map fAC$. Assume that the pointwise left Kan extension $\map{f^\sharp}{\ps A}{\ps C}$ of $\yon_C \of f$ along $\yon_{A*}$ exists, as is implied by the existence of the right pointwise cocartesian cell
			\begin{displaymath}
				\begin{tikzpicture}
					\matrix(m)[math35]{A & \ps A \\ C & \ps A. \\};
					\path[map]	(m-1-1) edge[barred] node[above] {$\yon_{A*}$} (m-1-2)
															edge node[left] {$f$} (m-2-1)
											(m-2-1) edge[barred] node[below] {$K$} (m-2-2);
					\path				(m-1-2) edge[eq, ps] (m-2-2);
					\draw				($(m-1-1)!0.5!(m-2-2)$) node[font=\scriptsize] {$\cocart$};
				\end{tikzpicture}
			\end{displaymath}
			The right adjoint to $f^\sharp$ exists if and only if the companion $\hmap{(\yon_C \of f)_*}A{\ps C}$ does, in which case $\map{\ps f}{\ps C}{\ps A}$, as defined in \eqref{definition of ps f}, forms the right adjoint.
	\end{proposition}
	\begin{proof}
		For the first assertion take $f^\sharp \dfn \cur K$ as given by the yoneda axiom. That it forms the pointwise weak left Kan extension of $\yon_C \of f$ along $\yon_{A*}$ follows from applying the vertical pasting lemma (\lemref{vertical pasting lemma}) to the cartesian cell defining $\cur K$ composed with the cocartesian cell above.
	
		The `if'"/part of the second assertion: since the companion $(\yon_C \of f)_*$ exists, it follows from the previous proposition that $f^\sharp$ has a right adjoint that is given by $\cur{(\yon_C \of f)_*}$. For the final assertion above notice that the latter is isomorphic to $\ps f$ by definition: see \eqref{definition of ps f}.
		
		For the `only if'"/part let us denote the right adjoint to $f^\sharp$ by $\map r{\ps C}{\ps A}$; we claim that the restriction $\hmap{\ps A(\yon_A, r)}A{\ps C}$, which exists because $y_A$ is a good yoneda embedding (see \defref{yoneda embedding}), forms the companion of $y_C \of f$.
		\begin{displaymath}
			\begin{tikzpicture}
				\matrix(m)[math35, column sep={1.75em,between origins}]{A & & \ps C & \\ & \ps A & & \\ & & \ps C & \\ };
				\path[map]	(m-1-1) edge[barred] node[above] {$\ps A(\yon_A, r)$} (m-1-3)
														edge[ps, transform canvas={xshift=-1pt}] node[left] {$\yon_A$} (m-2-2)
										(m-1-3) edge[ps, transform canvas={xshift=1pt}] node[right] {$r$} (m-2-2)
										(m-2-2) edge[bend right=18] node[below left] {$f^\sharp$} (m-3-3);
				\path				(m-1-3) edge[eq, ps, bend left=45] (m-3-3)
										(m-1-3) edge[cell, transform canvas={yshift=-1.625em}] node[right] {$\eps$} (m-2-3);
				\draw				([yshift=0.25em]$(m-1-2)!0.5!(m-2-2)$) node[font=\scriptsize] {$\cart$};
			\end{tikzpicture}
		\end{displaymath}
		To see this consider the composite above, where $\eps$ is the counit of $f^\sharp \ladj r$; it is cartesian by \lemref{horizontal composition with (co-)units preserves nullary cartesian cells}. Because $f^\sharp$ is the pointwise Kan extension along the full and faithful yoneda embedding $\yon_A$ (see \lemref{yoneda embedding full and faithful}), we have $\yon_C \of f \iso f^\sharp \of \yon_A$ by \propref{pointwise left Kan extension along full and faithful map}. Composing the composite above with this isomorphism we obtain the cartesian cell that defines $\ps A(\yon_A, r)$ as the companion of $\yon_C \of f$.
	\end{proof}
	
	In the proof of \propref{pointwise left Kan extensions along yoneda embeddings are left adjoints} we will use the following lemma.
	\begin{lemma}
		Let $\map\yon A{\ps A}$ be a good yoneda embedding in an augmented virtual double category $\K$. A vertical cell
		\begin{displaymath}
			\begin{tikzpicture}
				\matrix(m)[math35, column sep={1.75em,between origins}]{& \ps A & \\ & & B \\ & \ps A & \\};
				\path[map]	(m-1-2) edge[bend left = 18] node[above right] {$j$} (m-2-3)
										(m-2-3) edge[bend left = 18] node[below right] {$l$} (m-3-2);
				\path				(m-1-2) edge[eq, bend right = 45] (m-3-2);
				\path[transform canvas={yshift=-1.625em}]	(m-1-2) edge[cell] node[right] {$\eta$} (m-2-2);
			\end{tikzpicture}
		\end{displaymath}
		defines $l$ as the left Kan extension of $\id_{\ps A}$ along $j$ in $V(\K)$ precisely if $\eta \of \yon$ defines $l$ as the left Kan extension of $\yon$ along $j \of \yon$.
	\end{lemma}
	\begin{proof}
		Consider the commuting diagram of assignments below, between collections of cells that are of the form as shown. The two bijections here are given by composition with the cartesian and weakly cocartesian cell that define $\yon_*$; that they are bijections follows from the fact that the cartesian cell defines $\id_{\ps A}$ as a left Kan extension, and the universal property of weakly cocartesian cells.
		\begin{displaymath}

		\end{displaymath}
		The proof now follows from the fact that, by definition, $\eta$ defines $l$ as a left Kan extension if the top left assignment is a bijection, while $\eta \of \yon$ does so whenever the bottom assignment is a bijection.
	\end{proof}
	
	\begin{proof}[Proof of \propref{pointwise left Kan extensions along yoneda embeddings are left adjoints}]
		Let us write $\eta$ for the cell that defines $l$ as the pointwise weak left Kan extension of $d$ along $\yon_*$, as on the left below. Notice that, because $\yon$ is full and faithful (\lemref{yoneda embedding full and faithful}), composing $\eta$ with the weakly cocartesian cell defining $\yon_*$ results in an invertible vertical cell by \propref{pointwise left Kan extension along full and faithful map}. Writing $\eta'$ for the factorisation of $\eta$ through $d_*$, it follows that the composite of the top two cells in the left-hand side of the equation below is weakly cocartesian, so that the full composite defines $\cur{d_*}$ as the left Kan extension of $\yon$ along $l \of \yon$ in $V(\K)$, by \propref{weak left Kan extensions along companions}.
		\begin{displaymath}

	  \end{displaymath}
	  Now the composite of the bottom two cells in the left-hand side above factorises uniquely through the cartesian cell defining $\yon_*$ as a vertical cell $\zeta$, as shown in the first identity. Using the companion identity for $\yon_*$ the second identity follows. We conclude that $\zeta \of \yon$ defines $\cur{d_*}$ as the left Kan extension of $\yon$ along $l \of \yon$ in $V(\K)$ which, by the lemma above, implies that $\zeta$ defines $\cur{d_*}$ as the left Kan extension of $\id_{\ps A}$ along $l$. We will shown that it is preserved by $l$, so that the proof follows from \propref{adjunctions in terms of left Kan extensions}. To this end consider the identity below, where $\cart'$ denotes the factorisation of the cartesian cell defining $\cur{d_*}$ through $\yon_*$; the identity itself follows from the first identity above.
	  \begin{displaymath}

	  \end{displaymath}
	  By the pasting lemma $\cart'$ is again cartesian; hence, because $\eta$ defines a pointwise weak left Kan extension, so does the composite of the bottom two cells in the left-hand side above. Remembering that $\eta' \of \cocart$ is weakly cocartesian, it follows that the full left-hand side defines $l \of \cur{d_*}$ as a left Kan extension in $V(\K)$ by \propref{weak left Kan extensions along companions}. Since $\eta \of \cocart$ in the right-hand side is invertible we conclude that $l \of \zeta \of \yon$ defines $l \of \cur{d_*}$ as a left Kan extension which, by the previous lemma, implies that $l \of \zeta$ does too, as required.
	\end{proof}
	
	\subsection{Exact paths of cells}
	In the definition below the classical notion of `carr\'e exact', as studied by Guitart \cite{Guitart80}, is generalised to the setting of augmented virtual double categories. This notion will be used throughout the remainder.
	\begin{definition} \label{left exact}
		Consider a path $\ul \phi = (\phi_1, \dotsc, \phi_n)$ of unary cells as in the composite below,  where $n \geq 1$, and let $\map d{C_0}M$ be any vertical morphism. The path $\ul \phi$ is called \emph{(weak) left $d$-exact} if for any nullary cell $\eta$ as below, that defines $l$ as the (weak) left Kan extension of $d$ along $(K_1, \dotsc, K_n)$, the full composite defines $l \of f_n$ as the (weak) left Kan extension of $d \of f_0$ along $(J_{11}, \dotsc, J_{nm_n})$.
		\begin{displaymath}
			\begin{tikzpicture}
				\matrix(m)[math35, column sep={4em,between origins}]
					{	A_{10} & A_{11} & A_{1m'_1} & A_{1m_1} &[4em] A_{n0} & A_{n1} & A_{nm'_n} & A_{nm_n} \\
						C_0 & & & C_1 & C_{n'} & & & C_n \\};
				\draw				([yshift=-3.25em]$(m-2-1)!0.5!(m-2-8)$) node (M) {$M$};
				\path[map]	(m-1-1) edge[barred] node[above] {$J_{11}$} (m-1-2)
														edge node[left] {$f_0$} (m-2-1)
										(m-1-3) edge[barred] node[above] {$J_{1m_1}$} (m-1-4)
										(m-1-4) edge node[right] {$f_1$} (m-2-4)
										(m-1-5)	edge[barred] node[above] {$J_{n1}$} (m-1-6)
														edge node[left] {$f_{n'}$} (m-2-5)
										(m-1-7) edge[barred] node[above] {$J_{nm_n}$} (m-1-8)
										(m-1-8) edge node[right] {$f_n$} (m-2-8)
										(m-2-1) edge[barred] node[below] {$K_1$} (m-2-4)
														edge[transform canvas={yshift=-2pt}] node[below left] {$d$} (M)
										(m-2-5)	edge[barred] node[below] {$K_n$} (m-2-8)
										(m-2-8)	edge[transform canvas={yshift=-2pt}] node[below right] {$l$} (M);
				\path				($(m-1-1.south)!0.5!(m-1-4.south)$) edge[cell] node[right] {$\phi_1$} ($(m-2-1.north)!0.5!(m-2-4.north)$)
										($(m-1-5.south)!0.5!(m-1-8.south)$) edge[cell] node[right] {$\phi_n$} ($(m-2-5.north)!0.5!(m-2-8.north)$)
										($(m-2-1.south)!0.5!(m-2-8.south)$) edge[cell] node[right] {$\eta$} (M);
				\draw[transform canvas={xshift=-1pt}]	($(m-1-2)!0.5!(m-1-3)$) node {$\dotsb$}
										($(m-1-6)!0.5!(m-1-7)$) node {$\dotsb$};
				\draw				($(m-1-4)!0.5!(m-2-5)$) node {$\dotsb$};
			\end{tikzpicture}
		\end{displaymath}
		If $\ul \phi$ is (weak) left $d$-exact for any $\map d{C_0}M$, where $M$ varies, then it is called \emph{(weak) left exact}.
		
		Analogously, the path $\ul \phi$ above is called \emph{pointwise} (weak) left $d$-exact if for any cell $\eta$ above, that defines $l$ a pointwise (weak) left Kan extension of $d$ along $(K_1, \dotsc, K_n)$, the full composite defines $l \of f_n$ as the pointwise (weak) left Kan extension of $d \of f_0$ along $(J_{11}, \dotsc, J_{nm_n})$. A path that is pointwise (weak) left $d$-exact for all $\map d{C_0}M$ is called \emph{pointwise} (weak) left exact.
		
		Moreover, we say that $\ul \phi$ \emph{satisfies the left Beck-Chevalley condition} if the restriction $K_n(\id, f_n)$ exists and the path $(\phi_1, \dotsc, \phi_{n'}, \phi_n')$ is right pointwise cocartesian (\defref{pointwise cocartesian path}), where $\phi_n'$ is the unique factorisation in
		\begin{displaymath}
			\phi_n = \begin{tikzpicture}[textbaseline]
				\matrix(m)[math35, column sep={4em,between origins}]{A_{10} & A_{11} & A_{nm_n'} & A_{nm_n} \\ C_{n'} & & & A_{nm_n} \\ C_{n'} & & & C_n. \\};
				\path[map]	(m-1-1) edge[barred] node[above] {$J_{11}$} (m-1-2)
														edge node[left] {$f_{n'}$} (m-2-1)
										(m-1-3) edge[barred] node[above] {$J_{nm_n}$} (m-1-4)
										(m-2-1) edge[barred] node[below] {$K_n(\id, f_n)$} (m-2-4)
										(m-2-4) edge node[right] {$f_n$} (m-3-4)
										(m-3-1) edge[barred] node[below] {$K_n$} (m-3-4);
				\draw				([xshift=-1pt]$(m-1-2)!0.5!(m-1-3)$) node {$\dotsc$}
										([yshift=-0.333em]$(m-2-1)!0.5!(m-3-4)$) node[font=\scriptsize] {$\cart$};
				\path				(m-1-4) edge[eq] (m-2-4)
										(m-2-1) edge[eq] (m-3-1)
										($(m-1-1.south)!0.5!(m-1-4.south)$) edge[cell] node[right] {$\phi_n'$} ($(m-2-1.north)!0.5!(m-2-4.north)$);
			\end{tikzpicture}
		\end{displaymath}
		A single cell $\phi$ is said to satisfy the left Beck-Chevalley condition if the single path $(\phi)$ does.
		
		We shall also use the horizontally dual condition: we say that the path $\ul \phi$ above \emph{satisfies the right Beck-Chevalley condition} if the restriction $K_1(f_0, \id)$ exists and $(\phi_1', \phi_2, \dotsc, \phi_n)$ is left pointwise cocartesian, where $\phi_1'$ is the unique factorisation of $\phi_1$ through the cartesian cell defining $K_1(f_0, \id)$.
	\end{definition}
	
	\begin{example}
		Consider a cell
		\begin{displaymath}
			\begin{tikzpicture}
				\matrix(m)[math35]{A & B \\ C & D \\};
				\path[map]	(m-1-1) edge[barred] node[above] {$J$} (m-1-2)
														edge node[left] {$f$} (m-2-1)
										(m-1-2) edge node[right] {$g$} (m-2-2)
										(m-2-1) edge[barred] node[below] {$K$} (m-2-2);
				\path[transform canvas={xshift=1.75em}]	(m-1-1) edge[cell] node[right] {$\phi$} (m-2-1);
			\end{tikzpicture}
		\end{displaymath}
		in the augmented virtual equipment $\enProf{(\Set, \Set')}$ of (unenriched) profunctors, and assume that $C$ is small. It follows that both the conjoint $f^*$ and the pointwise composite $(f^* \hc J)$ exist; see \exref{restrictions of V-profunctors} and \exref{horizontal composites in (V, V')-Prof}. The latter forms the extension of $J$ along $f$ by \corref{extensions and composites}, whose defining cocartesian cell is right pointwise, so that $\phi$ satisfies the left Beck-Chevalley condition precisely if its factorisation $\cell{\phi''}{(f^* \hc J)}{K(\id, g)}$, that is induced by the cell
		\begin{displaymath}
			(f^*, J) \Rar K(\id, g) \colon (x \xrar s fy, y \xrar u z) \mapsto (x \xrar{s} fy \xrar{\phi u} gz),
		\end{displaymath}
		is invertible. If $J = j_*$ and $K = k_*$ for functors $\map jAB$ and $\map kCD$, so that $\phi$ corresponds to a transformation $\nat\phi{k \of f}{g \of j}$, this recovers the definition of exactness given in \cite{Guitart80}.
	\end{example}
	
	The Beck-Chevalley condition is preserved under concatenation, as follows. In particular, a path $(\phi_1, \dotsc, \phi_n)$ satisfies the left Beck-Chevalley condition whenever each of the cells $\phi_i$ do.
	\begin{lemma} \label{concatenation of paths satisfying the Beck-Chevalley condition}
		Consider composable paths of cells $\ul \phi = (\phi_1, \dotsc, \phi_n)$ and $\ul \psi = (\psi_1, \dotsc, \psi_n)$. If they both satisfy the left Beck-Chevalley condition then so does their concatenation $\ul \phi \conc \ul \psi = (\phi_1, \dotsc, \phi_n, \psi_1, \dotsc, \psi_m)$.
	\end{lemma}
	\begin{proof}
		That $\ul \psi$ satisfies the left Beck-Chevalley condition means that $(\psi_1, \dotsc, \psi_m')$ is right pointwise cocartesian, where $\psi_m'$ is the unique factorisation of $\psi_m$ through the restriction of its horizontal target along its vertical target, as in the definition above. We have to show that the concatenation $(\phi_1, \dotsc, \phi_n, \psi_1, \dotsc, \psi_n')$ is again right pointwise cocartesian. To see this we substitute $\phi_n = \cart \of \phi_n'$, again as in the definition above, thus obtaining the composite of paths that is drawn schematically below.
		\begin{displaymath}
			\begin{tikzpicture}[x=1.5em, y=1.5em, font=\scriptsize]
				\draw	(1,2) -- (0,2) -- (0,1) -- (3,1) -- (3,2) -- (2,2) (7,2) -- (6,2) -- (6,0) -- (12,0) -- (12,1) -- (11,1) (8,2) -- (9,2) -- (9,0) (6,1) -- (10,1) (16,1) -- (15,1) -- (15,0) -- (18,0) -- (18,1) -- (17,1);
				\draw	(1.5,1.5) node {$\phi_1$}
							(7.5,1.5) node {$\phi_n'$}
							(7.5,0.5) node {$\cart$}
							(10.5,0.5) node {$\psi_1$}
							(16.5,0.5) node {$\psi_m'$}
							(1.5,2) node {$\dotsb$}
							(7.5,2) node {$\dotsb$}
							(10.5,1) node {$\dotsb$}
							(16.5,1) node {$\dotsb$}
							(4.5,1.5) node[font=] {$\dotsb$}
							(13.5,0.5) node[font=] {$\dotsb$};
			\end{tikzpicture}
		\end{displaymath}
		It is clear form the definition of cocartesian paths (\defref{cocartesian paths}) that the bottom row here is right pointwise cocartesian, because $(\psi_1, \dotsc, \psi_m')$ is so. Clearly also the top row is right pointwise cocartesian, so that the full path is right-cocartesian by the pasting lemma \lemref{pasting lemma for right pointwise cocartesian paths}, as required.
	\end{proof}
	
	The following generalises Theorem 1.7 of \cite{Guitart80}.
	\begin{proposition} \label{Beck-Chevalley}
		For the path $\ul \phi = (\phi_1, \dotsc, \phi_n)$ of \defref{left exact} the implication $\textup{(a)} \Rightarrow \textup{(b)}$ holds for the following conditions:
		\begin{enumerate}[label=\textup{(\alph*)}]
			\item $\ul \phi$ satisfies the left Beck-Chevalley condition;
			\item $\ul \phi$ is pointwise weak left exact as well as pointwise left exact.
		\end{enumerate}
		Moreover, in the case that $n = 1$ while both the restriction $K_1(\id, f_1)$ and the yoneda embedding $\map\yon{C_0}{\ps{C_0}}$ exist, consider the further condition
		\begin{enumerate}[label=\textup{(c)}]
			\item $\phi_1$ is pointwise left $\yon$-exact.
		\end{enumerate}
		Then the conditions \textup{(a)} and \textup{(c)} are equivalent whenever the right pointwise cocartesian cell below exists, while \textup{(b)} and \textup{(c)} are equivalent as soon as both the conjoint $\hmap{f_0^*}{C_0}{A_{10}}$ and the right pointwise composite $(f_0^* \hc J_{11} \hc \dotsb \hc J_{1m_1})$ exist.
		\begin{displaymath}
			\begin{tikzpicture}
				\matrix(m)[math35, column sep={4em,between origins}]{A_{10} & A_{11} & A_{1m_1'} & A_{1m_1} \\ C_0 & & & A_{1m_1} \\};
				\path[map]	(m-1-1) edge[barred] node[above] {$J_{11}$} (m-1-2)
														edge node[left] {$f_0$} (m-2-1)
										(m-1-3) edge[barred] node[above] {$J_{1m_1}$} (m-1-4)
										(m-2-1) edge[barred] node[below] {$H$} (m-2-4);
				\path				(m-1-4) edge[eq] (m-2-4);
				\draw				([xshift=-1pt]$(m-1-2)!0.5!(m-1-3)$) node {$\dotsb$}
										($(m-1-1)!0.5!(m-2-4)$) node[font=\scriptsize] {$\cocart$};
			\end{tikzpicture}
		\end{displaymath}
	\end{proposition}
	\begin{proof}
		To show (a) $\Rightarrow$ (b), consider a nullary cell $\cell\eta{(K_1, \dotsc, K_n)}M$ that defines the pointwise (weak) left Kan extension of a vertical morphism $\map d{C_0}M$, as in \defref{left exact} above. Writing $\phi_n = \cart \of \phi_n'$ as before, it follows that $\eta \of (\id, \dotsc, \cart)$ defines a pointwise (weak) left Kan extension by \lemref{properties of pointwise left Kan extensions}(b), so that the full composite $\eta \of (\id, \dotsc, \cart) \of (\phi_1, \dotsc, \phi_n') = \eta \of (\phi_1, \dotsc, \phi_n)$ defines a pointwise (weak) left Kan extension by the vertical pasting lemma for Kan extensions (\lemref{vertical pasting lemma}), showing that (b) holds.
		
		Restricting to $n = 1$, it is clear that (b) $\Rightarrow$ (c), so that proving (c)~$\Rightarrow$~(a) and (c)~$\Rightarrow$~(b), under the respective assumptions, suffices. To prove (a) $\Rightarrow$ (c) consider the factorisation $\phi_1'$ of $\phi_1$ as given in \defref{left exact}; we have to show that it is right pointwise cocartesian. We write $\phi_1''$ for its further factorisation through the right pointwise cocartesian cell above, so that the identity on the left below follows. Here, on both sides, the bottom nullary cartesian cell is that given by the yoneda axiom (\defref{yoneda embedding}); remember it defines $g$ as a pointwise weak left Kan extension because $\yon$ is dense (\lemref{density axioms}). Assuming condition (c), it follows that the full composite on the left-hand side defines $l \of f_1$ as a pointwise left Kan extension.
		\begin{displaymath}

		\end{displaymath}
		Since the right-hand side of the identity above defines a pointwise left Kan extension, so does the composite of its bottom three cells, as follows from the vertical pasting lemma (\lemref{vertical pasting lemma}) applied to its right pointwise cocartesian top cell. From \lemref{left Kan extensions along yoneda embeddings} we conclude that the latter composite is cartesian so that, by the pasting lemma, the cell $\phi''_1$ is cartesian. Being a horizontal cell, it follows that $\phi''_1$ is invertible; hence $\phi'_1 = \phi''_1 \of \cocart$ is right pointwise cocartesian, as required.
		
		Finally, to prove (c) $\Rightarrow$ (b), consider the composite on the right above. Assuming that the right pointwise composite $(f_0^* \hc J_{11} \hc \dotsb \hc J_{1m_1})$ exists, we may apply (c)~$\Rightarrow$~(a) $\Rightarrow$ (b) to find that it is pointwise left exact. Using \propref{Kan extension along conjoints} we conclude that $\phi_1$ is itself pointwise left exact. This concludes the proof.
	\end{proof}
	
	\subsection{Presheaf objects as free cocompletions}
	We now turn to proving that, under mild conditions, a presheaf object $\ps M$ forms the `free cocompletion of $M$', in a sense that will be made precise. This generalises the situation of $\V$-enriched category theory, where the $\V$-category of presheaves $\ps M$ on a small $\V$-category $M$ forms the free cocompletion of $M$ under small weighted colimits; see for instance Theorem 4.51 of \cite{Kelly82}. On the other hand, for any yoneda embedding $\map\yon M{\ps M}$ in a $2$-category in the sense of \cite{Street-Walters78} and \cite{Weber07}, the conditions ensuring that $\ps M$ is cocomplete seem to be unclear.
	
	In the result below we start by characterising the cocompleteness of presheaf objects, that are defined by good yoneda embeddings, in terms of left exact cells. In particular it follows that presheaf objects are cocomplete in a strong sense, that is they admit many pointwise left Kan extensions. Remember that the yoneda embedding $\yon$ is called good if all restrictions $\ps M(\yon, d)$ exist, where $\map dA{\ps M}$; for the notion of pointwise composition see \defref{pointwise cocartesian path}.
	\begin{proposition} \label{cocompleteness of presheaf objects}
		Let $\map\yon M{\ps M}$ be a good yoneda embedding. The pointwise left Kan extension of $\map dA{\ps M}$ along $\hmap JAB$ exists precisely when a pointwise left $\yon$-exact cell of the form below does. This is the case, for instance, if $\phi$ defines $K$ as the right pointwise composite $\bigpars{\ps M(\yon, d) \hc J}$.
		\begin{displaymath}
			\begin{tikzpicture}
				\matrix(m)[math35]{ M & A & B \\ M & & B \\ };
				\path[map]	(m-1-1) edge[barred] node[above] {$\ps M(\yon, d)$} (m-1-2)
										(m-1-2) edge[barred] node[above] {$J$} (m-1-3)
										(m-2-1) edge[barred] node[below] {$K$} (m-2-3);
				\path				(m-1-1) edge[eq] (m-2-1)
										(m-1-3) edge[eq] (m-2-3)
										(m-1-2) edge[cell] node[right] {$\phi$} (m-2-2);
			\end{tikzpicture}
		\end{displaymath}
	\end{proposition}
	\begin{proof}
		For the `when'-part consider $l \dfn \cur K$ given by the yoneda axiom and defined by the cartesian cell in the left-hand side below. This cell defines $l$ as a pointwise left Kan extension because $\yon$ is dense (\lemref{density axioms}) so that, by the assumption on $\phi$, the full composite on the left-hand side again defines $l$ as a pointwise left Kan extension. Next consider the cartesian cell in the right-hand side, which exists because $\yon$ is a good yoneda embedding. This too defines a left Kan extension, again by density of $\yon$, so that the left-hand side factors through it as a cell $\eta$, as shown. Applying the horizontal pasting lemma (\lemref{horizontal pasting lemma}) we conclude that $\eta$ defines $l$ as the pointwise left Kan extension of $d$ along $J$.
		\begin{equation} \label{pointwise left Kan extension into presheaf object}
			\begin{tikzpicture}[textbaseline]
				\matrix(m)[math35]{M & A & B \\ M & & B \\ & \ps M & \\};
				\path[map]	(m-1-1) edge[barred] node[above] {$\ps M(\yon, d)$} (m-1-2)
										(m-1-2) edge[barred] node[above] {$J$} (m-1-3)
										(m-2-1) edge[barred] node[below, inner sep=2.5pt] {$K$} (m-2-3)
														edge node[below left] {$\yon$} (m-3-2)
										(m-2-3) edge node[below right] {$l$} (m-3-2);
				\path				(m-1-1) edge[eq] (m-2-1)
										(m-1-3) edge[eq] (m-2-3)
										(m-1-2) edge[cell] node[right] {$\phi$} (m-2-2);
				\draw[font=\scriptsize]	([yshift=0.25em]$(m-2-2)!0.5!(m-3-2)$) node {$\cart$};
			\end{tikzpicture} = \begin{tikzpicture}[textbaseline]
				\matrix(m)[math35]{M & A & B \\ & \ps M & \\};
				\path[map]	(m-1-1) edge[barred] node[above] {$\ps M(\yon, d)$} (m-1-2)
														edge[transform canvas={xshift=-2pt}] node[below left] {$\yon$} node[sloped, above, inner sep=6.5pt, font=\scriptsize, pos=0.55] {$\cart$} (m-2-2)
										(m-1-2) edge[barred] node[above] {$J$} (m-1-3)
														edge[ps] node[right, inner sep=1.5pt] {$d$} (m-2-2)
										(m-1-3) edge[transform canvas={xshift=2pt}] node[below right] {$l$} (m-2-2);
				\path				(m-1-2) edge[cell, transform canvas={shift={(1.1em,0.333em)}}] node[right] {$\eta$} (m-2-2);
			\end{tikzpicture}
		\end{equation}
		
		For the `precisely'-part assume that the cell $\eta$ in the right-hand side above defines $l$ as the pointwise left Kan extension of $d$ along $J$. We compose it with the cartesian cell defining the restriction $\ps M(\yon, d)$, which exists because $\yon$ is a good yoneda embedding, and factor the result through the cartesian cell defining the restriction $K \dfn \ps M(\yon, l)$, obtaining a horizontal cell $\phi$ as shown. Now the cartesian cells above define pointwise left Kan extensions because $\yon$ is dense; it follows that the full right-hand side does as well, by the horizontal pasting lemma. We conclude that composing $\phi$ with a cell that defines the pointwise left Kan extension of $\yon$ along $K$ results in a composite that again defines a pointwise left Kan extension, showing that $\phi$ is pointwise left $\yon$-exact as required. Completing the proof, the final assertion follows directly from \corref{Kan extensions and composites}.
	\end{proof}
	Notice in particular that presheaf objects $\ps N$ in augmented virtual equipments, that are defined by good yoneda embeddings, admit all pointwise left Kan extensions of the identity $\id_{\ps N}$, along any horizontal morphisms $\hmap J{\ps N}B$: indeed in this case the pointwise composite $\bigpars{\ps N(\yon, \id) \hc J}$ coincides with the restriction $J(\yon, \id)$; see \lemref{restrictions and composites}. Thus the object $M = \ps N$ statisfies the equivalent conditions of the lemma below; we will call such objects $M$ \emph{total}. This term is motivated by \cite{Street-Walters78}, where an `admissible' object $C$ in a $2$-category is called total when its yoneda embedding $\map\yon C{\ps C}$ admits a left adjoint. Moreover in \cite{Day-Street86} a $\V$"/category $M$ is called total when it admits, for all $\V$-presheaves $\map J{\op M}\V$, the $J$-weighted colimit of the identity $\id_M$; that this is equivalent to the condition (a) below, considered in $\enProf\V$, follows easily from the results of \secref{Kan extensions in (V, V')-Prof}. The equivalences (c)~$\Leftrightarrow$~(a)~$\Leftrightarrow$~(b) below are generalisations of Theorems 5.2 and 5.3 of \cite{Kelly86}, where they are proved in the case of $\enProf{(\V, \V')}$.
	\begin{lemma} \label{total objects}
		For an object $M$ the following conditions are equivalent:
		\begin{enumerate}[label=\textup{(\alph*)}]
			\item for any $\hmap JMB$ the pointwise left Kan extension of $\id_M$ along $J$ exists;
			\item for any right pointwise cocartesian cell
				\begin{displaymath}
					\begin{tikzpicture}
						\matrix(m)[math35]{A & B \\ M & B \\};
						\path[map]	(m-1-1) edge[barred] node[above] {$J$} (m-1-2)
																edge node[left] {$d$} (m-2-1)
												(m-2-1) edge[barred] node[below] {$K$} (m-2-2);
						\path				(m-1-2) edge[eq] (m-2-2);
						\draw				($(m-1-1)!0.5!(m-2-2)$) node[font=\scriptsize] {$\cocart$};
					\end{tikzpicture}
				\end{displaymath}
				the pointwise left Kan extension of $d$ along $J$ exists.
		\end{enumerate}
		Given a good yoneda embedding $\map\yon M{\ps M}$ the following condition is equivalent too:
		\begin{enumerate}
			\item[\textup{(c)}] $\yon$ admits a left adjoint $\map c{\ps M}M$.
		\end{enumerate}
	\end{lemma}
	\begin{proof}
		(a) $\Rightarrow$ (b) follows from applying the vertical pasting lemma (\lemref{vertical pasting lemma}) to the composites as on the left below, where the cell $\eta$ defines $\map lBM$ as the pointwise left Kan extension of $\id_M$ along $K$.
		
		(b) $\Rightarrow$ (a) simply follows from taking the identity on $J$ as cocartesian cell.
		
		(c) $\Rightarrow$ (a) follows from considering composites as on the second left below, where the cartesian cell is given by the yoneda axiom (\defref{yoneda embedding}) and $\cell\eps{c \of \yon}{\id_M}$ denotes the counit of $c \ladj \yon$; the latter is invertible because $\yon$ is full and faithful (see \lemref{yoneda embedding full and faithful} and the horizontal dual of \propref{adjunctions in terms of left Kan extensions}). Use the fact that the left adjoint $c$ is cocontinuous; see \propref{left adjoints preserve pointwise left Kan extensions}.
		\begin{displaymath}

		\end{displaymath}
		
		(a) $\Rightarrow$ (c). Let $\map c{\ps M}M$ denote the pointwise left Kan extension of $\id_M$ along $\yon_*$ and let $\gamma$ denotes its defining cell, as in the left-hand side of the identity on the right above; we claim that $c$ forms the left adjoint of $\yon$. To prove this, as the unit cell $\cell\zeta{\id_{\ps M}}{\yon \of c}$ we take the unique factorisation of $\yon \of \gamma$ through the cartesian cell defining $\yon_*$, as shown above; recall that the latter defines $\id_{\ps M}$ as a pointwise left Kan extension because $\yon$ is dense (\lemref{density axioms}). For the counit $\cell\delta{c \of \yon}{\id_M}$ remember that $\yon$ is full and faithful (\lemref{yoneda embedding full and faithful}) so that, by \propref{pointwise left Kan extension along full and faithful map}, the composite of $\gamma$ with the weakly cocartesian cell defining $\yon_*$ is an invertible vertical cell $\id_M \Rar c \of \yon$; we take $\delta$ to be its inverse. The triangle identity $\id_{\yon} = (\zeta \of \yon) \hc (c \of \delta)$ follows directly from the identity above, by composing it first with the weakly cocartesian cell defining $\yon_*$ and then with $\delta$.
		\begin{displaymath}

		\end{displaymath}
		To prove the second triangle identity $(c \of \zeta) \hc (\delta \of c) = \id_C$ notice that, after composition with the cartesian cell defining $\yon_*$, it coincides with the left-hand and right-hand sides of the equation above. The first identity here follows from the definition of $\zeta$, while the second one follows by precomposition with the weakly cocartesian cell $\cocart$ that defines $\yon_*$, the uniqueness of factorisations through the latter, the identity $\gamma \of \cocart = \inv\delta$ and the interchange axiom (\lemref{horizontal composition}). The second triangle identity now follows from the fact that the left-hand side above defines $c$ as a pointwise left Kan extension: indeed, its composition with the invertible cell $\inv\delta = \gamma \of \cocart$ equals $\gamma$ by the horizontal companion identity (see \lemref{companion identities lemma}), and $\gamma$ defines $c$ as a pointwise left Kan extension. This completes the proof.
	\end{proof}
	
	Having seen that presheaf objects are cocomplete in a strong sense, we now describe the way in which they form `small cocompletions'. As there is no clear notion of `smallness' for objects in a general augmented virtual double category, we regard the notion of `smallness' as variable, as follows.
	\begin{definition} \label{cocompletion}
		Let $\K$ be an augmented virtual double category. By a \emph{left extension diagram} in $\K$ we mean a span of the form $M \xlar d A \xbrar J B$. A collection $\mathcal S$ of left extension diagrams is called an \emph{ideal} if $(f \of d, J) \in \catvar S$ for all $(d, J) \in \catvar S$ and $f \in \K$ composable with $d$; given such an ideal and an object $M$ we write $\mathcal S(M) \subset \mathcal S$ for the subcollection of spans of the form $M \xlar d A \xbrar J B$. Moreover:
		\begin{enumerate}[label=-]
			\item $M$ is called \emph{$\mathcal S$-cocomplete} if, for any $(d, J) \in \mathcal S(M)$, the pointwise left Kan extension of $d$ along $J$ exists;
			\item $\map fMN$ is called \emph{$\mathcal S$-cocontinuous} if, for any $(d, J) \in \mathcal S(M)$ and any nullary cell $\eta$ that defines a morphism $\map lBM$ as the pointwise left Kan extension of $d$ along $J$, the composite $f \of \eta$ defines $f \of l$ as a pointwise left Kan extension;
			\item $\map wM{\bar M}$ is said to define $\bar M$ as the \emph{free $\mathcal S$-cocompletion} of $M$ if $\bar M$ is $\mathcal S$"/cocomplete and, for any $\mathcal S$-cocomplete $N$, the composite
			\begin{displaymath}
				V_\textup{$\mathcal S$-cocts}(\K)(\bar M, N) \subseteq V(\K)(\bar M, N) \xrar{V(\K)(w, N)} V(\K)(M, N)
			\end{displaymath}
			is an equivalence, where $V_\textup{$\mathcal S$-cocts}(\K)$ denotes the sub-$2$-category of $\mathcal S$"/cocontinuous morphisms in $V(\K)$.
		\end{enumerate}
	\end{definition}
	
	We can now state the main theorem of this subsection, using the notion of left exactness (\defref{left exact}). Afterwards some examples are given. Recall that, for condition (e) below, it suffices that the cell $\phi$ defines $K$ as the right pointwise composite $\bigpars{\ps M(\yon, d) \hc J}$ (\defref{pointwise cocartesian path}).
	\begin{theorem} \label{presheaf objects as free cocompletions}
		Let $\map\yon M{\ps M}$ be a good yoneda embedding in an augmented virtual double category $\K$ and let $\mathcal S$ be an ideal of left extension diagrams in $\K$. If
		\begin{enumerate}
			\item[\textup{(e)}] a pointwise left $\yon$-exact cell
				\begin{displaymath}
					\begin{tikzpicture}
						\matrix(m)[math35]{ M & A & B \\ M & & B \\ };
						\path[map]	(m-1-1) edge[barred] node[above] {$\ps M(\yon, d)$} (m-1-2)
												(m-1-2) edge[barred] node[above] {$J$} (m-1-3)
												(m-2-1) edge[barred] node[below] {$K$} (m-2-3);
						\path				(m-1-1) edge[eq] (m-2-1)
												(m-1-3) edge[eq] (m-2-3)
												(m-1-2) edge[cell] node[right] {$\phi$} (m-2-2);
					\end{tikzpicture}
				\end{displaymath}
				 exists for every $(d, J) \in \mathcal S(\ps M)$;
			\item[\textup{(y)}] $(f, \yon_*) \in \mathcal S$ for all $\map fMN$,
		\end{enumerate}
		then $\yon$ defines $\ps M$ as the free $\mathcal S$-cocompletion of $M$.
	\end{theorem}
	\begin{proof}
		Condition (e) ensures that $\ps M$ is $\mathcal S$-cocomplete by \propref{cocompleteness of presheaf objects}. That precomposition with $\yon$ induces an equivalence $V_\textup{$\mathcal S$-cocts}(\K)(\ps M, N) \simeq V(\K)(M, N)$, for any $\mathcal S$-cocomplete $N$, is shown in \lemref{cocompletion equivalence} below.
	\end{proof}
	
	\begin{example} \label{V-small}
		Taking $\K = \enProf{(\V, \V')}$ (\exref{(V, V')-Prof}), consider the ideal of left extension diagrams
		\begin{displaymath}
			\mathcal S = \set{(d, J) \mid \textup{$\hmap JAB$ is a $\V$-profunctor between $\V$-categories, with $A$ small}}.
		\end{displaymath}
		If $\V$ is closed monoidal, small cocomplete and small complete, and with $\tens$ preserving small colimits on both sides then, for any small $\V$-category $M$, the yoneda embedding $\map\yon M{\ps M}$ satisfies the conditions of the theorem (as follows from \exref{horizontal composites in (V, V')-Prof} and \propref{enriched yoneda embeddings}), so that the $\V$-category $\ps M$ of presheaves forms the free $\mathcal S$-cocompletion of $M$.
		
		If $\V$ is closed symmetric monoidal then, in view of the results of \secref{Kan extensions in (V, V')-Prof}, $\mathcal S$-cocompleteness coincides with the classical notion of (small) cocompleteness for $\V$-categories, in the sense of e.g.\ of Section 3.2 of \cite{Kelly82}. The theorem in this case recovers the fact that, for a small $\V$-category $M$, the $\V$-category of presheaves $\ps M$ forms the free small cocompletion of $M$; see Theorem 4.51 of \cite{Kelly82}.
	\end{example}
	\begin{example}
		To give a useful ideal $\mathcal C$ of left extension diagrams in a general augmented virtual double category $\K$, let us call a horizontal morphism $\hmap JAB$ \emph{left composable} when the right pointwise composite $(H \hc J)$ exists for any $\hmap HCA$; we set
		\begin{displaymath}
			\mathcal C = \set{(d, J) \mid \textup{$J$ is left composable}}.
		\end{displaymath}
		Clearly a good yoneda embedding $\map\yon M{\ps M}$ in $\K$ satisfies the conditions of the theorem as soon as its companion $\yon_*$ is left composable so that, in that case, $\ps M$ forms the free $\mathcal C$-cocompletion of $M$.
		
		Notice that if $\K = \enProf{(\V, \V')}$, with $\V$ as in the previous example, then the ideal $\mathcal S$ considered there is contained in $\mathcal C$, so that any $\mathcal C$-cocomplete $\V$-category is $\mathcal S$-cocomplete. Under what conditions the converse holds I do not know.
	\end{example}
	
	In the proof of \lemref{cocompletion equivalence} below the following proposition is used, which is a weak variant of \propref{pointwise left Kan extensions along yoneda embeddings are left adjoints}.
	\begin{lemma} \label{pointwise left Kan extensions along yoneda embeddings}
		Let $\map\yon M{\ps M}$ and $\mathcal S$ be as in \thmref{presheaf objects as free cocompletions}. Any pointwise left Kan extension along $\yon_*$ is $\mathcal S$-cocontinuous.
	\end{lemma}
	\begin{proof}
		Suppose that the cell $\zeta$ in the composite below defines $l$ as the pointwise left Kan extension of $e$ along $\yon_*$. We have to show that for any left extension diagram $(d, J) \in \mathcal S(\ps M)$ and any cell $\eta$, as below, that defines a morphism $\map lB{\ps M}$ as the pointwise left Kan extension of $d$ along $J$, the composite $k \of \eta$ defines $k \of l$ as a pointwise left Kan extension as well.
		\begin{displaymath}
			\begin{tikzpicture}
				\matrix(m)[math35, column sep={1.75em,between origins}]
					{ M & & A & & B \\
						& M & & \ps M & \\
						& & N & & \\ };
				\path[map]	(m-1-1) edge[barred] node[above] {$\ps M(\yon, d)$} (m-1-3)
										(m-1-3) edge[barred] node[above] {$J$} (m-1-5)
														edge[ps] node[left] {$d$} (m-2-4)
										(m-1-5) edge[ps, transform canvas={xshift=2pt}] node[right] {$l$} (m-2-4)
										(m-2-2) edge[barred] node[below, inner sep=2pt] {$\yon_*$} (m-2-4)
														edge[transform canvas={xshift=-2pt}] node[left] {$e$} (m-3-3)
										(m-2-4) edge[transform canvas={xshift=2pt}] node[right] {$k$} (m-3-3);
				\path				(m-1-1) edge[eq, transform canvas={xshift=-1pt}] (m-2-2)
										(m-2-3) edge[cell, transform canvas={yshift=0.25em}] node[right] {$\zeta$} (m-3-3)
										(m-1-4) edge[cell, transform canvas={yshift=0.25em}] node[right] {$\eta$} (m-2-4);
				\draw				($(m-1-1)!0.5!(m-2-4)$) node[font=\scriptsize] {$\cart'$};
			\end{tikzpicture}
		\end{displaymath}
		To see this first notice that, by the uniqueness of Kan extensions, we may without loss of generality assume that $\eta$ is the cell obtained in the `when'-part of \propref{cocompleteness of presheaf objects}, as the unique factorisation in the left-hand side of \eqref{pointwise left Kan extension into presheaf object}; here we use that the pointwise left $y$-exact cell $\cell\phi{\bigpars{\ps M(\yon, d), J}}K$, as considered there, exists by condition (e) of \thmref{presheaf objects as free cocompletions}. Now the full left-hand side of \eqref{pointwise left Kan extension into presheaf object}, after factorising it through the cartesian cell defining $\yon_*$, coincides with the top row in the composite above, while its right-hand side factors through $\yon_*$ as a cartesian cell (by the pasting lemma for cartesian cells), precomposed with $\phi$. Using \lemref{properties of pointwise left Kan extensions}(b) and the assumption that $\phi$ is pointwise left $\yon$-exact, we conclude that the full composite above defines $k \of l$ as a pointwise left Kan extension. Finally, because the factorisation $\cart'$ too is cartesian by the pasting lemma, we find that the first column $\zeta \of \cart'$ above defines a left Kan extension as well. Using the horizontal pasting lemma we conclude that second column, that is $k \of \eta$, also defines a pointwise left Kan extension, as required.
	\end{proof}
	
	\begin{lemma} \label{cocompletion equivalence}
		Let $\map\yon M{\ps M}$ and $\mathcal S$ be as in \thmref{presheaf objects as free cocompletions}. For any object $N$ the top leg of the diagram
		\begin{displaymath}
			\begin{tikzpicture}
				\matrix(m)[math35, column sep=3em]
					{	V(\K)(\ps M, N) & V(\K)(M, N) \\
						V_\textup{$\mathcal S$-cocts}(\K)(\ps M, N) & V(\K)(M, N)' \\ };
				\path[map]	(m-1-1) edge node[above] {$V(\K)(\yon, N)$} (m-1-2)
														edge[white] node[black, sloped, font=] {$\supseteq$} (m-2-1)
										(m-1-2) edge[white] node[sloped, black, font=] {$\supseteq$} (m-2-2)
										(m-2-1) edge[dashed] node[below] {$\eq$} (m-2-2);
			\end{tikzpicture}
		\end{displaymath}
			factors through the full subcategory $V(\K)(M, N)'$ of $V(\K)(M, N)$, that is generated by all $\map gMN$ whose pointwise left Kan extension along $\yon_*$ exists in $\K$, as an equivalence as shown.
		
		In particular, if $N$ is $\mathcal S$-cocomplete, then the factorisation above reduces to an equivalence $V_\textup{$\mathcal S$-cocts}(\K)(\ps M, N) \eq V(\K)(M, N)$.
	\end{lemma}
	\begin{proof}
		Firstly, for the final assertion, simply notice that condition (y) of \thmref{presheaf objects as free cocompletions} ensures that $V(\K)(M, N)' = V(K)(M, N)$ for all $\mathcal S$-cocomplete $N$. Next to see that the top leg of the diagram above factors consider any $\mathcal S$"/cocontinuous $\map f{\ps M}N$; we have to show that the pointwise left Kan extension of $f \of \yon$ along $\yon_*$ exists. Since $(\yon, \yon_*) \in \mathcal S$ by the same condition it clearly does: it is $f$ itself, defined by the composite on the left below.
		
		To prove that the factorisation is an equivalence we will show that it is essentially surjective and full and faithful; for the former consider any $g \in V(\K)(M, N)'$. By definition of $V(\K)(M, N)'$ the pointwise left Kan extension $\map l{\ps M}N$ of $g$ along $\yon_*$ exists; we denote its defining cell by $\eta$, as in the middle below. By \lemref{pointwise left Kan extensions along yoneda embeddings} $l$ is $\mathcal S$-cocontinuous while, by precomposing $\eta$ with the weakly cocartesian cell defining $\yon_*$, we obtain a vertical isomorphism $g \iso f \of \yon$, as follows from the fact that $\yon$ is full and faithful (\lemref{yoneda embedding full and faithful}) and \propref{pointwise left Kan extension along full and faithful map}. This shows essential surjectivity.
		
		Finally, to prove full and faithfulness, consider any vertical cell $\cell\phi{f \of \yon}{g \of \yon}$, where $f$ and $g$ are $\mathcal S$-cocontinuous. We have to show that there exists a unique vertical cell $\cell{\phi'}fg$ such that $\phi = \phi' \of \yon$.
		\begin{displaymath}

		\end{displaymath}
		Since $\yon$ is dense, the cartesian cell in the right-hand side of the identity above defines a left Kan extension, by \lemref{density axioms}. It follows that its composition with $f$ does too, by $\mathcal S$-cocompleteness of $f$ and condition (y) of \thmref{presheaf objects as free cocompletions} again, so that the composite on the left-hand side factors uniquely as a cell $\phi'$ as shown. Composing both sides with the weakly cocartesian cell that defines $\yon_*$, we conclude that $\phi'$ is unique such that $\phi = \phi' \of \yon$, as required. This completes the proof.
	\end{proof}
		
	\section{Algebras of monads}\label{algebras section}
	With our study of `ordinary' category theory within augmented virtual double categories complete, we now turn to such theory in the presence of `algebraic structures'; the archetypal example being that of monoidal structures on categories. Like in $2$"/dimensional category theory, algebraic structures in an augmented virtual double category $\K$ are defined by monads on $\K$. In fact, any monad $T$ on $\K$ induces a strict $2$"/monad $V(T)$ on the vertical $2$-category $V(\K)$, under the $2$-functor $\map V\AugVirtDblCat\twoCat$. Given a monad $T$ we will consider an augmented virtual double category $\Alg T$ of (weak) algebras of $T$, whose vertical part $V(\Alg T)$ coincides with the $2$-category $\Alg{V(T)}$ of (weak) $V(T)$-algebras in the classical sense. The notion of horizontal morphism in $\Alg T$ generalises that of  `horizontal $T$-morphism' introduced by Grandis and Par\'e in the setting of pseudo double categories \cite{Grandis-Pare04}; see also \cite{Koudenburg15a}.
	
	In closing this section we will use the characterisation of (op-)representable horizontal morphisms given in \thmref{lower star} to characterise (op-)representable horizontal $T$-morphisms.
	
	\subsection{Monads on an augmented virtual double category}
	We start with monads on augmented virtual double categories. They are simply monads in the $2$-category $\AugVirtDblCat$ as follows.
	\begin{definition} \label{monad}
		By a \emph{monad} $T$ on an augmented virtual double category $\K$ we mean a monad $T = (T, \mu, \iota)$ on $\K$ in the $2$-category $\AugVirtDblCat$. As usual, $T$ consists of an endofunctor $\map T\K\K$ equipped with \emph{multiplication} and \emph{unit} transformations $\nat\mu{T^2}T$ and $\nat\iota{\id_\K}T$, which satisfy the associativity and unit identities $T \of T \mu = T \of \mu T$ and $\mu \of T\iota = \id_T = \mu \of \iota T$.
		
		We call $T$ \emph{strong} whenever its underlying endofunctor is strong, that is preserves all horizontal cocartesian cells.
	\end{definition}
	Notice that, under the strict $2$-functor $\map V\AugVirtDblCat\twoCat$ of \propref{2-category of (fc, bl)-multicategories}, every monad $T$ on an augmented virtual double category $\K$ induces a strict $2$-monad $V(T)$ on its vertical $2$-category $V(\K)$.  Remember that most of our examples of augmented virtual double categories are obtained by considering monoids and bimodules in a virtual double category. Unsurprisingly our examples of monads are images under the strict $2$-functor $\map\Mod\VirtDblCat\AugVirtDblCat$ of \propref{Mod is a 2-functor} as well, as described below. We shall only describe three examples of monads; in the remainder their corresponding algebraic structures will be studied in detail. For other examples we refer to Section 3 of \cite{Cruttwell-Shulman10}, where several monads on virtual double categories are described; many of these can be regarded as monads on augmented virtual double categories under the equivalence of \thmref{unital fc-multicategories}. The examples of `cartesian monads' on categories $\E$ with pullbacks, that are given in Section~4.1 of \cite{Leinster04}, form another source of monads on augmented virtual double categories, by taking their images under the composite $2$-functor $\inProf{\dash} = \Mod \of \Span{\dash}$; see \exref{Span is a 2-functor} and \propref{Mod is a 2-functor}.
	
	\begin{example} \label{free strict monoidal V-category monad}
		Let $\V = (\V, \tens, I)$ be a symmetric monoidal category with an initial object $\emptyset$ preserved by $\tens$ on both sides. The \emph{`free monoid'-monad} $T$ on the virtual double category $\Mat\V$ of $\V$-matrices (\exref{enriched profunctors}) is given as follows. It maps each large set $A$ to its free monoid $TA \dfn \coprod_{n \geq 0} A^{\times n}$; likewise, on a function $\map fAC$ it is given by $Tf \dfn \coprod_{n \geq 0} f^{\times n}$, while its image of the $\V$-matrix $\hmap JAB$ is given by
		\begin{displaymath}
			(TJ)(\ul x, \ul y) = \begin{cases}
				\Tens_{i = 1}^n J(x_i, y_i) & \text{if $\lns{\ul x} = n = \lns{\ul y}$;}\\
				\emptyset	& \text{otherwise,}
			\end{cases}
		\end{displaymath}
		where $\lns{\ul x}$ and $\lns{\ul y}$ denote the lengths of $\ul x$ and $\ul y$. The components of the image $\cell{T\phi}{(TJ_1, \dotsc, TJ_n)}{TK}$ of a cell $\phi$ are given by tensor products of the components of $\phi$, after using the while the symmetry of $\V$ to put the tensor factors in the right order. The multiplication of $T$, finally, is given by concatenation of sequences on sets, while on $\V$-matrices it is induced by the associator of $\V$.
		
		Applying $\Mod$ to $T$ we obtain the \emph{`free strict monoidal $\V$-category'-monad} on $\enProf\V$, which we again denote $T$. It maps a $\V$-category $A$ to the freely generated strict monoidal $\V$-category $TA \dfn \coprod_{n \geq 0} A^{\tens n}$, while its image of a $\V$-profunctor $\hmap JAB$ consists of the $\V$-objects $(TJ)(\ul x, \ul y)$ above, on which $TA$ and $TB$ act indexwise. For further details we refer to Example 3.11 of \cite{Koudenburg15a} where, in the case of a cocomplete $\V$ whose tensor product preserves colimits on both sides, the restriction of $T$ to the double category of small $\V$-categories is described in detail.
		
		For a closed symmetric monoidal $\V'$ that is large cocomplete, so that $\enProf{\V'}$ is a double category (\exref{horizontal composites in (V, V')-Prof}), showing that $T$ is a strong monad on $\enProf{\V'}$ is straightforward. In proving this the ``Fubini theorem'' for coends is useful; see Section~2.1 of \cite{Kelly82} for the dual theorem for ends. Now, given a symmetric universe enlargement \mbox{$\V \to \V'$} with $\V$ as in the above, notice that $T$ restricts to a strong monad on $\enProf{(\V, \V')}$, whose strongness is a consequence of the reflection of cocartesian cells along the inclusion $\enProf{(\V, \V')} \to \enProf{\V'}$ (\propref{cocartesian cells in (V, V')-Prof}).
	\end{example}
	
	\begin{example} \label{free strict double category-monad}
		Let $\Gg_1$ be the indexing category $\Gg_1 = (1 \rightrightarrows 0)$, so that the functor category $\Set'^{\Gg_1}$ is that of large directed graphs. Consider the \emph{`free category'-monad} on $\Set'^{\Gg_1}$: it maps a graph $A = \bigl(\mspace{-11mu}\begin{tikzpicture}[textbaseline]
			\matrix(m)[math35, column sep=1em]{A_1 & A_0 \\};
			\path[map]	(m-1-1) edge[above, transform canvas={yshift=2pt}] node {$s$} (m-1-2)
													edge[below, transform canvas={yshift=-2pt}] node {$t$} (m-1-2);
		\end{tikzpicture}\mspace{-11mu}\bigr)$ to its free category $TA$, with the same vertices as $A$ and, as edges, (possibly empty) paths \mbox{$\ul e = (x_0 \xrar{e_1} x_1 \xrar{e_2} \dotsb \xrar{e_n} x_n)$} of edges in $A$; formally
		\begin{displaymath}
			(TA)_0 = A_0 \qquad \text{and} \qquad (TA)_1 = A_0 \amalg \coprod_{n \geq 1} \overbrace{A_1 \times_{A_0} A_1 \times_{A_0} \dotsb \times_{A_0} A_1}^{\text{$n$ times}}.
		\end{displaymath}
		Multiplication of $T$ is given by concatenation of paths, while its unit inserts edges as singleton paths.
		
		Checking that $T$ preserves the indexwise pullbacks of $\Set'^{\Gg_1}$ is straightforward; see for instance Proposition 2.3 of \cite{Dawson-Pare-Pronk06}. Hence we may apply $\inProf{\dash} = \Mod \of \Span{\dash}$ (see \exref{Span is a 2-functor} and \propref{Mod is a 2-functor}) to obtain a monad on the equipment $\enProf{\Set'}^{\Gg_1} \dfn \inProf{\Set'^{\Gg_1}}$ of large $\Gg_1$-indexed profunctors: this is the \emph{`free strict double category'-monad}, which we again denote $T$. We shall briefly describe $\enProf{\Set'}^{\Gg_1}$ and the action of $T$.
		
		A $\Gg_1$-indexed category $\A = \bigl(\mspace{-11mu}\begin{tikzpicture}[textbaseline]
			\matrix(m)[math35, column sep=1em]{\A_1 & \A_0 \\};
			\path[map]	(m-1-1) edge[above, transform canvas={yshift=2pt}] node {$\sigma$} (m-1-2)
													edge[below, transform canvas={yshift=-2pt}] node {$\tau$} (m-1-2);
		\end{tikzpicture}\mspace{-11mu}\bigr)$ will be regarded as being a virtual double category $\A$ that has $(1, 1)$-ary cells only; its objects and vertical morphisms form the category $\A_0$ while its horizontal morphisms and cells form $\A_1$. The free strict double category $T\A$ has the same objects and vertical morphisms as $\A$, while its horizontal morphisms and cells are (possible empty) paths of horizontal morphisms and cells in $\A$. A large $\Gg_1$-indexed profunctor $\hmap \J\A\B$ consists of large profunctors $\hmap{\J_0}{\A_0}{\B_0}$ and $\hmap{\J_1}{\A_1}{\B_1}$, equipped with natural transformations $\nat{\J_\sigma}{\J_1}{\J_0(\sigma_A, \sigma_B)}$ and $\nat{\J_\tau}{\J_1}{\J_0(\tau_A, \tau_B)}$. We think of the elements $u \in \J_0(A, C)$ as vertical morphisms $\map uAC$, and of those $\theta \in \J_1(J, K)$ as cells as below, where $u = \J_\sigma(\theta)$ and $v = \J_\tau(\theta)$.
		\begin{displaymath}
			\begin{tikzpicture}
				\matrix(m)[math35]{A & B \\ C & D \\};
				\path[map]	(m-1-1) edge[barred] node[above] {$J$} (m-1-2)
														edge node[left] {$u$} (m-2-1)
										(m-1-2) edge node[right] {$v$} (m-2-2)
										(m-2-1) edge[barred] node[below] {$K$} (m-2-2);
				\path[transform canvas={xshift=1.75em}]	(m-1-1) edge[cell] node[right] {$\theta$} (m-2-1);
			\end{tikzpicture}
		\end{displaymath}
		Analogous to its action on $\Gg_1$-indexed categories, $T$ maps $\hmap \J\A\B$ to the $\Gg_1$"/indexed profunctor $\hmap{T\J}{T\A}{T\B}$ that has the same vertical morphisms as $\J$ while its cells are paths of cells of $\J$.
		
		Finally, notice that $T$ restricts to the augmented virtual equipment $\enProf{(\Set, \Set')}^{\Gg_1}$ of small $\Gg_1$-indexed profunctors between large $\Gg_1$-indexed categories, that is those $\hmap\J\A\B$ with small sets $\J_0(A, C)$ and $\J_1(J, K)$ of vertical morphisms and cells. We remark that $T$ is not strong; for an example of a composite $\Gg_1$-indexed profunctor that is not preserved by $T$ see Proposition 3.25 of \cite{Koudenburg13}.
	\end{example}
	
	\subsection{Algebras of \texorpdfstring{$2$}2-monads}
	Here we recall from \cite{Street74b} several notions of weak algebra of a $2$-monad, and apply them to the vertical parts $V(T)$ of monads $T$ on augmented virtual double categories.
	\begin{definition} \label{lax algebra}
		Let $T = (T, \mu, \iota)$ be a strict $2$-monad on a $2$-category $\catvar C$.
		\begin{enumerate}[label=-]
			\item A \emph{lax $T$-algebra} $A$ is a quadruple $A = (A, a, \bar a, \tilde a)$ consisting of an object $A$ in $\catvar C$ equipped with a morphism $\map a{TA}A$, its \emph{structure morphism}, and cells $\bar a$ and $\tilde a$ as on the left below, its \emph{associator} and \emph{unitor}, that satisfy the usual coherence axioms; see for instance Section 2 of \cite{Street74b} or Section 3 of \cite{Koudenburg15a}.
			
			A lax $T$-algebra $A$ with an invertible unitor $\tilde a$ is called \emph{normal}. If both the associator and the unitor are invertible then $A$ is called a \emph{pseudo} $T$-algebra; if they are identity cells then $A$ is called \emph{strict}.
			\begin{displaymath}
				\begin{tikzpicture}[baseline]
					\matrix(m)[math35, column sep={1.75em,between origins}]{& T^2 A & \\ TA & & TA \\ & A & \\};
					\path[map]	(m-1-2) edge[bend right=18] node[above left] {$Ta$} (m-2-1)
															edge[bend left=18] node[above right] {$\mu_A$} (m-2-3)
											(m-2-1) edge[bend right=18] node[below left] {$a$} (m-3-2)
											(m-2-3) edge[bend left=18] node[below right] {$a$} (m-3-2);
					\path				(m-1-2) edge[cell, transform canvas={yshift=-1.625em}] node[right] {$\bar a$} (m-2-2);
				\end{tikzpicture} \qquad \qquad \begin{tikzpicture}[baseline]
					\matrix(m)[math35, column sep={1.75em,between origins}]{& A & \\ & & TA \\ & A & \\};
					\path[map]	(m-1-2) edge[bend left=18] node[above right] {$\iota_A$} (m-2-3)
											(m-2-3) edge[bend left=18] node[below right] {$a$} (m-3-2);
					\path				(m-1-2) edge[cell, transform canvas={yshift=-1.625em}] node[right] {$\tilde a$} (m-2-2)
											(m-1-2) edge[eq, bend right=45] (m-3-2);
				\end{tikzpicture} \qquad \qquad \begin{tikzpicture}[baseline]
					\matrix(m)[math35, column sep={1.75em,between origins}]{& TA & \\ TC & & A \\ & C & \\};
					\path[map]	(m-1-2) edge[bend right=18] node[above left] {$Tf$} (m-2-1)
															edge[bend left=18] node[above right] {$a$} (m-2-3)
											(m-2-1) edge[bend right=18] node[below left] {$c$} (m-3-2)
											(m-2-3) edge[bend left=18] node[below right] {$f$} (m-3-2);
					\path				(m-1-2) edge[cell, transform canvas={yshift=-1.625em}] node[right] {$\bar f$} (m-2-2);
				\end{tikzpicture}
			\end{displaymath}
			\item Given lax $T$-algebras $A = (A, a, \bar a, \tilde a)$ and $C = (C, c, \bar c, \tilde c)$, a \emph{lax $T$-morphism} $A \to C$ is a morphism $\map fAC$ of $\catvar C$ that is equipped with a \emph{structure cell} $\bar f$ as on the right above, which is required to satisfy an associativity and unit axiom; see \cite{Street74b} or \cite{Koudenburg15a}.
		
			Dually, in the notion of an \emph{colax $T$-morphism} $\map hAC$ the direction of the structure cell $\cell{\bar h}{h \of a}{c \of Th}$ is reversed. A (co-)lax morphism is called a \emph{pseudo $T$-morphism} if its structure cell is invertible.
			\item Given lax $T$-morphisms $f$ and $\map gAC$, a \emph{$T$-cell} $f \Rar g$ is a cell $\cell\phi fg$ in $\catvar C$ satisfying
			\begin{displaymath}
				\begin{tikzpicture}[textbaseline]
					\matrix(m)[math35, column sep={1.75em,between origins}]{& TA & \\ TC & & A \\ & C & \\};
					\path[map]	(m-1-2) edge[bend right=35] node[above left] {$Tg$} (m-2-1)
															edge[bend left=35] node[below right, inner sep=0pt, yshift=-2pt] {$Tf$} (m-2-1)
															edge[bend left=18] node[above right] {$a$} (m-2-3)
											(m-2-1) edge[bend right=18] node[below left] {$c$} (m-3-2)
											(m-2-3) edge[bend left=18] node[below right] {$f$} (m-3-2);
					\path				(m-1-1) edge[cell, transform canvas={xshift=0.4em}] node[right, inner sep=2pt] {$T\phi$} (m-2-1)
											(m-1-2) edge[cell, transform canvas={yshift=-2.4375em}] node[right] {$\bar f$} (m-2-2);
				\end{tikzpicture} = \begin{tikzpicture}[textbaseline]
					\matrix(m)[math35, column sep={1.75em,between origins}]{& TA & \\ TC & & A \\ & C & \\};
					\path[map]	(m-1-2) edge[bend right=18] node[above left] {$Tg$} (m-2-1)
															edge[bend left=18] node[above right] {$a$} (m-2-3)
											(m-2-1) edge[bend right=18] node[below left] {$c$} (m-3-2)
											(m-2-3) edge[bend right=35] node[above left, inner sep=1.5pt] {$g$} (m-3-2)
															edge[bend left=35] node[below right] {$f$} (m-3-2);
					\path				(m-1-2) edge[cell, transform canvas={yshift=-0.8125em}] node[right] {$\bar g$} (m-2-2)
											(m-2-2) edge[cell, transform canvas={xshift=0.875em}] node[right, inner sep=2.5pt] {$\phi$} (m-3-2);
				\end{tikzpicture}.
			\end{displaymath}
			Likewise a $T$-cell between colax $T$-morphisms $h$ and $\map kAC$ is a cell \mbox{$\cell\phi hk$} satisfying $\bar h \hc (c \of T\phi) = (\phi \of a) \hc \bar k$.
		\end{enumerate}
		Lax $T$-algebras, lax $T$-morphisms and the $T$-cells between them form a $2$-category denoted $\wAlg\lax\lax T$. We denote by $\wAlg{\textup{nl}}\lax T$ and $\wAlg\psd\lax T$ the sub-$2$-categories consisting of normal lax $T$-algebras and pseudo $T$-algebras. Likewise (normal) lax $T$"/algebras, or pseudo $T$-algebras, together with colax $T$-morphisms and their $T$"/cells form a hierarchy of $2$-categories $\wAlg\psd\clx T \subset \wAlg{\textup{nl}}\clx T \subset \wAlg\lax\clx T$.
		
		Writing $\co T$ for the induced strict $2$-monad on $\co{\mathcal C}$, by an \emph{colax $T$-algebra} we mean a lax $\co T$-algebra; that is the notion of colax $T$-algebra is obtained by reversing the direction of the associator and unitor cells of lax $T$-algebras. The notions of lax $T$"/morphism and colax $T$-morphism for colax $T$-algebras are defined analogously to those for lax $T$-algebras: in fact the $2$-category of colax $T$-algebras, lax $T$"/morphisms and $T$-cells is defined as $\wAlg\clx\lax T \dfn \co{\bigpars{\wAlg\lax\clx{\co T}}}$, while that of colax $T$"/algebras, colax $T$-morphism and $T$-cells is defined as $\wAlg\clx\clx T \dfn \co{\bigpars{\wAlg\lax\lax{\co T}}}$. As before we denote their sub-$2$-categories consisting of normal colax $T$-algebras by $\wAlg{\textup{no}}\lax T$ and $\wAlg{\textup{no}}\clx T$. Notice that we need not distinguish between ``pseudo lax algebras'' and ``pseudo colax algebras'', since inverting associator and unitor cells induces an isomorphism $\wAlg\psd\lax T \iso \co{\bigpars{\wAlg\psd\clx{\co T}}}$, and similar for colax morphisms.
	\end{definition}
	
	\begin{example} \label{monoidal V-categories}
		Consider the `free strict monoidal $\V$-category'-monad $T$ on $\enProf\V$ (\exref{free strict monoidal V-category monad}), and its vertical part $V(T)$ on $\enCat\V = V(\enProf\V)$. A lax $V(T)$"/algebra is a \emph{lax monoidal $\V$-category}, that is a large $\V$-category $A$ equipped with tensor product $\V$-functors
		\begin{displaymath}
			\map\oslash{TA}A\colon (x_1, \dotsc, x_n) \mapsto (x_1 \oslash \dotsb \oslash x_n),
		\end{displaymath}
		where $n \geq 0$, together with $\V$-natural \emph{unitors}\footnote{As is the custom, by a map $\map{\mathfrak i}x{(x)}$ in the $\V$-category $A$ we mean a $\V$-map $\map{\mathfrak i}I{A(x, (x))}$.} $\map{\mathfrak i}x{(x)}$ and \emph{associators}
		\begin{displaymath}
			\map{\mathfrak a}{\bigpars{(x_{11} \oslash \dotsb \oslash x_{1m_1}) \oslash \dotsb \oslash (x_{n1} \oslash \dotsb \oslash x_{nm_n})}}{(x_{11} \oslash \dotsb \oslash x_{nm_n})}
		\end{displaymath}
		satisfying the usual coherence axioms; see for instance Definition 3.1.1 of \cite{Leinster04}. In a \emph{monoidal $\V$-category} unitors and associators are invertible; in an \emph{colax monoidal $\V$-category} their direction is reversed.
		
		A lax $V(T)$-morphism $\map fAC$ between lax monoidal $\V$-categories is a \emph{lax monoidal $\V$-functor}, that is a $\V$-functor $\map fAC$ equipped with $\V$-natural \emph{compositors} $\map{f_\oslash}{(fx_1 \oslash \dotsb \oslash fx_n)}{f(x_1 \oslash \dotsb \oslash x_n)}$, that are compatible with the coherence maps of $A$ and $C$; see Definition 3.1.3 of \cite{Leinster04}. In the notion of \emph{colax monoidal $\V$-functor} the direction of the compositors is reversed. The $V(T)$-cells finally are \emph{monoidal $\V$-natural transformations} $\nat\xi fg$, whose components satisfy the coherence axiom $f_\oslash \of (\xi_{x_1} \oslash \dotsb \oslash \xi_{x_n}) = \xi_{(x_1 \oslash \dotsb \oslash x_n)} \of g_\oslash$. We remark that, in relation to the classical definitions of monoidal category/functor/transformation, the preceding (co-)lax definitions are often called \emph{unbiased} because, unlike the classical notions, they are not ``biased'' towards binary tensor products.
	\end{example}
	
	\begin{example} \label{lax double categories}
		Consider the `free strict double category'-monad $T$ on $\enProf{\Set'}^{\Gg_1}$, (\exref{free strict double category-monad}), and its vertical part $V(T)$ on the $2$-category $\Cat'^{\Gg_1}$ of large $\Gg_1$"/indexed categories. A lax $V(T)$-algebra $\A$ is a \emph{lax double category} as follows. It consists of a large $\Gg_1$-indexed category $\A$ equipped with compatible \emph{horizontal compositions} for its horizontal morphisms and cells:
		\begin{align*}
			(A_0 \xbrar{J_1} A_1, \dotsc, A_{n'} \xbrar{J_n} A_n) \quad &\mapsto \qquad\mspace{-2.5mu} A_0 \xbrar{(J_1 \hc \dotsb \hc J_n)} A_n \\
			\Biggl(\begin{tikzpicture}[textbaseline, ampersand replacement=\&]
				\matrix(m)[math35]{A_0 \& A_1 \\ C_0 \& C_1 \\};
				\path[map]	(m-1-1) edge[barred] node[above] {$j_1$} (m-1-2)
														edge node[left] {$f_0$} (m-2-1)
										(m-1-2) edge node[right] {$f_1$} (m-2-2)
										(m-2-1) edge[barred] node[below] {$K_1$} (m-2-2);
				\path[transform canvas={xshift=1.75em}]	(m-1-1) edge[cell] node[right] {$\phi_1$} (m-2-1);
			\end{tikzpicture}, \dotsc, \begin{tikzpicture}[textbaseline, ampersand replacement=\&]
				\matrix(m)[math35]{A_{n'} \& A_n \\ C_{n'} \& C_n \\};
				\path[map]	(m-1-1) edge[barred] node[above] {$J_n$} (m-1-2)
														edge node[left] {$f_{n'}$} (m-2-1)
										(m-1-2) edge node[right] {$f_n$} (m-2-2)
										(m-2-1) edge[barred] node[below] {$K_n$} (m-2-2);
				\path[transform canvas={xshift=1.75em}]	(m-1-1) edge[cell] node[right] {$\phi_n$} (m-2-1);
			\end{tikzpicture}\Biggr) \quad &\mapsto \quad \begin{tikzpicture}[textbaseline, ampersand replacement=\&]
				\matrix(m)[math35, column sep={7em,between origins}]{A_0 \& A_n \\ C_0 \& C_n. \\};
				\path[map]	(m-1-1) edge[barred] node[above] {$(J_1 \hc \dotsb \hc J_n)$} (m-1-2)
														edge node[left] {$f_0$} (m-2-1)
										(m-1-2) edge node[right] {$f_n$} (m-2-2)
										(m-2-1) edge[barred] node[below] {$(K_1 \hc \dotsb \hc K_n)$} (m-2-2);
				\path				(m-1-1) edge[cell, transform canvas={xshift=0.75em}] node[right] {$(\phi_1 \hc \dotsb \hc \phi_n)$} (m-2-1);
			\end{tikzpicture}
		\end{align*}
		The second of these is required to be functorial with respect to the vertical composition in $\A$; this condition is known as the \emph{interchange axiom}. Furthermore, $\A$ is equipped with horizontal \emph{unitor cells} $\cell{\mathfrak i}J{(J)}$, one for each horizontal morphism $\hmap JAB$, and horizontal \emph{associator cells}
		\begin{displaymath}
			\cell{\mathfrak a}{\bigpars{(J_{11} \hc \dotsb \hc J_{1m_1}) \hc \dotsb \hc (J_{n1} \hc \dotsb \hc J_{nm_n})}}{(J_{11} \hc \dotsb \hc J_{nm_n})},
		\end{displaymath}
		one for each double path $\ull J$ of horizonal morphisms. These coherence cells are required to be natural with respect to the vertical composition of cells in $\A$ and to satisfy the usual coherence axioms, analogous to those of a lax monoidal category. For each $A \in \A$ we call the horizontal composite $\hmap{(A)}AA$ of the empty path $(A)$ the \emph{(horizontal) unit} of $A$. By requiring that the associator and unitor cells be invertible we obtain the notion of `weak double category', considered in Section 5.2 of \cite{Leinster04}; notice that this an ``unbiased'' variant of the `pseudo double categories' that were considered after \propref{Mod as right pseudo-adjoint}.
		
		\emph{Lax double functors} $\map F\A{\catvar C}$ are defined analogously to lax monoidal functors: they are $\Gg_1$-indexed functors equipped with natural horizontal \emph{compositor cells}
		\begin{displaymath}
			\cell{F_\hc}{(Fj_1 \hc \dotsb \hc Fj_n)}{F(j_1 \hc \dotsb \hc j_n)}
		\end{displaymath}
		satisfying the usual coherence axioms. Likewise \emph{double transformations} $\nat\xi FG$ are $\Gg_1$"/indexed natural transformations satisfying the coherence axiom
		\begin{displaymath}
			G_\hc \of (\xi_{J_1} \hc \dotsb \hc \xi_{J_n}) = \xi_{(J_1 \hc \dotsb \hc J_n)} \of F_\hc.
		\end{displaymath}
	\end{example}
	
	\subsection{Augmented virtual double categories of \texorpdfstring{$T$}{T}"/algebras}
	We turn to the augmented virtual double categories $\wAlg vwT$ associated to a monad $T = (T, \mu, \iota)$ on an augmented virtual double category $\K$. Remember that $T$ induces a strict $2$"/monad on the vertical $2$"/category $V(\K)$ that is contained in $\K$. To start, by an \emph{(co"/)lax $T$"/algebra} we shall simply mean an (co"/)lax $V(T)$"/algebra, and by an \emph{(co"/)lax vertical $T$"/morphism} between such algebras we shall mean an (co"/)lax $V(T)$"/morphism. The following definition generalises to augmented virtual double categories the notion of `horizontal $T$"/morphism', given in \cite{Koudenburg15a}, which is itself a slight generalisation of the `algebraic arrows' that were introduced in Section 7 of \cite{Grandis-Pare04}.
	\begin{definition} \label{horizontal T-morphism}
		Let $T = (T, \mu, \iota)$ be a monad on an augmented virtual double category $\K$. Given lax $T$"/algebras $A = (A, a, \bar a, \tilde a)$ and $B = (B, b, \bar b, \tilde b)$, a \emph{horizontal $T$"/morphism} $A \slashedrightarrow B$ is a horizontal morphism $\hmap JAB$ equipped with a \emph{structure cell}
		\begin{displaymath}

		\end{displaymath}
		satisfying the associativity and unit axioms below.
		
		A horizontal $T$"/morphism $\hmap JAB$ is called \emph{(op"/)representable} whenever its underlying morphism $\hmap JAB$ is (op"/)representable in $\K$. It is said to satisfy the \emph{left Beck"/Chevalley condition} whenever its structure cell $\bar J$ satisfies the left Beck"/Chevalley condition (\defref{left exact}).
		\begin{align*}
			%
		\end{align*}
	\end{definition}
	
	\begin{example} \label{monoidal V-profunctor}
		Let $T$ be the `free strict monoidal $\V$-category'-monad on $\enProf{(\V, \V')}$, described in \exref{free strict monoidal V-category monad}. Given lax monoidal $\V'$-categories $A$ and $B$, a horizontal $T$-morphism $A \brar B$ is a \emph{monoidal $\V$-profunctor}, that is a $\V$-profunctor $\hmap JAB$ equipped with a monoidal structure given by $\V'$-maps
		\begin{displaymath}
			\map{J_\oslash}{J(x_1, y_1) \tens' \dotsb \tens' J(x_n, y_n)}{J\bigpars{(x_1 \oslash \dotsb \oslash x_n), (y_1 \oslash \dotsb \oslash y_n)}},
		\end{displaymath}
		where $x_1, \dotsc, x_n \in A$, $y_1, \dotsc, y_n \in B$ and, in the right-hand side, the tensor products of $A$ and $B$ are denoted by $\oslash$. These $\V'$-maps are required to be compatible with the actions of $TA$ and $TB$, and to satisfy associativity and unit axioms; see Example~3.23 of \cite{Koudenburg15a}.
		
		We remark that if $A$ is a monoidal $\V'$-category, that is with invertible associator and unitor, while $\V'$ is closed symmetric monoidal, so that $\V$ can be regarded as a monoidal $\V'$-category itself, then monoidal $\V'$-profunctors $\hmap JAB$ in the sense above can be identified with lax monoidal $\V'$-functors $\op A \tens' B \to \V'$ whose images are $\V$-objects.
		
		A monoidal $\V$-profunctor $\hmap JAB$ satisfies the left Beck-Chevalley condition whenever, for each $x \in A$ and $y_1, \dotsc, y_n \in B$, the family of $\V'$-maps
		\begin{displaymath}

		\end{displaymath}
		defines the $\V$-object $J\bigpars{x, (y_1 \oslash \dotsb \oslash y_n)}$ as the coend
		\begin{displaymath}
			\int^{u_1, \dotsc, u_n \in A} A\bigpars{x, (u_1 \oslash \dotsb \oslash u_n)} \tens' J(u_1, y_1) \tens' \dotsb \tens' J(u_n, y_n)
		\end{displaymath}
		in $\V'$, as follows from \defref{left exact} and \propref{cocartesian cells in (V, V')-Prof}. In the case $(\V, \V') = (\Set, \Set')$ this means that any morphism $\map px{(y_1 \oslash \dotsb \oslash y_n)}$ in $J$ factors as
		\begin{displaymath}
			%
		\end{displaymath}
		where $s \in A$ and each $q_i \in J$, while these factorisations are required to be unique up to the equivalence relation defined by the coend above.
	\end{example}
	
	\begin{example}
		Let $T$ be the `free strict double category'-monad on $\enProf{(\Set, \Set')}^{\Gg_1}$; we will call a horizontal $T$-morphism $\A \brar \B$, between large lax double categories $\A$ and $\B$, a \emph{(small) double profunctor}. It is given by a small $\Gg_1$-indexed profunctor $\hmap \J\A\B$ equipped with `horizontal compositions'
		\begin{displaymath}
			\Biggl(%
		\end{displaymath}
		one for each path $\cell{\ul \theta}{\ul J}{\ul K}$ of cells in $\J$. These compositions are required to be natural with repect to the actions of $\A$ and $\B$, and to satisfy the following coherence axioms. The associativity axioms states that the diagram below commutes in $\J$, for every double path $\cell{\ull \theta}{\ull J}{\ull K}$ of cells, while the unit axiom states that $\mathfrak i_\B \of \theta = (\theta) \of \mathfrak i_\A$, for every cell $\cell \theta JK$ in $\J$.
		\begin{displaymath}
			%
		\end{displaymath}
	\end{example}
	
	Next is the definition of $T$-cell. Notice that, when restricted to vertical cells, it coincides with that of `$V(T)$-cell' in the sense of \defref{lax algebra}. In \propref{(fc, bl)-multicategory of horizontal T-morphisms} below we will see that lax $T$-algebras, `weak' $T$-morphisms, horizontal $T$"/morphisms and the $T$-cells between them form an augmented virtual double category.
	\begin{definition} \label{T-cell}
		Let $T$ a monad on an augmented virtual double category and consider a cell
		\begin{displaymath}
			%
		\end{displaymath}
		where $f$ and $g$ are lax vertical $T$-morphisms while $\ul J$ and $\ul K$ are paths of horizontal $T$-morphisms. It is called a \emph{$T$-cell} whenever the identity
		\begin{displaymath}
			%
		\end{displaymath}
		holds, where in the left-hand side $\ul{\bar K} \dfn \bar K$ if $\ul K = (C \xbrar K D)$ and $\ul{\bar K} \dfn \id_c$ if $\ul K = (C)$; analogously the path $(\bar J_1, \dotsc, \bar J_n)$ of structure cells in the right-hand side is to be interpreted as the identity cell $\id_{A_0}$ in the case that $\ul J$ is empty.
		
		Likewise a cell $\phi$ as above, but with colax vertical $T$-morphisms $f$ and $g$, is called a \emph{$T$-cell} whenever $\bar f \hc (\ul{\bar K} \of T\phi) = \bigpars{\phi \of (\bar J_1, \dotsc, \bar J_n)} \hc \bar g$.
	\end{definition}
	
	In the case of the `free strict monoidal $\V$-category'-monad, $(1,1)$-ary $T$-cells were described in Example 3.25 of \cite{Koudenburg15a}.
	
	\begin{example}
		For a $\Gg_1$-indexed cell $\xi$ in $\enProf{(\Set, \Set')}^{\Gg_1}$, as in the top right below, where $F$ and $G$ are lax double functors and the paths $\ul \J$ and $\ul \K$ are small double profunctors, the $T$-cell axiom states the following. Consider matrices of cells as on the left, where each row $(\theta_{i1}, \dotsc, \theta_{im})$ forms a composable path of cells in $\J_i$; remember that $\xi$ consists of cells $\cell{\xi(\theta_{1k}, \dotsc, \theta_{nk})}{FJ_{0k}}{GJ_{nk}}$ in $\K$, one for each column $(\theta_{1k}, \dotsc, \theta_{nk})$. The $T$-axiom for $\xi$ means that the diagram on the bottom right commutes, for each matrix $(\theta_{ik})$; in that case we call $\xi$ a \emph{double transformation}.
		\begin{displaymath}

		\end{displaymath}
	\end{example}
	
	\begin{proposition} \label{(fc, bl)-multicategory of horizontal T-morphisms}
		Let $T$ be a monad on an augmented virtual double category $\K$ and let `weak' mean either `lax' or `colax'. The structure on $\K$ lifts to make lax $T$-algebras, weak vertical $T$-morphisms, horizontal $T$-morphisms and $T$-cells into an augmented virtual double category $\wAlg\lax wT$, such that $V(\wAlg\lax wT) = \wAlg\lax w{V(T)}$.
	\end{proposition}
	Analogous to the notation for $2$-monads we set $\wAlg\clx\lax T \dfn \co{\bigpars{\wAlg\lax\clx{\co T}}}$ and $\wAlg\clx\clx T \dfn \co{\bigpars{\wAlg\lax\lax{\co T}}}$. For $v \in \set{\clx, \lax}$ we write $\wAlg v\psd T$ for the sub"/augmented virtual double category of $\wAlg v\lax T$ (or, equivalently, $\wAlg v\clx T$) obtained by restricting to pseudo vertical $T$-morphisms. By $\lbcwAlg vT \subset \wAlg v\psd T$ we denote the sub-augmented virtual double category obtained by further restricting to horizontal $T$-morphisms that satisfy the left Beck-Chevalley condition (see \defref{horizontal T-morphism}).
	
	\begin{proof}
		We treat the case of lax morphisms; that of colax morphisms is similar. The structure on $\wAlg\lax\lax T$ is completely determined by the requirement that it is lifted from $\K$, and that it reduces to $\wAlg\lax\lax{V(T)}$. In particular the structure cell of a vertical composite $h \of f$ is given by $\overline{h \of f} \dfn (\bar h \of Tf) \hc (h \of \bar f)$, as usual. Checking that this structure on $\wAlg\lax\lax T$ is well-defined is straightforward. Indeed that the composite $\psi \of (\phi_1, \dotsc, \phi_n)$ of $T$-cells, as in \eqref{vertical composite}, is again a $T$-cell is shown by the equality (drawn schematically, leaving out all details except the shape of cells)
		\begin{align*}
			&
,
		\end{align*}
		where the first identity follows from the fact that $T$ preserves composites and the $T$-cell axiom for $\psi$, and the second from the $T$-cell axioms for $\phi_1, \dotsc, \phi_n$ and the interchange axioms (\lemref{horizontal composition}). This leaves checking that the identity cells of $\K$ form $T$-cells, which clearly is the case.
	\end{proof}
	
	The following example recovers, for a strict $2$-monad $T$, the strict double category of $T$-algebras, lax and colax $T$-morphisms, and the appropriate notion of cell between those, that is considered in Example 4.8 of \cite{Shulman11}.
	\begin{example} \label{T-quintets}
		Let $T = (T, \mu, \iota)$ be a strict $2$-monad on a $2$-category $\mathcal C$. Its image under the strict $2$-functor $\map Q\twoCat\AugVirtDblCat$ of \propref{2-functor Q} forms a monad on the strict double category $Q(\mathcal C)$ of quintets in $\mathcal C$ (see \defref{quintets} and \exref{strict double category of quintets}). Since $V \of Q = \id$, $Q(T)$"/algebras are simply $T$"/algebras, while the corresponding notions of vertical morphism coincide as well. Moreover horizontal $Q(T)$"/morphisms are precisely colax $T$"/morphisms; in fact we have $\wAlg v\clx{Q(T)} = Q(\wAlg v\clx T)$, for each $v \in \set{\clx, \lax, \psd}$, while $\wAlg v\lax{Q(T)}$ has as cells quintets $\phi$ in $\mathcal C$ as on the left below, where $\map fAC$ and $\map gBD$ are lax $T$"/morphisms and $\map jAB$ and $\map kCD$ are colax $T$"/morphisms, such that the identity on the right is satisfied. Quintets like these are called `generalised $T$-transformations' in \cite{Shulman11}; we will call them \emph{$T$-quintets}.
		\begin{displaymath}

		\end{displaymath}
		
		Horizontally dual $\wAlg v\lax{\co Q(T)} = \co Q\bigpars{\wAlg v\lax T}$ while $\wAlg v\clx{\co Q(T)}$ consists of weak $T$-algebras and both colax $T$-morphisms and lax $T$-morphisms as vertical and horizontal morphisms respectively.
	\end{example}
	
	\begin{example}
		If in the previous example $T$ is the restriction of `free strict double category'-monad, as in \exref{lax double categories}, to the $2$-category of $\Gg_1$-indexed categories, functors and transformations, then a $T$-quintet $\cell\xi{K \of F}{G \of H}$, where $F$ and $G$ are lax double functors and $H$ and $K$ are colax ones, is given by a $\Gg_1$-indexed transformation $\cell\xi{K \of F}{G \of J}$ that makes the diagrams
		\begin{displaymath}
			\begin{tikzpicture}[baseline]
				\matrix(m)[math35, column sep={12em,between origins}]
					{ K(FJ_1 \hc \dotsb \hc FJ_n) & (KFJ_1 \hc \dotsb \hc KFJ_n) \\
						KF(J_1 \hc \dotsb \hc J_n) & (GHJ_1 \hc \dotsb \hc GHJ_n) \\
						GH(J_1 \hc \dotsb \hc J_n) & G(HJ_1 \hc \dotsb \hc HJ_n) \\ };
				\path	(m-1-1) edge[cell] node[above] {$K_\hc$} (m-1-2)
											edge[cell] node[left] {$KF_\hc$} (m-2-1)
							(m-1-2) edge[cell] node[right] {$(\xi_{J_1} \hc \dotsb \hc \xi_{J_n})$} (m-2-2)
							(m-2-1) edge[cell] node[left] {$\xi_{(J_1 \hc \dotsb \hc J_n)}$} (m-3-1)
							(m-2-2) edge[cell] node[right] {$G_\hc$} (m-3-2)
							(m-3-1) edge[cell] node[below] {$GH_\hc$} (m-3-2);
			\end{tikzpicture}
		\end{displaymath}
		commute, for each path $\ul J$ of horizontal morphisms in the source of $F$ and $H$. We will call these quintets \emph{double quintets}.
		
		The (biased variant of the) notion of double quintet was introduced by Grandis and Paré in Section~2.2 of \cite{Grandis-Pare04}, without name. In fact, the sub-strict double category of $\wAlg\lax\lax{Q(T)}$ consisting of all pseudo double categories coincides with (the biased variant of) the strict double category $\mathbb D\text{bl}$ considered there, which forms an important object of study in the subsequent \cite{Grandis-Pare08} and \cite{Grandis-Pare07}. We remark that, unfortunately, when applied to $\mathbb D\text{bl}$ our notion of pointwise right Kan extension (horizontally dual to \defref{pointwise left Kan extension}) is not the right one. Indeed the right notion, which generalises those of e.g.\ restriction and tabulation in a pseudo double category (see \cite{Grandis-Pare07}), defines the pointwise right Kan extension of a lax double functor $\map DAM$ along a colax double functor $\map JAB$ as a \emph{normal} lax double functor $\map RBM$ (i.e.\ one that preserves horizontal units strictly), equipped with a universal double quintet $\nat\eps {R \of J}D$. I do not know how to (smoothly) reconcile these notions.
	\end{example}
	
	\subsection{Representable horizontal \texorpdfstring{$T$}{T}-morphisms} \label{representable horizontal T-morphisms section}
	Let $T$ be a monad on an augmented virtual double category $\K$. In this section we will characterise representable horizontal $T$-morphisms (\defref{horizontal T-morphism}) in terms of colax $T$"/morphisms. This characterisation is a consequence of that of \thmref{lower star}, which characterises representable horizontal morphisms in $\K$.
	
	In detail, we will characterise the full sub-augmented virtual double categories
	\begin{displaymath}
		\wAlg\lax\lax{\Rep(T)} \subseteq \wAlg\lax\lax T \supseteq \wAlg\lax\lax{\opRep(T)}
	\end{displaymath}
	generated by (op-)representable horizontal $T$-morphisms, where $\Rep$ and $\opRep$ denote the strict $2$-endofunctors on $\AugVirtDblCat$ of \propref{2-functor Rep}, that restrict to (op-)representable horizontal morphisms. To do so recall from \thmref{lower star} the sub-$2$-endofunctor $(Q \of V)_* \subseteq Q \of V$ on $\AugVirtDblCat$, that maps $\K$ to the sub"/augmented virtual double category $(Q \of V)_*(\K) \subseteq (Q \of V)(\K)$ consisting of all objects, all vertical morphisms, those horizontal morphisms $\hmap jAB$ that admit companions in $\K$, and all quintets between them. Horizontally dual, the sub"/$2$"/endofunctor $(\co Q \of V)^* \subseteq (\co Q \of V)$ restricts to horizontal morphisms $\hmap jAB$ that admit conjoints. We will show that the sub-augmented virtual double categories of $\wAlg\lax\lax T$ above are equivalent to augmented virtual double categories $\wAlg\lax\lax{(Q \of V)_*(T)}$ and $\wAlg\lax\lax{(\co Q \of V)^*(T)}$ of $T$-quintets respectively; see \exref{T-quintets}.
	
	\begin{theorem} \label{representable horizontal T-morphisms}
		Let $T$ be a monad on an augmented virtual double category $\K$ and let $v, w \in \set{\clx, \lax, \psd}$. Choosing, for each $\hmap jAB$ in $(Q \of V)_*(\K)$, a cartesian cell $\eps_j$ that defines the companion $j_*$ in $\K$ induces an equivalence
		\begin{displaymath}
			(\dash)_*^T\colon\wAlg vw{(Q \of V)_*(T)}\simeq\wAlg vw{\Rep(T)}
		\end{displaymath}
		of augmented virtual double categories. While it restricts to the identity on weak $T$"/algebras and weak $T$"/morphisms, $(\dash)_*^T$ maps each colax $T$"/morphism $\map{(j, \bar j)}AB$ to the horizontal $T$"/morphism $\hmap{(j_*, \bar{j_*})}AB$, where $\bar{j_*}$ is the unique factorisation in
		\begin{displaymath}
			\begin{tikzpicture}[textbaseline]
				\matrix(m)[math35, column sep={1.75em,between origins}]{& TA & & TB \\ A & & TB \\ & B & \\};
				\path[map]	(m-1-2) edge[barred] node[above] {$Tj_*$} (m-1-4)
														edge[transform canvas={xshift=-1pt}] node[left] {$Tj$} (m-2-3)
														edge[bend right = 18] node[left] {$a$} (m-2-1)
										(m-2-1) edge[bend right = 18] node[left] {$j$} (m-3-2)
										(m-2-3) edge[bend left = 18] node[right] {$b$} (m-3-2);
				\path				(m-1-2) edge[cell, transform canvas={yshift=-1.625em}] node[right] {$\bar j$} (m-2-2)
										(m-1-4) edge[eq, transform canvas={xshift=2pt}] (m-2-3)
										(m-1-3) edge[cell, transform canvas={shift={(-0.625em,0.25em)}}] node[right] {$T\eps_j$} (m-2-3);
			\end{tikzpicture} = \begin{tikzpicture}[textbaseline]
  			\matrix(m)[math35, column sep={1.75em,between origins}]{TA & & TB \\ A & & B \\ & B. & \\};
  			\path[map]	(m-1-1) edge[barred] node[above] {$Tj_*$} (m-1-3)
  													edge node[left] {$a$} (m-2-1)	
  									(m-1-3) edge node[right] {$b$} (m-2-3)
  									(m-2-1) edge[barred] node[below] {$j_*$} (m-2-3)
  													edge[transform canvas={xshift=-1pt}] node[left] {$j$} (m-3-2);
  			\path				(m-1-1) edge[cell, transform canvas={xshift=1.75em}] node[right] {$\bar{j_*}$} (m-2-1)
  									(m-2-3) edge[eq, transform canvas={xshift=1pt}] (m-3-2)
  									(m-2-2) edge[cell, transform canvas={xshift=-0.375em}] node[right] {$\eps_j$} (m-3-2);
  		\end{tikzpicture}
		\end{displaymath}
		The assignment $\bar j \mapsto \bar{j_*}$ forms a bijective correspondence between colax $T$"/morphism structures on $\map jAB$ and horizontal $T$"/morphism structures on $\hmap{j_*}AB$.
			
		Horizontally dual, choosing cartesian cells that define conjoints induces an equivalence $\wAlg vw{(\co Q \of V)^*(T)} \simeq \wAlg vw{\opRep(T)}$ while, for each $\map jBA$, lax $T$-morphism structures on $j$ correspond bijectively to horizontal $T$-morphism structures on $j_*$.
	\end{theorem}
	\begin{proof}
		By \thmref{lower star} a choice of companions $j_*$, for each horzontal morphism $\hmap jAB$ in $(Q \of V)_*(\K)$, induces an equivalence $(\dash)_*\colon (Q \of V)_*(\K) \simeq \Rep(\K)$ that restricts to the identity on objects and vertical morphisms, while it maps $\hmap jAB$ to its companion $\hmap{j_*}AB$. Define $\map{(\dash)_*^T}{\wAlg vw{(Q \of V)_*(T)}}{\wAlg vw{\Rep(T)}}$ on objects and morphisms as in the statement. To check that, for each colax $T$-morphism $\map{(j, \bar j)}AB$, its image $(j_*, \bar{j_*})$ satisfies the coherence axioms of \defref{horizontal T-morphism}, consider the invertible horizontal cell $\gamma_j\colon Tj_* \iso (Tj)_*$ that is the factorisation of $T\eps_j$ through $\eps_{Tj}$; here we use that $T$ preserves $\eps_j$ by \corref{functors preserve companions and conjoints}. Comparing the factorisation above with the action \eqref{lower star on quintets} of $(\dash)_*$ on quintets in $\K$, it follows that $\bar{j_*} = (\bar j)_* \of \gamma_j$. Similarly, by postcomposing the coherence axioms for $\bar{j_*}$ with $\eps_j$, we find that they coincide with the $(\dash)_*$ images of the coherence axioms for $\bar j$, followed by precomposition with $T\gamma_j$ in the case of the associativity axiom. We conclude that the coherence axioms for $\bar{j_*}$ are induced by those for $\bar j$. Even better, because $(\dash)_*$ is full and faithful, it follows that the assignment $\bar j \mapsto \bar{j_*}$ gives a bijection between the colax $T$-morphism structures on $j$ and the horizontal $T$-morphism structures on $j_*$, as asserted.
		
		Next, on $T$-quintets we let $(\dash)_*^T$ act simply by $\phi \mapsto \phi_*$, as in \eqref{lower star on quintets}. Similar to the argument above, postcomposing the $T$-cell axiom for $\phi_*$ with $\eps_k$, shows that it coincides with the $(\dash)_*$-image of the $T$-quintet axiom for $\phi$ after precomposing it with the invertible horizontal cells $\gamma_{j_1}, \dotsc, \gamma_{j_n}$. We conclude that the $T$-quintet axiom for $\phi$ and the $T$-cell axiom for $\phi_*$ are equivalent, showing that the assignment $\phi \mapsto \phi_*$ is well-defined. As before, combining this with the fact that $(\dash)_*$ is full and faithful we conclude that $\map{(\dash)_*^T}{\wAlg vw{(Q \of V)_*(T)}}{\wAlg vw{\Rep(T)}}$, whose definition is now complete, is again full and faithful.
		
		To prove that $(\dash)_*^T$ is an equivalence it remains to show that it is essentially surjective, by \propref{equivalences}. To see this consider any horizontal $T$-morphism $\hmap{(J, \bar J)}AB$, with $J$ represented by $\map jAB$. By factoring $\eps_j$ through the cartesian cell that defines $J$ as a companion of $j$ we obtain a horizontal isomorphism $\delta\colon j_* \iso J$. Clearly the composite $\bar{j_*} \dfn \brks{Tj_* \xRar{T\delta} TJ \xRar{\bar J} J \xRar{\inv\delta} j_*}$ forms a horizontal $T$-morphism structure on $j_*$ so that, as we have already seen, it is the image of $(j, \bar j)$ under $(\dash)_*^T$ for some colax $T$-morphism structure cell $\bar j$. Finally, by definition of $\bar{j_*}$, the cell $\delta$ forms a horizontal $T$-cell $(j_*, \bar{j_*}) \iso (J, \bar J)$, showing that $(\dash)_*^T$ is essentially surjective. This completes the proof.
	\end{proof}
	\begin{remark}
		We remark that the theorem can be proved by formal means as well, as follows. First show that the assignment \mbox{$T \mapsto \wAlg\lax\lax T$} extends to a strict $2$"/functor $\map{\wAlg\lax\lax{(\dash)}}\Mnd\AugVirtDblCat$, where $\Mnd$ denotes the $2$"/category of monads on augmented virtual double categories, the lax morphisms between them and the cells between those, analogous to the corresponding notions for monads on $2$"/categories; see \cite{Street72} or Section~6.1 of \cite{Leinster04}. The pseudonaturality of \mbox{$(\dash)_*\colon (Q \of V)_*(\K) \simeq \Rep(\K)$} (\thmref{lower star}) with respect to $T$ can then be used to lift $(\dash)_*$ to an equivalence $(Q \of V)_*(T) \simeq \Rep(T)$ in $\Mnd$; apply $\wAlg\lax\lax{(\dash)}$ to obtain $\wAlg\lax\lax{(Q \of V)_*(T)} \simeq \wAlg\lax\lax{\Rep(T)}$.
	\end{remark}
	
	\section{Creativity of the forgetful functors \texorpdfstring{$\map U{\Alg T}\K$}{U: T-Alg → K}}\label{creativity section}
	The classical notion of the creation of limits by ordinary functors can be generalised to the creation of ``universal constructions'' by functors between augmented virtual double categories, where by universal constructions we mean restrictions, horizontal composites, Kan extensions, et cetera. In this section, after making this notion precise, we describe the creation of universal constructions by the forgetful functors $\map U{\wAlg vwT}\K$, where $T$ is a monad on an augmented virtual double category $\K$ and $v, w \in \set{\clx, \lax, \psd}$.
	
	In any augmented virtual double category $\K$ consider a ``formal universal construction'' $\mathbf U$, whether it exists or not, that is defined by a universal cell. For example $\mathbf U$ can be one of
	\begin{enumerate}[label=-]
		\item ``the restriction of $\hmap{\ul K}CD$ along $\map fAC$ and $\map gBD$'';
		\item ``the horizontal composite of $A_0 \xbrar{J_1} A_1$, \dots, $A_{n'} \xbrar{J_n} A_n$'';
		\item ``the left Kan extension of $\map d{A_0}M$ along $A_0 \xbrar{J_1} A_1$, \dots, $A_{n'} \xbrar{J_n} A_n$''.
	\end{enumerate}
	Hence, for any universal construction $\mathbf U$ in $\K$ as above, as well as a cell $\phi$ in $\K$, we have a proposition ``$\phi$ defines $\mathbf U$ in $\K$'' that either holds or fails. Furthermore given any functor $\map F\K\L$, we can assign to each formal universal construction $\mathbf U$ in $\K$ its ``image'' $F\brks{\mathbf U}$, by replacing each reference to a $\K$-morphism in $\mathbf U$ by its image under $F$. For example, the image of the first universal property above is ``the restriction of $\hmap{F\ul K}{FC}{FD}$ along $\map{Ff}{FA}{FC}$ and $\map{Fg}{FB}{FD}$''. The following definition is a direct translation of the classical notion of limit-creating functor, see e.g.\ Section~V.1 of \cite{MacLane98}.
	\begin{definition} \label{creation}
		Let $\map F\K\L$ be a functor of augmented virtual double categories, and consider a formal universal construction $\mathbf U$ in $\K$. We say that $F$ \emph{creates $\mathbf U$} if for any cell $\psi$ that defines $F\brks{\mathbf U}$ in $\L$ the following holds:
		\begin{enumerate}[label=-]
			\item there exists a unique cell $\phi$ in $\K$ such that $F\phi = \psi$;
			\item $\phi$ defines $\mathbf U$ in $\K$.
		\end{enumerate}
	\end{definition}
	We remark that creation of $\mathbf U$ by $F$ implies that $F$ \emph{reflects} $\mathbf U$, that is for all cells $\phi$ in $\K$ we have: if $F\phi$ defines $F\brks{\mathbf U}$ in $\L$ then $\phi$ defines $\mathbf U$ in $\K$. Notice that creation of $\mathbf U$ by $F$, together with the existence of $F\brks{\mathbf U}$ in $\L$, implies  \emph{preservation} of $\mathbf U$ by $F$: if a cell $\phi$ defines $\mathbf U$ in $\K$ then $F\phi$ defines $F\brks{\mathbf U}$ in $\L$.
	
	\subsection{Restrictions}
	We start with the creation of restrictions in $\wAlg vwT$. The following result generalises Proposition 4.1 of \cite{Koudenburg15a}.
	\begin{proposition} \label{creating restrictions}
		For a monad $T$ on an augmented virtual double category $\K$ the following hold for each $v \in \set{\clx, \lax, \psd}$:
		\begin{enumerate}[label=\textup{(\alph*)}]
			\item $\map U{\wAlg v\clx T}\K$ creates restrictions $\ul K(f, g)$ where $g$ is a pseudo $T$"/morphism;
			\item $\map U{\wAlg v\lax T}\K$ creates restrictions $\ul K(f, g)$ where $f$ is a pseudo $T$"/morphism;
			\item $\map U{\wAlg v\psd T}\K$ creates all restrictions.
		\end{enumerate}
	\end{proposition}
	\begin{proof}[Sketch of the proof]
		We will sketch the proof of part (b) in the case $v = \lax$. Consider a pseudo $T$-morphism $\map fAC$, a lax vertical $T$-morphism $\map gBD$ and a path $\ul K$ of horizontal $T$-morphisms of length $\leq 1$, and assume given a cartesian cell $\phi$ in $\K$ as in the composite on the left-hand side below. We obtain a structure cell $\bar J$ on the horizontal source $\hmap JAB$ of $\phi$ by taking the unique factorisation in
		\begin{displaymath}
			\begin{tikzpicture}[textbaseline]
				\matrix(m)[math35, column sep={1.75em,between origins}]
					{ & TA & & & & TB & \\
						A & & TC & & TD & & B \\
						& C & & & & D & \\ };
				\path[map]	(m-1-2) edge[barred] node[above] {$TJ$} (m-1-6)
														edge node[left] {$a$} (m-2-1)
														edge node[right] {$Tf$} (m-2-3)
										(m-1-6) edge node[right, inner sep=2pt] {$Tg$} (m-2-5)
														edge node[right] {$b$} (m-2-7)
										(m-2-1) edge node[left] {$f$} (m-3-2)
										(m-2-3) edge[barred] node[below] {$T\ul K$} (m-2-5)
														edge node[right] {$c$} (m-3-2)
										(m-2-5) edge node[left] {$d$} (m-3-6)
										(m-2-7) edge node[right] {$g$} (m-3-6)
										(m-3-2) edge[barred] node[below] {$\ul K$} (m-3-6);
				\path				(m-1-2) edge[cell, transform canvas={shift={(-0.8em,-1.625em)}}] node[right] {$\inv{\bar f}$} (m-2-2)
										(m-1-4) edge[cell] node[right] {$T\phi$} (m-2-4)
										(m-2-4) edge[cell, transform canvas={yshift=-0.25em}] node[right] {$\ul{\bar K}$} (m-3-4)
										(m-1-6) edge[cell, transform canvas={yshift=-1.625em}] node[right] {$\bar g$} (m-2-6);
			\end{tikzpicture} = \begin{tikzpicture}[textbaseline]
    		\matrix(m)[math35]{TA & TB \\ A & B \\ C & D \\};
    		\path[map]  (m-1-1) edge[barred] node[above] {$TJ$} (m-1-2)
        		                edge node[left] {$a$} (m-2-1)
            		    (m-1-2) edge node[right] {$b$} (m-2-2)
            		    (m-2-1) edge[barred] node[below] {$J$} (m-2-2)
            		            edge node[left] {$f$} (m-3-1)
            		    (m-2-2) edge node[right] {$g$} (m-3-2)
            		    (m-3-1) edge[barred] node[below] {$\ul K$} (m-3-2);
    		\path[transform canvas={xshift=1.75em}]
        		        (m-1-1) edge[cell] node[right] {$\bar J$} (m-2-1)
        		        (m-2-1) edge[transform canvas={yshift=-0.25em}, cell] node[right] {$\phi$} (m-3-1);
  		\end{tikzpicture}
		\end{displaymath}
		where, in the left-hand side, $\inv{\bar f}$ denotes the inverse of the structure cell of $f$, while $\ul{\bar K} \dfn \bar K$ if $\ul K = (C \xbrar K D)$ and $\ul{\bar K} \dfn \id_c$ if $\ul K = (C)$. We claim that $\bar J$ forms a well-defined $T$-structure on $J$, that this structure is unique in making $\phi$ into a $T$-cell and that, thus a $T$-cell, $\phi$ is cartesian in $\wAlg\lax\lax T$. In checking these claims the proof of Proposition 4.1 of \cite{Koudenburg15a} applies almost verbatim, except that here $\ul K$ might be empty and that the factorisations $\ul H \Rar J$ through $\phi$ have paths $\ul H$ as horizontal source, rather than a single horizontal $T$-morphism $H$.
  \end{proof}		
  
  \begin{example}
		In $\enProf{(\V, \V')}$ the canonical monoidal structure on the restriction $K(f, g)$ of a monoidal $\V$-profunctor $\hmap KCD$ (\defref{monoidal V-profunctor}) along a monoidal $\V'$-functor $\map fAC$ and a lax monoidal $\V'$-functor $\map gBD$ is given by the $\V$-maps
		\begin{multline*}
			K(fx_1, gy_1) \tens' \dotsb \tens' K(fx_n, gy_n) \xrar{K_\oslash} K\bigpars{(fx_1 \oslash \dotsb \oslash fx_n), (gy_1 \oslash \dotsb \oslash gy_n)} \\
			\to K\bigpars{f(x_1 \oslash \dotsb \oslash x_n), g(y_1 \oslash \dotsb \oslash y_n)}
		\end{multline*}
		where the second map is given by the actions of the structure transformations $\inv f_\oslash$ and $g_\oslash$.
	\end{example}
  
  The following lemma shows the necessity of $f$ being a pseudo $T$-morphism for the creation of $f^*$ and $f_*$ by $\map U{\wAlg v\clx T}\K$ and $\map U{\wAlg v\lax T}\K$ respectively.
  \begin{lemma} \label{necessary condition for the creation of companions and conjoints}
    Let $T$ be a monad on an augmented virtual double category $\K$ and let \mbox{$\map fAC$} a lax $T$-morphism between weak $T$-algebras. If the companion of $f$ exists in $\K$ then $f_*$ is created by the forgetful functor $\map U{\wAlg v\lax T}\K$ precisely if $f$ is a pseudo $T$-morphism. An analogous result holds for the creation of conjoints of colax $T$-morphisms by the forgetful functor $\map U{\wAlg v\clx T}\K$.
  \end{lemma}
  \begin{proof}
    The `if'"/part follows immediately from \propref{creating restrictions} above. For the converse assume that $\map U{\wAlg v\lax T}\K$ creates $f_*$; that is, there exists a $T$-structure $\bar{f_*}$ on $f_*$, as in the composite below, that makes the cartesian cell defining $f_*$ into a cartesian $T$-cell. It follows that the corresponding cocartesian cell lifts to form a cocartesian $T$-cell as well.
    \begin{displaymath}
      \begin{tikzpicture}[textbaseline]
	    	\matrix(m)[math35, column sep={1.75em,between origins}]{& TA & \\ TA & & TC \\ A & & C \\ & C & \\};
  	  	\path[map]	(m-1-2) edge[transform canvas={xshift=3pt}] node[right] {$Tf$} (m-2-3)
  	  	            (m-2-1) edge[barred] node[below, inner sep=2pt] {$Tf_*$} (m-2-3)
                	  				edge node[left] {$a$} (m-3-1)
                	  (m-2-3) edge node[right] {$c$} (m-3-3)
  	  							(m-3-1) edge[barred] node[below, inner sep=2.5pt] {$f_*$} (m-3-3)
  	  											edge[transform canvas={xshift=-2pt}] node[left] {$f$} (m-4-2);
  	  	\path				(m-3-3)	edge[eq, transform canvas={xshift=1pt}] (m-4-2)
  	  							(m-2-2) edge[cell, transform canvas={yshift=-0.25em}] node[right] {$\bar{f_*}$} (m-3-2)
  	  							(m-1-2) edge[eq, transform canvas={xshift=-4pt}] (m-2-1);
  	  	\draw				([yshift=-0.5em]$(m-1-2)!0.5!(m-2-2)$) node[font=\scriptsize] {$T\!\cocart$}
  	  							([yshift=0.25em]$(m-3-2)!0.5!(m-4-2)$) node[font=\scriptsize] {$\cart$};
  		\end{tikzpicture}
    \end{displaymath}
    It is easily checked that the composite above forms the inverse for the structure cell $\bar f$ of $f$, by using the $T$-cell axioms for the cartesian and cocartesian cells above as well as the companion identities (\lemref{companion identities lemma}).
  \end{proof}
  
  Using \propref{creating restrictions} the following result describes the creation of restrictions $K(\id, g)$ in $\lbcwAlg vT$; recall that the latter has as horizontal morphisms the horizontal $T$-morphisms whose structure cells satisfy the left Beck-Chevalley condition (\defref{left exact}).
  \begin{proposition} \label{creating restrictions satisfying the left Beck-Chevalley condition}
    Let $T$ be a monad on an augmented virtual double category $\K$. The forgetful functors $\map U{\lbcwAlg vT}\K$, where $v \in \set{\clx, \lax, \psd}$, create all unary restrictions $K(\id, g)$ that are preserved by $T$.
  \end{proposition}
  \begin{proof}
	  First notice that the inclusion $\lbcwAlg vT \to \wAlg v\psd T$ reflects restrictions, and that $\map U{\wAlg v\psd T}\K$ creates the restriction $J \dfn K(\id, g)$ by \propref{creating restrictions}; as in the proof there we denote by $\cell\phi JK$ the defining cartesian cell. Thus it suffices to show that the structure cell $\bar J$ satisfies the left Beck-Chevalley condition whenever $\bar K$ does, provided that $g$ is a pseudo $T$-morphism and that $T\phi$ is again cartesian. To see this consider the following equation: in the first identity we factor $\bar K$ as a right pointwise cocartesian cell $\bar K'$ through $K(\id, d)$, as is possible by the Beck-Chevalley condition, while in the second identity the composite $\bar K' \of T\phi$ (where $T\phi$ is again cartesian) is factored through the restriction $H \dfn K(\id, d \of Tg)$; that the resulting cell $\bar K''$ is again right pointwise cocartesian follows from \lemref{coherence of pointwise cocartesian paths}.
		\begin{displaymath}

		\end{displaymath}
		Finally notice that the composite of the bottom two rows in the right-hand side above is cartesian, by the pasting lemma and the fact that $\bar g$ is invertible. Since, by the definition of $\bar J$ in the proof of \propref{creating restrictions}, each of the composites above factors through $\phi$ as the structure cell $\bar J$ we conclude, by factoring the right-hand side and using the pasting lemma again, that $\bar J$ coincides with $\bar K''$ composed with the cartesian cell that defines $H$ as the restriction of $J$ along $b$. Since $\bar K''$ is right pointwise cocartesian, this shows that $\bar J$ satisfies the left Beck-Chevalley condition, completing the proof.
	\end{proof}
	
	The following theorem describes the creation of adjunctions in $\wAlg v\lax T$. Its first assertion is Kelly's result on `doctrinal adjunction'; see Theorem~1.5 of \cite{Kelly74}. In Remark 5.25 of \cite{Shulman11} it is also observed that Kelly's result can be naturally described in the language of double categories, by using the so-called ``mate correspondence'' that is induced by companions.
  
  \begin{theorem} \label{creating adjunctions}
  	Let $T$ be a monad on an augmented virtual double category $\K$ and let $\map fAC$ be a lax $T$-morphism between weak $T$-algebras, where `weak' means either `colax', `lax' or `pseudo'. If $f$ admits both a companion $f_*$ and a right adjoint in $\K$ then $\map U{\wAlg v\lax T}\K$ creates the right adjoint of $f$ precisely if $f$ is a pseudo $T$"/morphism. In that case $\map U{\wAlg v\psd T}\K$ creates the right adjoint of $f$ precisely if the unique factorisation $\bar{f_*}$, in the right-hand side below, is left $\id_A$"/exact.
  	\begin{displaymath}
			\begin{tikzpicture}[textbaseline]
				\matrix(m)[math35, column sep={1.75em,between origins}]{& TA & & TC \\ A & & TC \\ & C & \\};
				\path[map]	(m-1-2) edge[barred] node[above] {$Tf_*$} (m-1-4)
														edge[transform canvas={xshift=-1pt}] node[left] {$Tf$} (m-2-3)
														edge[bend right = 18] node[left] {$a$} (m-2-1)
										(m-2-1) edge[bend right = 18] node[left] {$f$} (m-3-2)
										(m-2-3) edge[bend left = 18] node[right] {$c$} (m-3-2);
				\path				(m-1-2) edge[cell, transform canvas={shift={(-0.75em,-1.625em)}}] node[right] {$\inv{\bar f}$} (m-2-2)
										(m-1-4) edge[eq, transform canvas={xshift=2pt}] (m-2-3);
				\draw				([yshift=0.333em]$(m-1-3)!0.5!(m-2-3)$) node[font=\scriptsize] {$T\!\cart$};
			\end{tikzpicture} = \begin{tikzpicture}[textbaseline]
				\matrix(m)[math35, column sep={1.75em,between origins}]{TA & & TC \\ A & & C \\ & C & \\};
				\path[map]	(m-1-1) edge[barred] node[above] {$Tf_*$} (m-1-3)
														edge node[left] {$a$} (m-2-1)	
										(m-1-3) edge node[right] {$c$} (m-2-3)
										(m-2-1) edge[barred] node[below] {$f_*$} (m-2-3)
														edge[transform canvas={xshift=-1pt}] node[left] {$f$} (m-3-2);
				\path				(m-1-1) edge[cell, transform canvas={xshift=1.75em}] node[right] {$\bar{f_*}$} (m-2-1)
										(m-2-3) edge[eq, transform canvas={xshift=1pt}] (m-3-2);
				\draw				([yshift=0.25em]$(m-2-2)!0.5!(m-3-2)$) node[font=\scriptsize] {$\cart$};
 			\end{tikzpicture}
		\end{displaymath}
  \end{theorem}
  \begin{proof}
  	Let $\map gCA$ be the right adjoint of $f$ in $\K$, defined by unit and counit cells $\cell\eta{\id_A}{g \of f}$ and $\cell\eps{f \of g}{\id_C}$. Remember from \lemref{adjunctions} that the triangle identities for $\eta$ and $\eps$ are equivalent to the conjoint identities for the factorisations $\eta'$ and $\eps'$ in
  	\begin{displaymath}
  		\eta = \eta' \of \cocart \qquad \text{and} \qquad \eps = \cart \of \eps',
  	\end{displaymath}
  	where $\cart$ and $\cocart$ are the (co"/)cartesian cells defining $f_*$. For the `if'"/part of the first assertion, assume that $\map fAC$ is a pseudo $T$"/morphism. By \propref{creating restrictions}(b) the companion of $f$ is created by $\map U{\wAlg v\lax T}\K$ and, by applying \thmref{representable horizontal T-morphisms}, the resulting horizontal $T$-structure on $f_*$ induces a lax $T$"/morphism structure on $g$, making $\eta'$ into a $T$-cell. In fact $\eta'$ lifts to a cartesian $T$-cell, again by \propref{creating restrictions}, so that $\eps'$ lifts to form a $T$-cell as well: it is the unique factorisation of $\id_g$ through $\eta'$. We conclude that both $\eta$ and $\eps$ are composites of $T$-cells, and hence they are $T$-cells themselves.
  	
  	For the converse assume that $g$ admits the structure of a lax $T$"/morphism, such that $\eta$ and $\eps$ become $T$-cells. Applying \propref{creating restrictions} to $\eta'$, which defines $f_*$ as the conjoint of $g$, we obtain a horizontal $T$-morphism structure on $f_*$, making $\eta'$ into a cartesian $T$-cell. With respect to this structure $\eps'$, being the factorisation of $\id_g$ through $\eta'$, is a cocartesian $T$-cell; it follows that the (co"/)cartesian cells defining $f_*$ are $T$-cells as well. We conclude that $f_*$ is created by $\map U{\wAlg v\lax T}\K$, so that $f$ is a pseudo $T$"/morphism by \lemref{necessary condition for the creation of companions and conjoints}.
  	
  	For the final assertion assume that $f$ is a pseudo $T$"/morphism, so that the right adjoint $\map gCA$ of $f$ is created by $\map U{\wAlg v\lax T}\K$. Notice that the factorisation $\bar{f_*}$, as described in the statement, forms the structure cell making $f_*$ into a horizontal $T$-morphism. Applying \propref{adjunctions in terms of left Kan extensions}(c) to $Tf \ladj Tg$, combined with \propref{weak left Kan extensions along companions}, we find that $a \of T\eta'$ in the right"/hand side of
  	\begin{displaymath}
  		\eta' \of \bar{f_*} = (a \of T\eta') \hc \bar g,
  	\end{displaymath}
  	which is the $T$-cell axiom for $\eta'$, defines $a \of Tg$ as a left Kan extension. Because $\eta'$ itself defines $g$ as a left Kan extension too the final assertion follows. This concludes the proof.
  \end{proof}
	
	\subsection{Horizontal composites}
	The following proposition describes the creation of horizontal composites by the forgetful functors $\map U{\wAlg vwT}\K$. Remember that a horizontal $T$-morphism $\hmap{(J, \bar J)}AB$ is said to satisfy the left Beck-Chevalley condition as soon as its structure cell $\bar J$ does, in the sense of \defref{left exact}.
	\begin{proposition}
		Consider a monad $T$ on an augmented virtual double category $\K$ and let \mbox{$v, w \in \set{\clx, \lax, \psd}$}. A horizontal composite $(J_1 \hc \dotsb \hc J_n)$ of horizontal $T$-morphisms is created by $\map U{\wAlg vwT}\K$ whenever it is preserved by both $T$ and $T^2$.
		
		In the case $v = \lax$ or $\psd$ an analogous result holds for right pointwise horizontal composites $(J_1 \hc \dotsb \hc J_n)$ (\defref{pointwise cocartesian path}), provided that the restrictions $J_n(\id, f)$ and $(J_1 \hc \dotsb \hc J_n)(\id, f)$ exist in $\K$, for all $\map fB{A_n}$, and that they are preserved by $T$.
		
		Finally, if the horizontal $T$-morphisms $J_1, \dotsc, J_n$ satisfy the left Beck-Chevalley condition then, under the assumptions above, $\map U{\lbcwAlg vT}\K$ creates the horizontal composite $(J_1 \hc \dotsb \hc J_n)$ as well.
	\end{proposition}
	\begin{proof}
		We consider the case $(v, w) = (\lax, \lax)$; the other cases are similar. Let $\phi$ be a cocartesian cell in $\K$, as in the left-hand side below, and assume that $T\phi$ is again cocartesian. It follows that the left-hand side factors uniquely through $T\phi$ as a cell $\bar K$, as shown.
		\begin{displaymath}
			\begin{tikzpicture}[textbaseline]
				\matrix(m)[math35]
					{ TA_0 & TA_1 &[1em] TA_{n'} & TA_n \\
						A_0 & A_1 & A_{n'} & A_n \\
						A_0 & & & A_n \\ };
				\path[map]	(m-1-1) edge[barred] node[above] {$TJ_1$} (m-1-2)
														edge node[left] {$a_0$} (m-2-1)
										(m-1-2) edge node[right, inner sep=2pt] {$a_1$} (m-2-2)
										(m-1-3) edge[barred] node[above] {$TJ_n$} (m-1-4)
														edge node[left, inner sep=0pt] {$a_{n'}$} (m-2-3)
										(m-1-4) edge node[right] {$a_n$} (m-2-4)
										(m-2-1) edge[barred] node[below] {$J_1$} (m-2-2)
										(m-2-3) edge[barred] node[below] {$J_n$} (m-2-4)
										(m-3-1) edge[barred] node[below] {$K$} (m-3-4);
				\path				(m-2-1) edge[eq] (m-3-1)
										(m-2-4) edge[eq] (m-3-4);
				\path[transform canvas={xshift=1.75em}]	(m-1-1) edge[cell] node[right] {$\bar J_1$} (m-2-1)
										(m-1-3) edge[cell] node[right] {$\bar J_n$} (m-2-3)
										(m-2-2) edge[cell, transform canvas={xshift=0.5em}] node[right] {$\phi$} (m-3-2);
				\draw				($(m-1-2)!0.5!(m-2-3)$) node {$\dotsb$};
			\end{tikzpicture} = \begin{tikzpicture}[textbaseline]
				\matrix(m)[math35]
					{ TA_0 & TA_1 &[1em] TA_{n'} & TA_n \\
						TA_0 & & & TA_n \\
						A_0 & & & A_n \\ };
				\path[map]	(m-1-1) edge[barred] node[above] {$TJ_1$} (m-1-2)
										(m-1-3) edge[barred] node[above] {$TJ_n$} (m-1-4)
										(m-2-1) edge[barred] node[below] {$TK$} (m-2-4)
														edge node[left] {$a_0$} (m-3-1)
										(m-2-4) edge node[right] {$a_n$} (m-3-4)
										(m-3-1) edge[barred] node[below] {$K$} (m-3-4);
				\path				(m-1-1) edge[eq] (m-2-1)
										(m-1-4) edge[eq] (m-2-4);
				\path[transform canvas={xshift=2.25em}]	(m-1-2) edge[cell] node[right] {$T\phi$} (m-2-2)
										(m-2-2) edge[cell, transform canvas={yshift=-0.25em}] node[right] {$\bar K$} (m-3-2);
				\draw				($(m-1-2)!0.5!(m-1-3)$) node {$\dotsb$};
			\end{tikzpicture}
		\end{displaymath}
		We claim that $\bar K$ makes $K$ into a horizontal $T$-morphism, that is it satisfies the associativity and unit axioms of \defref{horizontal T-morphism}. To prove the associativity axiom consider the equality below, where the identities follow from the factorisation above and its $T$-image; the associativity axioms for $\bar J_1, \dotsc, \bar J_n$ and the interchange axiom; the factorisation above; the naturality of $\mu$ with respect to $\phi$. Notice the left-hand and right-hand sides below equal the corresponding sides of the associativity axiom for $\bar K$ after composition with $T^2\phi$. Since the latter is cocartesian the associativity axiom itself follows.
		\begin{align*}

		\end{align*}
		Using the fact that $\phi$ is cocartesian, the unit axiom for $\bar K$ can be deduced from that for $\bar J_1, \dotsc, \bar J_n$ in a similar way. The uniqueness of $\bar K$ follows from the fact that its defining unique factorisation above forms the $T$-cell axiom for $\phi$.
		
		It remains to show that $\phi$ is cocartesian as a $T$-cell. Hence consider paths $\hmap{\ul J'}{A'_0}A_0$ and $\hmap{\ul J''}{A_n}{A_q}$ of horizontal $T$"/morphisms; we have to prove that any $T$"/cell $\cell\psi{\ul J' \conc \ul J \conc \ul J''}K$ factors uniquely through $\phi$ in $\wAlg\lax\lax T$, in the sense of \defref{cocartesian paths}. Since $\phi$ is cocartesian in $\K$ such a unique factorisation $\psi = \psi' \of (\id_{J'_1}, \dotsc, \id_{J'_p}, \phi, \id_{J''_1}, \dotsc, \id_{J''_q})$ certainly exists in $\K$, and it suffices to show that $\psi'$ satisfies the $T$-cell axiom. To see this consider the equation below, where the identities follow from the $T$-image of the factorisation of $\psi$; the $T$-cell axiom for $\psi$; the factorisation of $\psi$; the $T$-cell axiom for $\phi$.
		\begin{align*}

		\end{align*}
		Since $T\phi$ is assumed to be cocartesian, the top rows in both the left-hand and right-hand sides above are cocartesian. Hence, because their bottom two rows equal the sides of the $T$-cell axiom for $\psi'$, the axiom follows. This completes the proof of the first assertion.
		
		For the second assertion, we now assume that both $\phi$ and $T\phi$, in the above, are right pointwise cocartesian in $\K$; we have to show that $\phi$, as a $T$-cell, is right pointwise cocartesian in $\wAlg\lax\lax T$. To this end consider any lax $T$-morphism $\map fB{A_n}$. By assumpion the restriction $J_n(\id, f)$ exists in $\K$ so that, because $\phi$ is right pointwise cocartesian, the restriction $K(\id, f)$ does too; it follows from the previous proposition that these restrictions created in $\wAlg\lax\lax T$. Thus we obtain the following factorisation in $\wAlg\lax\lax T$; we have to show that the $T$-cell $\phi'$ is cocartesian.
		\begin{displaymath}
			\begin{tikzpicture}[textbaseline]
				\matrix(m)[math35]
					{ A_0 & A_1 & A_{n'} & B \\
						A_0 & A_1 & A_{n'} & A_n \\
						A_0 & & & A_n \\ };
				\path[map]	(m-1-1) edge[barred] node[above] {$J_1$} (m-1-2)
										(m-1-3) edge[barred] node[above, xshift=-2pt] {$J_n(\id, f)$} (m-1-4)
										(m-1-4) edge node[right] {$f$} (m-2-4)
										(m-2-1) edge[barred] node[below] {$J_1$} (m-2-2)
										(m-2-3) edge[barred] node[below] {$J_n$} (m-2-4)
										(m-3-1) edge[barred] node[below] {$K$} (m-3-4);
				\path				(m-1-2) edge[eq] (m-2-2)
										(m-1-3) edge[eq] (m-2-3)
										(m-1-1) edge[eq] (m-2-1)
										(m-2-1) edge[eq] (m-3-1)
										(m-2-4) edge[eq] (m-3-4)
										(m-2-2) edge[cell, transform canvas={xshift=1.75em}] node[right] {$\phi$} (m-3-2);
				\draw[font=\scriptsize]	($(m-1-3)!0.5!(m-2-4)$) node {$\cart$};
				\draw				($(m-1-2)!0.5!(m-2-3)$) node {$\dotsb$};
			\end{tikzpicture} = \begin{tikzpicture}[textbaseline]
				\matrix(m)[math35]
					{ A_0 & A_1 & A_{n'} & B \\
						A_0 & & & B \\
						A_0 & & & A_n \\ };
				\path[map]	(m-1-1) edge[barred] node[above] {$J_1$} (m-1-2)
										(m-1-3) edge[barred] node[above, xshift=-2pt] {$J_n(\id, f)$} (m-1-4)
										(m-2-1) edge[barred] node[below] {$K(\id, f)$} (m-2-4)
										(m-2-4) edge node[right] {$f$} (m-3-4)
										(m-3-1) edge[barred] node[below] {$K$} (m-3-4);
				\path				(m-1-1) edge[eq] (m-2-1)
										(m-2-1) edge[eq] (m-3-1)
										(m-1-4) edge[eq] (m-2-4)
										(m-1-2) edge[cell, transform canvas={xshift=1.75em}] node[right] {$\phi'$} (m-2-2);
				\draw[font=\scriptsize]	([yshift=-0.25em]$(m-2-2)!0.5!(m-3-3)$) node {$\cart$};
				\draw				($(m-1-2)!0.5!(m-1-3)$) node {$\dotsb$};
			\end{tikzpicture}
		\end{displaymath}
		To see this, use the assumption that both $\phi$ and $T\phi$ are right pointwise cocartesian, and that $T$ preserves the cartesian cells above: it follows that both $\phi'$ and $T\phi'$ are cocartesian. Applying the second part of the proof of the non-pointwise case above to the $T$-cell $\phi'$, we find that it is cocartesian.
		
		Finally assume that each of the $J_1, \dotsc, J_n$ satisfies the Beck-Chevalley condition. Since the restrictions $J_n(\id, f)$ and $K(\id, f)$, where $\map fB{A_n}$, are assumed to exist and be preserved by $T$, it follows from \propref{creating restrictions satisfying the left Beck-Chevalley condition} and \propref{creating restrictions} that such restrictions are created by both $\map U{\lbcwAlg vT}\K$ and \mbox{$\map U{\wAlg v\psd T}\K$}, and hence are preserved by the inclusion $\lbcwAlg vT \to \wAlg v\psd T$. Since this inclusion (being locally full and faithful) clearly reflects cocartesian cells, it follows that it reflects right pointwise cocartesian cells as well. Therefore, to prove creation of \mbox{$(J_1 \hc \dotsb \hc J_n)$} along $\map U{\lbcwAlg vT}\K$, it suffices to prove that the structure cell $\bar K$, as obtained above, satisfies the left Beck-Chevalley condition; that is we have to show its factorisation $\bar K'$ through $K(\id, a_n)$ to be right pointwise cocartesian. By the pasting lemma (\lemref{pasting lemma for right pointwise cocartesian paths}) we may equivalently show that $\bar K' \of T\phi$ is right pointwise cocartesian; to do so consider the following equation whose identities follow form the definition of $\bar K$ and the factorisation of $\bar J_n$ through $J(\id, a_n)$ and that of $\phi \of (\id, \dotsb, \cart)$, in the third composite, through $K(\id, a_n)$.
		\begin{displaymath}
			\begin{tikzpicture}[scheme]
				\draw	(1,3) -- (0,3) -- (0,0) -- (3,0) -- (3,3) -- (2,3) (0,1) -- (3,1) (0,2) -- (3,2);
				\draw	(1.5,0.5) node {c}
							(1.5,1.5) node {$\bar K'$}
							(1.5,2.5) node {$T\phi$}
							(1.5,3) node {$\dotsb$};
			\end{tikzpicture} \mspace{12mu} = \mspace{12mu} \begin{tikzpicture}[scheme, yshift=0.8em]
				\draw (0,1) -- (1,1) -- (1,2) -- (0,2) -- (0,0) -- (3,0) -- (3,2) -- (2,2) -- (2,1) -- (3,1);
				\draw	(0.5,1.5) node {$\bar J_1$}
							(1.5,0.5) node {$\phi$}
							(1.5,1.5) node {$\dotsb$}
							(2.5,1.5) node {$\bar J_n$};
			\end{tikzpicture} \mspace{12mu} = \mspace{12mu} \begin{tikzpicture}[scheme]
				\draw (0,1) -- (1,1) -- (1,3) -- (0,3) -- (0,0) -- (3,0) -- (3,3) -- (2,3) -- (2,1) -- (3,1) (0,2) -- (1,2) (2,2) -- (3,2);
				\draw	(0.5,2.5) node {$\bar J_1$}
							(1.5,0.5) node {$\phi$}
							(1.5,1.5) node {$\dotsb$}
							(1.5,2.5) node {$\dotsb$}
							(2.5,2.5) node {$\bar J_n'$}
							(2.5,1.5) node {c};
			\end{tikzpicture} \mspace{12mu} = \mspace{12mu} \begin{tikzpicture}[scheme]
				\draw	(0,2) -- (1,2) -- (1,3) -- (0,3) -- (0,0) -- (3,0) -- (3,3) -- (2,3) -- (2,2) -- (3,2) (0,1) -- (3,1);
				\draw	(0.5,2.5) node {$\bar J_1$}
							(1.5,0.5) node {c}
							(1.5,1.5) node {$\phi'$}
							(1.5,2.5) node {$\dotsb$}
							(2.5,2.5) node {$\bar J_n'$};
			\end{tikzpicture}
		\end{displaymath}
		Now observe that $(\bar J_1, \dotsc, \bar J_n)$ satisfies the left Beck-Chevalley condition, because all of $\bar J_1, \dotsc, \bar J_n$ do and \lemref{concatenation of paths satisfying the Beck-Chevalley condition}. By definition it follows that $(\bar J_1, \dotsc, \bar J_n')$ is right pointwise cocartesian and, by \lemref{coherence of pointwise cocartesian paths}, so is $\phi'$. Using the pasting lemma again we conclude that $\phi' \of (\bar J_1, \dotsc, \bar J_n') = \bar K' \of T\phi$ is right pointwise cocartesian, as required. This completes the proof.
	\end{proof}
	
	\subsection{Kan extensions}
	The main results of this section, \thmref{creating Kan extensions between lax algebras} and \thmref{creating Kan extensions between colax algebras} below, describe the creation of algebraic Kan extensions by the forgetful functors \mbox{$\map U{\wAlg\lax wT}\K$} and $\map U{\wAlg\clx wT}\K$ respectively. The second of these, concerning colax $T$"/algebras, generalises the main result of \cite{Koudenburg15a}, which describes creation of algebraic Kan extensions in the case of an oplax monad on a double category $\K$. The latter in turn is a generalisation of a result by Getzler, on the lifting of pointwise left Kan extensions along symmetric monoidal enriched functors, that was given in \cite{Getzler09}, where it was used to obtain a coherent way of freely generating many types of generalised operad. In \cite{Weber15} a variant of \thmref{creating Kan extensions between lax algebras}, considered in the setting of $2$-categories, is used to obtain algebraic Kan extensions along `morphisms of internal algebra classifiers'. The creation of algebraic Kan extensions was first considered by Melli\`es and Tabareau in \cite{Mellies-Tabareau08}, in the case of a pseudomonad on a `bicategory equipped with proarrows' (see the remarks preceding \defref{augmented virtual equipment}), who assume the existence of a `yoneda situation' (see the remarks following \defref{yoneda embedding}) to state their result.
	
	\begin{remark}
		Related to the results of this section, in the forthcoming \cite{Koudenburg15b} the `lifting' of left Kan extensions that preserve algebraic structures defined by `colax-idempotent' $2$-monads $T$, in the sense of e.g.\ \cite{Kelly-Lack97}, is considered. The `free category with finite products'-monad as a prime example, such monads are special in that their pseudo $T$-algebra structures are `essentially unique', like the `structure of finite products' on a category is essentially unique. Moreover, any morphism \mbox{$\map lBM$} between pseudo $T$-algebras is uniquely a colax $T$-morphism. In \cite{Koudenburg15b} we consider the situation where such a colax $T$-morphism \mbox{$\map{(l, \bar l)}BM$} is the left Kan extension of morphisms $\map dAM$ and $\map jAB$, where $A$ is not necessarily a $T$-algebra. Generalising a classical result on finite-product-preserving left Kan extensions, by Ad\'amek and Rosick\'y \cite{Adamek-Rosicky01}, the main results of \cite{Koudenburg15b} relate the invertibility of $\bar l$ to conditions involving the morphisms $d$ and $j$.
	\end{remark}
	
	The following definition is used in stating the theorems.
	\begin{definition} \label{algebraic structure preserving left Kan extensions}
		Let $T$ be a monad on an augmented virtual double category $\K$. Consider a (co"/)lax $T$-algebra $M = (M, m, \bar m, \tilde m)$ as well as a vertical morphism $\map d{A_0}M$ and a path $\hmap{\ul J = (J_1, \dotsc, J_n)}{A_0}{A_n}$ of horizontal morphisms in $\K$. We say that the \emph{algebraic structure of $M$ preserves} the (pointwise) (weak) left Kan extension of $d$ along $(J_1, \dotsc, J_n)$ if, for any cell
		\begin{displaymath}
			\begin{tikzpicture}
				\matrix(m)[math35, yshift=1.625em]{A_0 & A_1 & A_{n'} & A_n \\};
				\draw	([yshift=-3.25em]$(m-1-1)!0.5!(m-1-4)$) node (M) {$M$};
				\path[map]	(m-1-1) edge[barred] node[above] {$J_1$} (m-1-2)
														edge[transform canvas={yshift=-2pt}] node[below left] {$d$} (M)
										(m-1-3) edge[barred] node[above] {$J_n$} (m-1-4)
										(m-1-4) edge[transform canvas={yshift=-2pt}] node[below right] {$l$} (M);
				\path				($(m-1-2.south)!0.5!(m-1-3.south)$) edge[cell] node[right] {$\eta$} (M);
				\draw				($(m-1-2)!0.5!(m-1-3)$) node {$\dotsb$};
										
			\end{tikzpicture}
		\end{displaymath}
		that defines $l$ as the (pointwise) (weak) left Kan extension of $d$ along $(J_1, \dotsc, J_n)$, the composite $m \of T\eta$ defines $m \of Tl$ as the (pointwise) (weak) left Kan extension of $m \of Td$ along $(TJ_1, \dotsc, TJ_n)$. 
	\end{definition}
	In \exref{tensor products preserving weighted colimits in each variable} we will see that the monoidal structure of any (op-)lax monoidal $\V'$-category $M$, whose tensor product preserves small $\V$"/weighted colimits in each variable, preserves pointwise left Kan extensions along $\V$-profunctors $\hmap JAB$ with $A$ a small $\V$-category.
	
	We start with the creation of algebraic left Kan extensions between lax algebras. For the notion of left exactness see \defref{left exact}.
	\begin{theorem} \label{creating Kan extensions between lax algebras}
  	Let $T = (T, \mu, \iota)$ be a monad on an augmented virtual double category $\K$ and let `weak' mean either `colax', `lax' or `pseudo'. Given lax $T$-algebras $A_0, \dotsc, A_n$ and $M$, consider the following conditions on a path of horizontal $T$-morphisms $(A_0 \xbrar{J_1} A_1, \dotsc, A_{n'} \xbrar{J_n} A_n)$ and a weak vertical $T$-morphism $\map d{A_0}M$, where $m$ and $a_0$ denote the structure maps of $M$ and $A_0$:
  	\begin{enumerate}
  		\item[\textup{(p)}]	the algebraic structure of $M$ preserves both the (weak) left Kan extension of $d$ along $(J_1, \dotsc, J_n)$ and that of $m \of Td$ along $(TJ_1, \dotsc, TJ_n)$;
  		\item[\textup{(e)}]	the path of structure cells $(\bar J_1, \dotsc, \bar J_n)$ is (weak) left $d$-exact, while its $T$-image $(T\bar J_1, \dotsc, T\bar J_n)$ is (weak) left $(d \of a_0)$-exact;
  		\item[\textup{(l)}]	the forgetful functor $\map U{\wAlg\lax wT}\K$ creates the (weak) left Kan extension of $d$ along $(J_1, \dotsc, J_n)$.
  	\end{enumerate}
  	The following hold:
  	\begin{enumerate}[label=\textup{(\alph*)}]
  		\item if `weak' means `lax' then \textup{(p)} implies \textup{(l)};
  		\item if `weak' means `colax' then \textup{(e)} implies \textup{(l)};
  		\item if `weak' means `pseudo' then any two of \textup{(p)}, \textup{(e)} and \textup{(l)} imply the third.
  	\end{enumerate}
  	Moreover if the restrictions $J_n(\id, f)$ exist in $\K$ and are preserved by $T$, for all $\map fB{A_n}$, then the result analogous to \textup{(a)} for pointwise (weak) left Kan extensions holds as well. In this case the forgetful functor $\map U{\lbcwAlg\lax T}\K$ too creates the pointwise left Kan extension of $d$ along $(J_1, \dotsc, J_n)$ (which now satisfy the left Beck-Chevalley condition), under the assumption of \textup{(p)}.
  \end{theorem}
  We remark that in the proof below of the pointwise case of (a) above we use the fact that, under the above assumptions, all restrictions in $\wAlg\lax\lax T$ of the form $J_n(\id, f)$ are created by $\map U{\wAlg\lax\lax T}\K$, as we have seen in \propref{creating restrictions}(b). That the same need not be true in the colax case $w = \clx$, prevents us from being able to prove a pointwise variant of (b) above. It is possible however to show that $\map U{\wAlg\lax\clx T}\K$ creates left Kan extensions satisfying a variation of \defref{pointwise left Kan extension}, by restricting it to restrictions $J_n(\id, f)$ in $\wAlg\lax\clx T$ with $\map fB{A_n}$ a pseudo $T$-morphism, for such restrictions are created by $\map U{\wAlg\lax\clx T}\K$; see \propref{creating restrictions}(a).
  \begin{proof}
		\emph{Part (a): `weak' means `lax'.} In this case $\map d{A_0}M$ is equipped with a lax $T$-morphism structure cell $\cell{\bar d}{m \of Td}{d \of a_0}$. We shall treat the creation of pointwise left Kan extensions; a proof for the other cases can be obtained by either choosing $f = \id_{A_n}$ in \eqref{T-cell factorisation} below, taking the path $\ul H = (H_1, \dotsc, H_m)$ to be empty, or both. Consider a nullary cell $\eta$ in $\K$, as in the composite on the left-hand side below, that defines $\map l{A_n}M$ as the pointwise left Kan extension of $d$ along $(J_1, \dotsc, J_n)$. We will equip $l$ with the structure of a lax $T$-morphism, with respect to which $\eta$ forms a $T$-cell. Thus a $T$-cell, we will show that $\eta$ defines $l$ as a pointwise left Kan extension in $\wAlg\lax\lax T$.
		
		Under the assumption of (p) the composite $m \of T\eta$ in the right-hand side below defines $m \of Tl$ as a left Kan extension so that the right-hand side, where $\bar d$ is the lax $T$-structure cell for $\map d{A_0}M$, factors as a vertical cell $\cell{\bar l}{m \of Tl}{l \of a_n}$ as shown.
  	\begin{equation} \label{lifted lax structure}

  	\end{equation}
  	
  	We claim that $\bar l$ makes $l$ into a lax $T$-morphism, that is it satisfies the associativity and unit axioms. To see that it satisfies the former consider the equation of composites below, whose identities are explained below. Notice that its left-hand and right-hand sides are given by the corresponding sides of the associativity axiom for $\bar l$, composed with $m \of Tm \of T^2 \eta$ on the left. By condition (p) the latter defines $m \of Tm \of T^2l$ as the left Kan extension of $m \of Tm \of T^2d$ along $(T^2J_1, \dotsc, T^2J_n)$, so that the axiom itself follows from uniqueness of factorisations through $m \of Tm \of T^2\eta$.
  	\begin{align*}

  	\end{align*}
  	The identities above follow from the $T$-image of \eqref{lifted lax structure}; the identity \eqref{lifted lax structure} itself; the associativity axioms for $\bar J_1, \dotsc \bar J_n$ (see \defref{horizontal T-morphism}); the associativity axiom for $\bar d$; the identity \eqref{lifted lax structure} again; the naturality of $\mu$.
  	
  	The unit axiom for $\bar l$ follows similarly from the equation below, whose sides form those of the unit axiom composed with $\eta$, and the fact that the latter defines $l$ as a left Kan extension. The identities here follow from the naturality of $\iota$; the identity \eqref{lifted lax structure}; the unit axiom for $\bar d$; the unit axioms for $\bar J_1, \dotsc, \bar J_n$.
  	\begin{align*}

  	\end{align*}
  	
  	Now notice that the factorisation \eqref{lifted lax structure} forms the $T$-cell axiom for $\eta$ (\defref{T-cell}). In fact, since the factorisation is unique, the lax $T$-structure cell $\bar l$ is unique in making $\eta$ into a $T$-cell. Thus a $T$-cell, it remains to prove that $\eta$ defines $(l, \bar l)$ as the pointwise left Kan extension of $d$ along $(J_1, \dotsc, J_n)$ in $\wAlg\lax\lax T$. In order to do so, consider any nullary $T$-cell $\phi$ as on the left-hand side below, where $J_n(\id, f)$ denotes any $T$-restriction of $J_n$ along any lax $T$-morphism $\map fB{A_n}$. Since we assume the underlying restriction $J_n(\id, f)$ to exist in $\K$ the cartesian $T$-cell defining $J_n(\id, f)$ is created, and thus preserved, by $\map U{\wAlg\lax\lax T}\K$.
  	\begin{multline} \label{T-cell factorisation}

  	\end{multline}
  	Because $\eta$ defines $l$ as a pointwise left Kan extension in $\K$, the cell $\phi$ factors uniquely in $\K$ as shown, and it remains to prove that the factorisation $\phi'$ is a $T$-cell. To see this consider the following equation of composites in $\K$, where `c' denotes the cartesian cell that defines $J_n(\id, f)$. Its identities follow from the $T$-cell axiom for $\eta$; the definition of $\overline{J_n(\id, f)}$ (see the proof of \propref{creating restrictions}); the identity above; the $T$-cell axiom for $\phi$; the $T$-image of the identity above. Now notice that the left-hand and right-hand sides of the equation below are the sides of the $T$-cell axiom for $\phi'$, composed with the composite $m \of T\eta \of (\id, \dotsc, T\cart)$ on the left. Since $m \of T\eta$ here defines a pointwise left Kan extension by the assumption of (p), the full composite does too, because we assume the cartesian cell to be preserved by $T$. From this we conclude that the $T$-cell axiom for $\phi'$ holds, which completes the proof of part (a).
  	\begin{align*}
  		&

  	\end{align*}
  	
  	\emph{Part (b): `weak' means `colax'.} In this case $\map d{A_0}M$ is equipped with a colax $T$-morphism structure cell $\cell{\bar d}{d \of a_0}{m \of Td}$. We shall treat the creation of left Kan extensions; a proof for weak Kan extensions is obtained by taking the path $\ul H = (H_1, \dotsc, H_m)$ in \eqref{T-cell factorisation} to be empty. Consider a nullary cell $\eta$ in $\K$, as in the composite on the right-hand side below, that defines $\map l{A_n}M$ as the pointwise left Kan extension of $d$ along $(J_1, \dotsc, J_n)$. Under the assumption of (e), the composite $\eta \of (\bar J_1, \dotsc, \bar J_n)$ in the left-hand side below defines $l \of a_n$ as a left Kan extension, so that the right-hand side factors as a vertical cell $\cell{\bar l}{l \of a_n}{m \of Tl}$ as shown.
  	\begin{equation} \label{lifted colax structure}

  	\end{equation}
  	Analogous to the proof of part (a) above, we can show that the cell $\bar l$ satisfies the associativity and unit axioms, so that it makes $\map l{A_n}M$ into a colax $T$"/morphism.
  	
  	Indeed, the associativity axiom follows from the identity below, which itself follows from the factorisation above, the associativity axiom for $\bar d$, the naturality of $\mu$ and the associativity axioms for $\bar J_1, \dotsc, \bar J_n$. As before notice that the identity below is the associativity axiom for $\bar l$ composed on the right with the composite \mbox{$\eta \of (\bar J_1 \of T\bar J_1, \dotsc, \bar J_n \of T\bar J_n)$}. As a consequence of condition (e), this composite defines a left Kan extension, so that the associativity axiom itself follows.
  	\begin{displaymath}

  	\end{displaymath}
  	The unit axiom for $\bar l$ follows in the same way from the identity on the left below, which itself follows from the unit axioms for $\bar J_1, \dotsc, \bar J_n$ and $\bar d$, the identity \eqref{lifted colax structure} and the naturality of $\iota$.
  	\begin{displaymath}
  		%
  	\end{displaymath}
  	
  	Analogous to the reasoning in the proof of part (a), the cell $\bar l$ is unique in making $\eta$ into a $T$-cell and, to show that $\eta$ defines $(l, \bar l)$ as a left Kan extension in $\wAlg\lax\clx T$, we have to prove that any unique factorisation $\phi'$ of a $T$-cell $\phi$, as in \eqref{T-cell factorisation}, where now $d$ and $k$ are colax $T$-morphisms, and where we take $f = \id_{A_n}$, is again a $T$"/cell. That its $T$-cell axiom holds after composition with $\eta \of (\bar J_1, \dotsc, \bar J_n)$, as shown below, follows from the $T$-cell axiom for $\eta$, the factorisation \eqref{T-cell factorisation} and the $T$-cell axiom for $\phi$. By assumption (e) the composite $\eta \of (\bar J_1, \dotsc, \bar J_n)$ defines a left Kan extension, so that the $T$-cell axiom follows. This completes the proof of part (b).
  	\begin{displaymath}
  		\begin{tikzpicture}[scheme, yshift=0.8em]
  			\draw (0,1) -- (1,1) -- (1,2) -- (0,2) -- (0,0) -- (4,0) -- (4,2) -- (5,2) (6,2) -- (7,2) -- (7,1) -- (4,1) (3,1) -- (2,1) -- (2,2) -- (4,2) (3,2) -- (3,0);
  			\draw	(0.5,1.5) node {$\bar J_1$}
  						(1.5,0.5) node {$\eta$}
  						(1.5,1.5) node[xshift=0.5pt] {$\dotsb$}
  						(2.5,1.5) node {$\bar J_n$}
  						(3.5,1) node {$\bar l$}
  						(5.5,1.5) node {$T\phi'$}
  						(5.5,2) node[xshift=0.5pt] {$\dotsb$};
  		\end{tikzpicture} \mspace{12mu} = \mspace{12mu} \begin{tikzpicture}[scheme, yshift=0.8em]
  			\draw (0,1) -- (1,1) -- (1,2) -- (0,2) -- (0,0) -- (7,0) -- (7,2) -- (5,2) -- (5,1) -- (6,1) (6,0) -- (6,2) (2,1) -- (2,2) -- (4,2) -- (4,1) -- (2,1) (3,0) -- (3,2);
  			\draw	(0.5,1.5) node {$\bar J_1$}
  						(1.5,0.5) node {$\eta$}
  						(1.5,1.5) node[xshift=0.5pt] {$\dotsb$}
  						(2.5,1.5) node {$\bar J_n$}
  						(3.5,1.5) node {$\bar H_1$}
  						(4.5,0.5) node {$\phi'$}
  						(4.5,1.5) node[xshift=0.5pt] {$\dotsb$}
  						(5.5,1.5) node {$\bar H_m$}
  						(6.5,1) node {$\bar k$};
  		\end{tikzpicture}
  	\end{displaymath}
  	
  	\emph{Part (c): `weak' means `pseudo'.} In this case $\map d{A_0}M$ is equipped with an invertible $T$-morphism structure cell $\cell{\bar d}{m \of Td}{d \of a_0}$. To see that any two of (e), (p) and (l) imply the third, first notice that the locally full inclusions \mbox{$\wAlg\lax\lax T \supset \wAlg\lax\psd T \subset \wAlg\lax\clx T$} reflect cells defining (weak) left Kan extensions. Assuming (e) and (p), we may apply (a) to find that the (weak) left Kan extension $\map{(l, \bar l)}{A_n}M$ of $d$ along $(J_1, \dotsc, J_n)$ is created by $\map U{\wAlg\lax\lax T}\K$, with structure cell $\bar l$ defined by the factorisation \eqref{lifted lax structure}. Using the fact that $(\bar J_1, \dotsc, \bar J_n)$ is (weak) left $d$-exact, by (e), and that $\bar d$ is invertible, it follows that the left-hand side of this factorisation defines a left-hand Kan extension, from which we conclude that $\bar l$ is invertible. We conclude that its defining $T$-cell $\eta$ defines $(l, \bar l)$ as a (weak) left Kan extension in $\wAlg\lax\psd T$, that is condition (l) holds.
  	\begin{displaymath}
  		\begin{tikzpicture}[scheme]
  			\draw	(1,1) -- (0,1) -- (0,3) -- (1,3) -- (1,0) -- (2,0) -- (2,3) -- (3,3) -- (3,1) -- (2,1) (1,2) -- (3,2) (2,0) -- (5,0) -- (5,3) -- (4,3) -- (4,1) -- (5,1) (4,2) -- (5,2);
  			\draw	(0.5,2) node {$T\bar d$}
  						(1.5,1) node {$\bar d$}
  						(2.5,1.5) node {$\bar J_1$}
  						(2.5,2.5) node {$T\bar J_1$}
  						(3.5,0.5) node {$\eta$}
  						(3.5,1.5) node[xshift=0.5pt] {$\dotsb$}
  						(3.5,2.5) node[xshift=0.5pt] {$\dotsb$}
  						(4.5,1.5) node {$\bar J_n$}
  						(4.5,2.5) node {$T\bar J_n$};
  		\end{tikzpicture} \mspace{12mu} = \mspace{12mu} \begin{tikzpicture}[scheme]
  			\draw	(1,3) -- (0,3) -- (0,2) -- (3,2) (3,3) -- (3,1) -- (4,1) (2,3) -- (4,3) -- (4,0) -- (5,0) -- (5,2) -- (4,2);
  			\draw	(1.5,2.5) node {$T^2\eta$}
  						(1.5,3) node[xshift=0.5pt] {$\dotsb$}
  						(3.5,2) node {$T\bar l$}
  						(4.5,1) node {$\bar l$};
  		\end{tikzpicture}
  	\end{displaymath}
  	
  	Next assume that both conditions (p) and (l) hold. From (l) it follows that the (weak) left Kan extension $\map l{A_0}M$ of $d$ along $(J_1, \dotsc, J_n)$ in $\K$ admits an invertible $T$-morphism structure cell that is unique in making its defining cell $\eta$ into a $T$-cell in $\wAlg\lax\psd T$. Of course, the same structure cell makes $\eta$ into a $T$"/cell in $\wAlg\lax\lax T$, so that it must coincide with the structure cell $\bar l$ that we obtain by applying part (a), as the unique factorisation in \eqref{lifted lax structure}. Thus, $\bar l$ in the right-hand side of \eqref{lifted lax structure} is invertible, and it follows that the full right-hand side defines a (weak) left Kan extension. Composing both sides on the left with $\inv{\bar d}$ we find that $\eta \of (J_1, \dotsc, J_n)$, in the left-hand side, defines a (weak) left Kan extension too, showing that $(J_1, \dotsc, J_n)$ is (weak) left $d$-exact. This proves the first half of (e); for the second half, that is $(T\bar J_1, \dotsc, T\bar J_n)$ is (weak) left $(d \of a_0)$-exact, a similar argument can applied to the identity above, which is itself a consequence of \eqref{lifted lax structure}. This shows that (p) and (l) together imply (e).
  	\begin{displaymath}
  		\begin{tikzpicture}[scheme]
  			\draw	(1,2) -- (0,2) -- (0,0) -- (1,0) -- (1,3) -- (3,3) (1,1) -- (2,1) -- (2,3) (2,2) -- (5,2) -- (5,3) -- (4,3);
  			\draw	(0.5,1) node {$\bar d$}
  						(1.5,2) node {$T\bar d$}
  						(3.5,2.5) node {$T^2\eta$}
  						(3.5,3) node[xshift=0.5pt] {$\dotsb$};
  		\end{tikzpicture} \mspace{12mu} = \mspace{12mu} \begin{tikzpicture}[scheme]
  			\draw	(0,1) -- (1,1) -- (1,3) -- (0,3) -- (0,0) -- (3,0) -- (3,3) -- (2,3) -- (2,1) -- (3,1) (0,2) -- (1,2) (2,2) -- (4,2) (3,0) -- (4,0) -- (4,3) -- (5,3) -- (5,1) -- (4,1);
  			\draw	(0.5,1.5) node {$\bar J_1$}
  						(0.5,2.5) node {$T\bar J_1$}
  						(1.5,0.5) node {$\eta$}
  						(1.5,1.5) node[xshift=0.5pt] {$\dotsb$}
  						(1.5,2.5) node[xshift=0.5pt] {$\dotsb$}
  						(2.5,1.5) node {$\bar J_n$}
  						(2.5,2.5) node {$T\bar J_n$}
  						(3.5,1) node {$\bar l$}
  						(4.5,2) node {$T\bar l$};
  		\end{tikzpicture}
  	\end{displaymath}
  	
  	The argument showing that (e) and (l) together imply (p) is horizontally dual to the previous one: in this case the unique factorisation $\bar l$ in \eqref{lifted colax structure} is invertible; it follows that $m \of T\eta$ in its left-hand side defines a (weak) left Kan extension. Hence the algebraic structure of $M$ preserves the (weak) left Kan extension of $d$ along $(J_1, \dotsc, J_n)$; that is the first half of (p) holds. That the (weak) left Kan extension of $m \of Td$ along $(TJ_1, \dotsc, TJ_n)$ is preserved as well follows from applying a similar argument to the identity above, which itself is a consequence of \eqref{lifted colax structure}.
  	
  	Finally we consider the creation of pointwise left Kan extensions by the forgetful functor $\map U{\lbcwAlg\lax T}\K$. Notice that, in the `pseudo'-case above, the proof that (p) and (e) implies (l) only uses the first part of (e), namely the left $d$-exactness of $(\bar J_1, \dotsc, \bar J_n)$. Analogously, if $(J_1, \dotsc, J_n)$ is a path in $\lbcwAlg\lax T$, so that $(\bar J_1, \dotsc, \bar J_n)$ satisfies the left Beck-Chevalley condition by \lemref{concatenation of paths satisfying the Beck-Chevalley condition} and thus is pointwise left $d$-exact by \propref{Beck-Chevalley}, it follows that the left Kan extension $(l, \bar l)$ created by $\map U{\wAlg\lax\lax T}\K$, of $d$ along $(J_1, \dotsc, J_n)$, is a pseudo $T$-morphism provided that $l$ is a pointwise Kan extension in $\K$. Now notice that, because the restrictions $J_n(\id, f)$ are assumed to exist in $\K$ and be preserved by $T$, it follows from \propref{creating restrictions satisfying the left Beck-Chevalley condition} that the inclusion $\lbcwAlg\lax T \to \wAlg\lax\lax T$ preserves such restrictions. Together with the fullness of this inclusion, from this we conclude that $(l, \bar l)$ forms a pointwise left Kan extension in $\lbcwAlg\lax T$ too, which completes the proof.
	\end{proof}
	
	Next is the colax case. We will only sketch its proof in the lax case $w = \lax$, since it is similar to parts of that above and parts of that of the main result of \cite{Koudenburg15a}.
	\begin{theorem} \label{creating Kan extensions between colax algebras}
  	Let $T = (T, \mu, \iota)$ be a monad on an augmented virtual double category $\K$ and let `weak' mean either `colax', `lax' or `pseudo'. Given colax $T$-algebras $A_0, \dotsc, A_n$ and $M$, consider the following conditions on a path of horizontal $T$-morphisms $(A_0 \xbrar{J_1} A_1, \dotsc, A_{n'} \xbrar{J_n} A_n)$ and a weak vertical $T$-morphism $\map d{A_0}M$, where $m$ and $a_0$ denote the structure maps of $M$ and $A_0$:
  	\begin{enumerate}
  		\item[\textup{(p)}]	the algebraic structure of $M$ preserves the (weak) left Kan extension of $d$ along $(J_1, \dotsc, J_n)$, while both paths $(\mu_{J_1}, \dotsc, \mu_{J_n})$ and $(\iota_{J_1}, \dotsc, \iota_{J_n})$ are (weak) left $(m \of Td)$-exact;
  		\item[\textup{(e)}]	the path of structure cells $(\bar J_1, \dotsc, \bar J_n)$ is (weak) left $d$-exact, while both paths $(\mu_{J_1}, \dotsc, \mu_{J_n})$ and $(\iota_{J_1}, \dotsc, \iota_{J_n})$ are (weak) left $(d \of a_0)$-exact;
  		\item[\textup{(l)}]	the forgetful functor $\map U{\wAlg\clx wT}\K$ creates the (weak) left Kan extension of $d$ along $(J_1, \dotsc, J_n)$.
  	\end{enumerate}
  	The following hold:
  	\begin{enumerate}[label=\textup{(\alph*)}]
  		\item if `weak' means `lax' then \textup{(p)} implies \textup{(l)};
  		\item if `weak' means `colax' then \textup{(e)} implies \textup{(l)};
  		\item if `weak' means `pseudo' then any two of \textup{(p)}, \textup{(e)} and \textup{(l)} imply the third.
  	\end{enumerate}
  	Moreover if the restrictions $J_n(\id, f)$ exist in $\K$ and are preserved by $T$, for all $\map fB{A_n}$, then the result analogous to \textup{(a)} for pointwise (weak) left Kan extensions holds as well. In this case the forgetful functor $\map U{\lbcwAlg\clx T}\K$ too creates the pointwise left Kan extension of $d$ along $(J_1, \dotsc, J_n)$ (which now satisfy the left Beck-Chevalley condition), under the assumption of \textup{(p)}
  \end{theorem}
  \begin{proof}[Sketch of the proof of part (a): `weak' means `lax'.]
  	In this case $\map d{A_0}M$ comes with a lax $T$-morphism structure cell $\cell{\bar d}{m \of Td}{d \of a_0}$. Using the assumption (p) we obtain, just like in the proof of the previous theorem, a cell $\cell{\bar l}{m \of Tl}{l \of a_n}$ as the unique factorisation in \eqref{lifted lax structure}. The proof that $\bar l$ makes $l$ into a lax $T$-morphism, that is it satisfies the associativity and unit axioms, is essentially the same as the horizontal dual of that given in proof of Theorem 5.7(a) of \cite{Koudenburg15a}, except that the setting there is that of double categories. Briefly, the proof follows from \eqref{lifted lax structure}, the associativity and unit axioms for $\bar d$, those for $\bar J_1$, \dots, $\bar J_n$ (see \defref{horizontal T-morphism}), the naturality of $\mu$ and $\iota$, and the fact that the composites
  	\begin{displaymath}
  		 m \of \mu_M \of T^2\eta = m \of T\eta \of (\mu_{J_1}, \dotsc, \mu_{J_n}) \quad \text{and} \quad m \of \iota_M \of \eta = m \of T\eta \of (\iota_{J_1}, \dotsc, \iota_{J_n})
  	\end{displaymath}
  	define left Kan extensions, where the latter is a consequence of the assumption (p) and the left exactness assumption on the paths $(\mu_{J_1}, \dotsc, \mu_{J_n})$ and $(\iota_{J_1}, \dotsc, \iota_{J_n})$. As before the unique factorisation \eqref{lifted lax structure} shows that $\bar l$ is unique in making $\eta$ into a $T$-cell in $\wAlg\clx\lax T$. Since the type of algebraic structure on objects does not figure in proving that $\eta$, as a $T$-cell, defines $(l, \bar l)$ as a (pointwise) (weak) left Kan extension in $\wAlg\clx\lax T$, the argument given in the proof of \thmref{creating Kan extensions between lax algebras}(a) applies verbatim.
  \end{proof}
  
  For completeness we show that, in the case of pseudo $T$-algebras, the conditions of \thmref{creating Kan extensions between lax algebras} and \thmref{creating Kan extensions between colax algebras} are equivalent, as follows.
  \begin{lemma}
  	Let $T$ be a monad on an augmented virtual double category. Consider pseudo $T$"/algebras $A_0, \dotsc, A_n$ and $M$, a path $(A_0 \xbrar{J_1} A_1, \dotsc, A_{n'} \xbrar{J_n} A_n)$ of horizontal $T$"/morphisms, and a vertical morphism $\map d{A_0}M$. If the (pointwise) (weak) left Kan extension of $d$ along $(J_1, \dotsc, J_n)$ exists then the conditions \textup{(p)}, of \thmref{creating Kan extensions between lax algebras} and \thmref{creating Kan extensions between colax algebras} respectively, are equivalent, and so are their conditions \textup{(e)}.
  \end{lemma}
  \begin{proof}
  	We treat the case of left Kan extensions; the proof for the others is the same. Writing $\eta$ for the cell defining the left Kan extension of $d$ along $(J_1, \dotsc, J_n)$, notice that the first halves of the conditions (p) concide: they say that $m \of T\eta$ defines the left Kan extension of $m \of Td$ along $(TJ_1, \dotsc, TJ_n)$. Showing that their other halves are equivalent too means showing that $m \of Tm \of T^2\eta$ defines a left Kan extension if and only if both $m \of T\eta \of (\mu_{J_1}, \dotsc, \mu_{J_n})$ and $m \of T\eta \of (\iota_{J_1}, \dotsc, \iota_{J_n})$ do. Using the fact that the coherence cells $\cell{\bar m}{m \of Tm}{m \of \mu_M}$ and $\cell{\tilde m}{\id_M}{m \of \iota_M}$ are invertible, the latter follows from the identities below, which themselves are consequences of the interchange axiom (\lemref{horizontal composition}) and the naturality of $\mu$ and $\iota$ (see \defref{transformation}).
  	\begin{gather*}
  		\bar m \of T^2\eta = (\bar m \of T^2d) \hc (m \of \mu_M \of T^2\eta) = (\bar m \of T^2d) \hc \pars{m \of T\eta \of (\mu_{J_1}, \dotsc, \mu_{J_n})} \\
  		\tilde m \of \eta = (\tilde m \of d) \hc (m \of \iota_M \of \eta) = (\tilde m \of d) \hc \pars{m \of T\eta \of (\iota_{J_1}, \dotsc, \iota_{J_n})}
  	\end{gather*}
  	
  	The first halves of (e) coincide as well: they say that $\eta \of (J_1, \dotsc, J_n)$ defines the left Kan extension of $d \of a_0$ along $(TJ_1, \dotsc, TJ_n)$. Showing that their other halves are equivalent means showing that $\eta \of (\bar J_1 \of T\bar J_1, \dotsc, \bar J_n \of T\bar J_n)$ defines a left Kan extension if and only if both $\eta \of (\bar J_1 \of \mu_{J_1}, \dotsc, \bar J_n \of \mu_{J_n})$ and $\eta \of (\bar J_1 \of \iota_{J_1}, \dotsc, \bar J_n \of \iota_{J_1})$ do. Using the fact that the coherence cells $\bar a_0$, $\tilde a_0$, $\bar a_n$ and $\tilde a_n$ are invertible, the latter follows from the identities
  	\begin{gather*}
  		\bigpars{\eta \of (\bar J_1 \of T\bar J_1, \dotsc, \bar J_n \of T\bar J_n)} \hc (l \of \bar a_n) = (d \of \bar a_0) \hc \bigpars{\eta \of (\bar J_1 \of \mu_{J_1}, \dotsc, \bar J_n \of \mu_{J_n})}; \\
  		\eta \hc (l \of \tilde a_n) = (d \of \tilde a_0) \hc \bigpars{\eta \of (\bar J_1 \of \iota_{J_1}, \dotsc, \bar J_n \of \iota_{J_1})},
  	\end{gather*}
  	which themselves follow from the coherence axioms for the structure cells $\bar J_1, \dotsc, \bar J_n$; see \defref{horizontal T-morphism}. This completes the proof.
  \end{proof}
  
  \subsection{Cocomplete \texorpdfstring{$T$}T-algebras}
  In this final subsection we show how the creation of left Kan extensions can be used to ``lift'' cocomplete $T$-algebras along the forgetful functors $\map U{\wAlg v\lax T}\K$ and $\map U{\lbcwAlg vT}\K$, where $v \in \set{\clx, \lax, \psd}$. Given an ideal $\catvar S$ of left extension diagrams $M \xlar d A \xbrar J B$ in $\K$ (see \defref{cocompletion}), we denote for each $v \in \set{\clx, \lax, \psd}$ by
  \begin{displaymath}
  	\edpi{\catvar S}v\lax \dfn \set{(d, J) \mid (Ud, UJ) \in \catvar S}
  \end{displaymath}
  the ideal of left extension diagrams in $\wAlg vlT$ that is the preimage of $\catvar S$ under the forgetful functor $\map U{\wAlg v\lax T}\K$. Similarly we write
  \begin{displaymath}
  	\edpilbc{\catvar S}v \dfn \set{(d, J) \in \edpi{\catvar S}v\psd \mid \text{$J$ satisfies the left Beck-Chevalley condition}}
  \end{displaymath}
  for the restriction of $\edpi{\catvar S}v\psd$ to horizontal $T$-morphisms that satisfy the left Beck-Chevalley condition (see \defref{horizontal T-morphism}).
  \begin{proposition} \label{lifted cocomplete algebras}
  	Let $T = (T, \mu, \iota)$ be a monad on an augmented virtual equipment $\K$ that preserves unary restrictions. Let $\catvar S$ be a ideal of left extension diagrams in $\K$, and let `weak' mean either `colax' or `lax'. Consider a weak $T$-algebra $M = (M, m, \bar m, \tilde m)$ whose underlying object $M$ is $\catvar S$-cocomplete in $\K$ and whose algebraic structure preserves, for each $(d, J) \in \catvar S(M)$, the pointwise left Kan extension of $d$ along $J$. The following hold:
  	\begin{enumerate}[label=\textup{(\alph*)}]
  		\item if `weak' means `colax' then $M$ is both $\edpi{\catvar S}\clx\lax$-cocomplete in $\wAlg\clx\lax T$ and $\edpilbc{\catvar S}\clx$-cocomplete in $\lbcwAlg\clx T$ provided that, for each $(d, J) \in \catvar S(M)$, the cells $\mu_J$ and $\iota_J$ of are pointwise left $(m \of Td)$-exact;
  		\item if `weak' means `lax' then $M$ is both $\edpi{\catvar S}\lax\lax$-cocomplete in $\wAlg\lax\lax T$ as well as $\edpilbc{\catvar S}\lax$-cocomplete in $\lbcwAlg\lax T$ provided that its algebraic structure preserves, for each $(d, J) \in \catvar S(M)$, the pointwise left Kan extension of $m \of Td$ along $TJ$.
  	\end{enumerate}
  	In either case, for any $(d, J) \in \edpi{\catvar S}v\lax(M)$ (resp.\ $\edpilbc{\catvar S}v(M)$), the pointwise left Kan extension of $d$ along $J$ is created by the forgetful functor $\map U{\wAlg v\lax T}\K$ (resp.\ $\map U{\lbcwAlg vT}\K)$).
  	
  	Finally, any lax (resp.\ pseudo) $T$-morphism $\map{(f, \bar f)}MN$, between weak $T$"/algebras that both satisfy the appropriate condition above, is $\edpi{\catvar S}v\lax$-cocontinuous (resp.\ $\edpilbc{\catvar S}v$-cocontinuous) as soon as its underlying morphism $f$ is $\catvar S$-cocontinuous.
  \end{proposition}
  Notice that, as a consequence of $\K$ being an augmented virtual equipment and \propref{creating restrictions}, the inclusion $\wAlg\psd\lax T \to \wAlg\lax\lax T$ preserves restrictions of the form $J(\id, f)$ and hence reflects pointwise left Kan extensions. It follows that any pseudo $T$-algebra $M$ is $\edpi{\catvar S}\psd\lax$-cocomplete in $\wAlg\psd\lax T$ whenever it is $\edpi{\catvar S}\lax\lax$-cocomplete in $\wAlg\lax\lax T$. An analogous implication holds for $\edpilbc{\catvar S}v$-cocompleteness.
  \begin{proof}
  	In either case consider $(d, J) \in \edpi{\catvar S}v\lax$ (resp.\ $(d, J) \in \edpilbc{\catvar S}v$); we have to show that the pointwise left Kan extension of $d$ along $J$ exists in $\wAlg v\lax T$ (resp.\ $\lbcwAlg vT$). The $\catvar S$-cocompleteness of $M$ in $\K$ ensures that the pointwise left Kan extension of the morphisms underlying $d$ and $J$ exists in $\K$. Since the remaining assumptions on $M$ allow us to apply the final part of either \thmref{creating Kan extensions between lax algebras} or \thmref{creating Kan extensions between colax algebras}, it follows that the pointwise left Kan extension of $d$ along $J$ is created by the forgetful functor $\map U{\wAlg v\lax T}\K$ (resp.\ $\map U{\lbcwAlg vT}\K$).
  	
  	For the final assertion consider a lax $T$-morphism $\map{(f, \bar f)}MN$ between weak $T$-algebras, whose $\edpi{\catvar S}v\lax$-cocompleteness is lifted along $\map U{\wAlg v\lax T}\K$ in the above sense, and assume that $f$ is $\catvar S$-cocontinuous. Let $\bigpars{(d, \bar d), (J, \bar J)} \in \edpi{\catvar S}v\lax(M)$ and let $\map{(l, \bar l)}BM$ be the lax $T$-morphism that is the pointwise left Kan extension of $(d, \bar d)$ along $(J, \bar J)$ created by $U$; we have to show that $(f \of l, \overline{f \of l})$ forms a pointwise left Kan extension as well. But this is clear: $\bigpars{(f \of d, \overline{f \of d}), (J, \bar J)} \in \edpi{\catvar S}v\lax$ because the latter is an ideal, so that the pointwise left Kan extension of $(f \of d, \overline{f \of d})$ along $(J, \bar J)$ is created, and hence reflected, along the forgetful functor $U$. By the $\catvar S$"/cocontinuity of $f$ the composite $f \of l$ forms the pointwise left Kan extension of $f \of d$ along $J$ in $\K$, from which we conclude that $(f \of l, \overline{f \of l})$ forms that of $(f \of d, \overline{f \of d})$ along $(J, \bar J)$. Clearly the previous argument applies to the case of $\edpilbc{\catvar S}v$-cocontinuity as well; this completes the proof.
  \end{proof}
  
  \begin{example} \label{monoidal cocompleteness}
  	Let $\V \to \V'$ be a symmetric universe enlargement (\defref{universe enlargement}) and let $\catvar S$ be the ideal of left extension diagrams in $\enProf{(\V, \V')}$ described in \exref{V-small}, consisting of pairs $(d, J)$ where $\hmap JAB$ is a $\V$-profunctor between $\V$"/categories with $A$ small. Consider the `free strict monoidal $\V'$-category'"/monad $T = (T, \mu, \iota)$ on $\enProf{(\V, \V')}$, that was described in \exref{free strict monoidal V-category monad}, whose (co-)lax algebras are (op-)lax monoidal $\V'$-categories $M = (M, \oslash, \mathfrak a, \mathfrak i)$ (\exref{monoidal V-categories}).
  	
  	Using \propref{cocartesian cells in (V, V')-Prof} it is straightforward to show that, for any $\V$-profunctor $\hmap JAB$, the cells $\mu_J$ and $\iota_J$ satisfy the left Beck-Chevalley condition (\defref{left exact}), so that by \propref{Beck-Chevalley} they are pointwise left exact. Moreover, given an (op-)lax monoidal category $M$ it follows from the next example that, for each $(d, J) \in \catvar S$, its monoidal structure preserves both the pointwise left Kan extension of $d$ along $J$ as well as that of $m \of Td$ along $J$, provided that its tensor product $\oslash$ preserves small $\V$-weighted colimits in each variable. Thus, if moreover $M$ is small $\V$-cocomplete as a $\V$-category, by the proposition above it admits all monoidal pointwise left Kan extensions of (op-)lax monoidal $\V'$-functors $\map dAM$ with $A$ small. Moreover, if $d$ is a monoidal $\V'$-functor (with invertible compositors) then a pointwise left Kan extension of $d$ along a monoidal $\V$-profunctor $\hmap JAB$ is again monoidal whenever $J$ satisfies the left Beck-Chevalley condition; see \exref{monoidal V-profunctor}.
  	
  	In \cite{Im-Kelly86} a monoidal $\V$-category $M$ is called `monoidally cocomplete' if it is small cocomplete as a $\V$-category and its tensor product $\dash \oslash \dash$ preserves small $\V$-weighted colimits on both sides. Thus, in terms of the previous and in view of the proposition above, monoidal cocompleteness implies $\edpi{\catvar S}\psd\lax$-cocompleteness in $\wAlg\psd\lax T$.
  \end{example}
  
  The following example is adapted from Example 5.3 in \cite{Koudenburg15a}. Recall the tensor product on $\enProf{(\V, \V')}$, as described before \lemref{enriched left Kan extensions along tensor products}.
	\begin{example} \label{tensor products preserving weighted colimits in each variable}
		Given an (op-)lax monoidal $\V'$-category $M = (M, \oslash, \mathfrak a, \mathfrak i)$, assume that each of its $n$-any tensor products $\oslash_n \dfn \brks{M^{\tens' n} \subset TM \xrar\oslash M}$ preserves small $\V$-weighted colimits in each variable; of course if $M$ is a monoidal category, with invertible $\mathfrak a$ and $\mathfrak i$, then it suffices that its binary tensor product $\map{\oslash_2}{M \tens' M}M$ does so.
		
		Given any $\V$-profunctor $\hmap JAB$ with $A$ a small $\V$-category, we claim that the algebraic structure of $M$ preserves all pointwise left Kan extensions along $J$, in the sense of \defref{algebraic structure preserving left Kan extensions}. To see this let $\eta$ be any cell in $\enProf{(\V, \V')}$ that defines a pointwise left Kan extension along $J$ into $M$; we have to show that $m \of T\eta$, as in the bottom of the left-hand side below, again defines a pointwise left Kan extension. Equivalently by \propref{enriched left Kan extensions in terms of weighted colimits} we may show that, for each \mbox{$\ul y = (y_1, \dotsc, y_n) \in TB$}, the composite of the bottom three cells below defines $(ly_1 \oslash \dotsb \oslash ly_n)$ as a $TJ(\id, \ul y)$"/weighted colimit of $m \of Td$.
		\begin{displaymath}

		\end{displaymath}
		Next we will describe a cell $\chi$, as in the left hand side above, that is right pointwise cocartesian, so that by the vertical pasting lemma (\lemref{vertical pasting lemma}), we may equivalently show that the whole composite on the left defines $(ly_1 \oslash \dotsb \oslash ly_n)$ as a weighted colimit. At the sequences of objects $\ul x_0 \in A^{\tens' n}, \ul x_1 \in A^{\tens' n'}, \dotsc, x_{n'} \in A$ the component of $\chi$ is given by
		\begin{align*}
			\sideset{}{'}\Tens_{i=1}^{n'}& A(x_{0i}, x_{1i}) \tens' J(x_{0n}, y_n) \tens' \sideset{}{'}\Tens_{i=1}^{n''} A(x_{1i}, x_{2i}) \tens' J(x_{1n'}, y_{n'}) \tens' \dotsb \tens' J(x_{n'}, y_1) \\
			&\mspace{-18mu}\iso \sideset{}{'}\Tens_{i=1}^n A(x_{0i}, x_{1i}) \tens' \dotsb \tens' A(x_{(n-i)'i}, x_{(n-i)i}) \tens' J(x_{(n-i)i}, y_i) \to \sideset{}{'}\Tens_{i=1}^n J(x_{0i}, y_i)
		\end{align*}
		where the isomorphism reorders the factors and the unlabelled map is given by the action of $A$ on $J$. Using \propref{cocartesian cells in (V, V')-Prof}, checking that $\chi$ is right pointwise cocartesian is straightforward.
		
		Finally consider the right-hand side above, where the nullary cells $\cell\delta{I_A}M$ are simply given by the action of $d$ on the hom-sets of $A$, and where each cell $\eta_i$ denotes the restriction of $\eta$ along $J(\id, y_i)$. That the two sides coincide follows from the equivariance of the cells $\eta_i$, with respect to the actions of $A$. Hence, using the horizontal pasting lemma (\lemref{horizontal pasting lemma}), we conclude that the claim follows if each of the cells in the right-hand side defines a pointwise left Kan extension. That that is the case, finally, follows easily from \lemref{enriched left Kan extensions along tensor products} and the assumption that the tensor products of $M$ preserve $J(\id, y_i)$-weighted colimits in each variable.
	\end{example}

  \section{Lifting algebraic yoneda embeddings}\label{algebraic yoneda embeddings section}
  The results of this section describe the lifting of yoneda embeddings along the forgetful functors $\map U{\wAlg vwT}\K$, for any monad $T$ on an augmented virtual double category $\K$. Stating and proving these results form the main motivation of this paper. Recall that an (co"/)lax $T$-algebra $A = (A, a, \bar a, \tilde a)$ is called normal if its unitor $\tilde a$ is invertible. Being slightly more natural, we start with the case of colax algebras.
  \begin{theorem} \label{lifting colax presheaf objects}
  	Let $T = (T, \mu, \iota)$ be a monad on an augmented virtual double category $\K$. Consider a colax $T$-algebra $A = (A, a, \bar a, \tilde a)$ and a good yoneda embedding $\map\yon A{\ps A}$ in $\K$. Assume that
  	\begin{enumerate}[label=-]
  		\item the conjoint $a^*$ exists;
  		\item the right pointwise composite $(a^* \hc Ty_*)$ exists (\defref{pointwise cocartesian path});
  		\item the cell $\mu_J$ is pointwise left $(\yon \of a)$-exact for each $\hmap JAB$ (\defref{left exact});
  		\item the cell $\mu_{T\yon_*}$ is pointwise left $(\yon \of a \of \mu_A)$-exact, while $\iota_{\yon_*}$ is pointwise left $(\yon \of a)$-exact;
  		\item $T$ preserves any unary cartesian cell with $\yon_*$ as horizontal target.
  	\end{enumerate}
  	The morphism $\map{v \dfn \cur{(a^* \hc T\yon_*)}}{T\ps A}{\ps A}$, that is given by the yoneda axiom (\defref{yoneda embedding}), and that comes equipped with a cartesian cell
  	\begin{displaymath}
  		\begin{tikzpicture}
				\matrix(m)[math35, column sep={2.75em,between origins}]{A & & T\ps A, \\ & \ps A & \\};
				\path[map]	(m-1-1) edge[barred] node[above] {$(a^* \hc T\yon_*)$} (m-1-3)
														edge node[below left] {$\yon$} (m-2-2)
										(m-1-3) edge node[below right] {$v$} (m-2-2);
				\draw				([yshift=0.333em]$(m-1-2)!0.5!(m-2-2)$) node[font=\scriptsize] {$\cart$};
			\end{tikzpicture}
  	\end{displaymath}
  	extends to a colax $T$-algebra structure $(v, \bar v, \tilde v)$ on $\ps A$. With respect to this structure $\map\yon A{\ps A}$ admits a pseudo $T$-morphism structure that makes it into a good yoneda embedding both in $\wAlg\clx\lax T$ and in $\lbcwAlg\clx T$.
  	
  	The following hold:
  	\begin{enumerate}[label=\textup{(\alph*)}]
  		\item	$\ps A$ is normal precisely if $A$ is;
  		\item assuming that $\bar a$ is invertible, the associator $\bar v$ is invertible if and only if the $T$-image of the cocartesian cell defining $(a^* \hc T\yon_*)$ is left $(\yon \of a)$-exact (e.g.\ if the cocartesian cell is preserved by $T$).
  	\end{enumerate}
  	If the conditions above are satisfied, so that the lift of $\yon$ forms a pseudo $T$-morphism between pseudo $T$-algebras, then $\yon$ forms a good yoneda embedding in $\wAlg\psd\lax T$ and $\lbcwAlg\psd T$ as well.
  \end{theorem}
  
  \begin{example} \label{Day convolution}
    Let $\V \to \V'$ be a symmetric universe enlargement, such that $\V$ has an initial object that is preserved by its tensor product $\tens$ on both sides, and let $T = (T, \mu, \iota)$ be the `free strict monoidal $\V'$-category'-monad on $\enProf{(\V, \V')}$, that was described in \exref{free strict monoidal V-category monad}. Remember that $T$ is strong and preserves cocartesian cells, while it is easily checked to preserve all cartesian cells. That the cells $\mu_J$ and $\iota_J$, for each $\V$"/profunctor $\hmap JAB$, are pointwise left exact was briefly described in \exref{monoidal cocompleteness}.
    
    Thus given a monoidal $\V$-category $A = (A, \oslash, \mathfrak a, \mathfrak i)$ (\exref{monoidal V-categories}), whose yoneda embedding $\map\yon A{\ps A}$ exists by \propref{enriched yoneda embeddings}, we may apply the theorem above as soon as the $\V'$-coends on the right below are $\V$-objects; e.g.\ in the case that $A$ is small and $\V$ is small cocomplete.
    \begin{displaymath}
      (p_1 \oslash \dotsb \oslash p_n)(x) \iso \int^{u_1, \dotsc, u_n \in A} A\bigpars{x, (u_1 \oslash \dotsb \oslash u_n)} \tens' p_1u_1 \tens' \dotsb \tens' p_nu_n
    \end{displaymath}
    In that case the monoidal structure on $A$ lifts to one on $\ps A$, which we will again denote by $\map\oslash{TA}A$. Its defining cartesian cell, in the theorem above, ensures that its action on $\V$-presheaves $p_1, \dotsc, p_n$ on $A$ is given by the coends above, as shown; that is $(p_1 \oslash \dotsb \oslash p_n)$ is the \emph{Day convolution} of the presheaves $p_1, \dotsc, p_n$, as introduced (in a biased form) in Section 4 of \cite{Day70}.
    
    With respect to the monoidal structure on $\ps A$ above, the yoneda embedding $\map\yon A{\ps A}$ lifts to form a pseudomonoidal functor between $A$ and $\ps A$ such that it becomes a good yoneda embedding in the augmented virtual double categories $\wAlg\psd\lax T$ and $\lbcwAlg\psd T$, of monoidal $\V$-profunctors and monoidal $\V$-profunctors that satisfy the left Beck-Chevalley condition (see \exref{monoidal V-profunctor}). In particular it satisfies the following \emph{monoidal Yoneda's lemma}: for any monoidal $\V$-profunctor $\hmap JAB$ there exists a lax monoidal functor $\map{\cur J}B{\ps A}$ such that $J \iso \ps A(\yon, \cur J)$ as monoidal $\V$-profunctors. On objects $\cur J$ is given by $\cur Jy \dfn J(\dash, y)$, while is lax monoidal structure $\nat{\cur J_\oslash}{(\cur Jy_1 \oslash \dotsb \oslash \cur Jy_n)}{\cur J(y_1 \oslash \dotsb \oslash y_n)}$ is induced by the $\V'$-maps below; it is invertible precisely when $J$ satisfies the left Beck-Chevalley condition.
    \begin{displaymath}
      \begin{tikzpicture}
        \matrix(m)[math35]
          { A\bigpars{x, (u_1 \oslash \dotsb \oslash u_n)} \tens' J(u_1, y_1) \tens' \dotsb \tens' J(u_n, y_n) \\
            A\bigpars{x, (u_1 \oslash \dotsb \oslash u_n)} \tens' J\bigpars{(u_1 \oslash \dotsb \oslash u_n), (y_1 \oslash \dotsb \oslash y_n)} \\
            J\bigpars{x, (y_1 \oslash \dotsb \oslash y_n)} \\ };
        \path[map]  (m-1-1) edge node[left] {$\id \tens' J_\oslash$} (m-2-1)
                    (m-2-1) edge node[left] {$\lambda$} (m-3-1);
      \end{tikzpicture}
    \end{displaymath}
  \end{example}
  
  \begin{proof}[Proof of \thmref{lifting colax presheaf objects}]
  	The proof consists of the following steps: complete the definition of the colax $T$-algebra structure on $\ps A$, define the pseudo $T$-morphism structure on $\map\yon A{\ps A}$, show that $(\yon, \bar\yon)$ forms a good yoneda embedding in $\wAlg\clx\lax T$ and $\lbcwAlg\clx T$ and, finally, prove the assertions (a) and (b). It is useful to abbreviate by $\eta$ the composite on the left below, and by $\eta'$ its factorisation through the cartesian cell defining $\yon_*$, as shown.
  	\begin{equation} \label{eta}
  		\eta \dfn \begin{tikzpicture}[textbaseline]
  			\matrix(m)[math35, column sep={1.75em,between origins}]
  				{ & TA & & T\ps A & \\
  					A & & TA & & T\ps A \\
  					A & & & & T\ps A \\
  					& & \ps A & & \\ };
  			\path[map]	(m-1-2) edge[barred] node[above] {$T\yon_*$} (m-1-4)
  													edge[transform canvas={xshift=-2pt}] node[left] {$a$} (m-2-1)
  									(m-2-1) edge[barred] node[below] {$a^*$} (m-2-3)
  									(m-2-3) edge[barred] node[above] {$T\yon_*$} (m-2-5)
  									(m-3-1) edge[barred] node[below, inner sep=1pt] {$(a^* \hc T\yon_*)$} (m-3-5)
  													edge node[below left] {$\yon$} (m-4-3)
  									(m-3-5) edge node[below right] {$v$} (m-4-3);
  			\path				(m-1-2) edge[eq, transform canvas={xshift=1pt}] (m-2-3)
  									(m-1-4) edge[eq, transform canvas={xshift=1pt}] (m-2-5)
  									(m-2-1) edge[eq] (m-3-1)
  									(m-2-5) edge[eq, ps] (m-3-5);
  			\draw[font=\scriptsize]	([shift={(-1pt,-0.5em)}]$(m-1-2)!0.5!(m-2-2)$) node {$\cocart$}
  									($(m-2-3)!0.5!(m-3-3)$) node {$\cocart$}
  									([yshift=0.25em]$(m-3-3)!0.5!(m-4-3)$) node {$\cart$};
  		\end{tikzpicture} = \begin{tikzpicture}[textbaseline]
  			\matrix(m)[math35, column sep={1.75em,between origins}]
  				{ TA & & T\ps A \\
  					A & & \ps A \\
  					& \ps A & \\ };
  			\path[map]	(m-1-1) edge[barred] node[above] {$T\yon_*$} (m-1-3)
  													edge node[left] {$a$} (m-2-1)
  									(m-1-3) edge[ps] node[right] {$v$} (m-2-3)
  									(m-2-1) edge[barred] node[below] {$\yon_*$} (m-2-3)
  													edge[transform canvas={xshift=-1pt}] node[left] {$\yon$} (m-3-2);
  			\path				(m-2-3) edge[transform canvas={xshift=2pt}, eq, ps] (m-3-2)
  									(m-1-2) edge[cell] node[right] {$\eta'$} (m-2-2);
  			\draw[font=\scriptsize]	([yshift=0.25em]$(m-2-2)!0.5!(m-3-2)$) node {$\cart$};
  		\end{tikzpicture}
  	\end{equation}
  	We claim that $\eta$ defines $v$ as a pointwise left Kan extension. To see this first notice that, by \propref{Kan extension along conjoints}, we may equivalently prove that its composition with the cartesian cell defining $a^*$ does which, by the conjoint identities (horizontally dual to those in \lemref{companion identities lemma}), coincides with the composite of the bottom two rows on the left above. To see that the latter composite defines a pointwise left Kan extension remember that $\yon$ is dense (\lemref{density axioms}), so that the bottom cartesian cell defines a pointwise left Kan extension, followed by applying \corref{Kan extensions and composites} to the cocartesian cell defining $(a^* \hc T\yon_*)$, which is right pointwise by assumption.
  	
  	\emph{Step 1: the colax $T$-algebra structure on $\ps A$.} Having obtained a structure morphism $\map v{T\ps A}{\ps A}$, it remains to define associator and unitor cells $\cell{\bar v}{v \of \mu_{\ps A}}{v \of Tv}$ and $\cell{\tilde v}{v \of \iota_{\ps A}}{\id_{\ps A}}$, and show that $(\ps A, v, \bar v, \tilde v)$ satisfies the usual coherence axioms for colax algebras; see \eqref{colax algebra axioms} below. In order to define the associator $\bar v$ notice that the composite $\eta \of \mu_{\yon_*}$ in the right-hand side below defines $v \of \mu_{\ps A}$ as a left Kan extension, since $\mu_{\yon_*}$ is assumed pointwise left $(\yon \of a)$-exact; we take $\bar v$ to be the unique factorisation of the left-hand side through this composite, as shown.
  	\begin{equation} \label{bar v}

  	\end{equation}
  	Analogously the composite $\eta \of \iota_{\yon_*}$ in the right-hand side below defines $v \of \iota_{\ps A}$ as a left-Kan extension, by the assumption that $\iota_{\yon_*}$ is pointwise left $(\yon \of a)$-exact; we take the unitor $\tilde v$ to be the factorisation as shown.
  	\begin{equation} \label{tilde v}
  		%
  	\end{equation}  	
  	We claim that $(v, \bar v, \tilde v)$ forms a well"/defined colax $T$"/structure on $\ps A$, that is $\bar v$ and $\tilde v$ satisfy the associativity and unit axioms, which are drawn schematically below; here, as always, empty cells denote identity cells.
  	\begin{equation} \label{colax algebra axioms}
  		%
 \mspace{4mu} = \id_v
  	\end{equation}
  	To prove the associativity axiom, on the left above, consider the equation below, whose identities follow the associativity axiom for $T$; the identity \eqref{bar v}; the $T$-image of \eqref{bar v} after both of its sides have been factorised through the cartesian cell defining $\yon_*$; the associativity axiom for $A$; the identity \eqref{bar v} again; the naturality of $\mu$; the identity \eqref{bar v} once more. Now notice that the left and right-hand sides below coincide with both sides of the associativity axiom for $\ps A$ after composition on the left with $\eta \of \mu_{\yon_*} \of \mu_{T\yon_*}$. By the exactness assumptions on $\mu_{\yon_*}$ and $\mu_{T\yon_*}$, the latter composite defines $v \of \mu_{\ps A} \of \mu_{T\ps A}$ as a left Kan extension, so that the associativity axiom follows.
  	\begin{align*}

  	\end{align*}
  	To prove the first unit axiom, in the middle of \eqref{colax algebra axioms}, consider the following equation, where `c' denotes the cartesian cell defining $\yon_*$, and whose identities follow from the identity \eqref{bar v}; the naturality of $\iota$; the identity \eqref{tilde v}; the first unit axiom for $A$; the identity \eqref{eta} and the unit axiom for $T$. Notice that the left and right-hand sides below coincide with the first unit axiom for $\ps A$ composed on the left with $\eta \of \mu_{\yon_*} \of \iota_{T\yon_*} = \eta$. Since the latter defines a left Kan extension the unit axiom follows.
  	\begin{align*}

  	\end{align*}
  	This leaves the second unit axiom, on the right of \eqref{colax algebra axioms}. Analogous to the arguments above, it follows from the equation below, whose identities follow from the unit axiom for $T$; the identity \eqref{bar v}; the $T$-image of \eqref{tilde v} after factorising it through the cartesian cell defining $\yon_*$; the second unit axiom for $A$. This completes the definition of the colax $T$-algebra $\ps A = (\ps A, v, \bar v, \tilde v)$.
  	\begin{align*}
  		%
  	\end{align*}
  	
  	\emph{Step 2: the pseudo $T$-structure on $\map\yon A{\ps A}$.} Consider the composite $\gamma$ on the left below; we claim that it makes $\map\yon A{\ps A}$ into a colax $T$-morphism, that is it satisfies the associativity and unit axioms. Indeed: these follow directly from composing both sides of the identities \eqref{bar v} and \eqref{tilde v} with the ($T^2$-image of) the cocartesian cell that defines $\yon_*$, by using the naturality of $\mu$ and $\iota$ together with the fact that $\eta' = (\cocart \of a) \hc \eta$ (for the latter notice that composing either side with the cartesian cell defining $\yon_*$ results in $\eta$).
  	\begin{equation} \label{structure cell for yoneda embedding}
  		\gamma \dfn \begin{tikzpicture}[textbaseline]
  			\matrix(m)[math35, column sep={0.875em,between origins}]
  				{	& & TA & & \\
  					TA & & & & T\ps A \\
  					& A & & & \\
  					& & \ps A & & \\ };
  			\path[map]	(m-1-3) edge[transform canvas={xshift=4pt}, ps] node[right] {$T\yon$} (m-2-5)
  									(m-2-1) edge[barred] node[below] {$T\yon_*$} (m-2-5)
  													edge[transform canvas={xshift=-1pt}] node[left] {$a$} (m-3-2)
  									(m-2-5) edge[transform canvas={xshift=1pt}, ps] node[right] {$v$} (m-4-3)
  									(m-3-2) edge[transform canvas={xshift=-1pt}, ps] node[left] {$\yon$} (m-4-3);
  			\path				(m-1-3) edge[transform canvas={xshift=-5pt}, eq] (m-2-1)
  									(m-2-3) edge[cell, transform canvas={yshift=-0.5em}] node[right] {$\eta$} (m-3-3);
  			\draw[font=\scriptsize] ([yshift=-0.6em]$(m-1-3)!0.5!(m-2-3)$) node {$T\!\cocart$};
  		\end{tikzpicture} \qquad\qquad\qquad\qquad \id_{\yon} = \begin{tikzpicture}[textbaseline]
	    	\matrix(m)[math35, column sep={1.75em,between origins}]{& A & \\ A & & A \\ A & & \ps A \\ & \ps A & \\};
  	  	\path[map]	(m-2-1) edge[barred] node[below, inner sep=2pt] {$\ps A(\yon, \yon)$} (m-2-3)
  	  							(m-2-3) edge[ps] node[right] {$\yon$} (m-3-3)
  	  							(m-3-1) edge[barred] node[below, inner sep=2.5pt] {$\yon_*$} (m-3-3)
  	  											edge[transform canvas={xshift=-2pt}] node[left] {$\yon$} (m-4-2);
  	  	\path				(m-3-3)	edge[eq, ps, transform canvas={xshift=1pt}] (m-4-2)
  	  							(m-1-2) edge[cell, transform canvas={shift={(-0.5em,-0.5em)}}] node[right] {$\id_{\yon}'$} (m-2-2)
  	  							(m-1-2) edge[eq, transform canvas={xshift=-2pt}] (m-2-1)
  	  											edge[eq, transform canvas={xshift=2pt}] (m-2-3)
  	  							(m-2-1) edge[eq] (m-3-1);
  	  	\draw				([yshift=-0.25em]$(m-2-2)!0.5!(m-3-2)$) node[font=\scriptsize] {$\cart$}
  	  							([yshift=0.25em]$(m-3-2)!0.5!(m-4-2)$) node[font=\scriptsize] {$\cart$};
  		\end{tikzpicture}
  	\end{equation}
  	It remains to show that $\gamma$ is invertible. Since the cell $\eta$, in the composite $\gamma$, defines $v$ as a pointwise left Kan extension it suffices to prove that $T\yon$ is full and faithful, by \propref{pointwise left Kan extension along full and faithful map}. To this end consider the the factorisation of $\id_{\yon}$ on the right above and notice that both cartesian cells here are preserved by $T$: the bottom one by \corref{functors preserve companions and conjoints} and the one in the middle row by assumption. Because $\yon$ is full and faithful (\lemref{yoneda embedding full and faithful}) the top cell $\id_{\yon}'$ is cocartesian by \corref{unit in terms of full and faithful map}; it is preserved by $T$ as well by \corref{functors preserve horizontal units}. We conclude that $\id_{T\yon}$ factors as a cocartesian cell through the restriction $T\ps A(\yon, \yon)$, so that $T\yon$ is full and faithful by applying \corref{unit in terms of full and faithful map} once more. We write $\bar\yon \dfn \inv\gamma$, completing the definition of $\map{(\yon, \bar \yon)}A{\ps A}$ as a pseudo $T$-morphism.

  	\emph{Step 3: $\map{(\yon, \bar\yon)} A{\ps A}$ forms a good yoneda embedding in $\wAlg\clx\lax T$ and in $\lbcwAlg\clx T$.} To complete the proof we show that, with the $T$-algebra structures $(v, \bar v, \tilde v)$ and $\bar\yon$ on $\ps A$ and $\yon$ above, $(\yon, \bar\yon)$ defines a good yoneda embedding both in $\wAlg\clx\lax T$ and in $\lbcwAlg\clx T$. First notice that, because $(\yon, \bar\yon)$ is a pseudo $T$"/morphism, for any lax $T$-morphism $\map gB{\ps A}$ the restriction $\ps A(\yon, g)$ in $\wAlg\clx\lax T$ is created by the forgetful functor $\map U{\wAlg\clx\lax T}\K$, by \propref{creating restrictions}(b). Moreover, as follows from the proof of the latter combined with \eqref{eta}, the $T$"/structure cell of $\yon_* = \ps A(\yon, \id)$ coincides with the factorisation $\eta'$ of $\eta$ through $\yon_*$. Because $\eta$ defines $v$ as a pointwise left Kan extension and $\yon$ is dense (\lemref{density axioms}), $\eta'$ is pointwise left $\yon$-exact. It follows from \propref{Beck-Chevalley} that $\eta'$, and thus $\yon_*$ by definition, satisfies the left Beck"/Chevalley condition. Combining this with \propref{creating restrictions satisfying the left Beck-Chevalley condition} we find that restrictions $\ps A(\yon, g)$ in $\lbcwAlg\clx T$, where now $g$ is a pseudo $T$"/morphism, are created by $\map U{\lbcwAlg\clx T}\K$ as well. Thus, once we have shown that $(\yon, \bar\yon)$ forms a yoneda embedding both in $\wAlg\clx\lax T$ and in $\lbcwAlg\clx T$, it forms in fact a good yoneda embedding, in the sense described after \defref{yoneda embedding}.
  	
  	To show that $(\yon, \bar\yon)$ satisfies the yoneda axiom (\defref{yoneda embedding}), consider a horizontal $T$-morphism $\hmap JAB$. By the yoneda axiom for $\yon$ in $\K$ there exists a cartesian cell as on the left below; we have make it into a cartesian cell in $\wAlg\clx\lax T$ and $\lbcwAlg\clx T$, by supplying a lax $T$-morphism structure for $\map{\cur J}B{\ps A}$.
  	\begin{equation} \label{cart'}
			\begin{tikzpicture}[textbaseline]
				\matrix(m)[math35, column sep={1.75em,between origins}]{A & & B \\ & \ps A & \\};
				\path[map]	(m-1-1) edge[barred] node[above] {$J$} (m-1-3)
														edge[transform canvas={xshift=-2pt}] node[left] {$\yon$} (m-2-2)
										(m-1-3) edge[transform canvas={xshift=2pt}] node[right] {$\cur J$} (m-2-2);
				\draw				([yshift=0.333em]$(m-1-2)!0.5!(m-2-2)$) node[font=\scriptsize] {$\cart$};
			\end{tikzpicture} = \begin{tikzpicture}[textbaseline]
  			\matrix(m)[math35, column sep={1.75em,between origins}]
  				{ A & & B \\
  					A & & \ps A \\
  					& \ps A & \\ };
  			\path[map]	(m-1-1) edge[barred] node[above] {$J$} (m-1-3)
  									(m-1-3) edge[ps] node[right] {$\cur J$} (m-2-3)
  									(m-2-1) edge[barred] node[below] {$\yon_*$} (m-2-3)
  													edge[transform canvas={xshift=-1pt}] node[left] {$\yon$} (m-3-2);
  			\path				(m-1-1) edge[eq] (m-2-1)
  									(m-2-3) edge[transform canvas={xshift=2pt}, eq, ps] (m-3-2);
  			\draw[font=\scriptsize]	([yshift=0.25em]$(m-2-2)!0.5!(m-3-2)$) node {$\cart$}
  									($(m-1-2)!0.5!(m-2-2)$) node {$\cart'$};
  		\end{tikzpicture}
	  \end{equation}
	  Writing $\cart'$ for the factorisation as shown above, remember that it is again cartesian by the pasting lemma, and that it is preserved by $T$ by assumption. Because $\eta$ defines a pointwise left Kan extension, it follows that the composite $\eta \of T\cart'$ in the right-hand side below defines $v \of T\cur J$ as a left Kan extension; consequently the left-hand side, where $\bar J$ denotes the structure cell of $J$, factors uniquely as a vertical cell $\bar{\cur J}$ as shown. Notice that if $\bar J$ satisfies the left Beck"/Chevalley condition, so that it is pointwise left exact by \propref{Beck-Chevalley}, then $\bar{\cur J}$ is invertible.
	  \begin{equation} \label{bar cur J}
	  	\begin{tikzpicture}[textbaseline]
  			\matrix(m)[math35, column sep={1.75em,between origins}]
  				{ TA & & TB \\
  					A & & B \\
  					& \ps A & \\ };
  			\path[map]	(m-1-1) edge[barred] node[above] {$TJ$} (m-1-3)
  													edge node[left] {$a$} (m-2-1)
  									(m-1-3) edge node[right] {$b$} (m-2-3)
  									(m-2-1) edge[barred] node[below] {$J$} (m-2-3)
  													edge[transform canvas={xshift=-1pt}] node[left] {$\yon$} (m-3-2)
  									(m-2-3) edge[ps, transform canvas={xshift=2pt}] node[right] {$\cur J$} (m-3-2);
  			\path				(m-1-2) edge[cell] node[right] {$\bar J$} (m-2-2);
  			\draw[font=\scriptsize]	([yshift=0.25em]$(m-2-2)!0.5!(m-3-2)$) node {$\cart$};
  		\end{tikzpicture} = \begin{tikzpicture}[textbaseline]
  			\matrix(m)[math35, column sep={1.75em,between origins}]
  				{ & TA & & TB \\
  					TA & & T\ps A & \\
  					& A & & \\
  					& & & \ps A \\ };
  			\draw				([shift={(2.25em,0.25em)}]$(m-1-4)!0.5!(m-4-4)$) node (P) {$B$};
  			\path[map]	(m-1-2) edge[barred] node[above] {$TJ$} (m-1-4)
  									(m-1-4) edge[ps] node[below right, inner sep=0.5pt] {$T\cur J$} (m-2-3)
  													edge[bend left=18, transform canvas={xshift=1pt}] node[above right] {$b$} (P)
  									(m-2-1) edge[barred] node[below] {$T\yon_*$} (m-2-3)
  													edge node[left] {$a$} (m-3-2)
  									(m-2-3) edge[ps] node[above right, inner sep=1.5pt] {$v$} (m-4-4)
  									(m-3-2) edge node[below left] {$\yon$} (m-4-4)
  									(P) edge[ps, bend left=18, transform canvas={xshift=1pt}] node[below right] {$\cur J$} (m-4-4);
  			\path				(m-1-2) edge[eq, transform canvas={xshift=-1pt}] (m-2-1)
  									(m-2-2) edge[cell, transform canvas={shift={(1.125em,-1.125em)}}] node[right] {$\eta$} (m-3-2)
  									(m-2-4) edge[cell, transform canvas={xshift=0.25em}] node[right] {$\bar{\cur J}$} (m-3-4);
  			\draw[font=\scriptsize]	($(m-1-2)!0.5!(m-2-3)$) node {$T\!\cart'$};
  		\end{tikzpicture}
	  \end{equation}
	  We shall prove that $\bar{\cur J}$ makes $\cur J$ into a lax $T$-morphism. Combined with $\bar \yon = \inv\gamma$, with $\gamma$ defined in \eqref{structure cell for yoneda embedding}, the identity above then implies the $T$"/cell axiom for the cartesian cell defining $\cur J$ so that, because the forgetful functors from $\wAlg\clx\lax T$ and $\lbcwAlg\clx T$ into $\K$ reflect such cartesian cells, as we have already seen, it forms a cartesian $T$"/cell in both $\wAlg\clx\lax T$ and $\lbcwAlg\clx T$ as well.
	  
	  We return to proving that $(\cur J, \bar{\cur J})$ forms a lax $T$-morphism; that is it satisfies the associativity and unit axioms. That the two sides of the associativity axiom coincide after composition on the left with the composite $\eta \of \mu_{\yon_*} \of T^2\cart'$ is shown below, where the identities follow from the identity \eqref{bar v}; the $T$-image of the identity above factorised through the cartesian cell defining $\yon_*$; the identity above; the associativity axiom for $J$ (\defref{horizontal T-morphism}); the identity above; the naturality of $\mu$. The associativity axiom itself now follows from the fact that $\eta \of \mu_{\yon_*} \of T^2\cart' = \eta \of T\cart' \of \mu_J$ defines a left Kan extension, as $\mu_J$ is assumed to be pointwise left $(\yon \of a)$-exact.
	  \begin{align*}

  	\end{align*}
  	The unit axiom for $\bar{\cur J}$ follows likewise from the following equation, where $\eta \of \iota_{\yon_*} \of \cart'$ defines a left Kan extension by the assumption on $\iota_{\yon_*}$, and whose identities follow from the naturality $\iota$; the identities \eqref{bar cur J}; the unit axiom for $\bar J$; the identity \eqref{tilde v}. This concludes the proof of the yoneda axiom for $(\yon, \bar \yon)$.
  	\begin{displaymath}
  		%
  	\end{displaymath}
  	
  	It remains to show that $\map{(\yon, \bar \yon)} A{\ps A}$ is dense in $\wAlg\clx\lax T$ and $\lbcwAlg\clx T$; see \lemref{density axioms}. Hence consider a cartesian $T$-cell as on the left below; we have to show that it defines the lax $T$-morphism $\map lB{\ps A}$ as a left Kan extension of $\yon$ along $J$. As before it will be useful to consider the factorisation $\zeta'$ of $\zeta$ through the cartesian cell defining $\yon_*$, as shown.
  	\begin{equation} \label{zeta}
  		%
  	\end{equation}
  	First remember that have already shown the forgetful functors $\map U{\wAlg\clx\lax T}\K$ and $\map U{\lbcwAlg\clx T}\K$ to create restrictions of the form $\ps A(\yon, g)$; it follows that they preserve $\zeta$, so that it is cartesian in $\K$. By the density of $\yon$ in $\K$ it follows that any $T$-cell $\phi$, as on the left below, factors uniquely through $\zeta$ in $\K$ as shown; we have to prove that the factorisation $\phi'$ forms a $T$-cell.
  	\begin{displaymath}
  		%
  	\end{displaymath}
  	That it does follows from the equation below, whose identities follow from the identity $\gamma \hc (v \of T\zeta) = \eta \of T\zeta'$ (a consequence of the definitions of $\gamma$ and $\zeta'$, as well as the horizontal companion identity for $\yon_*$); the $T$-cell axiom for $\zeta$; the factorisation above; the $T$-cell axiom for $\phi$; the $T$-image of the factorisation above; the identity $\gamma \hc (v \of T\zeta) = \eta \of T\zeta'$ again.
  	\begin{align*}
  		%
  	\end{align*}
  	Finally, notice that $\zeta'$ is cartesian in $\K$ because $\zeta$ is, by the pasting lemma; hence $T$ preserves $\zeta'$ by assumption. Because $\eta$ defines $v$ as a pointwise left Kan extension it follows that the first column $\eta \of T\zeta'$, in the left-hand and right-hand sides above, defines $v \of Tl$ as a pointwise left Kan extension. Since these sides coincide with $T$-cell axiom for $\phi'$ after it has been composed with $\eta \of T\zeta'$ on the left, we conclude that the $T$-cell axiom itself holds. This completes the proof of the main assertion of the theorem.
  	
  	\emph{Step 4: the assertions \textup{(a)} and \textup{(b)}.} To prove (a) first notice that $\tilde a$ is invertible if and only if $\yon \of \tilde a$ is, where we use that $\yon$ is full and faithful (\lemref{yoneda embedding full and faithful}) so that $\id_{\yon}$ is cartesian. Now consider the unique factorisation $(\yon \of \tilde a) \hc \cart = (\eta \of \iota_{\yon_*}) \hc \tilde v$, see \eqref{tilde v} above, which defines $\tilde v$; the cartesian cell here defines $\yon_*$. Both the first column $\yon \of \tilde a$ in the right-hand side as well as the cartesian cell define left Kan extensions, the latter because of the density of $\yon$ (\lemref{density axioms}). That $\tilde v$ being invertible is implied by $\yon \of \tilde a$ being so now follows from uniqueness of left Kan extensions. Conversely, if $\tilde v$ is invertible then both sides above define left Kan extensions so that, composing the left-hand side with the weakly cocartesian cell defining $\yon_*$, we find that $\yon \of \tilde a$ is invertible by using \propref{pointwise left Kan extension along full and faithful map}.
  	
  	To prove (b) assume that $\bar a$ is invertible. Composing the unique factorisation \eqref{bar v}, that defines $\bar v$, with $\inv{\bar a}$ gives $\eta \of T\eta' = (\yon \of \inv{\bar a}) \hc (\eta \of \mu_{\yon_*}) \hc \bar v$, so that $\bar v$ is invertible if and only if $\eta \of T\eta'$ defines a pointwise left Kan extension. The defining identity \eqref{eta} of $\eta'$ implies that we can rewrite $\eta' = \cart' \of \cocart \of (\cocart, \id_{T\yon_*})$, where the cocartesian cells define $(a^* \hc T\yon_*)$ and $a^*$ respectively and where $\cart'$ is the factorisation $(a^* \hc T\yon_*) \Rar \yon_*$ of the cartesian cell defining $v$ through $\yon_*$, which is again cartesian by the pasting lemma (\lemref{pasting lemma for cartesian cells}). Composing $\eta \of T\eta'$ on the left with the $T$-image of the weakly cocartesian cell that defines $a_*$, we conclude that it defines a pointwise left Kan extension if and only if $\eta \of T\cart' \of T\cocart$ does; here we use \propref{Kan extension along conjoints} and the horizontal conjoint identity for $a^*$ (horizontally dual to that in \lemref{companion identities lemma}). In the latter composite $\eta \of T\cart'$ defines the pointwise left Kan extension of $\yon \of a$ along $(a^* \hc T\yon_*)$, because $\eta$ defines $v$ as a pointwise left Kan extension while $T$ preserves the cartesian cell $\cart'$ by assumption. We conclude that $\eta \of T\cart' \of T\cocart$ defining a pointwise left Kan extension is the same as the $T$-image of the cocartesian defining $(a^* \hc T\yon_*)$ being left $(\yon \of a)$-exact; since we have shown that the former is equivalent to $\bar v$ being invertible, the proof of (b) follows.
  	
  	That, in the case that $A$ and $\ps A$ are pseudo $T$"/algebras, $(\yon, \bar\yon)$ forms a yoneda embedding in $\wAlg\psd\lax T$ and $\lbcwAlg\psd T$ as well follows immediately from the fact that the inclusions $\wAlg\psd\lax T \to \wAlg\clx\lax T$ and $\lbcwAlg\psd T \to \lbcwAlg\lax T$ are locally full (see \defref{full and faithful functor}), so that they create both cartesian cells as well as cells defining left Kan extensions. This completes the proof.
	\end{proof}
	
	Next is the case of lax algebras.
	\begin{theorem} \label{lifting lax presheaf objects}
		Let $T = (T, \mu, \iota)$ be a monad on an augmented virtual double category $\K$. Consider a lax $T$-algebra $A = (A, a, \bar a, \tilde a)$ and a good yoneda embedding $\map\yon A{\ps A}$ in $\K$ and assume that
  	\begin{enumerate}[label=-]
  		\item the conjoint $a^*$ exists;
  		\item the right pointwise composite $(a^* \hc Ty_*)$ exists (\defref{pointwise cocartesian path}) and is preserved by both $T$ and $T^2$;
  		\item both $T$ and $T^2$ preserve any unary cartesian cell with $\yon_*$ as horizontal target.
  	\end{enumerate}
  	The morphism $\map{v \dfn \cur{(a^* \hc T\yon_*)}}{T\ps A}{\ps A}$, that is given by the yoneda axiom (\defref{yoneda embedding}), extends to a lax $T$-algebra structure $(v, \bar v, \tilde v)$ on $\ps A$. With respect to this structure $\map\yon A{\ps A}$ admits a pseudo $T$-morphism structure that makes it into a good yoneda embedding both in $\wAlg\lax\lax T$ and in $\lbcwAlg\lax T$.
  	
  	The following hold:
  	\begin{enumerate}[label=\textup{(\alph*)}]
  		\item	if $A$ is normal then $\ps A$ is normal precisely if $\iota_{\yon_*}$ is pointwise left $(\yon \of a)$-exact;
  		\item assuming that $\bar a$ is invertible, the associator $\bar v$ is invertible precisely if $\mu_{\yon_*}$ is pointwise left $(\yon \of a)$-exact.
  	\end{enumerate}
  \end{theorem}
  \begin{proof}
  	The proof coincides for a large part with that for the colax case; we briefly describe the differences. In order to define the associator $\cell{\bar v}{v \of Tv}{v \of \mu_{\ps A}}$ and the unitor $\cell{\tilde v}{\id_{\ps A}}{v \of \iota_{\ps A}}$ we again consider the cells $\eta$ and $\eta'$ that were defined in \eqref{eta}, and remember that $\eta$ defines $v$ as the pointwise left Kan extension. The assumptions on $T$ ensure that the composite $\eta \of T\eta'$, in the right"/hand side below, defines $v \of Tv$ as a pointwise left Kan extension, as we have seen in the proof of assertion (b) of \thmref{lifting colax presheaf objects}; that the composite $\eta \of T\eta' \of T^2\eta'$ defines a pointwise left Kan extension as well follows analogously from the assumptions on $T^2$. We can thus obtain the associator $\bar v$ as the unique factorisation in
  	\begin{equation} \label{lax bar v}
\mspace{-3mu};
  	\end{equation}
  	likewise using the density of $\yon$ (\lemref{density axioms}) we can obtain the unitor $\tilde v$ as the unique factorisation below. Notice that, once we have shown that $(\ps A, v, \bar v, \tilde v)$ forms a lax $T$"/algebra, the assertions (a) and (b) of the theorem follow easily from the identities defining $\bar v$ and $\tilde v$, by composing them on the left with $\inv{\bar a}$ or $\inv{\tilde a}$ respectively and the uniqueness of left Kan extensions.
  	\begin{equation} \label{lax tilde v}
  		%
  	\end{equation}
  	Since the composite $\eta \of T\eta' \of T^2\eta'$ defines $v \of Tv \of T^2v$ as a left Kan extension, the associativity axiom for $\bar v$ is equivalent to the equality of its two sides both composed with this composite, as shown schematically below. That this equality holds follows from the $T$"/image of \eqref{lax bar v} factorised through $T\yon_*$, \eqref{lax bar v} itself, the associativity and naturality of $\mu$, and the associativity axiom for $\bar a$.
  	\begin{displaymath}
  		%
  	\end{displaymath}
  	Likewise the unit axioms for $\tilde v$ are equivalent to the equalities of both pairs of sides after composing them with $\eta = \cart \of \eta'$, as shown below. That the first of these equalities holds follows from \eqref{lax tilde v}, the naturality and unitality of $\iota$, \eqref{lax bar v}, the corresponding unit axiom for $\tilde a$. For the second we use the $T$"/image of \eqref{lax tilde v} factorised through $T\yon_*$, \eqref{lax bar v}, the unitality of $\iota$ and the corresponding unix axiom for $\tilde a$.
  	\begin{displaymath}
  		%
 \mspace{12mu} = \eta
  	\end{displaymath}
  	
  	This completes the definition of the lax $T$"/algebra $\ps A = (\ps A, v, \bar v, \tilde v)$. Just like in the colax case the invertible vertical cell $\cell{\bar\yon \coloneqq \inv\gamma}{v \of T\yon}{\yon \of a}$, with $\gamma$ as defined in \eqref{structure cell for yoneda embedding}, makes $\map\yon A{\ps A}$ into a pseudo $T$"/morphism: the coherence axioms for $\gamma$ are obtained by precomposing \eqref{lax tilde v} and \eqref{lax bar v} above with the cocartesian cell defining $\yon_*$ and its $T^2$"/image respectively. As before all restrictions of the form $\ps A(\yon, g)$ in $\wAlg\lax\lax T$ are created by the forgetful functor $\map U{\wAlg\lax\lax T}\K$. Also the companion of $(\yon, \bar\yon)$ again satisfies the left Beck"/Chevalley condition so that by \propref{creating restrictions satisfying the left Beck-Chevalley condition} all restrictions $\ps A(\yon, g)$, now with $g$ a pseudo $T$"/morphism, are created by $\map U{\lbcwAlg\lax T}\K$ as well.
  	
  	In order to show that $(\yon, \bar \yon)$ satisfies the yoneda axiom both in $\wAlg\lax\lax T$ and $\lbcwAlg\lax T$, let $\hmap JAB$ be any horizontal $T$"/morphism. Consider the morphism $\map{\cur J}B{\ps A}$ given by the yoneda axiom for $\yon$ in $\K$, that is defined by the cartesian cell in \eqref{cart'}; we claim that the vertical cell $\cell{\bar{\cur J}}{v \of T\cur J}{\cur J \of b}$, defined by the unique factorisation \eqref{bar cur J}, in the lax case too forms a lax $T$"/morphism structure for $\cur J$. Indeed its associativity and unit axioms follow from the equalities below, which form the composition of either axiom with the composite $\eta \of T\eta' \of T^2\cart'$ and the cartesian cell defining $\yon_*$ respectively, because both define left Kan extensions; here we use that the cartesian cell $\cart'$, defined in \eqref{cart'}, is preserved by $T^2$. The first equality here follows from the $T$"/image of \eqref{bar cur J} factorised through $T\yon_*$; \eqref{bar cur J} itself; the associativity axiom for the structure cell $\bar J$ of $J$; the naturality of $\mu$; \eqref{lax bar v}, while the second one follows from \eqref{lax tilde v}; the naturality of $\iota$; \eqref{bar cur J} again; the unit axiom for $\bar J$.
  	\begin{displaymath}
  		\begin{tikzpicture}[scheme, yshift=-0.8em]
  			\draw	(1,4) -- (1,0) -- (0,0) -- (0,4) -- (2,4) -- (2,0) -- (3,0) -- (3,4) -- (4,4) -- (4,2) -- (3,2) (0,2) -- (2,2) (0,3) -- (1,3) (2,3) -- (3,3);
  			\draw	(0.5,1) node {$\eta$}
  						(0.5,2.5) node {$T\eta'$}
  						(0.5,3.5) node {$T^2\mspace{-2mu}\textup c'$}
  						(1.5,3) node {$T\bar{\cur J}$}
  						(2.5,1.5) node {$\bar{\cur J}$}
  						(3.5,3) node {$\bar b$};
  		\end{tikzpicture} \mspace{12mu} = \mspace{12mu} \begin{tikzpicture}[scheme, yshift=-0.8em]
  			\draw	(1,0) -- (1,4) -- (0,4) -- (0,0) -- (2,0) -- (2,4) -- (3,4) -- (3,0) -- (4,0) -- (4,3) -- (3,3) (0,2) -- (1,2) (2,2) -- (3,2) (0,3) -- (2,3);
  			\draw	(0.5,1) node {$\eta$}
  						(0.5,2.5) node {$T\eta'$}
  						(0.5,3.5) node {$T^2\mspace{-2mu}\textup c'$}
  						(1.5,1.5) node {$\bar v$}
  						(3.5,1.5) node {$\bar{\cur J}$};
  		\end{tikzpicture} \qquad\qquad \begin{tikzpicture}[scheme]
  			\draw	(2,2) -- (0,2) -- (0,3) -- (1,3) -- (1,0) -- (2,0) -- (2,3) -- (3,3) -- (3,0) -- (4,0) -- (4,2) -- (3,2) (2,1) -- (3,1);
  			\draw	(0.5,2.5) node {c}
  						(1.5,1) node {$\tilde v$}
  						(3.5,1) node {$\bar{\cur J}$};
  		\end{tikzpicture} \mspace{12mu} = \mspace{12mu} \begin{tikzpicture}[scheme]
  			\draw	(1,3) -- (1,2) -- (0,2) -- (0,3) -- (3,3) -- (3,1) -- (2,1) -- (2,3) (1,2) -- (1,0) -- (2,0) -- (2,1);
  			\draw	(0.5,2.5) node {c}
  						(2.5,2) node {$\tilde b$};
  		\end{tikzpicture}
  	\end{displaymath}
  	Finally notice that, as in the colax case, the structure cell $\bar{\cur J}$ is invertible whenever $J$ satisfies the left Beck"/Chevalley condition, as follows from \eqref{bar cur J}. This identity also implies the $T$"/cell axiom for the cartesian cell that defines $\cur J$; that it is again cartesian as a $T$"/cell then follows from the fact that the forgetful functors of $\wAlg\lax\lax T$ and $\lbcwAlg\lax T$ create all restrictions of the form $\ps A(\yon, g)$. This completes the proof of the yoneda axiom for $(\yon, \bar \yon)$. It now remains to prove that $(\yon, \bar\yon)$ is dense both in $\wAlg\lax\lax T$ and $\lbcwAlg\lax T$: here the corresponding argument of the colax case applies verbatim. This completes the proof.
  \end{proof}
	
	In the following proposition we consider the uniqueness of the colax structures $(v, \bar v, \tilde v)$ on presheaf objects $\ps A$, as obtained in the previous theorem.
	\begin{proposition}
		When regarded as a colax $T$-morphism the lifted yoneda embedding $\map{(\yon, \gamma)}{(A, a, \bar a, \tilde a)}{(\ps A, v, \bar v, \tilde v)}$, obtained in \thmref{lifting colax presheaf objects}, is unique as follows. Any colax $T$-morphism $\map{(\yon, \phi)}{(A, a, \bar a, \tilde a)}{(\ps A, w, \bar w, \tilde w)}$ between colax $T$-algebras factors uniquely through $(\yon, \gamma)$ as a colax $T$-morphism $\map{(\id, \phi')}{\ps A}{\ps A}$. If $(\yon, \phi)$ is a pseudo $T$-morphism then $(\id, \phi')$ is too precisely if the composite below defines $w$ as a pointwise left Kan extension of $w \of T\yon$ along $T\yon_*$.
		\begin{displaymath}

		\end{displaymath}
	\end{proposition}
	\begin{proof}
		The invertible colax structure cell $\cell\gamma{\yon \of a}{v \of T\yon}$ for $\yon$ was defined in \eqref{structure cell for yoneda embedding} in terms of the composite cell $\eta$, as in the right-hand side below, which in turn was defined in \eqref{eta}. In the discussion that followed $\eta$ was shown to define a pointwise left Kan extension; we take the colax structure cell $\phi'$ on $\id_{\ps A}$ to be the unique factorisation through $\eta$ as shown below.
		\begin{displaymath}
			%
		\end{displaymath}
		We claim that $\phi'$ makes $\id_{\ps A}$ into a colax $T$-morphism $(\ps A, v, \bar v, \tilde v) \to (\ps A, w, \bar w, \tilde w)$; notice that, by precomposing both sides above with the $T$-image of the weakly cocartesian cell defining $\yon_*$, it follows from the definition \eqref{structure cell for yoneda embedding} of $\gamma$ that $\phi'$ is unique such that $(\id, \phi') \of (\yon, \gamma) = (\yon, \phi)$. Also notice that the final assertion can be read off: if $\phi$ is invertible then, as follows from the identity obtained by composing both sides above with $\inv\phi$ on the left together with uniqueness of Kan extensions, $\phi'$ is invertible precisely if $w \of T\cart$ defines $w$ as a pointwise left Kan extension.
		
		We return to the claim that $\phi'$ forms a colax $T$-structure cell for $\id_{\ps A}$; that is it satisfies the associativity and unit axioms. The former follows from the equation below, where $\eta \of \mu_{\yon_*}$ defines a left Kan extension because $\eta$ does and the left exactness assumption on $\mu_{\yon_*}$ in \thmref{lifting colax presheaf objects}, and where the identities follow from the definition \eqref{bar v} of $\bar v$; the identity above and the definition \eqref{eta} of $\eta'$; the $T$-image of the identity above; the associativity axiom for $\phi$; the naturality of $\mu$; the identity above.
		\begin{align*}

		\end{align*}
		Similarly the unit axiom follows from the following equation, where $\eta \of \iota_{\yon_*}$ defines a left Kan extension by the exactness assumption on $\iota_{\yon_*}$, and where the identities follow from the identity defining $\phi'$ above; the naturality of $\iota$; the unit axiom for $\phi$; the definition \eqref{tilde v} of $\tilde v$.
		\begin{displaymath}
			%
		\end{displaymath}
		This completes the proof.
	\end{proof}
	
	By combining \thmref{presheaf objects as free cocompletions}, \propref{lifted cocomplete algebras} and the previous proposition, the theorem below describes the sense in which a lifted algebraic yoneda embedding, as obtained in \thmref{lifting colax presheaf objects}, defines a free cocompletion.
	\begin{theorem}
		Let $T = (T, \mu, \iota)$ be a monad on an augmented virtual equipment $\K$ whose underlying endofunctor preserves all unary restrictions. Given a colax $T$-algebra $M = (M, m, \bar m, \tilde m)$ and a good yoneda embedding $\map\yon M{\ps M}$ in $\K$ assume that the hypotheses of \thmref{lifting colax presheaf objects} are satisfied, so that $\yon$ lifts to a yoneda embedding \mbox{$\map{(\yon, \bar\yon)}{(M, m, \bar m, \tilde m)}{(\ps M, v, \bar v, \tilde v)}$} in $\wAlg\clx\lax T$ and $\lbcwAlg\clx T$. Let $\catvar S$ be an ideal of left extension diagrams in $\K$ such that
		\begin{enumerate}
			\item[\textup{(e)}] for each $(d, J) \in \catvar S(\ps M)$ there exists a pointwise left $\yon$-exact cell
				\begin{displaymath}
					\begin{tikzpicture}
						\matrix(m)[math35]{ M & A & B \\ M & & B \\ };
						\path[map]	(m-1-1) edge[barred] node[above] {$\ps M(\yon, d)$} (m-1-2)
												(m-1-2) edge[barred] node[above] {$J$} (m-1-3)
												(m-2-1) edge[barred] node[below] {$K$} (m-2-3);
						\path				(m-1-1) edge[eq] (m-2-1)
												(m-1-3) edge[eq] (m-2-3)
												(m-1-2) edge[cell] node[right] {$\phi$} (m-2-2);
					\end{tikzpicture}
				\end{displaymath}
				whose $T$-image is pointwise left $(\yon \of m)$-exact;
			\item[\textup{(m)}]	for each $(d, J) \in \catvar S(\ps M)$ both cells $\mu_J$ and $\iota_J$ are pointwise left $(m \of Td)$-exact;
			\item[\textup{(y)}] $(f, \yon_*) \in \catvar S$ for all $\map fMN$.
		\end{enumerate}
		The yoneda embedding $(\yon, \bar\yon)$ defines $\ps M$ both as the free $\edpi{\catvar S}\clx\lax$-cocompletion of $M$ in $\wAlg\clx\lax T$ and as the free $\edpilbc{\catvar S}\clx$-cocompletion of $M$ in $\lbcwAlg\clx T$. Moreover $\ps M$ is $\catvar S$-cocomplete in $\K$ and, for each $(d, J) \in \edpi{\catvar S}\clx\lax(\ps M)$ (resp.\ $\edpilbc{\catvar S}\clx(\ps M)$), the pointwise left Kan extension of $d$ along $J$ is created by $\map U{\wAlg\clx\lax T}\K$ (resp.\ $\map U{\lbcwAlg\clx T}\K$).
		
		Finally let $N$ be a colax $T$-algebra that satisfies the hypotheses of \propref{lifted cocomplete algebras}, so that its $\edpi{\catvar S}\clx\lax$-cocompleteness is lifted along $\map U{\wAlg\clx\lax T}\K$. A lax $T$-morphism $\map{(f, \bar f)}{\ps M}N$ is $\edpi{\catvar S}\clx\lax$-cocontinuous if and only if its underlying morphism $f$ is $\catvar S$"/cocontinuous. The analogous result for $\edpilbc{\catvar S}\clx$-cocontinuity of pseudo $T$"/morphisms $\ps M \to N$ in $\lbcwAlg\clx T$ holds as well.
	\end{theorem}
	\begin{proof}
		We first prove the final assertion. The assumptions (e) and (y) above ensure that we may apply \thmref{presheaf objects as free cocompletions} to find that $\ps M$ is $\catvar S$-cocomplete in $\K$. To prove that the forgetful functors $\map U{\wAlg\clx\lax T}\K$ and $\map U{\lbcwAlg\clx T}\K$ create all algebraic pointwise left Kan extensions of $T$-morphisms $d$ along a horizontal $T$"/morphisms $J$, with $(d, J)$ ranging over $\edpi{\catvar S}\clx\lax(\ps M)$ (resp.\ $\edpilbc{\catvar S}\clx(\ps M)$), it suffices to show that $\ps M = (\ps M, v, \bar v, \tilde v)$ satisfies the hypotheses of \propref{lifted cocomplete algebras}(a). Besides $\catvar S$-cocompleteness of $\ps M$ and assumption (m) above, the hypothesis that remains to be checked asks that the algebraic structure of $\ps M$ preserves the pointwise left Kan extension of $d$ along $J$, for each $(d, J) \in \catvar S(\ps M)$. From the proof of \propref{cocompleteness of presheaf objects} we know how to construct this Kan extension: it is defined by the cell $\eta$ that is the factorisation defined in \eqref{pointwise left Kan extension into presheaf object}, and we have to show that $v \of T\eta$ again defines a pointwise left Kan extension.
		\begin{equation} \label{pointwise left Kan extension into colax presheaf object}

		\end{equation}
		
		To see this we first factor either side of \eqref{pointwise left Kan extension into presheaf object} through $\yon_*$, as shown on the left above. Here $\phi$ is the pointwise left $\yon$-exact cell whose existence is assumed by condition~(e) above, while $\cart'$ and $\cart''$ are the factorisations of the cartesian cells in \eqref{pointwise left Kan extension into presheaf object} through $\yon_*$; they are again cartesian by the pasting lemma. Notice that $T\cart'$ is again cartesian by the hypotheses of \thmref{lifting colax presheaf objects}, so that it is pointwise left $(y \of m)$-exact by \lemref{properties of pointwise left Kan extensions}(b). It follows that the $T$-image of the composite $\cart' \of \phi$ in the left-hand side is pointwise left $(v \of T\yon)$-exact, by assumption (e) and the fact that $\yon \of m \iso v \of T\yon$; a consequence of $(\yon, \bar\yon)$ being a pseudo $T$-morphism.		
		
		Now consider the composite on the right above, whose top row is the $T$-image of the right-hand side of the identity on the left. Applying the previous proposition to the identity on $(\ps M, v, \bar v, \tilde v)$ we find that the composite $v \of T\cart$ here defines $v$ as a pointwise left Kan extension. As we have seen the $T$-image of both factorisations through $\yon_*$, on either side of the identity on the left, is pointwise left $(v \of T\yon)$-exact, so that it follows that the full composite on the right defines $v \of Tl$ as a pointwise left Kan extension. Finally, since $T$ preserves the cartesian cell $\cart'$, its first column defines $v \of Td$ as a pointwise left Kan extension, so that by the horizontal pasting lemma (\lemref{horizontal pasting lemma}) we may conclude that its second column, that is $v \of T\eta$, defines a $v \of Tl$ as a pointwise left Kan extension, as required. This completes the proof of the final assertion of the theorem.
		
		To prove the first assertion, that is $(\yon, \bar\yon)$ defines $\ps M = (\ps M, v, \bar v, \tilde v)$ as the free $\edpi{\catvar S}\clx\lax$-cocompletion of $M$ in $\wAlg\clx\lax T$, it suffices to show that we can apply \thmref{presheaf objects as free cocompletions} to the lifted yoneda embedding $(\yon, \bar \yon)$ in $\wAlg\clx\lax T$. That the second hypothesis of the latter theorem is satisfied, that is $\bigpars{(f, \bar f), (\yon, \bar \yon)} \in \edpi{\catvar S}\clx\lax$ for each lax $T$-morphism $\map{(f, \bar f)}MN$, follows immediately from the assumption (y) above. It remains to prove its first hypothesis, which asks for a pointwise left $(y, \bar y)$-exact horizontal $T$-cell $\bigpars{\ps M(\yon, d), J} \Rar K$ to exist for each $\bigpars{(d, \bar d), (J, \bar J)} \in \edpi{\catvar S}\clx\lax(\ps M)$. We will show that the horizontal cell $\phi$, that is supplied by the assumption (e) above, lifts to form such a $T$-cell.
		
		To see this first notice that $\bigpars{(d, \bar d), (J, \bar J)} \in \edpi{\catvar S}\clx\lax(\ps M)$ implies $(d, J) \in \catvar S(\ps M)$, so that the pointwise left Kan extension $\map lB{\ps M}$ of $d$ along $J$ exists in $\K$: indeed, it is defined by the cell $\eta$ in the left-hand side of the identity on the left above. We already know that the pointwise left Kan extension of $(d, \bar d)$ along $(J, \bar J)$ is created by $\map U{\wAlg\clx\lax T}\K$, that is there exists a lax $T$-structure $\bar l$ on $l$ that is unique in making $\eta$ a $T$-cell that defines $(l, \bar l)$ as this Kan extension. Likewise the nullary restriction $K$ of $\yon$ along $l$ in $\wAlg\clx\lax T$, as in the left-hand side of the identity above, is created by $U$; see \propref{creating restrictions}. We claim that the thus created horizontal $T$-structure cell $\bar K$ on $K$ makes $\phi$ into a $T$-cell. Indeed the equation below shows that its $T$-cell axiom holds after composition with the cartesian cell that defines $K$, where `c' has been used to denote both the cartesian cell defining $K$ and that defining $\ps M(\yon, d)$, and where $\gamma \dfn \inv{\bar\yon}$. The identities follow from the definition of $\bar K$ (see the proof of \propref{creating restrictions}); the $T$-image of the identity on the left of \eqref{pointwise left Kan extension into colax presheaf object}; the $T$-cell axiom for $\eta$; the definition of $\overline{\ps M(\yon, d)}$ (again see \propref{creating restrictions}); the identity on the left of \eqref{pointwise left Kan extension into colax presheaf object} itself.
		\begin{align*}
			\begin{tikzpicture}[scheme]
				\draw (0,0) -- (2,0) -- (2,3) -- (0,3) -- (0,0) (0,1) -- (2,1) (0,2) -- (2,2);
				\draw	(1,0.5) node {$\textup c$}
							(1,1.5) node {$\bar K$}
							(1,2.5) node {$T\phi$};
			\end{tikzpicture} \mspace{12mu} &= \mspace{12mu} \begin{tikzpicture}[scheme]
				\draw (1,3) -- (1,0) -- (0,0) -- (0,3) -- (4,3) -- (4,0) -- (3,0) -- (3,3) (1,1) -- (3,1) (1,2) -- (3,2);
				\draw	(0.5,1.5) node {$\gamma$}
							(2,1.5) node {$T\textup c$}
							(2,2.5) node {$T\phi$}
							(3.5,1.5) node {$\bar l$};
			\end{tikzpicture} \mspace{12mu} = \mspace{12mu} \begin{tikzpicture}[scheme, yshift=0.8em]
				\draw	(1,2) -- (1,0) -- (0,0) -- (0,2) -- (4,2) -- (4,0) -- (3,0) -- (3,2) (1,1) -- (3,1) (2,1) -- (2,2);
				\draw	(0.5,1) node {$\gamma$}
							(1.5,1.5) node {$T\textup c$}
							(2.5,1.5) node {$T\eta$}
							(3.5,1) node {$\bar l$};
			\end{tikzpicture} \\
			&= \mspace{12mu} \begin{tikzpicture}[scheme, yshift=0.8em]
				\draw	(1,2) -- (1,0) -- (0,0) -- (0,2) -- (4,2) -- (4,0) -- (2,0) -- (2,2) (1,1) -- (2,1) (3,0) -- (3,2) (3,1) -- (4,1);
				\draw	(0.5,1) node {$\gamma$}
							(1.5,1.5) node {$T\textup c$}
							(2.5,1) node {$\bar d$}
							(3.5,0.5) node {$\eta$}
							(3.5,1.5) node {$\bar J$};
			\end{tikzpicture} \mspace{12mu} = \mspace{12mu} \begin{tikzpicture}[scheme, yshift=0.8em]
				\draw	(0,0) -- (3,0) -- (3,2) -- (0,2) -- (0,0) (0,1) -- (3,1) (2,0) -- (2,2);
				\draw	(1,0.5) node {$\textup c$}
							(1,1.5) node {$\overline{\ps M(\yon, d)}$}
							(2.5,0.5) node {$\eta$}
							(2.5,1.5) node {$\bar J$};
			\end{tikzpicture} \mspace{12mu} = \mspace{12mu} \begin{tikzpicture}[scheme]
				\draw	(0,0) -- (3,0) -- (3,3) -- (0,3) -- (0,0) (0,1) -- (3,1) (0,2) -- (3,2) (2,2) -- (2,3);
				\draw	(1.5,0.5) node {$\textup c$}
							(1.5,1.5) node {$\phi$}
							(1,2.5) node {$\overline{\ps M(\yon, d)}$}
							(2.5,2.5) node {$\bar J$};
			\end{tikzpicture}
		\end{align*}
		Thus a $T$-cell, it remains to prove that $\phi$ is pointwise left $(\yon, \bar\yon)$-exact in $\wAlg\clx\lax T$. To see this we return to the identity on the left of \eqref{pointwise left Kan extension into colax presheaf object} once more, now regarding both its sides as compositions of $T$-cells. By the density of the lifted yoneda embedding $(\yon, \bar\yon)$ (\lemref{density axioms}) both composites of cartesian $T$-cells, in either side, define pointwise left Kan extensions. Because $\eta$ was reflected as a $T$-cell that defines $(l, \bar l)$ as a pointwise left Kan extension, it follows from the horizontal pasting lemma that the full left-hand side does so too. From this we conclude that $\phi$, now as a $T$-cell, is pointwise left $(\yon, \bar \yon)$-exact. This concludes the proof of the first assertion.
		
		Next, if $\bigpars{(d, \bar d), (J, \bar J)} \in \edpilbc{\catvar S}\clx$ then the $T$-cell $\cell\phi{\bigpars{\ps M(\yon, d), J}}K$ considered above is in fact a $T$-cell between horizontal $T$-morphisms satisfying the left Beck-Chevalley condition. Indeed, in that case both $\ps M(\yon, d)$ and $K$ are nullary restrictions of $\ps M$ along $(\yon, \bar\yon)$ and a pseudo $T$-morphism ($d$ and $l$ respectively), and such restrictions satisfy the the left Beck-Chevalley condition in $\lbcwAlg\clx T$: indeed $\yon_*$ does because $(\yon, \bar\yon)$ forms a good yoneda embedding, so that $\ps M(\yon, d)$ and $K$ do by \propref{creating restrictions satisfying the left Beck-Chevalley condition}. We conclude that in this case too the hypotheses of \thmref{presheaf objects as free cocompletions} are satisfied, so that the lift of $\yon$ defines $\ps M$ as the free $\edpilbc{\catvar S}\clx$"/cocompletion of $M$ in $\lbcwAlg\clx T$ as well.
		
		To complete the proof consider a lax $T$-morphism $\map{(f, \bar f)}{\ps M}N$ where $N$ is a colax $T$-algebra satisfying the hypotheses of \propref{lifted cocomplete algebras} so that its $\edpi{\catvar S}\clx\lax$"/cocompleteness lifts along $\map U{\wAlg\clx\lax T}\K$, while $(f, \bar f)$ is $\edpi{\catvar S}\clx\lax$-cocontinuous whenever $f$ is $\catvar S$-cocontinuous. Conversely assume that $(f, \bar f)$ is  $\edpi{\catvar S}\clx\lax$-cocontinuous. Using the density of $\yon$ in $\wAlg\clx\lax T$ (\lemref{density axioms}) it follows that the composite of $f$ and the cartesian $T$-cell defining $\yon_*$ defines $f$ as the pointwise left Kan extension of $f \of \yon$ along $\yon_*$. Because $N$ is assumed to be $\catvar S$-cocomplete in $\K$, the latter pointwise left Kan extension is created, and hence preserved, by $\map U{\wAlg\clx\lax T}\K$; it follows that $f$ forms a pointwise left Kan extension along $\yon_*$ in $\K$ too, so that it is $\catvar S$-cocontinuous by \propref{pointwise left Kan extensions along yoneda embeddings}. That the same argument applies in the case of $\edpilbc{\catvar S}\clx$-cocontinuity too is clear. This concludes the proof.
	\end{proof}
	
	\begin{example}
		Let $\V \to \V'$ be a symmetric universe enlargement such that $\V$ has an initial object preserved by $\tens$ on both sides, let $T$ be the `free strict monoidal $\V'$-category'-monad on $\enProf{(\V, \V')}$ (\exref{free strict monoidal V-category monad}), and let $\catvar S$ be the ideal of left extension diagrams in $\enProf{(\V, \V')}$ described in \exref{V-small}, consisting of pairs $(d, J)$ where $\hmap JAB$ is a $\V$-profunctor between $\V$-categories with $A$ small. It follows from the discussion at the beginning of \exref{Day convolution} that the hypotheses of the previous theorem are satisfied so that, for any monoidal $\V$-category $M = (M, \oslash, \mathfrak a, \mathfrak i)$, the presheaf $\V$-category $\ps M$ equipped with Day convolution (\exref{Day convolution}) forms both the free $\edpi{\catvar S}\clx\lax$-cocompletion of $M$ in $\wAlg\clx\lax T$ as well as the free $\edpilbc{\catvar S}\clx$"/cocompletion of $M$ in $\lbcwAlg\clx T$. In particular for any small $\V$-cocomplete monoidal $\V$-category $N$, whose tensor product $\oslash$ preserves small $\V$-weighted colimits in each variable (see \exref{tensor products preserving weighted colimits in each variable}), we obtain an equivalence of categories of ($\V$-small cocontinuous) lax monoidal functors
		\begin{displaymath}
				\MonCat_\textup{(l, cocts)}(\ps M, N) \simeq \MonCat_\textup{l}(M, N)
		\end{displaymath}
		that is given by precomposition with $\map\yon M{\ps M}$ and that restricts to an equivalence of categories of monoidal functors. This recovers Theorem 5.1 of \cite{Im-Kelly86}.
	\end{example}

  \bibliographystyle{alpha}
  \bibliography{../_other/main}

\begin{thebibliography}{Kou15b}

\bibitem[AR01]{Adamek-Rosicky01}
J.~Ad{\'a}mek and J.~Rosick{\'y}.
\newblock On sifted colimits and generalized varieties.
\newblock {\em Theory and Applications of Categories}, 8:33--53, 2001.

\bibitem[BK75]{Borceux-Kelly75}
F.~Borceux and G.~M. Kelly.
\newblock A notion of limit for enriched categories.
\newblock {\em Bulletin of the Australian Mathematical Society}, 12(1):49--72,
  1975.

\bibitem[Bur71]{Burroni71}
A.~Burroni.
\newblock {$T$}-cat\'egories ({C}at\'egories dans un triple).
\newblock {\em Cahiers de Topologie et G\'eom\'etrie Diff\'erentielle
  Cat\'egoriques}, 12(3):215--321, 1971.

\bibitem[CS10]{Cruttwell-Shulman10}
G.~S.~H. Cruttwell and M.~A. Shulman.
\newblock A unified framework for generalized multicategories.
\newblock {\em Theory and Applications of Categories}, 24(21):580--655, 2010.

\bibitem[Day70]{Day70}
B.~Day.
\newblock On closed categories of functors.
\newblock In {\em Reports of the Midwest Category Seminar IV}, volume 137 of
  {\em Lecture Notes in Mathematics}, pages 1--38. Springer-Verlag, 1970.

\bibitem[DK69]{Day-Kelly69}
B.~J. Day and G.~M. Kelly.
\newblock Enriched functor categories.
\newblock In {\em Reports of the Midwest Category Seminar III}, volume 106 of
  {\em Lecture Notes in Mathematics}, pages 178--191. Springer-Verlag, 1969.

\bibitem[DL07]{Day-Lack07}
B.~J. Day and S.~Lack.
\newblock Limits of small functors.
\newblock {\em Journal of Pure and Applied Algebra}, 210(3):651--663, 2007.

\bibitem[DPP06]{Dawson-Pare-Pronk06}
R.~J.~MacG. Dawson, R.~Par{\'e}, and D.~A. Pronk.
\newblock Paths in double categories.
\newblock {\em Theory and Applications of Categories}, 16(18):460--521, 2006.

\bibitem[DS86]{Day-Street86}
B.~Day and R.~Street.
\newblock Categories in which all strong generators are dense.
\newblock {\em Journal of Pure and Applied Algebra}, 43:235--242, 1986.

\bibitem[Dub70]{Dubuc70}
E.~J. Dubuc.
\newblock {\em Kan Extensions in Enriched Category Theory}, volume 145 of {\em
  Lecture Notes in Mathematics}.
\newblock Springer-Verlag, 1970.

\bibitem[Ehr63]{Ehresmann63}
C.~Ehresmann.
\newblock Cat\'egories structur\'ees: {III}. {Q}uintettes et applications
  covariantes.
\newblock In {\em Topologie et G\'eom\'etrie Diff\'erentielle: S\'eminaire
  Ehresmann}, volume~5, pages 1--21. Institut H. Poincar\'e, 1963.

\bibitem[FS95]{Freyd-Street95}
P.~Freyd and R.~Street.
\newblock On the size of categories.
\newblock {\em Theory and Applications of Categories}, 1(9):174--178, 1995.

\bibitem[Get09]{Getzler09}
E.~Getzler.
\newblock Operads revisited.
\newblock In {\em Algebra, Arithmetic, and Geometry: In Honor of {Y}u. {I}.
  {M}anin. {V}olume {I}}, volume 269 of {\em Progress in Mathematics}, pages
  675--698. Birkh\"auser, 2009.

\bibitem[GP99]{Grandis-Pare99}
M.~Grandis and R.~Par{\'e}.
\newblock Limits in double categories.
\newblock {\em Cahiers de Topologie et G\'eom\'etrie Diff\'erentielle
  Cat\'egoriques}, 40(3):162--220, 1999.

\bibitem[GP04]{Grandis-Pare04}
M.~Grandis and R.~Par{\'e}.
\newblock Adjoint for double categories.
\newblock {\em Cahiers de Topologie et G\'eom\'etrie Diff\'erentielle
  Cat\'egoriques}, 45(3):193--240, 2004.

\bibitem[GP07]{Grandis-Pare07}
M.~Grandis and R.~Par{\'e}.
\newblock Lax {K}an extensions for double categories ({O}n weak double
  categories, {P}art {IV}).
\newblock {\em Cahiers de Topologie et G\'eom\'etrie Diff\'erentielle
  Cat\'egoriques}, 48(3):163--199, 2007.

\bibitem[GP08]{Grandis-Pare08}
M.~Grandis and R.~Par{\'e}.
\newblock Kan extensions in double categories ({O}n weak double categories,
  {P}art {III}).
\newblock {\em Theory and Applications of Categories}, 20(8):152--185, 2008.

\bibitem[Gui80]{Guitart80}
R.~Guitart.
\newblock Relations et carr\'es exacts.
\newblock {\em Les Annales des Sciences Math\'ematiques du Qu\'ebec},
  4(2):103--125, 1980.

\bibitem[HST14]{Hofmann-Seal-Tholen14}
D.~Hofmann, G.~J. Seal, and W.~Tholen, editors.
\newblock {\em Monoidal Topology: A Categorical Approach to Order, Metric, and
  Topology}, volume 153 of {\em Encyclopedia of Mathematics and Its
  Applications}.
\newblock Cambridge University Press, 2014.

\bibitem[IK86]{Im-Kelly86}
G.~B. Im and G.~M. Kelly.
\newblock A universal property of the convolution monoidal structure.
\newblock {\em Journal of Pure and Applied Algebra}, 43(1):75--88, 1986.

\bibitem[Kel74]{Kelly74}
G.~M. Kelly.
\newblock Doctrinal adjunction.
\newblock In {\em Proceedings Sydney Category Theory Seminar 1972/1973}, volume
  420 of {\em Lecture Notes in Mathematics}, pages 257--280. Springer"/Verlag,
  1974.

\bibitem[Kel82]{Kelly82}
G.~M. Kelly.
\newblock {\em Basic Concepts of Enriched Category Theory}, volume~64 of {\em
  London Mathematical Society Lecture Note Series}.
\newblock Cambridge University Press, 1982.

\bibitem[Kel86]{Kelly86}
G.~M. Kelly.
\newblock A survey of totality for enriched and ordinary categories.
\newblock {\em Cahiers de Topologie et G\'eom\'etrie Diff\'erentielle
  Cat\'egoriques}, 27(2):109--132, 1986.

\bibitem[KL97]{Kelly-Lack97}
G.~M. Kelly and S.~Lack.
\newblock On property-like structures.
\newblock {\em Theory and Applications of Categories}, 3(9):213--250, 1997.

\bibitem[Kou13]{Koudenburg13}
S.~R. Koudenburg.
\newblock {\em Algebraic weighted colimits}.
\newblock PhD thesis, University of Sheffield, 2013.
\newblock Available as
  \href{http://arxiv.org/abs/1304.4079}{\texttt{arXiv:1304.4079}}.

\bibitem[Kou14]{Koudenburg14a}
S.~R. Koudenburg.
\newblock On pointwise {K}an extensions in double categories.
\newblock {\em Theory and Applications of Categories}, 29(27):781--818, 2014.

\bibitem[Kou15a]{Koudenburg15a}
S.~R. Koudenburg.
\newblock Algebraic {K}an extensions in double categories.
\newblock {\em Theory and Applications of Categories}, 30(5):86--146, 2015.

\bibitem[Kou15b]{Koudenburg15b}
S.~R. Koudenburg.
\newblock A double-dimensional approach to formal category theory.
\newblock Draft, available as
  \href{http://arxiv.org/abs/1511.04070}{\texttt{arXiv:1511.04070}}, November
  2015.

\bibitem[Kou22]{Koudenburg22}
S.~R. Koudenburg.
\newblock Formal category theory in augmented virtual double categories.
\newblock Available as
  \href{https://arxiv.org/abs/2205.04890}{\texttt{arxiv:2205.04890}}, 2022.

\bibitem[Law73]{Lawvere73}
F.~W. Lawvere.
\newblock Metric spaces, generalized logic, and closed categories.
\newblock {\em Rendiconti del Seminario Mat\'ematico e Fisico di Milano},
  43:135--166, 1973.

\bibitem[Lei04]{Leinster04}
T.~Leinster.
\newblock {\em Higher Operads, Higher Categories}, volume 298 of {\em London
  Mathematical Society Lecture Note Series}.
\newblock Cambridge University Press, 2004.

\bibitem[ML98]{MacLane98}
S.~Mac~Lane.
\newblock {\em Categories for the Working Mathematician}, volume~5 of {\em
  Graduate Texts in Mathematics}.
\newblock Springer, second edition, 1998.

\bibitem[MT08]{Mellies-Tabareau08}
P.-A. Melli{\`e}s and N.~Tabareau.
\newblock Free models of {$T$}-algebraic theories computed as {K}an extensions.
\newblock Preprint, 2008.

\bibitem[Shu08]{Shulman08}
M.~Shulman.
\newblock Framed bicategories and monoidal fibrations.
\newblock {\em Theory and Applications of Categories}, 20(18):650--738, 2008.

\bibitem[Shu11]{Shulman11}
M.~Shulman.
\newblock Comparing composites of left and right derived functors.
\newblock {\em New York Journal of Mathematics}, 17:75--125, 2011.

\bibitem[Str72]{Street72}
R.~Street.
\newblock The formal theory of monads.
\newblock {\em Journal of Pure and Applied Algebra}, 2(2):149--168, 1972.

\bibitem[Str74]{Street74b}
R.~Street.
\newblock Fibrations and {Y}oneda's lemma in a {$2$}-category.
\newblock In {\em Proceedings of the Sydney Category Theory Seminar 1972/1973},
  volume 420 of {\em Lecture Notes in Mathematics}, pages 104--133.
  Springer"/Verlag, 1974.

\bibitem[SW78]{Street-Walters78}
R.~Street and R.~Walters.
\newblock Yoneda structures on 2-categories.
\newblock {\em Journal of Algebra}, 50(2):350--379, 1978.

\bibitem[Szl12]{Szlachanyi12}
K.~Szlach\'anyi.
\newblock Skew-monoidal categories and bialgebroids.
\newblock {\em Advances in Mathematics}, 231(3--4):1694--1730, 2012.

\bibitem[Wal18]{Walker18}
C.~Walker.
\newblock Yoneda structures and {KZ} doctrines.
\newblock {\em Journal of Pure and Applied Algebra}, 222(6):1375--1387, 2018.

\bibitem[Web07]{Weber07}
M.~Weber.
\newblock Yoneda structures from 2-toposes.
\newblock {\em Applied Categorical Structures}, 15(3):259--323, 2007.

\bibitem[Web15]{Weber15}
M.~Weber.
\newblock Internal algebra classifiers as codescent objects of crossed internal
  categories.
\newblock {\em Theory and Applications of Categories}, 30(50):1713--1792, 2015.

\bibitem[Wil15]{Willerton15}
S.~Willerton.
\newblock The {L}egendre-{F}enchel transform from a category theoretic
  perspective.
\newblock Preprint, 2015.

\bibitem[Woo82]{Wood82}
R.~J. Wood.
\newblock Abstract proarrows {I}.
\newblock {\em Cahiers de Topologie et G\'eom\'etrie Diff\'erentielle},
  23(3):279--290, 1982.

\end{thebibliography}
\end{document}